\documentclass[11pt,reqno,a4paper]{amsbook}

\usepackage{cite}
\usepackage{enumitem}
\usepackage{tikz}
\usepackage{tikz-cd}
\usepackage{mathtools}
\usetikzlibrary{matrix}
\usepackage[all]{xy}
\usepackage{MnSymbol}
\usepackage[titletoc]{appendix}
\usepackage{mathtools}
\usepackage{bbm}
\usepackage{amsfonts}
\usepackage{url}

\usepackage[%
  verbose,
  colorlinks=true,
  naturalnames=true,
  linkcolor=blue,
]{hyperref}

\usepackage{afterpage}

\calclayout

\usepackage[english]{babel}
\usepackage{mathrsfs}

\usepackage[T2A,T1]{fontenc}
\newcommand{\BE}{\mbox{\usefont{T2A}{\rmdefault}{m}{n}\CYRB}}

\theoremstyle{plain}

\newtheorem{thm}{Theorem}[section]
\newtheorem{prop}[thm]{Proposition}
\newtheorem{lem}[thm]{Lemma}
\newtheorem{cor}[thm]{Corollary}
\newtheorem*{ThmA}{Theorem A}
\newtheorem*{ThmB}{Theorem B}
\newtheorem*{ThmC}{Theorem C}

\theoremstyle{definition}
\newtheorem{definition}[thm]{Definition}
\newtheorem{example}[thm]{Example}
\newtheorem{remark}[thm]{Remark}
\newtheorem{noname}[thm]{}
\newtheorem{construction}[thm]{Construction}

\numberwithin{equation}{thm}

\newcommand{\A}{\mathbb{A}}

\renewcommand{\C}{\mathcal{C}}
\newcommand{\D}{\mathbb{D}}

\newcommand{\F}{\mathcal{F}}
\newcommand{\G}{\mathbb{G}}
\newcommand{\N}{\mathbb{N}}
\newcommand{\Log}{\mathcal{L}og}

\renewcommand{\Pr}{\mathcal{P}r}

\newcommand{\SH}{\mathbf{SH}}

\renewcommand{\1}{\mathbbm{1}}
\newcommand{\Q}{\mathbb{Q}}
\newcommand{\Z}{\mathbb{Z}}
\newcommand{\Dual}{\mathbb{D}}
\newcommand{\T}{\mathcal{D}}
\newcommand{\R}{\mathfrak{R}}

\renewcommand{\sp}{\text{sp}}
\renewcommand{\lim}{\text{lim}}

\DeclareMathOperator{\id}{id} 
 
\DeclareMathOperator{\colim}{colim}
\DeclareMathOperator{\slim}{lim}  
\DeclareMathOperator{\Sh}{Sh} 
\DeclareMathOperator{\PSh}{PSh} 
\DeclareMathOperator{\map}{map} 
\DeclareMathOperator{\Fun}{Fun} 
 
\DeclareMathOperator{\Spc}{Spc} 
\DeclareMathOperator{\y}{y} 
\DeclareMathOperator{\const}{const}

\DeclareMathOperator{\CAlg}{CAlg} 
\DeclareMathOperator{\Mod}{Mod} 
\DeclareMathOperator{\Fin}{Fin} 
 
\DeclareMathOperator{\BarConstr}{Bar}

\DeclareMathOperator{\Hom}{\underline{Hom}}

\DeclareMathOperator{\DM}{DM}
\DeclareMathOperator{\DA}{DA}
\DeclareMathOperator{\et}{\text{\'et}}
\DeclareMathOperator{\Spec}{Spec}

\DeclareMathOperator{\Sch}{Sch}
\DeclareMathOperator{\Sm}{Sm}
\DeclareMathOperator{\Et}{\text{\'Et}}

\DeclareMathOperator{\can}{\alpha}
\DeclareMathOperator{\Ex}{Ex}
\DeclareMathOperator{\comp}{comp}
\DeclareMathOperator{\cons}{cons}
\DeclareMathOperator{\qcqs}{qcqs}
\DeclareMathOperator{\ft}{ft}
\DeclareMathOperator{\Cat}{Cat}

\DeclareMathOperator{\tame}{tame}
\DeclareMathOperator{\unit}{unit}
\DeclareMathOperator{\counit}{counit}

\DeclareMathOperator{\eff}{eff}
\DeclareMathOperator{\hyp}{hyp}
\DeclareMathOperator{\st}{st}
\DeclareMathOperator{\Ab}{Ab}
\DeclareMathOperator{\DiaSch}{DiaSch}
\DeclareMathOperator{\op}{op}
\DeclareMathOperator{\sep}{sep}
\DeclareMathOperator{\Gal}{Gal}

\DeclareMathOperator{\Spt}{Spt}
\DeclareMathOperator{\Ob}{Ob}

\DeclareMathOperator{\ULA}{ULA}
\DeclareMathOperator{\Pro}{Pro}
\DeclareMathOperator{\Ind}{Ind}

\DeclareMathOperator{\CRing}{CRing}
\DeclareMathOperator{\Uni}{Uni}
\DeclareMathOperator{\Frac}{Frac}

\DeclareMathOperator{\modulo}{mod}
\DeclareMathOperator{\ex}{ex}
\DeclareMathOperator{\Sp}{Sp}
\DeclareMathOperator{\End}{End}
\DeclareMathOperator{\HOM}{Hom}

\tikzset{%
    symbol/.style={%
        draw=none,
        every to/.append style={%
            edge node={node [sloped, allow upside down, auto=false]{$#1$}}}
    }
}

\newcommand{\leftrarrows}{\mathrel{\raise.75ex\hbox{\oalign{%
  $\scriptstyle\leftarrow$\cr
  \vrule width0pt height.75ex$\hfil\scriptstyle\relbar$\cr}}}}
\newcommand{\lrightarrows}{\mathrel{\raise.75ex\hbox{\oalign{%
  $\scriptstyle\relbar$\hfil\cr
  $\scriptstyle\vrule width0pt height.75ex\smash\rightarrow$\cr}}}}
\newcommand{\Rrelbar}{\mathrel{\raise.75ex\hbox{\oalign{%
  $\scriptstyle\relbar$\cr
  \vrule width0pt height.75ex$\scriptstyle\relbar$}}}}
\newcommand{\longleftrightarrows}{\leftrarrows\joinrel\Rrelbar\joinrel\lrightarrows}

\usepackage{microtype}

\newcommand{\MatrixForThm}{\begin{psmallmatrix} 0 & \comp_\Upsilon (-1)[-1]  \\ \comp_\Upsilon & 0  \end{psmallmatrix}}

\begin{document}
\pagenumbering{gobble}

\begin{titlepage}\begin{center}
	\vspace*{\fill}
	\huge{\textbf{Motivic nearby cycles functors, local monodromy and universal local acyclicity}} \\[0.7cm]
	
	\vspace{1cm}
	
	\includegraphics[scale=0.2]{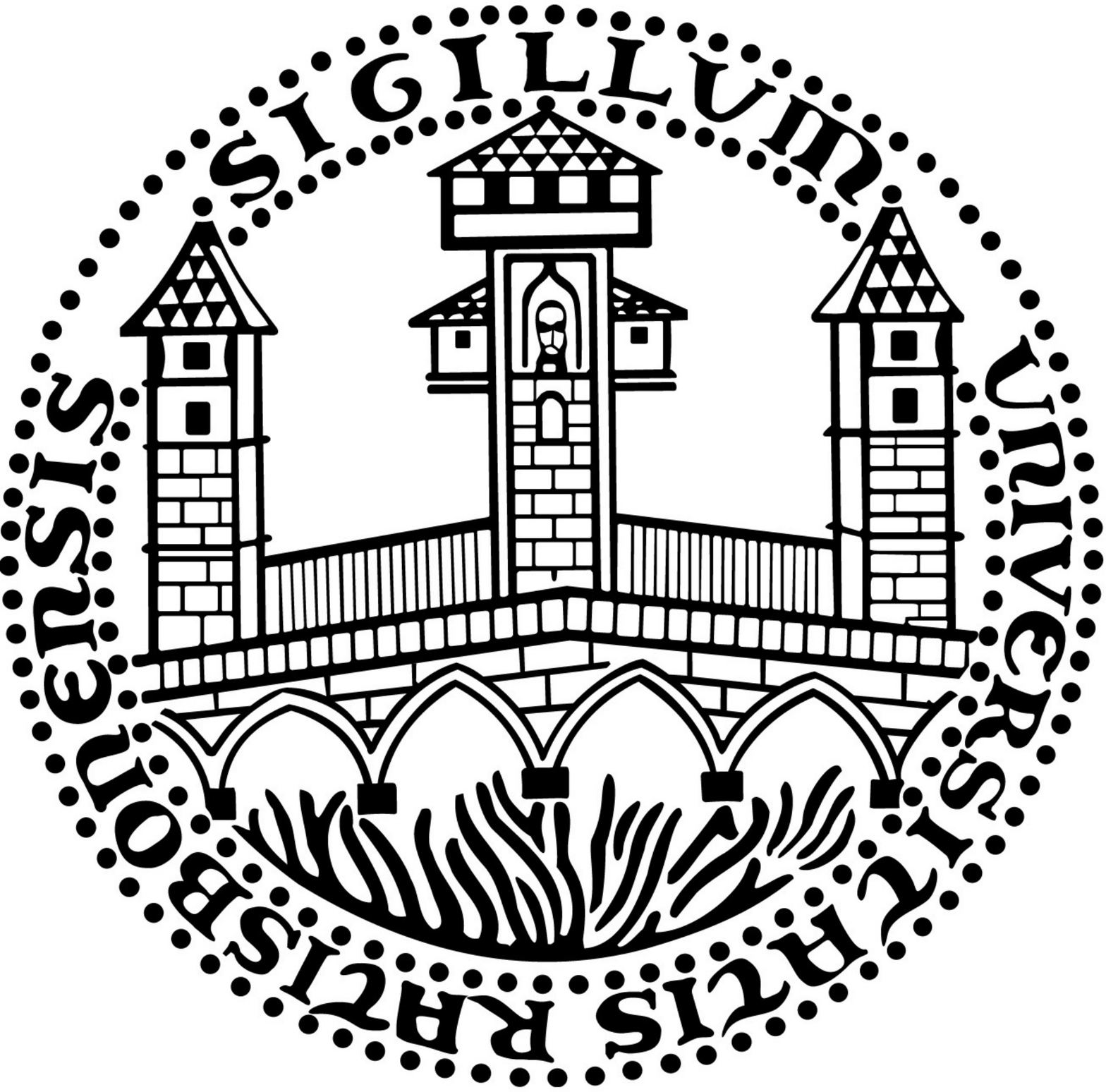} \\[0.5cm]	
	
	\huge{\textsc{Dissertation}} \\[0.1cm]
	\LARGE{\textsc{zur Erlangung des Doktorgrades \\
			der Naturwissenschaften (Dr. rer. nat.) \\
			der Fakult\"at f\"ur Mathematik \\
			der Universit\"at Regensburg}} \\[0.7cm]	
	
	\LARGE{vorgelegt von \\[0.1cm]
		\textbf{Benedikt Preis} \\[0.1cm]
		aus \\[0.1cm]
		\textbf{Roding} \\[0.1cm]
		im Jahr 2023}
	\vfill	
	\phantom{\LARGE{\textbf{\textcolor{red}{Version as of \today}}}}
\end{center}\end{titlepage}


	\vspace{2cm}

\noindent\begin{tabular}{@{}ll}

	Promotionsgesuch eingereicht am: & 01.02.2023 \\ 
	\\
	Die Arbeit wurde angeleitet von: & Prof. Dr. Denis-Charles Cisinski \\	
	\\

	Pr\"ufungsausschuss: & ~ \\
\\	
	
	Vorsitzender: & Prof. Dr. Harald Garcke \\ 
	Erstgutachter: & Prof. Dr. Denis-Charles Cisinski \\ 
	Zweitgutachter: & Prof. Dr. Alberto Vezzani \\ 
	weiterer Pr\"ufer: & Prof. Dr. Marc Hoyois \\
	Ersatzpr\"ufer: & Prof. Dr. Uli Bunke  
\end{tabular}

\newpage
\begin{center}
{\Large
\textbf{Abstract}}
\end{center}

\vspace{1cm}
In this thesis we give two applications of Ayoub's motivic nearby cycles functor: First we give a generalization of Grothendieck's classical local monodromy theorem. In the same setup we show that the inertia group acts quasi-unipotently on the \'etale cohomology of  sheaves "coming from motives". Second we study the notion of universal local acyclicity for motives and show that for \'etale motives universal local acyclicity over an excellent 1-dimensional regular base scheme is detected by the motivic nearby cycles functor. Along the way we prove properties of the motivic nearby cycles functor which might be of independent interest.


\tableofcontents
\pagenumbering{arabic}
\chapter*{Introduction}

\section*{Overview}
The theory of \'etale Voevodsky motives over a scheme $X$ was developed by Ayoub in \cite{AyoubRealizationEtale} and a slight variation thereof at the same time by Cisinksi-Deglise \cite{CisinskiDegliseEtale}. To any scheme $X$ and commutative ring $\Lambda$ one associates a triangulated category $\DA_{\et}(X, \Lambda)$ of \'etale motives. This formation admits a six functor formalism similiar to the six functor formalism of \'etale torsion sheaves as developed in SGA4 and SGA5. More precisely, for any morphism of schemes $f:X \rightarrow Y$ one gets a pair of adjoint functors
\[
f^* : \DA_{\et}(Y, \Lambda) \rightleftarrows \DA_{\et}(X, \Lambda): f_*,
\]
if $f:X \rightarrow S$ is of finite type between qcqs schemes one gets an adjunction
\[
f_! : \DA_{\et}(Y, \Lambda) \rightleftarrows \DA_{\et}(X, \Lambda): f^!
\]
and $\DA_{\et}(X, \Lambda)$ comes equipped with a tensor product $\otimes$ which is closed (i.e. $\_ \otimes M$ admits a right adjoint $\Hom(M, \_)$). These six functors satisfy various properties and compatibilities. 

Let $\ell$ be a prime number invertible in $\mathcal{O}(X)$. Then under mild assumptions on $X$ we can define the \textit{$\ell$-adic realization functor}
\[
\R_\ell: \DA^{\cons}_{\et}(X, \Q) \longrightarrow \hat{\T}^{\cons}_{\et}(X, \Q_\ell)
\]
from constructible \'etale motives to the the derived category of constructible $\ell-$adic systems of  \'etale  sheaves on $X$. Both sides of this functor admit the six functors and under mild assumptions the $\ell$-adic realization functor commutes with the six functors.

This realization functor opens up an interesting point of view on \'etale motives: Understanding a phenomenon in $\DA^{\cons}_{\et}(X, \Q)$ will make you understand the $\ell$-adic version of the phenomenon for all  primes $\ell$ invertible in $\mathcal{O}(X)$ simultaneously. In particular, these results have a built-in  "independence of $\ell$". Conversely, an $\ell$-adic phenomenon which is independent of $\ell$ and is "of geometric nature" can be expected to be the shadow of a motivic phenomenon under the $\ell$-adic realization.  

Let $S$ be the spectrum of a strictly henselian discrete valuation ring and $f: X \rightarrow S$ a morphism of finite type. Denote the closed point of $S$ by $\sigma$ and the open point by $\eta$. Then via pullback we get a decomposition
\[
\begin{tikzcd}
X_\eta \arrow[r, "j"] \arrow[d, "f_\eta"] & X \arrow[d, "f"] & X_\sigma \arrow[l, "i"'] \arrow[d, "f_\sigma"] \\
\eta \arrow[r, "j"]                       & S                & \sigma \arrow[l, "i"']                        
\end{tikzcd}
\]
of $X$ into its generic and special fiber. In this setup one can use the six operations to define the $\ell$-adic nearby cycles functor
\[
\Psi_f^\ell: \hat{\T}^{\cons}_{\et}(X_\eta, \Q_\ell) \longrightarrow \hat{\T}^{\cons}_{\et}(X_\sigma, \Q_\ell)
\]
(see \cite[Exp. XIII]{SGA7.2}, \cite[4.4]{BBD}). In \cite{AyoubRealizationEtale} Ayoub defined a motivic nearby cycles functor 
\[
\Psi_f: \DA^{\cons}_{\et}(X_\eta, \Lambda) \longrightarrow \DA^{\cons}_{\et}(X_\sigma, \Lambda),
\]
and proved that it satisfies a lot of desirable properties. In particular if $\Lambda= \Q$ it satisfies under mild assumptions that $\R_\ell \circ \Psi_f^\ell \simeq \Psi_f \circ \R_\ell$. 

In this thesis we will give two applications of the motivic nearby cycles functor for \'etale motives: We generalize Grothendieck's local monodromy theorem to "sheaves coming from motives" and we show that universal local acyclicity over a 1-dimensional excellent regular base can be detected by the nearby cycles functor. \\

Let us recall Grothendieck's famous local monodromy theorem: Let $K$ be the fraction field of a henselian discrete valuation ring $S$ and $X$ a separated $K$-scheme of finite type. Let $\bar{K}$ be a separable closure of $K$ and denote the pullback of $X$ to $\bar{K}$ by $\bar{X}$.  Then the \'etale cohomology groups $H^{i}_{\et}(\bar{X}, \Q_\ell)$ come canonically equipped with a group action
\[
\rho: \Gal(\bar{K}/K) \rightarrow \End_{\Q_\ell}(H^{i}_{\et}(\bar{X}, \Q_\ell))  
\]  
of the absolute Galois group of $K$. We say that $\lambda \in \Gal(\bar{K}/K)$ acts unipotent on $H^{i}_{\et}(\bar{X}, \Q_\ell)$ if there exists an integer $m$ such that $(\rho(\lambda)- \id )^m = 0$. Let $I \subset \Gal(\bar{K}/K)$ denote the inertia subgroup. Then the local monodromy theorem asserts that for all $i \in \Z$ there exists an open subgroup $H \subset I$ such that for all $\lambda \in H$  the action of $\lambda$ on $H^{i}_{\et}(\bar{X}, \Q_\ell)$ is unipotent. Moreover the analogue statement is true for the cohomology groups with compact support $H^{i}_{\et, c}(\bar{X}, \Q_\ell)$. 

Grothendieck proved this first in \cite[Ex. I]{SGA7.1} in a very arithmetic fashion. He later gave a more geometric proof using the theory of nearby cycles which was conditional at that time, since it relied on the absolute purity conjecture (now a theorem of Gabber) and resolution of singularities (now one can use de Jong's alterations). The local monodromy theorem has many applications. For example, it was a crucial input for Grothendieck's semi-stable reduction theorem for abelian varieties \cite[Ex. IX, 3.6]{SGA7.1}. Moreover, it implies the existence of a nilpotent operator $N: H^{i}_{\et}(\bar{X}, \Q_\ell) \rightarrow H^{i}_{\et}(\bar{X}, \Q_\ell)(-1)$ which gives rise to a filtration of $H^{i}_{\et}(\bar{X}, \Q_\ell)$ called the monodromy filtration. This filtration was introduced and studied by Deligne in \cite{Weil2} and is subject of the Monodromy-Weight Conjecture (see \cite[3.9]{IllusieAutour}).

We generalize this in the following sense: Let $\mathcal{F}$ in $\hat\T_{\et}^{\cons}(X, \Q_\ell)$ be either 
\begin{enumerate}
\item $\mathcal{F} = \R_\ell (M)$ for some constructible motive $M$ in $\DA^{\cons}_{\et}(X, \Q)$, or
\item$\mathcal{F} = {}^p\mathcal{H}^k(\R_\ell (M))$ (i.e. the $k$-th perverse cohomology sheaf) for some constructible motive $M$ in $\DA^{\cons}_{\et}(X, \Q)$ and some $k \in \Z$.
\end{enumerate}
Again the cohomology groups $H^{i}_{\et}(\bar{X}, \mathcal{F}|_{\bar{X}})$ and $H^{i}_{\et,c}(\bar{X}, \mathcal{F}|_{\bar{X}})$ come equipped with an action of $\Gal(\bar{K}/K)$. We show:

\begin{ThmA}[\ref{cor:MonodromyThm},\ref{cor:LocMonForPerverseThings}]
Assume that the henselian discrete valuation ring $S$ is excellent. Then there exists an open subgroup $H \subset I$ such that for all $\lambda \in H$ and all $i \in \Z$ the action of $\lambda$ on $H^{i}_{\et}(\bar{X}, \mathcal{F}|_{\bar{X}})$ is unipotent. The analogue statement is true for $H^{i}_{\et,c}(\bar{X}, \mathcal{F}|_{\bar{X}})$.
\end{ThmA}
 
In particular, we recover Grothendieck's local monodromy theorem (under the additional excellency assumption) if we plug in $M=\Q$. Let us note that our proof is completely independent of the existing proofs of the local monodromy theorem. Moreover this confirms (even generalizes!) an expectation stated by Illusie in \cite[\S 1]{IllusieAutour}. \\

The notion of local acyclicity goes back to \cite[Exp. XV]{SGA4.3}. It was a key tool to prove the smooth base change theorem for \'etale torsion sheaves. Consider a morphism of schemes $f: X \rightarrow S$ and an \'etale torsion sheaf $\mathcal{F}$ on $X$. Then $\mathcal{F}$ is called \textit{locally acyclic with respect to $f$} if for all geometric points $x$ of $X$ and $t$ of $S_{(f(x))}$ the canonical map
\[
\mathcal{F}_x \simeq R\Gamma(X_{(x)}, \mathcal{F}) \longrightarrow R\Gamma(X_{(x)} \times_{S_{(f(x))}} t, \mathcal{F})
\]
is an isomorphism.  $\mathcal{F}$ is called \textit{universally locally acyclic with respect to $f$} if the analogue is true after base change along any $S' \rightarrow S$. In the case where $S$ is the spectrum of a strictly henselian discrete valuation ring we may consider the associated nearby cycles functor $\Psi_f$. It comes with a canonical map 
\[\alpha: i^* \mathcal{F} \longrightarrow \Psi_f(j^* \mathcal{F}). \]
 It is easy to see that $\mathcal{F}$ is locally acyclic with respect to $f$ if and only if $\alpha$ is an equivalence: This can be checked on stalks of geometric points $x$ of $X_\sigma$ where we have 
 \[\Psi_f(j^* \mathcal{F})_x \simeq R \Gamma(X_{(x)} \times_S \eta, \F).\]
  If $f$ is moreover of finite type, being locally acyclic is in fact equivalent to being universally locally acyclic by \cite[6.6]{LuZhengDuality}. 

Recently, Lu-Zheng \cite{lu_zheng_2022} gave an equivalent characterisation of universal local acyclicity which makes sense in any six functor formalism. This was for example used by Hansen-Scholze to define a relative perverse t-structure in \cite{HansenScholzeRelative}. In particular for any six functor formalism with a theory of nearby cycles $\Psi$ one can ask: Is there a relation between $\Psi$ and the property of being universally locally acyclic? We give a positive answer for \'etale motives:

\begin{ThmB}[\ref{thm:ULAinDAdetectedByNearby}]
Let $f: X \rightarrow S$ be of finite type, where $S$ is the spectrum of an excellent strictly henselian discrete valuation ring. Let $\Lambda$ be a noetherian ring flat over $\Z$ and $M$ a motive in $\DA_{\et}^{\cons}(X, \Lambda)$. Then $M$ is universally locally acyclic with respect to $f$ if and only if the canonical map
\[
 i^*M \longrightarrow \Psi_f(j^*M)
\]
is an equivalence. 
\end{ThmB}

As an application  we show in Proposition \ref{prop:SSviaPsi} that the weak singular support of a motive can be determined using motivic nearby cycles functors. We hope that this tool will be useful for further study of the (weak) singular support of a motive.

\vspace{1cm}

Along the way we study the motivic nearby cycles functor in some depth. Given a morphism of schemes $f: X \rightarrow S$, where $S$ is the spectrum of a strictly henselian discrete valuation ring, Ayoub defines not only the motivic nearby cycles functor $\Psi_f$ but also a functor
\[
\Upsilon_f: \DA_{\et}(X_\eta, \Lambda) \longrightarrow  \DA_{\et}(X_\sigma, \Lambda)
\]
which he calls the \textit{unipotent nearby cycles functor}. There is a canonical natural transformation $ \Upsilon_f \rightarrow \Psi_f$. Our main result concerning this is the following:
\begin{ThmC}[\ref{cor:UpsilonRetractOfPsi}]
Assume that $\Lambda$ is a $\Q$-algebra and $S$ is excellent. Then for every morphism of finite type $f:X \rightarrow S$ and  every $M$ in $\DA^{\cons}_{\et}(X_\eta, \Lambda)$  the canonical map 
\[ \Upsilon_{f}(M) \longrightarrow \Psi_f(M)\]
is the inclusion of a direct summand.
\end{ThmC}
Whenever $\Lambda$ is a $\Q$-algebra Ayoub constructs an interesting monodromy operator $N: \Upsilon_f \rightarrow \Upsilon_f(-1)$. The theorem allows us to use the monodromy operator $N$ as a very effective tool to study $\Psi_f$.

\section*{Leitfaden}
We start off Chapter 1 with introducing \'etale motives. We do this in the language of $\infty$-categories and take special care to eliminate various finiteness hypotheses. Once this bookkeeping duty is done we recall Ayoub's formalism of specialization systems and the construction of the motivic nearby cycles functors $\Upsilon, \Psi^{\tame}$ and $\Psi$. Using our $\infty$-categorical setup we can describe these functors in terms of colimits (Proposition \ref{prop:NearbyCyclesViaLogAndColimits}). Finally we use this description and some observations concerning the logarithm motive to prove Theorem \ref{thm:UpsilonRetractOfPsiTame}, which will be a key technical tool. \\

In Chapter 2 we prove our generalization of the local monodromy theorem (Corollaries \ref{cor:MonodromyThm} and \ref{cor:LocMonForPerverseThings}). With Theorem \ref{thm:UpsilonRetractOfPsiTame} at our disposal this is a rather easy consequence of a theorem of Ayoub (Theorem \ref{thm:MonodromySquareAyoub}). Before that we give some background on the $J$-adic realization and the classical $J$-adic nearby cycles functor.\\

In Chapter 3 we introduce the notion of universal local acyclicity for motives. We give an equivalent characterisation of universal local acyclicity in terms of K\"unneth-type formulas (Proposition \ref{thm:EquivCharOfULA}) and prove a generic universal local acyclicity theorem (Proposition \ref{thm:GenericULA}). Then in Theorem \ref{thm:ULAinDAdetectedByNearby} we finally relate the notion of universal local acyclicity with the motivic nearby cycles functor. This is quite technical and covers a good part of the chapter. As an application we show that the weak singular support of a motive is determined by the motivic nearby cycles functor (Proposition \ref{prop:SSviaPsi}).\\

In Appendix A we state some facts about dualizable objects in a bicategory which we use in the proof of Theorem \ref{thm:ULAinDAdetectedByNearby}.
\section*{Acknowledgements}

I wish to thank my advisor Denis-Charles Cisinski for his hearty support and for the countless hours he spent answering my questions and sharing his insights. \\

I wish to thank Han-Ung Kufner and Sebastian Wolf for our coffee break discussions on the $\Log$-motive. In retrospect these meetings influenced this thesis a lot. Especially \S \ref{section:OnTheLog} is a direct result of the ideas conceived together. \\

Above all I am grateful for all the friends I made along the way as a PhD student in Regensburg and all the good moments we shared.

\newpage

\section*{Notations and conventions}

We freely use the language of $\infty$-category theory as developed by Lurie in \cite{lurie2009higher}, \cite{lurie2016higher} and \cite{lurie2018SAG}. Our notations and conventions often coincide with the ones in \textit{loc. cit.}. Let us still recall some basic and frequently used notations.\\

For two objects $X$ and $Y$ in an $\infty$-category $\C$ we denote by $\map_\C(X,Y)$ the mapping space. To an $\infty$-category $\C$ we can associate its homotopy category $h \C$, which is an ordinary category with homomorphism sets $\HOM_{h\mathcal{C}}(X,Y) = \pi_ 0 \map_\C(X,Y)$. A functor $F: \C \rightarrow \mathcal{D}$ of $\infty$-categories gives rise to a functor $F: h\C \rightarrow h\mathcal{D}$ between the homotopy categories. The $\infty$-category of functors between $\C$ and $\mathcal{D}$ is denoted by $\Fun(\C, \mathcal{D})$. \\

We say that a diagram 
\[
\begin{tikzcd}
X \arrow[r, "a"] \arrow[d, "c"'] & Y \arrow[d, "b"] \\
Z \arrow[r, "d"']                & W               
\end{tikzcd}
\]
in an $\infty$-category $\C$ commutes if there exists an equivalence $ba \simeq dc$ in $\map_\C(X,W)$. Equivalently the induced diagram in $h\C$ commutes in the classical sense.  \\

We have to talk about size. For this we use Grothendieck universes to talk about small sets. An $\infty$-category is called small if it is equivalent to an $\infty$-category whose underlying simplicial set is small. We say a set is possibly non-small if it is small for the next Grothendieck universe in the hierarchy. These two steps suffice for our purpose. \\

We denote the $\infty$-category of small $\infty$-categories by $\Cat_\infty$ and the $\infty$-category of possibly non-small $\infty$-categories by $\widehat{\Cat}_\infty$. We write $\Pr^L$ for the non-full subcategory of $\widehat{\Cat}_\infty$ consisting of presentable $\infty$-categories with small colimit preserving functors between them. We write $\Pr^{L,\st}$ for its full subcategory consisting of stable presentable $\infty$-categories. \\

We denote by $\Spc$ the $\infty$-category of small spaces and by $\Spt$ its stabilization, the $\infty-$category of spectra. Given a small category $\C$ we write $\PSh(\C) := \Fun(\C^{\op}, \Spc)$ for the $\infty$-category of presheaves on $\C$. Note that $\PSh(\C)$ is a presentable $\infty$-category which is non-small. \\

By a ring we will always mean a \textbf{commutative ring}. \\

By a smooth (resp. \'etale) morphism we always mean smooth (resp. \'etale) \textbf{of finite type}.

\chapter{\'Etale motives and the motivic nearby cycles functor}

We start with lifting the classical theory of \'etale motives as developed in \cite{AyoubRealizationEtale} to the level of $\infty$-categories. This is the language we want to speak later on and moreover it allows us to remove assumptions on separatedness. We give a complete proof of the fact that Ayoub's \'etale motives agree with Cisinski-Deglise's $h$-motives when restricted to finite dimensional noetherian schemes (Theorem \ref{thm:CompDAvsDM}). This allows us to use results from both worlds.

Next we introduce Ayoub's formalism of specialization systems and define the motivic nearby cycles functors. We give an alternative description of these functors in terms of ($\infty$-categorical) colimits in Proposition \ref{prop:NearbyCyclesViaLogAndColimits}. 

We make a small digression to a more general setup in Section \ref{section:OnTheLog}. We can associate to a smooth commutative group scheme $X$ a cosimplicial motive which for $X = \G_{m}$ already appears in Ayoub's construction of the unipotent nearby cycles functor $\Upsilon$. We show that this object can be considered as a cosimplicial representation of the logarithm motive associated to $X$ (Corollary \ref{cor:CosimplcialLogAgreesWithClassicalLogByHK}).

Finally our study of the logarithm motive allows us to prove our key technical tool: In Theorem \ref{thm:UpsilonRetractOfPsiTame} we show that with rational coefficients the unipotent nearby cycles functor is actually a direct factor of the tame nearby cycles functor.

\section{Motivic $\infty$-categories}

\noname \label{noname:MotivicInftyCat} Let $S$ be a quasi compact quasi separated (qcqs) scheme and denote by $\Sch^{\qcqs}_{/S}$ the category of qcqs schemes over $S$. Let $\Sm$ denote the collection of smooth morphisms in $\Sch^{\qcqs}_{/S}$. A \textit{motivic $\infty$-category over $S$} is a functor
\[
\T(\_): (\Sch_{/S}^{\text{qcqs}} )^{\op}\longrightarrow \CAlg( \Pr^{L, \st})
\]
to the $\infty$-category of stable presentable symmetric monoidal $\infty$-categories, which is a $(*, \sharp, \otimes)$-formalism on $(\Sch^{\qcqs}_{/S}, \Sm)$ satisfying the Voevodsky conditions in the sense of \cite[\S 2]{Adeel6Functor}. Given a morphism of schemes $f: X \rightarrow Y$ in $\Sch_{/S}^{\text{qcqs}}$ we write $f^*: \T(Y) \rightarrow \T(X)$ for the functor induced by $\T(\_)$ and $f_* : \T(X) \rightarrow \T(Y)$ for its right adjoint. 

Let us recall some properties:
\begin{enumerate}
\item For all $X$ in $\Sch^{\qcqs}_{/S}$ the monoidal structure of $\T(X)$ is closed. We denote the internal Hom-object by $\Hom(\_, \_)$.
\item Let $\Sch^{\qcqs,F}_{/S}$ denote the (non-full) subcategory of $\Sch^{\qcqs}_{/S}$ whose objects are qcqs schemes over $S$ and whose maps are morphisms of finite type between these. Then there exists a functor 
\[
\T(\_)_!: \Sch_{/S}^{\qcqs, F}\longrightarrow \Pr^{L, \st}
\]
which sends a morphism of finite type $f: X \rightarrow Y$ to  functor $f_!: \T(X) \rightarrow \T(Y)$. We denote the right adjoint of $f_!$ by $f^!$. 
\item For all $f: X \rightarrow Y$ of finite type in $\Sch^{\qcqs}_{/S}$ there exists a natural transformation $f_! \rightarrow f_*$ which is an equivalence whenever $f$ is proper.
\item For all smooth morphisms $f: X \rightarrow Y$ of relative dimension $d$ there exists a natural equivalence $f^* \overset{\sim}\rightarrow f^!(-d)[-2d]$. Here $(\_)$ denotes the Tate twist and $[\_]$ denotes the suspension.
\item (Base Change) For any cartesian square
\[
\begin{tikzcd}
Y' \arrow[d, "f'"'] \arrow[r, "g'"] & Y \arrow[d, "f"] \\
X' \arrow[r, "g"']                  & X               
\end{tikzcd}
\]
in $\Sch^{\qcqs}_{/ S}$, where  $f$ is of finite type, there exist natural equivalences
\begin{align*}
g^* f_! &\overset{\sim}\longrightarrow f'_! g'^*, \\
g'_* f'^! &\overset{\sim}\longrightarrow f^! g_*.
\end{align*}

\item (Projection Formula) For any morphism $f:X \rightarrow Y$ of finite type in $\Sch^{\qcqs}_{/ S}$ there are canonical equivalences
\begin{align*}
(f_! A) \otimes B &\overset{\sim}\longrightarrow f_!(A \otimes f^*B), \\
\Hom(f_! A, B) &\overset{\sim}\longrightarrow f_* \Hom(A, f^! B), \\
f^!\Hom(A, C) &\overset{\sim}\longrightarrow \Hom(f^*A, f^! B) \\
\end{align*}
for all $A,C$ in $\T(X)$ and $B$ in $\T(Y)$.
\item (Localization Sequence) Consider a closed immersion $i: Z \rightarrow X$ in $\Sch^{\qcqs}_{/S}$ whose open complement $j:U \rightarrow X$ lies in $\Sch^{\qcqs}_{/S}$. Then there are (co)fiber sequences
\[
j_! j^! \overset{\counit}\longrightarrow \id \overset{\unit}\longrightarrow i_*i^*
\]
and
\[
i_! i^! \overset{\counit}\longrightarrow \id \overset{\unit}\longrightarrow j_*j^*.
\]
\end{enumerate}

A motivic $\infty$-category satisfies several more properties such as excision, descent and purity. We refer to \cite[\S 2]{Adeel6Functor} for more details. 

\noname \label{noname:ExchangeMaps} In fact the natural transformations in (5) above are particular instances of a more general formalism: Consider a natural transformation
\[
\begin{tikzcd}
E \arrow[r, "c"] \arrow[d, "a"'] & F \arrow[d, "d"]\\
G \arrow[r, "b"']      \arrow[ru, Rightarrow, "\alpha", shorten <=2.5ex, shorten >=2.5ex]            & H                           
\end{tikzcd}
\]
of functors between $\infty$-categories and assume that $a$ and $d$ admit left adjoints which we denote by $a^\star$ and $d^\star$ respectively. To this we may associate a natural transformation 
\[
\begin{tikzcd}
G \arrow[d, "b"'] \arrow[r, "a^\star"] & E \arrow[d, "c"] \\
H \arrow[r, "d^\star"'] \arrow[ru, Rightarrow, "\Ex_\alpha", shorten <=2.5ex, shorten >=2.5ex]     & F               
\end{tikzcd}
\]
by defining $\Ex_\alpha$ to be the composition
\[
d^\star b \overset{\unit}\longrightarrow d^\star b a a^\star \overset{\id \alpha \id }\longrightarrow d^\star  d c a^\star \overset{\counit }\longrightarrow c a^\star. 
\]
We call $\Ex_\alpha$ the associated \textit{exchange map}. The formation of exchange maps has the following pasting property: Consider two natural transformations of the form
\[
\begin{tikzcd}
E \arrow[r, "c"] \arrow[d, "a"'] & F \arrow[d, "d"] \arrow[r, "e"] & I \arrow[d, "g"] \\
G \arrow[r, "b"'] \arrow[ru, Rightarrow, "\alpha", shorten <=2.5ex, shorten >=2.5ex]     & H \arrow[r, "f"'] \arrow[ru, Rightarrow, "\beta", shorten <=2.5ex, shorten >=2.5ex]     & J,               
\end{tikzcd}
\]
where $a,d $ and $g$ admit left adjoints denoted by $a^\star, d^{\star}$ and $g^{\star}$. Then 
\[
\begin{tikzcd}
g^{\star}fb \arrow[rr, "\Ex_{\beta \circ \alpha}"] \arrow[rd, "\Ex_{\beta} b"'] &                                           & ec a^{\star} \\
                                                                                & ed^{\star}b \arrow[ru, "e \Ex_{\alpha}"'] &             
\end{tikzcd}
\]
commutes (see \cite[1.1.7]{CisinskiDegliseBook}).

\section{\'Etale motives}

\noname \label{noname:ConstrOfDA} Let $X$ be a scheme, $\Lambda$ a ring and write $\mathcal{D}(\Lambda):= \Mod_{H\Lambda}(\Spt)$. Here $H\Lambda$ denotes the Eilenberg-Maclane spectrum associated to $\Lambda$ and $\Mod_{H\Lambda}(\Spt)$ denotes the $\infty$-category of $H \Lambda$-modules in the $\infty$-category of spectra. Note that $\T(\Lambda)$ is in fact equivalent to the unbounded derived $\infty$-category of $\Lambda$-modules by the Schwede-Shipley Theorem (see \cite[7.1.2.1]{lurie2016higher}). Let $\Sm_{/X}$ denote the category of smooth schemes over $X$ and write
\[
\PSh(\Sm_{/X}, \T(\Lambda)):= \PSh(\Sm_{/X}) \otimes \T(\Lambda),
\]
where $\_\otimes \T(\Lambda)$ denotes the Lurie tensor product of presentable $\infty-$categories (see \cite[\S 4.8.1]{lurie2016higher}). The canonical colimit preserving functor $\Spc \rightarrow \T(\Lambda)$ induces a Yoneda functor
\[
\y_\Lambda: \Sm_{/X} \longrightarrow \PSh(\Sm_{/X}) \longrightarrow \PSh(\Sm_{/X}, \T(\Lambda)).
\]
Let $\Sh^{\text{hyp}}_{\et}(\Sm_{/X}, \mathcal{D}(\Lambda))$ be the full subcategory of $\PSh(\Sm_{/X}, \T(\Lambda))$ consisting of those objects $\mathcal{F}$ which are local with respect to \'etale hyper-covers. The inclusion 
\[
\Sh^{\text{hyp}}_{\et}(\Sm_{/X}, \mathcal{D}(\Lambda)) \subset \PSh(\Sm_{/X}, \T(\Lambda))
\]
admits by \cite[5.5.4.15]{lurie2009higher} a left adjoint  which we denote by $L_{\et}$. An element $\mathcal{F}$ in $\Sh^{\text{hyp}}_{\et}(\Sm_{/X}, \mathcal{D}(\Lambda))$ is called \textit{$\A^1$-invariant} if $\mathcal{F}(\pi_Y): \mathcal{F}(Y) \rightarrow \mathcal{F}(\A^1_Y)$ is an equivalence for $\A^1$-projections $\pi_Y: \A^1_Y \rightarrow Y$ in $\Sm_{/X}$. We define $\DA_{\et}^{\text{eff}}(X, \Lambda)$ to be the full subcategory of $\Sh^{\text{hyp}}_{\et}(\Sm_{/X}, \mathcal{D}(\Lambda))$ consisting of $\A^1$-invariant objects. Again the inclusion 
\[
\DA_{\et}^{\text{eff}}(X, \Lambda) \subset \Sh^{\text{hyp}}_{\et}(\Sm_{/X}, \mathcal{D}(\Lambda))
\]
admits a left adjoint by \cite[5.5.4.15]{lurie2009higher} which we denote by $L_{\A^1}$. Let us denote the motive associated to a $Y$ in $\Sm_{/X}$ via
\[
\Sm_{/X} \overset{\y_\Lambda}\longrightarrow  \PSh(\Sm_{/X}, \mathcal{D}(\Lambda)\overset{L_{\et}}\longrightarrow \Sh^{\text{hyp}}_{\et}(\Sm_{/X}, \mathcal{D}(\Lambda)) \overset{L_{\A^1}}\longrightarrow \DA_{\et}^{\text{eff}}(X, \Lambda)
\]
by $\Lambda (Y)$. 

\noname \label{noname:SymMonStrOfDAeff} $\T( \Lambda)$ is the underlying $\infty$-category of a symmetric monoidal $\infty$-category $\T( \Lambda)^\otimes$ which comes equipped with a canonical symmetric monoidal functor 
\begin{equation} \label{eqn:SpcToDLambda}
\Spc^\times \overset{\Sigma^\infty}\longrightarrow \Spt^\otimes \longrightarrow \T( \Lambda)^\otimes
\end{equation}
of presentable $\infty$-categories, where $\Spc$ is equipped with its cartesian monoidal structure and $\Spt$ with its smash product as defined in \cite[\S 4.8.2]{lurie2016higher}. The constant sheaf functor $\Gamma^*: \Spc \longrightarrow \PSh(\Sm_{/X})$  preserves finite limits and hence equips $\PSh(\Sm_{/X})^\times$ with the structure of a $\Spc^\times$-algebra in $\CAlg(\Pr^{L})$. By \cite[4.5.3.1]{lurie2016higher} the symmetric monoidal functor (\ref{eqn:SpcToDLambda}) induces a symmetric monoidal functor 
\[
\_ \otimes \T(\Lambda): \Mod_{\Spc^\times}(\Pr^L)^\otimes \longrightarrow \Mod_{\T(\Lambda)^\otimes}(\Pr^L)^\otimes. 
\]
The symmetric monoidal structure of $\PSh(\Sm_{/X})$ is expressed by a functor \[\Fin_* \longrightarrow \Mod_{\Spc^\times}(\Pr^L)^\otimes\]
 over $\Fin_*$.
The composition 
\[
\Fin_* \longrightarrow \Mod_{\Spc^\times}(\Pr^L)^\otimes \overset{\_ \otimes \T(\Lambda)}\longrightarrow \Mod_{\T(\Lambda)^\otimes}(\Pr^L)^\otimes 
\]
over $\Fin_*$ defines a $\T(\Lambda)$-algebra in $\Pr^L$ which we denote by $\PSh(\Sm_{/X}, \T(\Lambda))^\otimes$. It is clear from the construction that its underlying $\infty$-category is $\PSh(\Sm_{/X}, \T(\Lambda))$. It follows from \cite[2.2.1.9]{lurie2016higher} that $\Sh^{\text{hyp}}_{\et}(\Sm_{/X}, \mathcal{D}(\Lambda))$ is the underlying $\infty$-category of a symmetric monoidal $\infty$-category $\Sh^{\text{hyp}}_{\et}(\Sm_{/X}, \mathcal{D}(\Lambda))^\otimes$ such that the sheafification functor lifts to a functor
\[
L_{\et}^\otimes: \PSh( \Sm_{/X}, \T(\Lambda))^\otimes \longrightarrow \Sh^{\text{hyp}}_{\et}(\Sm_{/X}, \mathcal{D}(\Lambda))^\otimes
\]
of symmetric monoidal $\infty$-categories. Similarly $L_{\A^1}$ admits a lift 
\[
L_{\A^1}^\otimes: \Sh^{\text{hyp}}_{\et}(\Sm_{/X}, \mathcal{D}(\Lambda))^\otimes \rightarrow \DA_{\et}^{\text{eff}}(X, \Lambda)^\otimes
\]
to a functor of symmetric monoidal $\infty$-categories.

\noname\label{noname:SymMonStrOfDA} The unit section $1: X \rightarrow \G_{m,X}$ induces a morphism 
\begin{equation} \label{eqn:Lamda1}
\Lambda (X) \rightarrow \Lambda(\G_{m,X})
\end{equation}
in $\DA_{\et}^{\text{eff}}(X, \Lambda)$. We denote by $\Lambda(1)$ the object such that $\Lambda(1)[1]$ is the cofiber of (\ref{eqn:Lamda1}) and call it \textit{the Tate object}. We define \textit{the category of \'etale motives with $\Lambda$-coefficients} as the colimit  of the diagram
\[
\DA_{\et}^{\text{eff}}(X, \Lambda) \overset{\_ \otimes \Lambda(1)} \longrightarrow\DA_{\et}^{\text{eff}}(X, \Lambda) \overset{\_ \otimes \Lambda(1)}\longrightarrow \DA_{\et}^{\text{eff}}(X, \Lambda) \overset{\_ \otimes \Lambda(1)} \longrightarrow \dots 
\]
indexed by the poset $\mathbb{N}= \{0 \rightarrow 1  \rightarrow 2 \rightarrow \dots \}$ in the $\infty$-category $\Pr^L$ and denote it by $\DA_{\et}(X, \Lambda)$. We write
\[
\Sigma^\infty: \DA_{\et}^{\text{eff}}(X, \Lambda) \longrightarrow \DA_{\et}(X, \Lambda)
\]
for the functor induced by the canonical functor into the 0-th level of the $\mathbb{N}$-indexed diagram above. By slight abuse of notation we denote for any $Y$ in $\Sm_{/X}$ the object $\Sigma^\infty (\Lambda(Y))$ again by $\Lambda(Y).$ Since the object $\Lambda(1)$ is symmetric by \cite[4.4]{VoevodskyProceedings1998} and \cite[2.16]{Robalo15} we get that $\DA_{\et}(X, \Lambda)$ underlies a symmetric monoidal $\infty$-category $\DA_{\et}(X, \Lambda)^\otimes$ such that $\Sigma^\infty$ lifts to a symmetric monoidal functor (see \cite[2.2, 2.22]{Robalo15}). As in \cite[\S 2.4]{Robalo15} it is straightforward to check that the homotopy category of $\DA_{\et}(X, \Lambda)$ with its induced monoidal structure is equivalent to the symmetric monoidal triangulated category considered \cite{AyoubRealizationEtale}. 

\noname \label{noname:f*onDA} For any map of schemes $f: X \rightarrow Y$ the functor
\[
\_ \times_Y X : \Sm_{/Y} \longrightarrow \Sm_{/X}
\]
induces by pre-composition a functor
\[
f_{\PSh*}: \PSh(\Sm_{/X}, \T(\Lambda)) \longrightarrow \PSh(\Sm_{/Y}, \T(\Lambda))
\]
which admits a left adjoint
\[
f_{\PSh}^*: \PSh(\Sm_{/Y}, \T(\Lambda)) \longrightarrow \PSh(\Sm_{/X}, \T(\Lambda))
\]
given by left Kan extension. As $\_ \times_X Y$ commutes with finite products its Kan extension $f_{\PSh}^*$ lifts to a symmetric monoidal functor. The functor $\_ \times_Y X$ preserves \'etale hyper-covers and maps $\A^1$-projections to $\A^1$-projections. Thus $f_{\PSh *}$ restricts to a functor
\[
f_{\eff *}: \DA_{\et}^{\text{eff}}(X, \Lambda) \longrightarrow \DA_{\et}^{\text{eff}}(Y, \Lambda). 
\]
Define
\[
f_{\eff}^*: \DA_{\et}^{\text{eff}}(Y, \Lambda) \longrightarrow \DA_{\et}^{\text{eff}}(X, \Lambda)
\] 
as the compsition
\[
\DA_{\et}^{\text{eff}}(Y, \Lambda) \overset{f_{\PSh}^*|_{\DA_{\et}^{\text{eff}}(Y, \Lambda)}}\longrightarrow \PSh(\Sm_{/X}, \T(\Lambda)) \overset{L_{\A^1} \circ L_{\et}}\longrightarrow \DA_{\et}^{\text{eff}}(X, \Lambda).
\]
Then $f_{\eff}^*$ is left adjoint to $f_{\eff *}$. As $L_{\et}$ and $L_{\A^1}$ are symmetric monoidal (see \ref{noname:SymMonStrOfDA}) we can deduce that $f_{\eff}^*$ is symmetric monoidal. Combining this with the fact that $f_{\eff}^* \Lambda(1) \simeq \Lambda(1)$ we see that the two left adjoints $f_{\eff}^*$ and $\_ \otimes \Lambda(1)$ commute. Thus by \cite[2.9, 2.22]{Robalo15} $f_{\eff}^*$ induces a functor
\[
f^*: \DA_{\et}(Y, \Lambda) \longrightarrow \DA_{\et}(X, \Lambda)
\]
in $\Pr^L$ such that $\Sigma_\infty \circ f_{\eff}^* \simeq f^* \circ \Sigma_\infty$ which moreover lifts to a symmetric monoidal functor. We denote the right adjoint of $f^*$ by $f_*$. In the case where $f: X \rightarrow Y$ is smooth we can argue as in \cite[1.23, 1.26]{Adeel6Functor} to show that $f^*$ admits a left adjoint $f_\#$ that satisfies smooth base change and the projection formula.  


\noname \label{noname:FunctorialityOfDA}
Let $S$ be a qcqs scheme. Proceeding as in \cite[\S 9.1, Step 1)]{Robalo14} one can make the formation of $\DA_{\et}(\_, \Lambda)$ functorial in the sense that we can construct a functor of $\infty$-categories
\begin{equation} \label{eqn:FunctorDA*}
\DA_{\et}( \_, \Lambda): (\Sch_{/S}^{\qcqs})^{\op} \longrightarrow \CAlg(\Pr^{L,\text{st}})
\end{equation}
which sends a morphism $f: X \rightarrow Y$ to
\[
f^* : \DA_{\et}(Y, \Lambda) \longrightarrow \DA_{\et}(X, \Lambda).
\] 
From the observations above we see immediately that this is a $(*,\#, \otimes)$-formalism as defined in \cite[Definition 2.2]{Adeel6Functor}. Moreover one shows that $\DA_{\et}( \_, \Lambda)$ satisfies the Voevodsky conditions (see \cite[2.4]{Adeel6Functor}) analogous to \cite[2.5]{Adeel6Functor}. In particular $\DA_{\et}( \_, \Lambda)$ is  a motivic $\infty$-category over $S$ in the sense of \ref{noname:MotivicInftyCat}.

\noname We may adapt the steps (3.3)-(3.8) in \cite{LiuZhengEnhanced} in order to extend (\ref{eqn:FunctorDA*}) to a functor
\begin{equation} \label{eqn:FunctorFromDelta*22ToCat}
\delta^*_{2,\lbrace 2 \rbrace}(((\Sch_{/S}^{\qcqs})^{\op})^{\sqcup, \op})^{cart}_{F, all} \longrightarrow \Pr^{L, \st}
\end{equation}
(using the notations of \textit{loc. cit.}) where $F$ denotes the set of separated morphisms locally of finite type. Using the descent machinery developed in \cite[\S 4]{LiuZhengEnhanced} we can extend (\ref{eqn:FunctorFromDelta*22ToCat}) further to a functor 
\begin{equation} \label{eqn:FunctorFromDelta*22ToCatFT}
\delta^*_{2,\lbrace 2 \rbrace}(((\Sch_{/S}^{\qcqs})^{\op})^{\sqcup, \op})^{cart}_{F', all} \longrightarrow \Pr^{L, \st}
\end{equation}
where $F'$ denotes the set of morphisms locally of finite type which are not necessarily separated. Restricting (\ref{eqn:FunctorFromDelta*22ToCatFT}) to the first direction and forgetting the operadic structures gives rise to a functor
\begin{equation} \label{eqn:D!SchRingsCat}
\DA_{\et}(\_ , \Lambda)_!: \Sch_{/S}^{\qcqs,F'} \longrightarrow \Pr^{L, \st},
\end{equation}
where $\Sch_{/S}^{\qcqs,F'}$ denotes the subcategory of $\Sch_{/S}^{\qcqs}$ consisting of morphisms locally of finite type. This functor sends a morphism $f: X \rightarrow Y$ locally of finite type to the exceptional push forward
\[
f_!: \DA_{\et}(X, \Lambda) \longrightarrow \DA_{\et}(Y, \Lambda).
\]
We denote the right adjoint of $f_!$ by $f^!$. Again it is straightforward to check that when restricted to quasi projective morphisms the functors induced by $f^*,f_*,f^!, f_!$ on the homotopy categories are equivalent to the ones defined in \cite{AyoubRealizationEtale}.

The six functors $f^*,f_*,f^!, f_!, \otimes$ and $\Hom(\_, \_)$ satisfy various properties and compatibilities which are for example discussed in \cite[\S 2]{Adeel6Functor}.

\begin{remark} \label{rem:DAasSHotimesD}
\begin{enumerate}
\item Let $X$ be a scheme. If we replace $\T(\Lambda)$ with $\Spt$ in the constructions of $\DA^{\eff}_{\et}(X, \Lambda)$ and $\DA_{\et}(X, \Lambda)$ above we end up with $\SH^{S^1}_{\et}(X)$ and $\SH_{\et}(X)$ (as for example considered in \cite{BachmannRigidity}). One calls $\SH_{\et}(X)$ the \'etale motivic stable homotopy category. As above this construction can be made functorial giving rise to a functor
\begin{equation} \label{eqn:FunctorSH*}
\SH_{\et}( \_): (\Sch_{/S}^{\qcqs})^{\op} \longrightarrow \CAlg(\Pr^{L,\text{st}}).
\end{equation}

\item There is a canonical equivalence
\[
\DA_{\et}(X, \Lambda) \simeq \SH_{\et}(X) \otimes \T(\Lambda).
\]
Indeed since the Lurie tensor product interacts well with localisations (see the proof of \cite[4.8.1.15]{lurie2016higher}) we see that $\SH^{S^1}_{\et}(X) \otimes \T(\Lambda)$ and $\DA_{\et}^{\eff}(X, \Lambda)$ identify with the same subcategories in $\PSh(\Sm_{/X}, \T(\Lambda)) \simeq \PSh(\Sm_{/X}, \Spt) \otimes\T(\Lambda)$. Let us denote the image of a $Y$ in $\Sm_{/Y}$ under 
\[
\Sm_{/X} \overset{\text{Yoneda}}\longrightarrow \PSh(\Sm_{/X}, \Spt) \overset{L_{\A^1} \circ L_{\et}}\longrightarrow \SH^{S^1}_{\et}(X)
\]
by $\mathbb{S}(Y)$ and define the Tate twist $\mathbb{S}(1)$ analogous as in \ref{noname:SymMonStrOfDA}. The canonical symmetric monoidal functor $\Spt \rightarrow \T(\Lambda)$ in $\Pr^{L,\st}$ induces a symmetric monoidal functor $\SH^{S^1}_{\et}(X) \rightarrow \DA^{\eff}(X, \Lambda)$ in $\Pr^{L,\st}$ which maps $\mathbb{S}(1)$ to $\Lambda(1)$. Hence we have a chain of equivalences
\begin{align*}
\SH_{\et}(X) \otimes \T(\Lambda) &\simeq \left( \colim_{\N} (\SH^{S^1}_{\et}(X) \overset{\_ \otimes \mathbb{S}(1)}\longrightarrow \dots ) \right) \otimes \T(\Lambda) \\
& \simeq  \colim_{\N} \left( \SH^{S^1}_{\et}(X) \otimes \T(\Lambda)  \overset{\_ \otimes \Lambda(1)}\longrightarrow \dots  \right) \\
& \simeq \DA_{\et}(X, \Lambda).  
\end{align*}

\item In fact one can replace $\T(\Lambda)$ with $\Mod_A(\Spt)$ for any $E_\infty$-ring $A$ in the constructions above in order to get a theory of \'etale motives with values in $E_\infty$-rings. Since we will not need this in the following we chose to restrict ourselves to ordinary rings in order to avoid confusion.   
\end{enumerate}
\end{remark}

\noname \label{noname:DAChangeOfCoeff}
The $\infty$-category $\Spt$ comes equipped with a $t$-structure whose heart is equivalent to the ordinary category of abelian groups (see \cite[1.4.3.6]{lurie2016higher}). Hence there is a canonical fully faithful functor
\[
H: \CRing \simeq \CAlg(\Ab) \longrightarrow \CAlg(\Spt)
\]
which maps a ring to its Eilenberg-Maclane spectrum. Consider the composition
\begin{equation} \label{eqn:FunctorCRingToPrL}
\T(\_): \CRing \overset{H}\longrightarrow \CAlg(\Spt) \overset{\Mod{\_}(\Spt)}\longrightarrow \CAlg(\Pr^{L,\st}),
\end{equation}
where the second functor is obtained by straightening the coCartesian fibration of \cite[4.5.3.1]{lurie2016higher}. This functor sends a morphism of rings $\rho: \Lambda \rightarrow \Lambda'$ to a symmetric monoidal functor
\[
\rho^*: \T(\Lambda) \longrightarrow \T(\Lambda')
\]
in $\Pr^{L,\st}$ which gives rise to a symmetric monoidal change of coefficients functor
\[
\rho^*: \DA_{\et}(X, \Lambda) \simeq \SH_{\et}(X) \otimes \T(\Lambda) \overset{\id \otimes \rho^*}\longrightarrow \SH_{\et}(X) \otimes \T(\Lambda') \simeq \DA_{\et}(X, \Lambda')
\]
in $\Pr^{L,\st}$.

\noname \label{noname:Rigidity} Let $\Et_{/X}$ denote the category of \'etale schemes of finite type over $X$. For a ring $\Lambda$ and a scheme $X$ let us denote by $\T_{\et}(X, \Lambda)$ the unbounded derived $\infty-$category associated to the abelian category $\Sh_{\et}( \Et_{/X}, \Lambda)$ of \'etale sheaves on $\Et_{/X}$ with $\Lambda$ coefficients.  By \cite[2.1.2.2]{lurie2018SAG} there is a canonical equivalence
\[
\T_{\et}(X, \Lambda) \overset{\sim}\longrightarrow \Sh^{\hyp}_{\et}(\Et_{/X}, \T(\Lambda))
\]
of symmetric monoidal stable presentable $\infty$-categories. Hence the canonical morphism of \'etale sites
\[
\gamma: (\Et_{/X}, \et) \rightarrow (\Sm_{/X}, \et)
\]
induces an symmetric monoidal and colimit preserving functor
\[
\iota^*: \T_{\et}(X, \Lambda) \simeq \Sh^{\hyp}_{\et}(\Et_{/X}, \T(\Lambda)) \overset{\gamma^*}\rightarrow   \Sh^{\hyp}_{\et}(\Sm_{/X}, \T(\Lambda)) \overset{L_{\A^1}}\rightarrow \DA_{\et}^{\text{eff}}(X, \Lambda) \overset{\Sigma^\infty}\rightarrow \DA_{\et}(X, \Lambda)
\]
between stable presentable $\infty$-categories.

\begin{thm}[Rigidity] \label{thm:Rigidity} Let $X$ be a locally noetherian scheme, $n$ a positive integer invertible in $\mathcal{O}(X)$ and $\Lambda$ a ring satisfying $n\Lambda=0$. Then 
\[
\iota^*: \T_{\et}(X, \Lambda) \longrightarrow \DA_{\et}(X, \Lambda)
\]
is an equivalence.
\end{thm}

\begin{proof}
It is easy to deduce from \cite{BachmannRigidity} a proof independend of the one given in \cite[4.1]{AyoubRealizationEtale}. For the sake of completeness we will sketch how. We freely use the notations of \cite{BachmannRigidity}.

Let us write $\omega^* := L_{\A^1} \circ \gamma^*$. The $\A^1$- invariance of \'etale cohomology (see \cite[1.3.2]{CisinskiDegliseEtale}) implies that 
\[
{\omega^* = L_{\A^1} \circ \gamma^*}: \T_{\et}(X, \Lambda) \longrightarrow \DA^{\eff}_{\et}(X, \Lambda) 
\]
is fully faithful. We claim that $\Lambda(1)$ is already tensor invertible in $\DA^{\eff}_{\et}(X, \Lambda)$. This implies that
\[
\Sigma^\infty:\DA^{\eff}_{\et}(X, \Lambda) \longrightarrow \DA_{\et}(X, \Lambda)
\]
is an equivalence. In particular $\iota^*$ is fully faithful.

It suffices to treat the case $\Lambda = \Z/ n \Z$. Consider the morphism 
\[
e_n : \G_{m,X} \longrightarrow \G_{m,X}
\]
given by elevating to the $n$-th power and let $q: \G_{m,X} \rightarrow X$ be the projection. By \cite[3.11]{BachmannRigidity} $e_n$ is canonically a $\mu_n$ torsor which becomes trivial when pulled back along the unit map $1: X \rightarrow \G_{m,S}$. In particular the torsor $e_n$ is classified by an element $\tilde{\sigma}_n$ in $H^1_{\et}(\G_{m,X}, \mu_n) \simeq \pi_0 \map_{\T_{\et}(\G_{m,S}, \Lambda)}(\1, \mu_n [1])$. The map $\tilde{\sigma}_n$ gets mapped via $\omega^*$  to a morphism $\1 \rightarrow \omega^* \mu_n[1] \simeq q^* \omega^* \mu_n[1]$ in $\DA^{\eff}_{\et}(\G_{m,X}, \Lambda)$ which corresponds to a map $\Lambda( \G_{m,X}) \rightarrow \omega^* \mu_n[1]$ in $\DA^{\eff}_{\et}(X, \Lambda)$ such that the composition
\[
\Lambda(X) \longrightarrow \Lambda( \G_{m,X}) \longrightarrow \omega^* \mu_n[1]
\]
induced by the unit map is equivalent to the zero morphism. In particular we get a map $\Lambda(1)[1] \rightarrow \omega^* \mu_n[1]$ which we denote by $\sigma_n$. Since $\mu_n$ is tensor-invertible in $\T_{\et}(X, \Lambda)$ and $\omega^*$ is symmetric monoidal it suffices to show that $\sigma_n$ is an equivalence. We may check this after pulling back to strict localizations of $X$ and hence by local noetherianess of $X$ and \cite[1.1.5]{CisinskiDegliseEtale} we can assume that $X$ is uniformly of finite \'etale cohomological dimension. 

Let $\ell^k$ be a  maximal prime power dividing $n$. The construction of $\sigma_n$ above is compatible with change of coefficients along the ring map $ \rho_{\ell^k}: \Z/n\Z \rightarrow \Z/ \ell^k \Z$ in the sense that $\rho_{\ell^k}^* \sigma_n \simeq \sigma_{\ell^k}$. The collection of functors $\rho_{\ell^k}^*$ where $\ell^k$ runs through the maximal prime powers dividing $n$  is conservative. Thus we may assume that $\Lambda = \Z / \ell^k \Z$ for a prime number $\ell$ invertible in $\mathcal{O}(X)$.

Consider the map $\sigma: \1(1)[1] \rightarrow \hat{\1}_p (1)[1]$ in  $\SH^{S^1}_{\et}(X)^\wedge_\ell$ defined in the beginning of \cite[\S 6]{BachmannRigidity}. By \cite[6.6]{BachmannRigidity} $\sigma$ is an equivalence. The canonical functor $ \_ \otimes H \Lambda:  \Spt \rightarrow \T(\Lambda)$ induces a functor
\[
\rho^*_\Lambda: \SH_{\et}^{S^1}(X)^\wedge_\ell \longrightarrow \DA_{\et}^{\eff}(X, \Lambda). 
\]
Here we used that $\DA^{\eff}_{\et}(X, \Lambda)^\wedge_\ell \simeq \DA^{\eff}_{\et}(X, \Lambda)$ since $\ell^k\Lambda = 0$. We claim that $\rho^*_\Lambda \sigma$ and $\sigma_{\ell^k}$ are equivalent. This follows from the fact that by the proof of \cite[4.5]{BachmannRigidity} both maps classify the torsor $e_{\ell^k}$.

Every object of $\DA_{\et}(X, \Lambda)$ can be written as a colimit of objects of the form $p_* \1(n)$ for $p$ proper and $n \in \Z$ (see \cite[2.2.23]{AyoubThesisI} ). As in \cite[4.4.3]{CisinskiDegliseEtale} one deduces that $\iota^*$ commutes with $p_*$ for $p$ proper. Moreover we have shown above that $\1(n)$ for $n \in \Z$ lies in the image of $\iota^*$. Combining this shows that $\iota^*$ is essentially surjective and therefore an equivalence.
\end{proof}

\noname Let $X$ be a noetherian scheme and denote by $\Sch^{\ft}_{/X}$ the category of schemes of finite type over $X$. The $h$-topology on $\Sch^{\ft}_{/X}$ is the Grothendieck-topology whose covers are universal topological epimorphisms (see \cite[3.1.2]{VoevodskyHomOfSchemes}). Let $\Sh^{\text{hyp}}_{h}(\Sch^{\ft}_{/X}, \mathcal{D}(\Lambda))$ denote the $\infty$-category of $h$-hypersheaves on $\Sch^{\ft}_{/X}$ with values in $\T(\Lambda)$ and write $\underline{\DM}^{\eff}_h(X, \Lambda)$ for the full subcategory consisting of $\A^1$-invariant objects.  The inclusion
\[
\underline{\DM}^{\eff}_h(X, \Lambda) \subset \Sh^{\text{hyp}}_{h}(\Sch^{\ft}_{/X}, \mathcal{D}(\Lambda))
\]
admits a left adjoint $L_{\A^1}$ and we can define the object $\Lambda(1)$ analogously as in \ref{noname:SymMonStrOfDA}. Let $\underline{\DM}_h(X, \Lambda)$ denote the $\N$-indexed colimit of the diagram
\[
\underline\DM_{h}^{\text{eff}}(X, \Lambda) \overset{\_ \otimes \Lambda(1)} \longrightarrow \underline\DM_{h}^{\text{eff}}(X, \Lambda) \overset{\_ \otimes \Lambda(1)}\longrightarrow \underline\DM_{h}^{\text{eff}}(X, \Lambda) \overset{\_ \otimes \Lambda(1)} \longrightarrow \dots 
\]
in $\Pr^{L}$ and denote by ${\DM}_h(X, \Lambda)$ its smallest full stable subcategory closed under small colimits and containing the objects of the form $\Lambda(Y)(n)$ for all $Y \rightarrow X$ smooth and $n$ in $\Z$. The commutative diagram of sites
\[
\begin{tikzcd}
                             & {(\Et_{/X}, \et)} \arrow[ld] \arrow[rd] &                        \\
{(\Sm_{/X}, \et)} \arrow[rr] &                                         & {(\Sch^{\ft}_{/X}, h)}
\end{tikzcd}
\]
induces a commutative diagram
\begin{equation} \label{eqn:TriangleDetDAetDMh}
\begin{tikzcd}
                                                & {\T_{\et}(X,\Lambda)} \arrow[ld, "\iota^*"'] \arrow[rd, "\iota_h^*"] &                       \\
{\DA_{\et}(X, \Lambda)} \arrow[rr, "\varphi^*"] &                                                                      & {\DM_{h}(X, \Lambda)}
\end{tikzcd}
\end{equation}
in $\Pr^{L,\st}$. As in \ref{noname:SymMonStrOfDA} we can equip ${\DM_{h}(X, \Lambda)}$ with a symmetric monoidal structure and lift the diagram above to a diagram in $\CAlg(\Pr^{L, \st})$. Note that the formation of $\varphi^*$ is compatible with $f^*$ for any $f: X \rightarrow Y$ in $\Sch^{\qcqs}_{/S}$ as well as $f_\#$ for $f$ smooth.

\begin{prop} \label{prop:pOnBaseAutomaticallyInverted}
Let $X$ be a qcqs scheme and $\Lambda$ a $\Z/ N\Z$-algebra where $N = p^r$ for some prime $p$ and positive integer $r$. We write $X[1/p] := X \times_{\Spec Z} \Spec \Z[1/p]$ and $j: X[1/p] \rightarrow X$ for the canonical open immersion. Then
\[
j^* : \DA_{\et}(X, \Lambda) \longrightarrow \DA_{\et}(X[1/p], \Lambda)
\]
and 
\[
j^* : \DM_{h}(X, \Lambda) \longrightarrow \DM_{h}(X[1/p], \Lambda)
\]
are equivalences of categories. 
\end{prop}

\begin{proof}
This is \cite[A.3.4]{CisinskiDegliseEtale}.
\end{proof}

\noname \label{noname:DADMAndBasechange} Completely analogously as in \ref{noname:DAChangeOfCoeff} we obtain for  a morphism of rings $\rho: \Lambda \rightarrow \Lambda'$ a symmetric monoidal change of coefficients functor
\[
\rho^*: \DM_{h}(X, \Lambda) \longrightarrow \DM_{h}(X, \Lambda')
\]
in $\Pr^{L,\st}$ which is compatible with the one of $\DA_{\et}$ in the sense that
\[
\begin{tikzcd}
{\DA_{\et}(X, \Lambda)} \arrow[d, "\rho^*"'] \arrow[r, "\varphi^*"] & {\DM_{h}(X, \Lambda)} \arrow[d, "\rho^*"] \\
{\DA_{\et}(X, \Lambda')} \arrow[r, "\varphi^*"]                                  & {\DM_{h}(X, \Lambda')}                                
\end{tikzcd}
\]
commutes.

\noname \label{noname:rhoQrholambda} Given a ring $\Lambda$ we write $\rho_\Q: \Lambda \rightarrow \Lambda \otimes_\Z \Q $ and $\rho_{\Z/ \ell \Z}: \Lambda \rightarrow \Lambda \otimes_ \Z \Z/ \ell \Z $ for the canonical ring morphisms, where $\ell$ ranges through the set of all prime numbers.

\begin{thm} \label{thm:CompDAvsDM}
Let $X$ be a noetherian scheme of finite dimension. Then the functor
\[
\varphi^*: \DA_{\et}(X, \Lambda) \longrightarrow \DM_{h}(X, \Lambda)
\]
is an equivalence. Moreover it commutes with the six operations when restricted to noetherian schemes of finite dimension. 
\end{thm}

\begin{proof} We use the argument of \cite[5.5.7]{CisinskiDegliseEtale}. First note that it suffices to treat the case $\Lambda = \Z$ since we obtain $\varphi^*$ from this case by applying $\_ \otimes_{\T(\Z)} \T(\Lambda)$.
Since $\varphi^*$ is compatible with $f_\#$ for smooth $f$ and commutes with small colimits it follows right away from the definition of $\DM_{h}(X, \Lambda)$ that $\varphi^*$ is essentially surjective. Hence it suffices to show that the homomorphism of abelian groups
\[
\pi_0 \map_{\DA_{\et}(X,\Z)}(M,N) \longrightarrow \pi_0 \map_{\DM_{h}(X,\Z)}(\varphi^* M, \varphi^* N)
\]
is an isomorphism for all $M$ in $\DA^{\cons}_{\et}(X,\Lambda)$ and $N$ in $\DA_{\et}(X,\Lambda)$. There are equivalences 
\[
\pi_0 \map_{\DM_{h}(X,\Z)}(M,N)/\ell \overset{\sim}\longrightarrow \pi_0 \map_{\DM_{h}(X,\Z/\ell\Z)}(\rho_{\Z/\ell \Z}^*M, \rho_{\Z/\ell \Z}^*N)
\]
by \cite[5.4.5]{CisinskiDegliseEtale} and 
\[
\pi_0 \map_{\DM_{h}(X,\Z)}(M,N) \otimes \Q \overset{\sim}\longrightarrow \pi_0 \map_{\DM_{h}(X,\Q)}(\rho_{\Q}^*M, \rho_{\Q}^*N)
\]
by \cite[5.4.9]{CisinskiDegliseEtale}. It is easy to adapt these proofs and show the analogous statements for $\DA_{\et}$. Since $\varphi^*$ is compatible with change of coefficients it therefore suffices to show that $\varphi^*$ is an equivalence for $\Lambda= \Z/\ell \Z$ where $\ell$ runs through all prime numbers and for $\Lambda = \Q$. The case where $\Lambda=\Q$ follows from \cite[5.2.2]{CisinskiDegliseEtale} and \cite[16.2.18]{CisinskiDegliseBook}: Both sides are equivalent to Beilinson motives. For the case where $\Lambda= \Z/\ell \Z$ note that we may assume that $\ell$ is invertible in $\mathcal{O}(X)$ by Proposition \ref{prop:pOnBaseAutomaticallyInverted}. Hence the claim follows from the triangle (\ref{eqn:TriangleDetDAetDMh}) and the fact that both diagonal arrows are equivalences by Theorem \ref{thm:Rigidity} and \cite[5.5.4]{CisinskiDegliseEtale} respectively. The last sentence is clear since $\varphi$ is symmetric monoidal and commutes with $f^*$ for all $f$ and with $f_\#$ for smooth $f$. 
\end{proof}

\begin{prop} \label{prop:QandZ/plinearizationIsConsFam}
Let $\Lambda$ be a flat $\Z$-algebra and $X$ a finite dimensional noetherian scheme. Then the family of functors
\begin{align*}
\rho_{\mathbb{Q}}^* &: \DA_{\et}(X, \Lambda) \longrightarrow \DA_{\et}(X, \Lambda \otimes_{\Z} \mathbb{Q}) \\
\rho_{\Z/ \ell \Z}^* &: \DA_{\et}(X, \Lambda) \longrightarrow \DA_{\et}(X, \Lambda \otimes_{\Z} \Z/ \ell \Z), 
\end{align*}
is conservative, where $\ell$ runs through the set of all prime numbers.
\end{prop}

\begin{proof}
Since $\varphi^*$ in Theorem \ref{thm:CompDAvsDM} is an equivalence which is compatible with change of coefficients we may apply \cite[5.4.12]{CisinskiDegliseEtale}.
\end{proof}

\begin{definition}
Let $\DA^{\cons}_{\et}(X, \Lambda) \subset \DA_{\et}(X, \Lambda)$ be the smallest idempotent complete full stable subcategory consisting of the objects of the form $\Lambda(Y)(n)$ for all $n \in \Z$ and $f: Y \rightarrow X$ in $\Sm_{/X}$. We call an element $A$ in $\DA_{\et}(X, \Lambda)$ \textit{constructible} if it lies in $\DA^{\cons}_{\et}(X, \Lambda)$.
\end{definition}

\begin{thm}\label{thm:SixFunctorsPresConstr} Let $f: X \rightarrow Y$ be a morphism of schemes and $\Lambda$ any ring. We have the following stability properties for constructible objects under the six functors:
\begin{enumerate}
\item $f^*$ preserves constructibility.
\item $f_!$ preserves constructibility whenever $f$ is of finite type between qcqs schemes.
\item $f_*$ and $f^!$ preserve constructibility when $f$ is of finite type between quasi-excellent noetherian schemes of finite dimension. 
\item If $M,N$ are in $\DA_{\et}^{\cons}(X, \Lambda)$, then $M \otimes N$ is constructible.
\item If $M,N$ are in $\DA_{\et}^{\cons}(X, \Lambda)$ and $X$ is quasi-excellent noetherian of finite dimension, then $\Hom(M,N)$ is constructible.
\end{enumerate}
\end{thm}

\begin{proof}
The statements (1), (2) and (4) are standard and for example shown in \cite[2.56, 2.60]{Adeel6Functor}. The statements (3) and (5) are shown in \cite[6.2.14]{CisinskiDegliseEtale} for $h$-motives which implies the analogue for \'etale motives by Theorem \ref{thm:CompDAvsDM}.
\end{proof}

\noname \label{noname:DAoverDVRconstGen} Let $S$ be a strictly local noetherian scheme. Then it is shown in \cite[1.1.5]{CisinskiDegliseEtale} that any scheme $f:X \rightarrow S$ of finite type over $S$ is of finite \'etale cohomological dimension and the residue fields of $X$ are uniformly of finite \'etale cohomological dimension. This has the following consequence:

\begin{prop} \label{prop:SstrLocImpliesCompactGen}
Let $S$ be a strictly local noetherian scheme and $f:X \rightarrow S$ a morphism of finite type. Then an object $M$ in $\DA_{\et}(X, \Lambda)$ is constructible if and only if it is compact. In particular if $h:X \rightarrow Y$ is a morphism between schemes of finite type over $S$ then $h_*$ and $h^!$ preserve small colimits.
\end{prop}
\begin{proof}
The first part is \cite[5.2.4]{CisinskiDegliseEtale}. The last sentence is an easy consequence of the first part: Since $h_!$ and $h^*$ preserve compact objects by Theorem \ref{thm:SixFunctorsPresConstr} it is straightforward to check that their right adjoints commute with small colimits.
\end{proof}

\begin{remark}
In the following we will almost exclusively be concerned with schemes over a base $S$, where $S$ is the spectrum of a strictly henselian discrete valuation ring. Theorem \ref{thm:SixFunctorsPresConstr} tells us that we have to impose excellency on $S$ whenever we want to preserve constructibility. Moreover whenever we want compact generation by constructible objects Proposition \ref{prop:SstrLocImpliesCompactGen} tells us that we have to consider schemes of finite type over $S$.
\end{remark}

\begin{remark} Let us fix a qcqs base scheme $S$. The change of coefficients functor constructed in \ref{noname:DAChangeOfCoeff} commutes with the formation of $f^*$ for $f:X \rightarrow Y$ in $\Sch_{/S}^{\qcqs}$. We can express this by a functor
\begin{equation} \label{eqn:FunctorSchxRingToPrL}
\DA_{\et}(\_, \_): (\Sch_{/S}^{\op}) \times \CRing \longrightarrow \CAlg(\Pr^{L,\st}), 
\end{equation}
which restricts to (\ref{eqn:FunctorDA*}) when fixing a ring $\Lambda$. 

Indeed giving a functor $F: \C \rightarrow \CAlg(\Pr^{L,\st})$, where $\C$ is a $\infty$-category with coproducts, is equivalent to giving a map of $\infty$-operads $\tilde{F}: \C^{\coprod} \rightarrow (\widehat{\Cat}_\infty)^\times$ which factors through $(\Pr^{L,\st})^{\otimes} \subset (\widehat{\Cat}_\infty)^\times$ by \cite[2.4.3.18]{lurie2016higher}. Hence (\ref{eqn:FunctorSH*}) and (\ref{eqn:FunctorCRingToPrL}) give rise to maps of infinity operads $(\Sch_{/S}^{\op})^{\coprod} \rightarrow (\widehat{\Cat}_\infty)^\times$ and $\CRing^{\coprod} \rightarrow (\widehat{\Cat}_\infty)^\times$ respectively. Equivalently they give rise to weak cartesian structures $(\Sch_{/S}^{\op})^{\coprod} \rightarrow \widehat{\Cat}_\infty$ and $\CRing^{\coprod} \rightarrow \widehat{\Cat}_\infty$ by \cite[2.4.1.7]{lurie2016higher}. Given two $\infty$-categories $\C$ and $\mathcal{D}$ let us note that by universal property (see \cite[2.4.3.1]{lurie2016higher}) there exists a canonical map of simplicial sets $(\mathcal{C} \times \mathcal{D})^{\coprod} \rightarrow \mathcal{C}^{\coprod} \times \mathcal{D}^{\coprod}$ which makes the diagram
\[
\begin{tikzcd}
(\mathcal{C} \times \mathcal{D})^{\coprod} \arrow[d,"p_{\mathcal{C} \times \mathcal{D}}"'] \arrow[r] & \mathcal{C}^{\coprod} \times \mathcal{D}^{\coprod} \arrow[d,"p_{\mathcal{C}} \times p_{\mathcal{D}}"] \\
\Fin_* \arrow[r, "\Delta"]                                   & \Fin_* \times \Fin_*                                    
\end{tikzcd}
\]
commute. Consider the composition
\[
(\Sch_{/S}^{\op} \times \CRing)^{\coprod}  \overset{(1)}\longrightarrow (\Sch_{/S}^{\op})^{\coprod} \times \CRing^{\coprod} 
\overset{(2)}\longrightarrow \widehat{\Cat}_\infty \times \widehat{\Cat}_\infty 
\overset{(3)}\longrightarrow \widehat{\Cat}_\infty,
\]
where $(1)$ is the canonical map considered above, (2) is the product of the weak cartesian structures associated to (\ref{eqn:FunctorSH*}) and (\ref{eqn:FunctorCRingToPrL}) respectively and (3) is the product given by the cartesian monoidal structure of $\widehat{\Cat}_\infty$. It is straightforward to see that this is a weak cartesian structure and the associated map of $\infty$-operads $(\Sch_{/S}^{\op} \times \CRing)^{\coprod} \rightarrow (\widehat{\Cat}_\infty)^\times$ factors through $(\Pr^{L,\st})^\otimes$. This produces the desired functor (\ref{eqn:FunctorSchxRingToPrL}).

For a morphism $f: X \rightarrow Y$ in $\Sch^{\qcqs}_{/S}$ and a morphism of rings $\rho: \Lambda \rightarrow \Lambda'$ the functoriality exhibited by  (\ref{eqn:FunctorSchxRingToPrL}) gives rise to an equivalence 
\[
\rho^* f^* \overset{\sim}\longrightarrow  f^* \rho^*
\]
for which we get an associated exchange map
\[\rho^* f_* \longrightarrow  f_* \rho^*.\]
As in \ref{noname:FunctorialityOfDA} we may apply the formalism developed in \cite[\S 3, \S 4]{LiuZhengEnhanced} in order to extend (\ref{eqn:FunctorSchxRingToPrL}) to a functor
\begin{equation} \label{eqn:FunctorFromDelta*22(Sch x ring)ToCat}
\delta^*_{2,\lbrace 2 \rbrace}(((\Sch_{/S}^{\qcqs})^{\op} \times \CRing)^{\sqcup, \op})^{cart}_{F', all} \longrightarrow \Pr^{L, \st},
\end{equation}
where $F'$ denotes the set of morphisms locally of finite type. Restricting (\ref{eqn:FunctorFromDelta*22(Sch x ring)ToCat}) to the first direction and forgetting the operadic structures gives rise to a functor
\begin{equation} \label{eqn:D!SchRingsCat}
\DA_{\et}(\_ , \_)_!: \Sch_{/S}^{\qcqs,F'} \times \CRing \longrightarrow \Pr^{L, \st}.
\end{equation}
In particular for a morphism $f: X \rightarrow Y$ in $\Sch^{\qcqs}_{/S}$  which is locally of finite type and a morphism of rings $\rho: \Lambda \rightarrow \Lambda'$ the functoriality of (\ref{eqn:D!SchRingsCat}) gives rise to a natural equivalence 
\[
\rho^* f_! \overset{\sim}\longrightarrow  f_! \rho^*.
\]
To this natural equivalence we can associate the exchange map
\[\rho^* f^! \longrightarrow  f^! \rho^*.\]
\end{remark}

\begin{prop} \label{prop:ChangeOfCoeffCompWSixFunctors}
Consider a morphism of rings $\rho: \Lambda \rightarrow \Lambda'$.
\begin{enumerate}
\item Let $f: X \rightarrow Y$ be a morphism of finite type between noetherian schemes of finite dimension. Then the comparison maps
\[
\rho^* f_* M \longrightarrow  f_* \rho^* M
\]
and
\[
\rho^* f^! N \longrightarrow  f^! \rho^* N
\]
constructed above are equivalences for all $M$ in $\DA_{\et}(X, \Lambda)$ and $N$ in $\DA_{\et}(Y, \Lambda)$.
\item Let $X$ be a noetherian scheme of finite dimension, $M$ in $\DA_{\et}^{\cons}(X, \Lambda)$ and $N$ in $\DA_{\et}(X, \Lambda)$. Then the comparison map
\[
\rho^* \Hom(M,N) \longrightarrow  \Hom(\rho^*M, \rho^*N )
\]
obtained as the transpose of
\[
\Hom(M,N) \rightarrow \Hom(M, \rho_*\rho^*N ) \simeq \rho_* \Hom(\rho^*M, \rho^*N ) 
\]
is an equivalence. 
\end{enumerate}
\end{prop}

\begin{proof}
Part (1) can be proven as in \cite[6.3]{AyoubRealizationEtale} where we can drop the Hypothesis on the cohomological dimension by using \cite[5.5.10]{CisinskiDegliseEtale}. It suffices to prove (2) in the case where $M = f_\# \1$ for some smooth $f: W \rightarrow X$. Then (2) reduces to showing that the comparison map $\rho^* f_* f^* N \rightarrow  f_* f^* \rho^* N$ is an equivalence, which is true by (1).
\end{proof}

\begin{definition}
\label{def:Realization} Let $X$ be a locally noetherian scheme, $\Lambda$  a ring and $J \subset \Lambda$ an ideal such that $\Lambda/J$ is of positive characteristic invertible in $\mathcal{O}(X)$. Write $\rho_{J}$ for the canonical ring map $\Lambda \rightarrow \Lambda / J$. We call the composition 
\[
\mathfrak{R}_{\modulo J}: \DA_{\et}(X, \Lambda) \overset{\rho_J^*} \longrightarrow \DA_{\et}(X, \Lambda/ J) \overset{(\iota^*)^{-1}}\longrightarrow \T_{\et}(X, \Lambda/ J)
\]
the \textit{mod $J$ \'etale realization functor.} 
\end{definition}

\begin{remark}
The formation of $\mathfrak{R}_{\modulo J}$ commutes with the six functors under the mild conditions of Proposition \ref{prop:ChangeOfCoeffCompWSixFunctors}. Moreover if $\Lambda$ is flat over $\Z$ and $J = n\Lambda$ for some positive integer $n$ we can deduce as in \cite[5.4.5]{CisinskiDegliseEtale} that the natural transformations of Proposition \ref{prop:ChangeOfCoeffCompWSixFunctors} are equivalences without any finiteness conditions. 
\end{remark}

\section{Specialization systems}

\noname Let us recall Ayoub's formalism of specialization systems (see \cite[\S 3.2]{AyoubThesisII}).

\noname \label{noname:DefOfSpezSystem} Let $S$ be a qcqs scheme together with a decomposition of qcqs schemes
\[
\begin{tikzcd}
\sigma \arrow[r, "i"] & S & \eta = S\setminus \sigma \arrow[l, "j"']
\end{tikzcd}\]
where $i$ is a closed immersion and $j$ is its open complement. A morphism of schemes  $f: X \rightarrow S$ gives rise to a diagram
\[
\begin{tikzcd}
X_ \sigma \arrow[d, "f_\sigma"'] \arrow[r, "i"] & X \arrow[d, "f"] & X_\eta \arrow[d, "f_\eta"] \arrow[l, "j"'] \\
\sigma \arrow[r, "i"']                          & S                & \eta \arrow[l, "j"]                       .
\end{tikzcd}
\]
Here and throughout the following by slight abuse of notation we denote any pullback of $j$ and $i$ by a map $f$ again by $j$ and $i$. 

\begin{definition} \label{def:SpezSys}
A \textit{specialization system} (for $\DA_{\et}( \_ , \Lambda)$) over $(S,i,j)$ is a collection of functors $\sp_f: \DA_{\et}(X_\eta, \Lambda) \rightarrow \DA_{\et}(X_\sigma, \Lambda)$ for every $f: X \rightarrow S$ in $\Sch^{\qcqs}_{/S}$ such that the following hold:
\begin{enumerate}
\item For any map
\begin{equation} \label{eqn:TriangleForNonameSpecSys}
\begin{tikzcd}
X \arrow[rd, "f"'] \arrow[rr, "h"] &   & Y \arrow[ld, "g"] \\
                                   & S &                  
\end{tikzcd}
\end{equation}
in $\Sch^{\qcqs}_{/S}$ there is a natural transformation
\[
\Ex^*: (h_\sigma)^* \sp_g\rightarrow  \sp_f (h_\eta)^* ,
\]
which is an equivalence whenever $h$ is smooth. 
\item For 
\[
\begin{tikzcd}
X \arrow[rrd, "f"'] \arrow[rr, "h"] &  & Y \arrow[d, "g"] \arrow[rr, "k"] &  & Z \arrow[lld, "l"] \\
                                    &  & S                                &  &                   
\end{tikzcd}
\]
in $\Sch^{\qcqs}_{/S}$ the square
\[
\begin{tikzcd}
(h_\sigma)^* (k_\sigma)^* \sp_l \arrow[d, "\sim", no head] \arrow[r, "\Ex^*"] & (h_\sigma)^*\sp_g (k_\eta)^* \arrow[r, "\Ex^*"] & \sp_f (h_\eta)^* (k_\eta)^* \arrow[d, "\sim", no head] \\
(kh)_\sigma^*\sp_l \arrow[rr, "\Ex^*"]                             &                                 & \sp_f (kh)_\eta^*                            
\end{tikzcd}
\]
commutes. 

\item For a map of the form (\ref{eqn:TriangleForNonameSpecSys}) the natural transformation
\[
\Ex_*: \sp_g (h_\eta)_* \rightarrow   (h_\sigma)_* \sp_f,
\]
obtained as the composition
\[
\sp_g (h_\eta)_* \overset{\unit}\longrightarrow (h_\sigma)_* (h_\sigma^*) \sp_g (h_\eta)_* \overset{\Ex^*}\longrightarrow (h_\sigma)_*  \sp_f (h_\eta)^*(h_\eta)_* \overset{\counit}\longrightarrow (h_\sigma)_* \sp_f
\]
is an equivalence whenever $h$ is proper. 
\end{enumerate}
\end{definition} 

\noname It is shown in \cite[\S 3.2]{AyoubThesisII} that any specialization system is compatible with Thom twists from which one can deduce for any map of the form (\ref{eqn:TriangleForNonameSpecSys}) the existence of natural transformations
\[
\Ex^!: \sp_f (h_\eta)^! \rightarrow   (h_\sigma)^! \sp_g
\]
and
\[
\Ex_!: (h_\sigma)_! \sp_f \rightarrow  \sp_g (h_\eta)_!
\]
satisfying various compatibilities (see \cite[\S 3.2]{AyoubThesisII} for more details). 

\begin{definition}
Given two specialization systems $\sp$ and $\sp'$ over $(S,i,j)$ a \textit{morphism of specialization systems} is a family of natural transformations $\sp_f \rightarrow \sp_f'$ for every $f: X \rightarrow S$ in $\Sch^{\qcqs}_{/S}$ such that for any map of the form (\ref{eqn:TriangleForNonameSpecSys}) the induced square
\[
\begin{tikzcd}
(h_\sigma)^* \sp_g \arrow[r, "\Ex^*"] \arrow[d] &  \sp_f (h_\eta)^* \arrow[d] \\
(h_\sigma)^* \sp'_g \arrow[r, "\Ex^*"]          &  \sp'_f (h_\eta)^*         
\end{tikzcd}
\]
commutes.
\end{definition}

\begin{definition}
We call a specialization system $\sp$ over $(S,i,j)$ \textit{lax-monoidal} if it induces on homotopy categories a pseudo-monoidal specialization system in the sense of \cite[3.1.12]{AyoubThesisII}.
Given two lax-monoidal specialization systems $\sp$ and $\sp'$ a morphism of specialization systems $\sp \rightarrow \sp'$ is called \textit{lax-monoidal} if for all $f:X \rightarrow S$ in $\Sch^{\qcqs}_{/S}$ the natural transformation $\sp_f \rightarrow \sp'_f$ induces a monoidal natural transformation on homotopy categories.
\end{definition}

\begin{example}
Consider the collection of functors $\chi_f := i^* j_*: \DA_{\et}(X_\eta, \Lambda) \rightarrow \DA_{\et}(X_\sigma, \Lambda)$ for all $f:X \rightarrow S$ in $\Sch^{\qcqs}_{/S}$. Smooth and proper base change imply immediately that this gives rise to a specialization system. Moreover it is lax monoidal since $j_* $ is lax-monoidal. 
\end{example}

\section{Nearby cycles functors for \'etale motives}

\noname We want to lift Ayoub's construction of the motivic nearby cycles functors in \cite{AyoubThesisII} and \cite{AyoubRealizationEtale} to our $\infty$-categorical setup. For this we have to talk about motives on a diagram of schemes.

\begin{construction}  \label{noname:ConstOfDAForiagrams} Consider the category $\DiaSch_{/S}$ whose objects $(\mathcal{F},I)$ are functors $\mathcal{F}: I \rightarrow \Sch^{\qcqs}_{/S}$, where $I$ is a small category, and a morphism $(\theta, \alpha): (\mathcal{G},J) \rightarrow (\mathcal{F},I)$ is the data of a functor $\alpha: J \rightarrow I$ and a natural transformation $\theta: \mathcal{G} \rightarrow \mathcal{F} \circ\alpha$. The functor
\[
\Sm_{/\_}: \left(\Sch^{\qcqs}_{/S}\right)^{\op} \longrightarrow \Cat_1 \subset \Cat_\infty,
\]
which assigns to a qcqs scheme $X$ over $S$ the category $\Sm_{/X}$, gives rise to a cartesian fibration $\underline{\Sm} \rightarrow \Sch^{\qcqs}_{/S}$. For a $(\mathcal{F},I)$ in $\DiaSch_{/S}$ we define the  category $\Sm_{/(\mathcal{F},I)}$ via the cartesian diagram of categories
\[
\begin{tikzcd}
{\Sm_{/(\mathcal{F} ,I)}} \arrow[r] \arrow[d] & \underline{\Sm} \arrow[d] \\
I \arrow[r, "\mathcal{F}"]                   & \Sch^{\qcqs}_{/S}.          
\end{tikzcd}
\]
An object of ${\Sm_{/(\mathcal{F},I)}}$ is a tuple $\{U \rightarrow \mathcal{F}(i), i \}$, where $i \in I$ and $U \rightarrow \mathcal{F}(i)$ is a smooth morphism, and a map $\{U \rightarrow \mathcal{F}(i), i \} \rightarrow \{V \rightarrow \mathcal{F}(j), j \}$ is the data of a map $i \rightarrow j$ in $I$ together with a map $U \rightarrow V \times_{\mathcal{F}(j)} \mathcal{F}(i) $ in $\Sm_{/\mathcal{F}(i)}$. 

For any $(\mathcal{F},I)$ in $\DiaSch_{/S}$ we define the \'etale topology on  $\Sm_{/(\mathcal{F},I)}$ as the Grothendieck topology generated by the families of maps $\{ (U_\alpha, i) \rightarrow (U,i)\}_\alpha$ where $\{U_\alpha \rightarrow U \}_\alpha$ is a covering family for the ordinary \'etale topology. 

\end{construction}

\begin{construction}  Let $\mathcal{F}: I \rightarrow \Sch^{\qcqs}_{/S}$ be a functor and $\Lambda$ a ring. We set
\[
\PSh(\Sm_{/(\mathcal{F},I)}, \T(\Lambda)):= \PSh(\Sm_{/(\mathcal{F},I)})\otimes \T(\Lambda).
\]
Let 
\[
\Sh^{\hyp}_{\et}(\Sm_{/(\mathcal{F},I)}, \T(\Lambda)) \subset \PSh(\Sm_{/(\mathcal{F},I)}, \T(\Lambda))
\]
be the full subcategory consisting of hypersheaves with respect to the \'etale topology. A sheaf $\mathcal{G}$ in $\Sh^{\hyp}_{\et}(\Sm_{/(\mathcal{F},I)}, \T(\Lambda))$ is called $\A^1$-local if $\mathcal{G}(\pi_X,i): \mathcal{G}(X,i) \rightarrow \mathcal{G}(\A^1_X,i)$ is an equivalence for all projection maps $(\pi_X,i):(\A^1_X,i) \rightarrow (X,i)$ in $\Sm_{/(\mathcal{F},I)}$. We denote by
\[
\DA^{\eff}_{\et}((\mathcal{F},I), \Lambda) \subset \Sh^{\hyp}_{\et}(\Sm_{/(\mathcal{F},I)}, \T(\Lambda))
\]
the full subcategory consisting of $\A^1$-local objects. The inclusion $\DA^{\eff}_{\et}((\mathcal{F},I), \Lambda) \subset \Sh_{\et}(\Sm_{/(\mathcal{F},I)}, \T(\Lambda))$ admits a left adjoint which we denote by $L_{\A^1}$. As in \ref{noname:SymMonStrOfDAeff} all the categories constructed above underlie a symmetric monoidal $\infty$-category and the localization functors are symmetric monoidal.  Clearly by construction we have for any scheme $X$ considered as a diagram $X: * \rightarrow \Sch^{\qcqs}_{/S}$ that $\DA^{\eff}_{\et}((X,*), \Lambda) \simeq \DA^{\eff}_{\et}(X, \Lambda)$, where $\DA^{\eff}_{\et}(X, \Lambda)$ is defined as in \ref{noname:ConstrOfDA}.
\end{construction} 

\begin{construction} \label{constr:MapsOfDiagrams} We may always factor a morphism $(\theta, \alpha): (\mathcal{G},J) \rightarrow (\mathcal{F},I)$ in $\DiaSch_{/S}$ as
\[
(\mathcal{G},J) \overset{(\theta, \id)}\longrightarrow (\mathcal{F} \circ \alpha, J) \overset{(\id, \alpha)}\longrightarrow (\mathcal{F},I). 
\]
For a morphism of the form $(\theta, \id): (\mathcal{G},J) \rightarrow (\mathcal{H},J)$ we can define a functor 
\[
\bar{\theta}: \Sm_{/(\mathcal{H},J)} \longrightarrow \Sm_{/(\mathcal{G},J)}
\]
by mapping a smooth morphism $U \rightarrow \mathcal{H}(j)$ to $U \times_{\mathcal{H}(j)} \mathcal{G}(j) \rightarrow \mathcal{G}(j)$. This induces via left Kan extension a functor
\[
(\theta_{\PSh})^*: \PSh(\Sm_{/(\mathcal{H},J)}, \T(\Lambda)) \longrightarrow \PSh(\Sm_{/(\mathcal{G},J)}, \T(\Lambda)),
\]
in $\Pr^{L,\st}$ which can be lifted to a symmetric monoidal functor since $\bar{\theta}$ preserves finite products. We denote the right adjoint of $(\theta_{\PSh})^*$ by $(\theta_{\PSh})_*$. If $\theta$ is levelwise smooth, i.e. $\theta(j): \mathcal{G}(j) \rightarrow \mathcal{H}(j)$ is smooth for all $j$ in $J$, then we may define a functor 
\[
\tilde\theta: \Sm_{/(\mathcal{G},J)} \longrightarrow \Sm_{/(\mathcal{H},J)}
\]
by sending $\{ U \rightarrow \mathcal{G}(j) ,j \}$ to $\{ U \rightarrow \mathcal{G}(j) \rightarrow \mathcal{H}(j) ,j \}$. This induces by left Kan extension a functor
\[
(\theta_{\PSh})_\#: \PSh(\Sm_{/(\mathcal{G},J)}, \T(\Lambda)) \longrightarrow \PSh(\Sm_{/(\mathcal{H},J)}, \T(\Lambda))
\]
whose right adjoint is equivalent to $(\theta_{\PSh})^*$. A morphism of the form $(\id, \alpha): (\mathcal{H} \circ \alpha, J) \rightarrow (\mathcal{H},I)$ for some functor $\alpha: J \rightarrow I$ gives rise to a functor
\[
\bar{\alpha}:  \Sm_{/(\mathcal{H} \circ \alpha ,J)} \longrightarrow \Sm_{/(\mathcal{H} ,I)},
\]
which sends $\{ U \rightarrow \mathcal{H}(j), j\}$ to $\{ U \rightarrow \mathcal{H}(j), \alpha(j)\}$. This induces by precomposition a functor
\[
\bar{\alpha}_*: \PSh(\Sm_{/(\mathcal{H} \circ \alpha,J)}, \T(\Lambda)) \longrightarrow \PSh(\Sm_{/(\mathcal{H},I)}, \T(\Lambda))
\]
which admits a left and a right adjoint. We write $(\alpha_{\PSh})^*:= \bar{\alpha}_*$ and denote by $(\alpha_{\PSh})_*$ its  right adjoint and by $(\alpha_{\PSh})_\#$  its left adjoint. 

Given a morphism of the form $(\theta, \id): (\mathcal{G},J) \rightarrow (\mathcal{H},J)$ we can observe that $(\theta_{\PSh})_*$ preserves $\A^1$-local \'etale hypersheaves. Hence it restricts to a functor
\[ 
(\theta_{\eff})_* : \DA^{\eff}_{\et}((\mathcal{G},J), \Lambda) \longrightarrow \DA^{\eff}_{\et}((\mathcal{H},J), \Lambda),
\]
whose left adjoint is given by 
\[
(\theta_{\eff})^*: \DA^{\eff}_{\et}((\mathcal{H},J), \Lambda) \overset{\theta^*|_{\DA^{\eff}_{\et}((\mathcal{H},J), \Lambda)}}\longrightarrow \PSh(\Sm_{/(\mathcal{G},J)}, \T(\Lambda)) \overset{L_{\A^1} \circ L_{\et}}\longrightarrow \DA^{\eff}_{\et}((\mathcal{G},J), \Lambda).
\]
In the case that $(\theta, \id)$ is level-wise smooth the left adjoint of $(\theta_{\eff})^*$ is given by the composition
\[
(\theta_{\eff})_\#:\DA^{\eff}_{\et}((\mathcal{G},J), \Lambda) \overset{(\theta_{\PSh})_\#|_{\DA^{\eff}_{\et}((\mathcal{H},J), \Lambda)}}\longrightarrow \PSh(\Sm_{/(\mathcal{H},J)}, \T(\Lambda)) \overset{L_{\A^1} \circ L_{\et}}\longrightarrow \DA^{\eff}_{\et}((\mathcal{H},J), \Lambda).
\] 
Similarly for a morphism of the form $(\id, \alpha): (\mathcal{H} \circ \alpha, J) \rightarrow (\mathcal{H},I)$ we observe that $\bar{\alpha}_* = \alpha^*$ preserves \'etale hypersheaves and $\A^1$-local objects. As above we get an adjunction
\[
(\alpha_{\eff})_\#: \DA^{\eff}_{\et}((\mathcal{H} \circ \alpha,J), \Lambda) \longleftrightarrows \DA^{\eff}_{\et}((\mathcal{H},I), \Lambda): (\alpha_{\eff})^*.
\]

\end{construction} 

\begin{construction} \label{remark:ConstOfDAForiagramsMaps} For any $(\mathcal{F},I)$ in $\DiaSch_{/S}$ there is a unique structure map $\pi_{(\mathcal{F},I)}:(\mathcal{F},I) \rightarrow  (S, *)$ which induces a functor
\[
\pi_{(\mathcal{F},I)}^* : \DA_{\et}^{\text{eff}}(S, \Lambda)\longrightarrow \DA_{\et}^{\text{eff}}((\mathcal{F},I), \Lambda).
\]
Consider the Tate object $\Lambda(1)$ in $\DA_{\et}^{\text{eff}}(S, \Lambda)$ as defined in \ref{noname:SymMonStrOfDA} and let us write by slight abuse of notation $\Lambda(1) := \pi_{(\mathcal{F},I)}^*\Lambda(1)$. We define $\DA_{\et}((\mathcal{F},I), \Lambda)$ as the colimit in $\Pr^L$ of the $\mathbb{N}$-indexed diagram
\[
\DA_{\et}^{\text{eff}}((\mathcal{F},I), \Lambda) \overset{\_ \otimes \Lambda(1)} \longrightarrow\DA_{\et}^{\text{eff}}((\mathcal{F},I), \Lambda) \overset{\_ \otimes \Lambda(1)}\longrightarrow \DA_{\et}^{\text{eff}}((\mathcal{F},I), \Lambda) \overset{\_ \otimes \Lambda(1)} \longrightarrow \dots 
\]
Similarly as in \ref{noname:f*onDA} we get induced adjoint pairs
\[ 
\theta^* : \DA_{\et}((\mathcal{H},J), \Lambda) \longleftrightarrows \DA_{\et}((\mathcal{G},J), \Lambda): \theta_*
\]
for morphisms of the form $(\theta, \id): (G,J) \rightarrow (H,J)$,
\[ 
\theta_\# : \DA_{\et}((\mathcal{G},J), \Lambda) \longleftrightarrows \DA_{\et}((\mathcal{H},J), \Lambda): \theta^*
\]
if $(\theta, \id)$ is level-wise smooth and
\[
\alpha_\#: \DA_{\et}((\mathcal{H} \circ \alpha,J), \Lambda) \longleftrightarrows \DA_{\et}((\mathcal{H},I), \Lambda): \alpha^*
\]
for morphisms of the form $(\id, \alpha): (\mathcal{H} \circ \alpha, J) \rightarrow (\mathcal{H},I)$. As in \ref{noname:SymMonStrOfDA} and \ref{noname:f*onDA} we can deduce that $\DA_{\et}((\mathcal{F},I), \Lambda)$ underlies a symmetric monoidal $\infty$-category and $\theta^*$ as well as $\alpha^*$ lift to symmetric monoidal functors. 
\end{construction} 

\begin{remark} \label{rem:DA(diagram)ConstCoinidesWAyoub}
It is not hard to check that the homotopy category of $\DA_{\et}((\mathcal{F},I), \Lambda)$ is equivalent to the triangulated category constructed in \cite[4.5.2]{AyoubThesisII}. Moreover the functors on homotopy categories induced by $\theta_\# \dashv \theta^* \dashv  \theta_*$ and $\alpha_\#  \dashv \alpha^*$ agree with the ones constructed in \textit{loc. cit}.
\end{remark}  

\noname \label{noname:SimplObjAssToSpan} Let $\C$ be a 1-category with finite products and consider a cospan
\[
A \overset{f}\longrightarrow B \overset{g}\longleftarrow C
\]
in $\C$. Then as in \cite[3.4.1]{AyoubThesisII} we can associate to this cospan a cosimplicial object
\[
\BarConstr(f,g): \Delta \longrightarrow \C
\]
with
\[
\BarConstr(f,g)([i]) = A \times B^{i} \times C.
\]
for $[i]$ in $\Delta$. The coface and codegeneracy maps are given by:
\begin{enumerate}
\item $d_0: (a, b_0,\dots, b_i, c) \mapsto (a, f(a), b_0,\dots, b_i, c)$.
\item $d_k: (a, b_0,\dots, b_i, c) \mapsto (a, b_0,\dots, b_k, b_k, \dots b_i, c)$ for $0<k<i$.
\item $d_i: (a, b_0,\dots, b_i, c) \mapsto (a, b_0,\dots, b_i, g(c), c)$.
\item $s_k:(a, b_0,\dots, b_i, c) \mapsto (a, b_0,\dots, \widehat{b_k}, \dots, b_i, c)$ for $0 \leq k \leq i$.
\end{enumerate}

Let us denote the category of cospans in $\C$ by $\C^{\rightarrow \leftarrow}$. It is easy to see that this construction gives rise to a functor
\[
\BarConstr(\_, \_) : \C^{\rightarrow \leftarrow} \longrightarrow \Fun (\Delta , \C). 
\]

\begin{construction} \label{noname:ConstructionOfMAthcalA} Let $S$ be a qcqs scheme and consider the cospan
\[
\G_{m,S} \overset{\id}\longrightarrow \G_{m,S} \overset{1}\longleftarrow S
\] 
in $\Sch^{\qcqs}_{/S}$ where $1$ denotes the unit section. As in \ref{noname:SimplObjAssToSpan} we can associate to such a cospan a cosimplicial object
\[
\mathcal{A}_S:= \BarConstr(\id, 1) : \Delta \longrightarrow \Sch^{\qcqs}_{/S}
\]
with
\[
\mathcal{A}_S([i]) \simeq \G_{m,S} \times_S (\G_{m,S})^{\times i} 
\]
for $[i]$ in $\Delta.$ Let us denote the constant diagram $\Delta \rightarrow * \overset{\G_{m,S}}\rightarrow \Sch^{\qcqs}_{/S}$ in $\DiaSch_{/ S}$ by $(\G_{m,S}, \Delta)$. The first projections $p_1: \G_{m,S} \times_S (\G_{m,S})^{\times i} \rightarrow \G_{m,S}$ induce a canonical map 
\begin{equation} \label{eqn:ThetaA}
\theta^\mathcal{A}: (\mathcal{A}_S, \Delta) \rightarrow (\G_{m,S}, \Delta)
\end{equation}
in $\DiaSch_{/S}$. In particular we can observe that $\mathcal{A}_S$ factors as 
\[
\mathcal{A}_S: \Delta \longrightarrow \Sm_{/{\G_{m,S}}} \longrightarrow \Sch^{\qcqs}_{/S}. 
\]
For a positive integer $n$ let $e_n: \G_{m,S} \rightarrow \G_{m,S}$ be the morphism given by elevating to the $n$-th power.

The morphism of cospans
\[
\begin{tikzcd}
\G_{m,S} \arrow[d, "e_n"'] \arrow[r, "\id"] & \G_{m,S}  \arrow[d, "e_n"] & S \arrow[d, "\id"] \arrow[l, "1"'] \\
\G_{m,S} \arrow[r, "\id"']                  & \G_{m,S}                   & S \arrow[l, "1"]                  
\end{tikzcd}
\]
gives rise to a map 
\begin{equation} \label{eqn:enAtoA}
e_n: \mathcal{A}_S \rightarrow \mathcal{A}_S
\end{equation}
in $\Fun( \Delta, \Sch^{\qcqs}_{/ S})$ which is given level-wise by 
\[
(e_n, (e_n)^i): \mathcal{A}_S([i]) \simeq \G_{m,S} \times_S (\G_{m,S})^{\times i}  \longrightarrow \G_{m,S} \times_S (\G_{m,S})^{\times i}\simeq  \mathcal{A}_S([i]).
\]

Let us denote the object in $\Fun(\Delta, \Sch^{\qcqs}_{/S})$ associated to the cospan
\[
\G_{m,S} \overset{e_n}\longrightarrow \G_{m,S}  \overset{1}\longleftarrow S
\] 
by $e_n^* \mathcal{A}_S$. As above one observes that $e_n^* \mathcal{A}_S$ takes values in $\Sm_{/ \G_{m,S}}$. Moreover there is a canonical map 
\begin{equation} \label{eqn:varphin:Atoen*A}
\varphi_n: \mathcal{A}_S \longrightarrow e_n^* \mathcal{A}_S
\end{equation}
in $\Fun( \Delta, \Sm_{/\G_{m,S}})$ induced by the morphism of cospans
\[
\begin{tikzcd}
\G_{m,S} \arrow[d, "\id"'] \arrow[r, "\id"] & \G_{m,S}  \arrow[d, "e_n"] & S \arrow[d, "\id"] \arrow[l, "1"'] \\
\G_{m,S} \arrow[r, "e_n"']                  & \G_{m,S}                   & S. \arrow[l, "1"]                  
\end{tikzcd}
\]

\end{construction}

\noname \label{noname:SetupEtaleNearbyCycle} Let $S$ be the spectrum of a strictly henselian discrete valuation ring $R$ with fixed uniformizer $\pi$. We denote the open point of $S$ by $\eta$ and the closed point by $\sigma$. The uniformizer $\pi$ gives rise to morphisms $\pi: S \rightarrow \A^1_S$ and $\pi: \eta \rightarrow \G_{m,S}$. To fix notations we consider for a morphism of schemes $f: X \rightarrow S$ the diagram
\[
\begin{tikzcd}
X_\eta \arrow[r, "j"] \arrow[d, "f_\eta"'] & X \arrow[d, "f"']  & X_\sigma \arrow[l, "i"'] \arrow[d, "f_\sigma"] \\
\eta \arrow[d, "\pi"'] \arrow[r, "j"]              & S \arrow[d, "\pi"'] & \sigma \arrow[d, "\pi_\sigma"] \arrow[l, "i"']               \\
{\G_{m,S}} \arrow[r, "j"]                       & \A^1_S             & S \arrow[l, "0"']                             
\end{tikzcd}
\]
consisting of pullback squares.

For a small category $I$ let us denote by $(X,I)$ the constant diagram $I \rightarrow * \overset{X} \rightarrow \Sch^{\qcqs}_{/S}$. We write $p_I:(X, I) \rightarrow (X, *)$ for the canonical morphism in $\DiaSch_{/S}$. For a diagram $\mathcal{B}_S: I \rightarrow \Sch^{\qcqs}_{/S}$ equipped with a morphism $\theta^\mathcal{B}: (\mathcal{B}_S,I) \rightarrow (\G_{m,S},I)$  in $\DiaSch_{/S}$ we consider the following diagram by taking pullbacks in $\DiaSch_{/S}$:
\[
\begin{tikzcd}
{(\mathcal{B}_{f}, I)} \arrow[d] \arrow[r, "\theta_f^\mathcal{B}"] & {(X_\eta, I)} \arrow[d, "{(f_\eta,I)}"] \arrow[r, "j"] & {(X,I)} \arrow[d, "{(f, I)}"]   & {(X_\sigma,I)} \arrow[d, "{(f_\sigma, I)}"] \arrow[l, "i"'] \\
{(\mathcal{B}_{\id}, I)} \arrow[d] \arrow[r, "\theta_{\id}^\mathcal{B}"]     & {(\eta, I)} \arrow[d, "{(\pi,I)}"] \arrow[r, "j"] & {(S, I)} \arrow[d, "{(\pi,I)}"] & {(\sigma,I)} \arrow[d, "{(\pi_\sigma,I)}"] \arrow[l, "i"']  \\
{(\mathcal{B}_S, I)} \arrow[r, "\theta^\mathcal{B}"]                      & {(\G_{m,S}, I)} \arrow[r, "j"]                         & {(\A^1_S, I)}                   & {(0,I).} \arrow[l, "i"']                                    
\end{tikzcd}
\]

\begin{construction}  \label{noname:ConstrOfPsi}
\begin{enumerate}
\item In the situation of \ref{noname:SetupEtaleNearbyCycle} consider the diagram $(\mathcal{A}_S, \Delta)$ equipped with the morphism $\theta^\mathcal{A}: (\mathcal{A}_S, \Delta) \rightarrow (\G_{m,S}, \Delta)$ (\ref{eqn:ThetaA}). Then using the notations of \ref{noname:SetupEtaleNearbyCycle} we define for any  morphism of schemes $f:X \rightarrow S$ a functor
\[
\Upsilon_f: \DA_{\et}(X_\eta, \Lambda) \longrightarrow \DA_{\et}(X_\sigma, \Lambda)
\]
by
\[
\Upsilon_f := (p_\Delta)_\# i^*j_* (\theta_f^{\mathcal{A}})_* (\theta_f^{\mathcal{A}})^* (p_\Delta)^*
\]
and call it the \textit{unipotent nearby cycles functor}.

\item Let $p$ be the characteristic of the residue field of $S$. Let us denote by $\mathbb{N}'^\times$ the poset consisting of positive natural numbers which are not divisble by $p$ with a unique morphism $n \rightarrow m$ whenever $m$ divides $n$. For any $n$ in $\mathbb{N}'^\times$ we denote by  $e_n: \G_{m,S} \rightarrow \G_{m,S} $ the morphism given by elevating to the $n$-th power. Let us define a diagram $\mathcal{R}_S$ in $\Fun(\N'^\times \times \Delta, \Sch^{\qcqs}_{/S}) \simeq \Fun(\N'^\times, \Fun( \Delta, \Sch^{\qcqs}_{/S}))$ by sending a map $n \rightarrow m$ in $\N'^\times$ with $k= n/m$ to the map $e_k: \mathcal{A}_S \rightarrow \mathcal{A}_S$ (\ref{eqn:enAtoA}) in $\Fun( \Delta, \Sch^{\qcqs}_{/ S})$. The map $\theta^{\mathcal{A}}$ induces a canonical map  $(\mathcal{R}_S, \N'^\times \times \Delta) \longrightarrow(\G_{m,S}, \N'^\times \times \Delta)$ which we denote by $\theta^\mathcal{R}$. We define the functor
\[
\Psi_f^{\tame}: \DA_{\et}(X_\eta, \Lambda) \longrightarrow \DA_{\et}(X_\sigma, \Lambda)
\]
by
\[
\Psi_f^{\tame} := (p_{\Delta\times \N'^\times})_\# i^*j_* (\theta_f^{\mathcal{R}})_* (\theta_f^{\mathcal{R}})^* (p_{\Delta\times \N'^\times})^* 
\]
and call it the \textit{tame nearby cycles functor}. 
\end{enumerate}

\end{construction}

\begin{construction} \label{const:TotalPsi}In the situation of $\ref{noname:SetupEtaleNearbyCycle}$ let us denote the function field of $S$ by $K$ and fix a separable closure $\bar{K}$ of $K$. Since $S$ is strictly henselian, it follows from Hensel's Lemma that $K$ contains all $n$-th roots of unity for $n$ in $\N'^\times$. This implies that $\tilde{K}:= K(\pi^{1/n} | n \in \N'^\times) \subset \bar{K}$ is a Galois extension of $K$. Let $\mu_n(K)$ denote the group of $n$-th roots of unity in $K$. The Galois group $\Gal(\tilde{K}/K)$ is isomorphic to $\widehat{\Z}'(1) := \lim_{n \in \N'^\times} \mu_n(K)$, where the transition maps $\mu_n(K) \rightarrow \mu_m(K)$ for $n \rightarrow m$ in $\N'^\times$ are given by elevating a $n$-th root $\xi$ to its $k$-th power $\xi^k$ for $k= n/m$. There is a short exact sequence of groups 
\[
0 \longrightarrow P \longrightarrow \Gal(\bar{K}/K) \overset{\chi}\longrightarrow \widehat{\Z}'(1) \longrightarrow 0,
\]
where $P$ is a maximal pro-$p$-subgroup of $\bar{K}$ and $\chi: \Gal(K^{\sep}/K) \rightarrow \widehat{\Z'}(1)$ is given by mapping a $\lambda \in \Gal(K^{\sep}/K)$ to the system of roots of unity $\left\lbrace \lambda(\pi^{1/k})/ \pi^{1/k} \right\rbrace_{k \in \N'^\times}$. By the Schur-Zassenhaus-Theorem \cite[8.10]{SuzukiGroupTheroy} this sequence splits. Let us fix a splitting $\tau$ and denote by $M_\tau$ the extension of $K$ corresponding to $\widehat{\mathbb{Z}}'(1)$ considered as a closed subgroup of $\Gal(\bar{K}/K)$ via $\tau$. We write $\Xi_{\tau}$ for the poset consisting of finite intermediate extensions of $M_\tau/K$ ordered by inclusion.

Let $f: X \rightarrow S$ be a morphism of schemes. For $L$ in $\Xi_{\tau}$ we denote the normalization of $S$ in $L$ by ${S}_L$.  Let us write $\eta_L := \Spec L$ denote by $t_L: \eta_L \rightarrow \eta$ and $t_L: S_L \rightarrow S$ the induced morphisms of spectra. For a morphism of schemes $f: X \rightarrow S$ consider the square
\[
\begin{tikzcd}
X_L \arrow[d, "f_L"'] \arrow[r, "t_L"] & X \arrow[d, "f"] \\
S_L \arrow[r, "t_L"']                  & S               
\end{tikzcd}
\]
obtained by pullback. Let us consider $X_L$ as a $S$ scheme via $t_L \circ f_L$ and let $\mathcal{R}_{t_L \circ f_L}$ be the object in $\Fun(\N'^\times \times \Delta, \Sch^{\qcqs}_{/S})$ defined as in \ref{noname:ConstrOfPsi}(2) above. Then we can define an object $\mathcal{T}_f$ in $\Fun(\Xi_{\tau}^{\op} \times \N'^\times \times \Delta, \Sch^{\qcqs}_{/S})$ by sending an arrow $L \rightarrow L'$ in $\Xi_{\tau}$ to the canonical morphism $\mathcal{R}_{t_{L'} \circ f_{L'}} \rightarrow \mathcal{R}_{t_L \circ f_L}$ in $\Fun(\N'^\times \times \Delta, \Sch^{\qcqs}_{/S})$. The map $\theta^{\mathcal{R}}_f$ defined above induces a canonical map $\theta_f^{\mathcal{T}}: (\mathcal{T}_{f}, \Xi_\tau^{\op} \times \N'^\times \times\Delta) \rightarrow (X_\eta,\Xi_\tau^{\op} \times \N'^\times \times\Delta)$. We define the functor 
\[
\Psi_f: \DA_{\et}(X_\eta, \Lambda) \longrightarrow \DA_{\et}(X_\sigma, \Lambda)
\]
by
\[
\Psi_f := (p_{\Xi_{\tau}^{\op} \times\Delta\times \N'^\times})_\# i^*j_* (\theta_f^{\mathcal{T}})_*(\theta_f^{\mathcal{T}})^* (p_{\Xi_{\tau}^{\op} \times\Delta\times \N'^\times})^*  
\]
and call it the \textit{total nearby cycles functor.}
 
\end{construction}

\begin{remark} \label{remark:AllNearbyCyclesCoincWAyoubAndAreSpecSystems}
\begin{enumerate}
\item From the construction and Remark \ref{rem:DA(diagram)ConstCoinidesWAyoub} it follows right away that the functors induced by $\Upsilon_f, \Psi^{\tame}_f$ and $\Psi_f$ on the homotopy categories coincide with the functors defined in \cite{AyoubRealizationEtale} whenever $f$ is quasi-projective (for $\Psi_f$ see Proposition \ref{prop:NearbyCyclesViaLogAndColimits} below).
\item  It is shown in \cite[3.2.9]{AyoubThesisII} (robust enough to apply to our setting) that $\Upsilon_f, \Psi^{\tame}_f$ and $\Psi_f$ indeed give rise to specialization systems over $(S,i,j)$. Moreover by \cite[3.2.12]{AyoubThesisII} the canonical maps $(\mathcal{A}, \Delta) \rightarrow (\G_{m,S}, *)$, $(\mathcal{R}, \N'^\times \times \Delta) \rightarrow (\mathcal{A}, \Delta)$ and $(\mathcal{T}, \Xi_\tau^{\op} \times \N'^\times \times \Delta) \rightarrow (\mathcal{R}, \N'^\times \times \Delta)$ in $\DiaSch_{/S}$ give rise to morphisms of specialization systems $\chi \rightarrow \Upsilon$, $\Upsilon \rightarrow \Psi^{\tame}$ and $\Psi^{\tame} \rightarrow \Psi$ respectively.
\item By \cite[3.2.17, 3.2.18]{AyoubThesisII} the specialization systems $\Upsilon, \Psi^{\tame}$ and $\Psi$ are lax-monoidal and the morphisms of specialisation systems above are lax-monoidal
\end{enumerate} 
\end{remark}

\section{An alternative description of the nearby cycles functors}

\noname Let $K$ be a simplicial set and $\C$ a $\infty-$category which admits $K$-shaped colimits. The canonical map of simplicial sets $K \rightarrow *$ induces via precomposition the diagonal functor
\[
\delta: \C \longrightarrow \Fun(K, \C).
\]
Since $\C$ admits $K$-shaped colimits $\delta$ admits a left adjoint
\[
\colim_{K}: \Fun(K, \C) \longrightarrow \C
\]
which sends a diagram $F: K \rightarrow \C$ to its colimit (see \cite[4.2.4.3]{lurie2009higher}). 

\begin{lem} \label{lem:DA(X,I)eqFun(IDA)} 
Let $X$ be a scheme, $I$ a small category and consider the constant diagram $(X,I)$ in $\DiaSch_{/S}$. Then there is an equivalence
\[\DA_{\et}((X,I), \Lambda ) \simeq \Fun\left(I^{\op}, \DA_{\et}(X, \Lambda )\right)\]
such that the composition
\[
\DA_{\et}(X, \Lambda ) \overset{(p_I)^*}\longrightarrow \DA_{\et}((X,I), \Lambda ) \simeq \Fun\left(I^{\op}, \DA_{\et}(X, \Lambda )\right)
\]
is equivalent to the diagonal functor $\delta$. In particular the composition
\[
\Fun\left(I^{\op}, \DA_{\et}(X, \Lambda )\right) \simeq \DA_{\et}((X,I), \Lambda ) \overset{(p_I)_\#}\longrightarrow \DA_{\et}(X, \Lambda ) 
\]
is equivalent to $\colim_{I^{\op}}$.
\end{lem}

\begin{proof}
Let us note that $\Sm_{/(X,I)} \simeq I\times \Sm_{/X}$ and therefore
\[
\PSh(\Sm_{/(X,I)}, \T(\Lambda)) \simeq \Fun(I^{\op} \times \Sm_{/X}^{\op}, \T(\Lambda)) \simeq \Fun(I^{\op}, \PSh(\Sm_{/X}, \T(\Lambda))).
\]
From the definitions it follows right away that there is a cartesian square
\[
\begin{tikzcd}
{\PSh(\Sm_{/(X,I)}, \T(\Lambda))} \arrow[d, "\sim"] & {\DA_{\et}^{\eff}((X,I), \Lambda)} \arrow[d, "\sim"] \arrow[l, hook'] \\
{\Fun(I^{\op}, \PSh(\Sm_{/X}, \T(\Lambda)))}                 & {\Fun(I^{\op},\DA_{\et}^{\eff}(X, \Lambda))} \arrow[l, hook']                
\end{tikzcd}
\]
of $\infty$-categories. Moreover it is straightforward to check that the composition
\[
\DA_{\et}^{\eff}(X, \Lambda) \overset{p_{I}^*}\longrightarrow \DA_{\et}^{\eff}((X,I), \Lambda) \simeq \Fun(I^{\op},\DA_{\et}^{\eff}(X, \Lambda))
\]
is simply the diagonal functor $\delta$. Let us denote the right adjoints of $\_ \otimes \Lambda(1) : \DA^{\eff}_{\et}(X, \Lambda ) \rightarrow \DA^{\eff}_{\et}(X, \Lambda )$ and $\_ \otimes \Lambda(1) : \DA^{\eff}_{\et}((X,I), \Lambda ) \rightarrow \DA^{\eff}_{\et}((X,I), \Lambda )$ by $\Omega$. Then there are equivalences
\[
 \DA_{\et}(X, \Lambda ) \simeq \lim_{\N^{\op}}\left( \dots \overset{\Omega}\rightarrow \DA^{\eff}_{\et}(X, \Lambda ) \overset{\Omega}\rightarrow \DA^{\eff}_{\et}(X, \Lambda )\right)
\]
and
\begin{align*}
 \DA_{\et}((X,I), \Lambda ) &\simeq \lim_{\N^{\op}}\left( \dots \overset{\Omega}\rightarrow \DA^{\eff}_{\et}((X,I), \Lambda ) \overset{\Omega}\rightarrow \DA^{\eff}_{\et}((X,I), \Lambda ) \right) \\
 &\simeq \lim_{\N^{\op}}\left( \dots \overset{\Omega_*}\rightarrow \Fun \left(I^{\op},\DA^{\eff}_{\et}(X, \Lambda )\right)  \overset{\Omega_*}\rightarrow \Fun \left(I^{\op},\DA^{\eff}_{\et}(X, \Lambda )\right) \right) \\
&\simeq \Fun\left(I^{\op},\lim_{\N^{\op}}\left( \dots \overset{\Omega}\rightarrow \DA^{\eff}_{\et}(X, \Lambda ) \overset{\Omega}\rightarrow \DA^{\eff}_{\et}(X, \Lambda )\right) \right) \\
&\simeq \Fun\left(I^{\op},\DA_{\et}(X, \Lambda )\right),
\end{align*}
where all limits are taken in $\Pr^R$ (or equivalently in $\widehat{\Cat}_\infty$ by \cite[5.5.3.13]{lurie2009higher}) and $\Omega_*$ denotes the functor given by post-composition with $\Omega$. 

The diagrams of $\infty-$categories
\[
\begin{tikzcd}
{\DA_{\et}^{\eff}(X, \Lambda)} \arrow[r, "\delta"] \arrow[d, "\Sigma^\infty"'] & {\Fun(I^{\op},\DA_{\et}^{\eff}(X, \Lambda))} \arrow[d, "(\Sigma^{\infty})_*"] \\
{\DA_{\et}(X, \Lambda)} \arrow[r, "\delta"]                                    & {\Fun(I^{\op},\DA_{\et}(X, \Lambda))}                                    
\end{tikzcd}
\]
and
\[
\begin{tikzcd}
{\DA_{\et}^{\eff}(X, \Lambda)} \arrow[d, "\Sigma^\infty"'] \arrow[r, "(p_I)^*"] & {\DA_{\et}^{\eff}((X,I), \Lambda)} \arrow[d, "\Sigma^\infty"] \arrow[r, "\sim"] & {\Fun(I^{\op},\DA_{\et}^{\eff}(X, \Lambda))} \arrow[d, "(\Sigma^{\infty})_*"] \\
{\DA_{\et}(X, \Lambda)} \arrow[r, "(p_I)^*"]                                    & {\DA_{\et}((X,I), \Lambda)} \arrow[r, "\sim"]                                   & {\Fun(I^{\op},\DA_{\et}(X, \Lambda))}                                        
\end{tikzcd}
\]
commute and can be in fact be canonically lifted to diagrams of symmetric monoidal $\infty$-categories. Hence the universal property of $\DA_{\et}(X, \Lambda) \simeq \DA_{\et}^{\eff}(X, \Lambda)[\Lambda(1)^{-1}]$ (see \cite[2.9]{Robalo15}) implies that $\delta$ is equivalent to the composition
\[
\DA_{\et}(X, \Lambda ) \overset{(p_I)^*}\longrightarrow \DA_{\et}((X,I), \Lambda ) \overset{\sim}\longrightarrow\Fun\left(I^{\op}, \DA_{\et}(X, \Lambda )\right)
\]
as desired.

\end{proof}

\begin{construction} \label{noname:Abullet} Let $S$ be a scheme and write
\[
\Lambda(\_): \Sm_{/\G_{m,S}} \overset{\y_\Lambda}\longrightarrow \PSh(\Sm_{/\G_{m,S}}, \T(\Lambda)) \overset{L_{\A^1} \circ L_{\et}}\longrightarrow \DA^{\eff}_{\et}(\G_{m,S},\Lambda) \overset{\Sigma^\infty}\longrightarrow \DA_{\et}(\G_{m,S},\Lambda)
\]
for the composition .

Let us define the cosimplicial object
\[
\mathscr{A}_S := \Lambda(\mathcal{A}_S) : \Delta \longrightarrow \DA_{\et}(\G_{m,S},\Lambda)
\]
where $\mathcal{A}_S: \Delta \rightarrow \Sm_{/\G_{m,S}}$ is defined in Construction \ref{noname:ConstructionOfMAthcalA}. By construction we have 
\[
\mathscr{A}_S([i]) \simeq  \Lambda(\G_{m,S} \times_S (\G_{m,S})^{\times i}),
\]
where $\G_{m,S} \times_S (\G_{m,S})^{\times n}$ is considered as a $\G_{m,S}$-scheme via the first projection. Let us write $q: \G_{m,S} \rightarrow S$ for the structure morphism. Note that 
\[\Lambda({\G_{m,S}}) \simeq q_{\#}\1_{\G_{m,S}} \simeq \1_S \oplus \1_S (1)[1] \]
 in $\DA_{\et}(S,\Lambda)$ and hence $\Lambda(\G_{m,S})$ is dualizable in $\DA_{\et}(S, \Lambda)$ in the sense of \cite[4.6.1.7]{lurie2016higher}. Therefore 
 \[\mathscr{A}_S([i]) \simeq\Lambda(\G_{m,S} \times_S (\G_{m,S})^{\times n}) \simeq q^* \Lambda((\G_{m,S})^{\times n})\]
   is  dualizable in $\DA_{\et}(\G_{m,S}, \Lambda)$ for all $[i]$ in $\Delta$. 
   
The structure maps $p_1: \G_{m,S} \times_S (\G_{m,S})^{\times n} \rightarrow \G_{m,S}$ equip $\mathscr{A}_S$ with a canonical map
\begin{equation} \label{eqn:Ato1}
\mathscr{A}_S \longrightarrow \underline\1_{\G_{m,S}}
\end{equation} 
in $\Fun(\Delta, \DA_{\et}(\G_{m,S}, \Lambda))$, where $\underline\1_{\G_{m,S}}$ denotes the constant functor with value $\1_{\G_{m,S}}$.  

In Construction \ref{noname:ConstructionOfMAthcalA} we defined a further cosimplicial object $e_k^* \mathcal{A}_S: \Delta \rightarrow \Sm_{/\G_{m,S}}$. It is straightforward to check that $\Lambda(e_k^* \mathcal{A}_S) \simeq e_k^* \mathscr{A}_S$ and hence the map (\ref{eqn:varphin:Atoen*A}) in $\Fun(\Delta, \Sm_{/ \G_{m,S}})$ induces a map 
\begin{equation} \label{eqn:Atoen*A}
\varphi_k: \mathscr{A}_S \longrightarrow e_k^* \mathscr{A}_S
\end{equation}
in $\Fun(\Delta, \DA_{\et}(\G_{m,S}, \Lambda))$.
\end{construction}

\noname \label{noname:ModelOfDADiag} Let us recall some notation of \cite{CisinskiDegliseBook}. Fix a base scheme $S$. Then there is a $\Sm$ fibred model category $C(\Sp(\PSh(\Sm, \Lambda)))$ over $\Sch^{\qcqs}_{/S}$ such that for $X$ in $\Sch^{\qcqs}_{/S}$ we have
\[
C(\Sp(\PSh(\Sm_{/X}, \Lambda))) [W^{-1}] \simeq h \DA_{\et}(X, \Lambda)
\]
(see \cite[5.3.31]{CisinskiDegliseBook}).
Here we wrote
\[
C(\Sp(\PSh(\Sm_{/X}, \Lambda))) :=C(\Sp(\PSh(\Sm, \Lambda)))(X)
\]
and $W$ denotes the weak equivalences of its model structure. As in \cite[\S 3.1]{CisinskiDegliseBook} we can define for a $\mathcal{X}: I \rightarrow \Sch^{\qcqs}_{/S}$ in $\DiaSch_{/S}$ a category $D(\Sp(\PSh(\Sm_{/(\mathcal{X},I)}, \Lambda)))$ and endow it with two model structures: The projective model structure (see \cite[\S 3.1.6]{CisinskiDegliseBook}) and the injective model structure (see \cite[\S 3.1.7]{CisinskiDegliseBook}). The two model structures share the same weak equivalences $W$ and it is not hard to show that 
\[
C(\Sp(\PSh(\Sm_{/(\mathcal{X},I)}, \Lambda))) [W^{-1}] \simeq h \DA_{\et}((\mathcal{X},I), \Lambda).
\]
Let us denote by
\[
\gamma: C(\Sp(\PSh(\Sm_{/(\mathcal{X},I)}, \Lambda))) \rightarrow h \DA_{\et}((\mathcal{X},I), \Lambda)
\]
the localization functor.

A map $(\theta,\alpha): (\mathcal{X}, I) \rightarrow (\mathcal{Y}, J)$ in $\DiaSch_{/S}$ induces a Quillen adjunction
\[
(\theta, \alpha)^*: C(\Sp(\PSh(\Sm_{/(\mathcal{X},I)}, \Lambda))) \longleftrightarrows C(\Sp(\PSh(\Sm_{/(\mathcal{X},I)}, \Lambda))) : (\theta, \alpha)_*
\]
for the injective model structures. Taking derived functors gives rise to an adjunction 
\[
L (\theta, \alpha)^*: h\DA_{\et}(( \mathcal{Y},J), \Lambda) \longleftrightarrows h\DA_{\et}(( \mathcal{X},I), \Lambda) : R(\theta, \alpha)_*.
\]
If $(\theta, \alpha)$ is moreover level-wise smooth then $(\theta, \alpha)^*$ admits a Quillen left adjoint $(\theta, \alpha)_\#$ for the projective model structure and $(\theta, \alpha)^*$ preserves weak equivalences. In particular we have $L(\theta, \alpha)^* = (\theta, \alpha)^* = R (\theta, \alpha)^*$ and taking derived functors yields an adjunction
\[
L (\theta, \alpha)_\#: h\DA_{\et}(( \mathcal{X},I), \Lambda) \longleftrightarrows h\DA_{\et}(( \mathcal{Y},J), \Lambda) : R(\theta, \alpha)^* = (\theta, \alpha)^* =R (\theta, \alpha)^*.
\]
This is shown in \cite[3.1.11]{CisinskiDegliseBook}.  

It is not hard to check that these functors agree with the ones defined in Construction \ref{remark:ConstOfDAForiagramsMaps} on homotopy categories.

\begin{lem} \label{lem:CoherenceLemma}
Consider the situation of \ref{noname:ConstrOfPsi} (1). For a  morphism of schemes $f: X \rightarrow S$ and $M$ in $\DA_{\et}(X_\eta, \Lambda)$ let us write
\[
\Hom(f_\eta^* \pi^* \mathscr{A}_S, M): \Delta^{\op} \overset{(  f_\eta^* \pi^* \mathscr{A}_S)^{\op}}\longrightarrow \DA_{\et}(X_\eta, \Lambda)^{\op} \overset{\Hom(\_, M)}\longrightarrow \DA_{\et}(X_\eta, \Lambda).
\]
Then there is an equivalence
\[
\Hom(f_\eta^* \pi^* \mathscr{A}_S, M) \simeq (\theta_f^\mathcal{A})_*(\theta_f^\mathcal{A})^* (p_\Delta)^* M
\]
in $\Fun(\Delta^{\op}, \DA_{\et}(X_\eta, \Lambda)) \simeq \DA_{\et}((X_\eta, \Delta), \Lambda)$. 
\end{lem}

\begin{proof}The idea is to construct this equivalence of functors on the level of model categories. For this we will use some notations of \cite{CisinskiDegliseBook} and consider the model categories of \ref{noname:ModelOfDADiag}.

For simplicity let us write $\theta := \theta_f^\mathcal{A}: (\mathcal{A}_f, \Delta) \rightarrow (X_\eta, \Delta)$ and ${\mathscr{A}} := f_\eta^* \pi^* \mathscr{A}_S$. We define $\tilde{\mathscr{A}}$ as the diagram
\[
\Delta \overset{\mathcal{A}_f}\longrightarrow \Sm_{/X_\eta} \overset{\text{Yoneda}}\longrightarrow C(\PSh(\Sm_{/X_\eta}, \Lambda)) \overset{\Sigma^\infty}\longrightarrow C(\Sp(\PSh(\Sm_{/X_\eta}, \Lambda))).
\]
Then by construction there is an equivalence $\gamma (\tilde{\mathscr{A}}) \simeq \mathscr{A}$ in $h \DA_{\et}((X_\eta, \Delta), \Lambda)$. The tensor product $\_ \otimes N$ for $N$ in $C(\Sp(\PSh(\Sm_{/X_\eta}, \Lambda)))$ admits a Quillen right adjoint which we denote by $Hom(N, \_)$.

Denote by $\theta_i: \mathcal{A}_f([i]) \rightarrow X_\eta$ the canonical  map for all $[i]$ in $\Delta$. Then there is an equivalence $\tilde{\mathscr{A}}([i]) \simeq (\theta_{i}) _\# (\theta_i)^* \1$ and thus we have
\begin{equation} \label{eqn:UnderivedThetaStarThetaStarIsoHom}
(\theta_{i})_* (\theta_i)^* \_ \simeq Hom((\theta_{i})_\# (\theta_i)^* \1 , \_ ) \simeq Hom(\tilde{\mathscr{A}}([i]) ,\_ ).
\end{equation} 
For a map $\varphi: [j] \rightarrow [i]$ in $\Delta$ consider the induced map $\theta_\varphi: \mathcal{A}([j]) \rightarrow \mathcal{A}([i])$. Then the diagram
\[
\begin{tikzcd}
(\theta_i)_*(\theta_i)^* \arrow[d, "\unit"] \arrow[r, "\sim", no head]                            & {Hom(\theta_{i \#} \theta_i^* \1 , \_ )} \arrow[r, "\sim", no head] & {Hom(\tilde{\mathscr{A}}([i]) ,\_ )} \arrow[dd, "\theta_\varphi^*"] \\
(\theta_i)_*(\theta_\varphi)_*(\theta_\varphi)^*(\theta_i)^* \arrow[d, "\sim"', no head] &                                                                      &                                                  \\
(\theta_j)_*(\theta_j)^* \arrow[r, "\sim", no head]                                      & {Hom(\theta_{j \#} \theta_j^* \1 , \_ )} \arrow[r, "\sim", no head] & {Hom(\tilde{\mathscr{A}}([j]) ,\_ )}           
\end{tikzcd}
\]
commutes. This shows that there is a natural equivalence
 \[
 \theta_* \theta^* p_{\Delta}^* \simeq Hom(\tilde{\mathscr{A}}^{}, \_): C(\Sp(\PSh(\Sm_{/X_\eta}, \Lambda))) \rightarrow C(\Sp(\PSh(\Sm_{/(X_\eta, \Delta)}, \Lambda))).
 \]
and therefore there is an equivalence of right derived functors
 \[
 R( \theta_* \theta^* p_{\Delta}^*) \simeq RHom(\mathscr{A}^{}, \_): h\DA_{\et}(X_\eta , \Lambda) \longrightarrow h\DA_{\et}((X_\eta, \Delta), \Lambda). 
 \]
Recall that since $\theta$ and $p_\Delta$ are level-wise smooth we have $L (\theta^*) = \theta = R (\theta^*)$ and $L (p_\Delta^*) =p_\Delta^* = R (p_\Delta^*) $. Hence the universal property of derived functors induces a natural transformation
\begin{equation} \label{eqn:CoherenceLemmaRelevantMap}
RHom(\mathscr{A}^{}, \_) \simeq R( \theta_* \theta^* p_{\Delta}^*)  \longrightarrow (R \theta_*) \theta^* p_{\Delta}^*.
\end{equation}
The Lemma follows if we can show that this is an equivalence. 

For all $[i]$ in $\Delta$ let us denote by $i: (X_\eta, *) \rightarrow (X_\eta, \Delta)$ the canonical map in $\DiaSch_{/S}$ induced by the functor $i: * \rightarrow \Delta$ which maps to $[i]$. Then the family of functors 
\[
Li^* = i^* = Ri^*: h\DA_{\et}((X_\eta, \Delta), \Lambda) \longrightarrow h\DA_{\et}((X_\eta), \Lambda)
\]
running through all $[i]$ in $\Delta$ is conservative by definition of the weak equivalences in the injective  and projective model structure. Since $i^*$ is a Quillen right adjoint with respect to the injective model structure by \cite[3.1.13]{CisinskiDegliseBook} and $i^* = R i^*$ the canonical natural transformation
\[i^* R \theta_* \rightarrow (R\theta_{i*}) i^*\]
is an equivalence (see \cite[3.1.15]{CisinskiDegliseBook}). Moreover the canonical natural transformation
\[
RHom(\mathscr{A}([i]) ,\_ ) \rightarrow i^* RHom(\mathscr{A} , \_ )
\] 
is an equivalence. 

Since $\theta_i^* \1$ is cofibrant in $C(\Sp(\PSh(\Sm_{/\mathcal{A}([i])}, \Lambda)))$ we have that 
\[ \mathscr{A}([i]) \simeq \gamma (\tilde{\mathscr{A}}([i])) \simeq L(\theta_{i \#}) \theta_i^* \1.\] 
In particular there are equivalences
\[
(R\theta_{i*})\theta_i^*  \simeq RHom((L \theta_{i\#}) \theta_i^* \1 , \_ ) \simeq RHom(\mathscr{A}([i]) ,\_ ).
\]
Hence in order to show that (\ref{eqn:CoherenceLemmaRelevantMap}) is an equivalence we may show that the composition
\begin{align*}
RHom(\mathscr{A}([i]) ,\_ ) &\simeq i^* RHom(\mathscr{A} , \_ ) \\
&\simeq i^* R( \theta_* \theta^* p_{\Delta}^*)  \\
&\rightarrow i^* (R \theta_*) \theta^* p_{\Delta}^*  \\
&\simeq  (R\theta_{i*}) \theta_i^*  \\
&\simeq RHom(\mathscr{A}([i]) ,\_ )
\end{align*}
is the identity, where the arrow in the third row is obtained by applying $i^*$ to (\ref{eqn:CoherenceLemmaRelevantMap}). 

By the universal property of right derived functors this amounts to proving that
\[
\begin{tikzcd}
{ \gamma \circ Hom(\tilde{\mathscr{A}}([i]), \_)} \arrow[r] \arrow[ddddd, "\id"', bend right=60] \arrow[d, "\sim"] & {RHom({\mathscr{A}}([i]), \gamma(\_))} \arrow[d, "\sim"]           \\
{ \gamma \circ i^* Hom(\tilde{\mathscr{A}}^{{}}, \_)} \arrow[d, "\sim"]                                      & {i^*RHom({\mathscr{A}}^{{}}, \gamma (\_))} \arrow[d, "\sim"] \\
\gamma \circ i^* \theta_* \theta^* p_{\Delta}^* \arrow[r] \arrow[rd] \arrow[dd, "\sim"]                             & i^* (R\theta_* \theta^* p_{\Delta}^*) \circ \gamma \arrow[d]         \\
                                                                                                                    & i^* (R\theta_*) \theta^* p_{\Delta}^* \circ \gamma \arrow[d, "\sim"] \\
\gamma \circ \theta_{i *}\theta_i^*  \arrow[r] \arrow[d, "\sim"]                                                  & (R\theta_{i*})\theta_i^* \circ \gamma \arrow[d, "\sim"]             \\
{ \gamma \circ Hom(\tilde{\mathscr{A}}([i]), \_)} \arrow[r]                                                                  & { RHom({\mathscr{A}}([i]), \gamma (\_) )}                            
\end{tikzcd}
\]
commutes. Here the maps from left to right are the canonical maps induced by the universal property of the respective right derived functors. This is straightforward to check. 
\end{proof}

\noname \label{noname:SetupPropNearbyCycleViaColimits} Let us fix some notations: $S$ is the spectrum of a strictly henselian discrete valuation ring with uniformizer $\pi.$ For $n$ in $\N'^\times$ let $S_n$ denote the normalisation of $S$ in $K(\pi^{1/n})$. Since $S$ is strictly henselian, $S_n$ is again a strictly henselian discrete valuation ring with uniformizer $\pi_n := \pi^{1/n}$ for some $n$-th root of $\pi$. Let us write $\eta_n := \Spec K(\pi^{1/n})$ denote by $t_n: \eta_n \rightarrow \eta$ and $t_n: S_n \rightarrow S$ the induced morphisms of spectra. For a morphism of schemes $f: X \rightarrow S$ consider the square
\[
\begin{tikzcd}
X_n \arrow[d, "f_n"'] \arrow[r, "t_n"] & X \arrow[d, "f"] \\
S_n \arrow[r, "t_n"']                  & S               
\end{tikzcd}
\]
obtained by pullback. Its generic fiber is the top square of the commutative diagram
\begin{equation} \label{eqn:DiagramSnfn}
\begin{tikzcd}
                                & (X_n)_\eta \arrow[r, "t_n"] \arrow[d, "(f_n)_\eta"']          & X_\eta \arrow[d, "f_\eta"] \\
                                & \eta_n \arrow[ld, "\pi_n"'] \arrow[d, "\pi"] \arrow[r, "t_n"] & \eta \arrow[d, "\pi"]      \\
{\G_{m,{S_n}}} \arrow[r, "e_n"] & {\G_{m,{S_n}}} \arrow[r, "t_n^{\G_m}"]                                      & {\G_{m,S}},                
\end{tikzcd}
\end{equation}
where $t_n^{\G_m}$ denotes the base change of $t_n: S_n \rightarrow S$ along the projection $q: \G_{m,S} \rightarrow S$.

As in \ref{noname:ConstrOfPsi}(1) we obtain a morphism of diagrams  $\theta_{f_n}^\mathcal{A}: (\mathcal{A}_{f_n}, \Delta) \rightarrow (X_n, \Delta)$ from $f_n: X_n \rightarrow S_n$ and $\theta^\mathcal{A}: (\mathcal{A}_{S_n}, \Delta) \rightarrow (\G_{m, S_n}, \Delta)$ via the uniformizer $\pi_n: S_n \rightarrow \A^1_{S_n}$. 

\begin{prop} \label{prop:NearbyCyclesViaLogAndColimits}
Consider the situation of $\ref{noname:SetupEtaleNearbyCycle}$. For any morphism of schemes $f:X \rightarrow S$ and any $M$ in $\DA_{\et}(X_\eta, \Lambda)$ there are equivalences
\[
\Upsilon_f(M) \simeq \colim_{\Delta^{op}} i^*j_* \Hom(f_\eta^* \pi^*\mathscr{A}_S, M),
\]
\[
\Psi_f^{\tame}(M) \simeq \colim_{n \in (\N'^\times)^{\op}} \Upsilon_{{f}_n}(t_n^* M)
\]
and
\[
\Psi_f(M) \simeq \colim_{L \in \Xi_\tau} \Psi_{t_L \circ f_L}^{\tame}(t_L^*M).
\]
\end{prop}

\begin{proof}
Using the identification $\DA_{\et}((X,I), \Lambda ) \simeq \Fun\left(I^{op}, \DA_{\et}(X, \Lambda )\right)$ of Lemma \ref{lem:DA(X,I)eqFun(IDA)} and Lemma \ref{lem:CoherenceLemma} we have
\begin{align*}
\Upsilon_f(M) &= (p_\Delta)_\# i^*j_* (\theta_f^\mathcal{A})_* (\theta_f^\mathcal{A})^*  (p_\Delta)^* M \\
& \simeq \colim_{i \in \Delta^{\op}} i^*j_* (\theta_f^\mathcal{A})_* (\theta_f^\mathcal{A})^*  (p_\Delta)^* M \\
& \simeq \colim_{i \in \Delta^{\op}} i^*j_* \Hom(f_\eta^* \pi^*\mathscr{A}_S, M). 
\end{align*}

Let us write $q: (\G_{m,S}, \N'^\times \times \Delta) \rightarrow (\G_{m,S}, \N'^\times)$ for the map of constant diagrams in $\DiaSch_{/S}$ induced by the first projection $\N'^\times \times \Delta \rightarrow  \N'^\times$. Then we have
\begin{align*}
\Psi_f^{\tame} M &= (p_{\Delta\times \N^\times})_\# i^*j_* (\theta_f^{\mathcal{R}})_* (\theta_f^{\mathcal{R}})^* (p_{\Delta\times \N^\times})^*M \\
& \simeq (p_{\N^\times})_\# q_\# i^*j_* (\theta_f^{\mathcal{R}})_* (\theta_f^{\mathcal{R}})^* (p_{\Delta\times \N^\times})^*M \\
& \overset{(1)}\simeq \colim_{n \in (\N'^\times)^{\op}} q_{\#} i^* j_* (t_n)_* (\theta^{\mathcal{A}}_{f_n})_* (\theta^{\mathcal{A}}_{f_n})^*  t_n^* M \\
& \overset{(2)}\simeq \colim_{n \in (\N'^\times)^{\op}} q_{\#} i^* j_* (\theta^{\mathcal{A}}_{f_n})_* (\theta^{\mathcal{A}}_{f_n})^*  t_n^* M \\
& \simeq \colim_{n \in (\N'^\times)^{\op}} \Upsilon_{{f}_n}(t_n^* M).
\end{align*}
Here (1) follows from Lemma \ref{lem:DA(X,I)eqFun(IDA)} and (2) follows from the fact that the squares in
\[
\begin{tikzcd}
\eta_n \arrow[d, "t_n"'] \arrow[r, "j"] & S_n \arrow[d, "t_n"'] & \sigma \arrow[l, "i"'] \arrow[d, "\id"] \\
\eta \arrow[r, "j"]                     & S                     & \sigma \arrow[l, "i"']                 
\end{tikzcd}
\]
are cartesian and since $t_n$ is finite proper base change implies that $i^* j_* (t_n)_* \simeq i^* (t_n)_*j_* \simeq i^*j_* $.

The last equivalence follows analogously. 
\end{proof}

\begin{remark} \label{remark:MapsOfSpezSystemsForDA} 
From the descriptions of the nearby cycles functors in Proposition \ref{prop:NearbyCyclesViaLogAndColimits} above we get for all $f: X \rightarrow S$ a canonical natural transformation
\[
i^*j_* \longrightarrow \Upsilon_f\]
induced by the canonical map $\mathscr{A}_S \longrightarrow \underline{\1}$ (\ref{eqn:Ato1}) as well as natural transformations 
\[ \Upsilon_f \longrightarrow \Psi^{\tame}_f \text{ and } \Psi^{\tame}_f \longrightarrow \Psi_f  
\] 
given by the canonical maps into the respective colimits. It is not hard to verify that these natural transformations agree with to the ones considered in Remark \ref{remark:AllNearbyCyclesCoincWAyoubAndAreSpecSystems} (2). 
\end{remark} 

\begin{remark}
\label{remark:ExplicitTransitionMapsNearby} It is straightforward to check that the transition maps of the colimits in Proposition \ref{prop:NearbyCyclesViaLogAndColimits} above are given as follows: 
\begin{enumerate}
\item For a map $ \varphi: n \rightarrow m$ in $\mathbb{N}^\times$ write $k=\frac{n}{m}$. Moreover write $t_k: \eta_n \rightarrow \eta_m$ and $t_k: (X_n)_\eta \rightarrow (X_m)_\eta$ for its basechange along $(f_m)_\eta$. Then the transition map $\tau_\varphi$ is given by
\begin{align*}
\Upsilon_{f_m}(t_m^* M) & \overset{\sim}\longrightarrow \colim_{\Delta^{op}} i^*j_* \Hom((f_m)_\eta^* \pi_m^*\mathscr{A}_{S_m}, t_m^*M) \\
&\overset{\text{unit}}  \longrightarrow \colim_{\Delta^{op}} i^*j_* t_{k *} t_k^* \Hom( (f_m)_\eta^* \pi_m^*\mathscr{A}_{S_m}, t_m^* M)\\
&\overset{\sim}\longrightarrow \colim_{\Delta^{op}} i^*j_* t_{k*}  \Hom(t_k^* (f_m)_\eta^* \pi_m^*\mathscr{A}_{S_m}, t_k^* t_m^* M)\\
&\overset{\sim}\longrightarrow \colim_{\Delta^{op}} i^*j_*  \Hom(t_k^* (f_m)_\eta^* \pi_m^*\mathscr{A}_{S_m}, t_n^* M)\\
&\overset{\sim}\longrightarrow \colim_{\Delta^{op}} i^*j_*  \Hom((f_n)_\eta^* \pi_n^* e_k^*\mathscr{A}_{S_n}, t_n^*M)\\
&\longrightarrow   \colim_{\Delta^{op}} i^*j_* \Hom((f_n)_\eta^* \pi_n^* \mathscr{A}_{S_n}, t_n^* M)  \\
& \overset{\sim}\longrightarrow \Upsilon_{f_n}(t_n^* M).
\end{align*}
As in the proof above the equivalence in the fourth row follows from the fact that 
\[
\begin{tikzcd}
S_n \arrow[d, "t_k"'] & \sigma \arrow[d, "\id"] \arrow[l, "i"'] \\
S_m                     & \sigma \arrow[l, "i"']                 
\end{tikzcd}
\]
is a pullback square and proper basechange. The equivalence in the fifth row is induced by the equivalence 
\[
t_k^* \pi_m^* \simeq  \pi_n^* e_n^*(t_k^{\G_m})^*,
\]
which is exhibited by the diagram (\ref{eqn:DiagramSnfn}), together with the observation that
\[
(t_k^{\G_m})^* \mathscr{A}_{S_m} \simeq \mathscr{A}_{S_n}.
\]
Finally the second to last map is induced by the canonical map $\varphi_k: \mathscr{A}_{S_n} \rightarrow e_k^* \mathscr{A}_{S_n}$ (see (\ref{eqn:Atoen*A})).
 
\item  For a morphism $u: L \rightarrow L'$ in $\Xi_\tau$ we denote the induced morphisms of schemes $S_{L'} \rightarrow S_L$ and $X_{L'} \rightarrow X_L$ also by $u.$ Consider the commutative diagram
\[
\begin{tikzcd}
X_{L'} \arrow[d, "f_{L'}"'] \arrow[r, "u"]  & X_L \arrow[d, "f_L"] \\
S_{L'} \arrow[rd, "t_{L'}"'] \arrow[r, "u"] & S_L \arrow[d, "t_L"] \\
                                            & S.                   
\end{tikzcd}
\]
Then the transition map $\tau_u$ is simply given by the composition
\[
\Psi^{\tame}_{t_L \circ f_L}(t_L^* M) \simeq u_\sigma ^* \Psi^{\tame}_{t_L \circ f_L}(t_L^* M) \overset{\Ex^*}\longrightarrow \Psi^{\tame}_{t_L \circ f_L \circ u}(u_\eta^* t_L^* M) \simeq \Psi^{\tame}_{t_{L'} \circ f_{L'}}(t_{L'}^* M).
\]
Where $\Ex^*$ is the exchange map describing the functoriality of the specialization system $\Psi^{\tame}$ (see Definition \ref{def:SpezSys}(1)).
\end{enumerate}

\end{remark} 

\noname \label{noname:CompPsi} Let $f:X \rightarrow S$ be a morphism of finite type and $M$ in $\DA_{\et}(X_\eta, \Lambda)$. Then we obtain a comparison map
\[
\comp_\Psi: \Psi_f(\D_\eta(M)) \longrightarrow \D_\sigma (\Psi_f (M))
\]
as the transpose of the composition
\begin{align*}
\Psi_f(\D_\eta(M)) \otimes \Psi_f(M) &\longrightarrow \Psi_f(\D_\eta(M) \otimes M) \\
&\overset{\Psi_f \id^t}\longrightarrow \Psi_f (f_\eta^!\1) \\
&\overset{\Ex^!}\longrightarrow f_\sigma^! \Psi_{\id} (\1) \\
&\overset{\sim}\longrightarrow f_\sigma^! \1.
\end{align*}
Here
\[
\id^t: M \otimes \D_\eta(M) \longrightarrow f_\eta^!\1
\]
is the transpose of $\id: \D_\eta(M) \rightarrow \D_\eta(M) $ with respect to the $\otimes \dashv \Hom$ adjunction.

\noname Let us collect the main properties of the nearby cycles functor proven in \cite{AyoubRealizationEtale}:

\begin{thm} \label{thm:PropertiesOfEtaleMotivicNearby}
Consider the situation of \ref{noname:SetupEtaleNearbyCycle}. Assume that $S$ is excellent and the residue characteristic of $S$ is invertible in $\Lambda$. Then the following hold:
\begin{enumerate}
\item The composition
\[
\1 \overset{\unit}\longrightarrow \chi_{\id}(\1) \longrightarrow \Psi_{\id}(\1)
\]
is an equivalence. 
\item Let $f: X \rightarrow S$ be of finite type and $M$ in $\DA_{\et}^{\cons}(X_\eta, \Lambda)$. Then there exists a $L$ in $\Xi_\tau$ such that for all $u: L \rightarrow L'$ in $\Xi_\tau$ the transition map
\[
\tau_u: \Psi_{t_L \circ f_L}^{\tame}(t_L^* M) \longrightarrow \Psi_{t_{L'} \circ f_{L'}}^{\tame}(t_{L'}^* M)
\]
is an equivalence. In particular the canonical map
\[
\Psi_{t_L \circ f_L}^{\tame}(t_L^* M) \longrightarrow \Psi_f(M)
\]
is an equivalence. 
\item Let $f: X \rightarrow S$ be of finite type. Then $\Psi^{\tame}_f(\_)$ and $\Psi_f(\_)$ preserve constructibiltiy. 
\item Let $f: X \rightarrow S$ and $g: Y \rightarrow S$ be of finite type. Then for any $M$ in $\DA(X_\eta, \Lambda)$ and $N$ in $\DA_{\et}(Y_\eta, \Lambda)$ the canonical comparison map
\[
\Psi_f(M) \boxtimes \Psi_g(N) \longrightarrow \Psi_{f \times g}(M \boxtimes N)
\]
induced by the lax-monoidal structure of $\Psi$ is an equivalence.
\item Let $f:X \rightarrow S$ be of finite type and $M$ in $\DA^{\cons}_{\et}(X_\eta, \Lambda)$. Then the comparison map
\[
\comp_\Psi: \Psi_f(\D_\eta(M)) \longrightarrow \D_\sigma (\Psi_f (M))
\]
is an equivalence. 
\end{enumerate}
\end{thm}

\begin{proof}
Statement (1) is proven in \cite[10.18]{AyoubRealizationEtale}. Let us note that (2)-(5) may be checked Zariski-locally on $X$. Hence we may assume that $f:X \rightarrow S$ is quasi-projective. Then (2) is \cite[10.13]{AyoubRealizationEtale}, (3) follows from (2) and \cite[10.9]{AyoubRealizationEtale}, (4) is  \cite[10.19]{AyoubRealizationEtale} and (5) is \cite[10.20]{AyoubRealizationEtale}.
\end{proof}

\section{Digression: On the logarithm motive}
\label{section:OnTheLog}

\noname \label{noname:ConstrCosimplLog}Let $S$ be a qcqs scheme and $(X,s)$ a smooth pointed $S$-scheme. By this we mean a smooth scheme $\pi: X \rightarrow S$ over $S$ together with a section $s: S \rightarrow X$ of $\pi$. Consider the cospan
\[
[X \overset{\id}\rightarrow X] \overset{\Delta}\longrightarrow [X \times_S X \overset{p_1}\rightarrow X] \overset{(\id, s)}\longleftarrow [X \overset{\id}\rightarrow X]
\] 
in $\Sm_{/X}$. By \ref{noname:SimplObjAssToSpan} we can associate to such a cospan a cosimplicial object
\[
\mathcal{A}_{(X,s)}:= \BarConstr(\Delta, (\id,s)): \Delta \longrightarrow \Sm_{/X}
\]
with
\[
\mathcal{A}_{(X,s)}([i]) = [X \times_S X^{\times i} \overset{p_1}\rightarrow X]
\]
for $[i]$ in $\Delta$. 

Alternatively we can associate to the cospan
\[
[X \overset{\pi}\rightarrow S] \overset{\id}\longrightarrow [X \overset{\pi}\rightarrow S] \overset{s}\longleftarrow [S \overset{\id}\rightarrow S].
\]
in $\Sch^{\qcqs}_{/S}$ the cosimplicial object 
\[
\mathcal{A}'_{(X,s)}:= \BarConstr(\id, s): \Delta \longrightarrow \Sch^{\qcqs}_{/S}
\]
with 
\[
\mathcal{A}'_{(X,s)}([i]) = X \times_S X^{\times i}. 
\]
The first projection $p_1: X \times_S X^{\times i} \rightarrow X$ equips
$\mathcal{A}'_{(X,s)}([i])$ with the structure of a smooth $X$-scheme. This is compatible for all $[i]$ in $\Delta$ in the sense that $\mathcal{A}'_{(X,s)}$ factors as
\[
\mathcal{A}'_{(X,s)}: \Delta \longrightarrow \Sm_{/X} \longrightarrow\Sch^{\qcqs}_{/S}
\] 
Note that this is what we do in \ref{noname:ConstructionOfMAthcalA} for $(X,s) = (\G_{m,S},1)$. It is straightforward to check that the two cosimplicial objects $\mathcal{A}_{(X,s)}$ and $\mathcal{A}'_{(X,s)}$  agree in $\Sm_{/X}$.

\noname Consider the stable motivic homotopy category $\SH(X)$ as for example considered in \cite[\S 2.4.3]{Robalo15}. Via the functor
\[
\Sigma^\infty: \Sm_X \longrightarrow \SH(X)
\]
(see \cite[2.39]{Robalo15}) the cosimplicial scheme $\mathcal{A}_{(X,s)}$ gives rise to a cosimplicial object
\[
\mathscr{A}_{(X,s)}: \Delta \overset{\mathcal{A}_{(X,S)}}\longrightarrow \Sm_{/X} \overset{\Sigma^\infty}\longrightarrow \SH(X).
\]

Let $\T(\_)$ be a motivic $\infty$-category over $S$. By this we mean a $(*,\#, \otimes)$-formalism on $(\Sch^{\qcqs}_{/S}, \Sm_{/S})$ satisfying the Voevodsky-conditions in the sense of \cite{Adeel6Functor}. Then there exists a unique system of colimit preserving functors
\[
R_Y: \SH(Y) \longrightarrow \T(Y)
\]
for all $Y$ in $\Sch^{\qcqs}_{/S}$ which commute with $f^*$ , tensor-product and $f_\#$ for smooth $f$ (see \cite[2.14]{Adeel6Functor}). We denote the cosimplicial object
\[
\Delta \overset{\mathscr{A}_{(X,s)}}\longrightarrow \SH(X) \overset{R_X}\longrightarrow \T(X)
\]
again by $\mathscr{A}_{(X,s)}$.

The structure of $\mathcal{A}_{(X,s)}$ as a cosimplicial $X$-scheme induces a canonical map
\begin{equation} \label{eqn:EpsilonLog}
\varepsilon: \mathscr{A}_{(X,s)} \rightarrow \underline\1_X
\end{equation}
in $\Fun(\Delta,\T(X))$, where $\underline\1_X$ denotes the constant functor with value in the tensor unit $\1_X$. It is easy to check that the cosimplicial object 
\[s^* \mathscr{A}_{(X,s)}: \Delta \overset{\mathscr{A}_{(X,s)}}\longrightarrow \T(X) \overset{s^*}\longrightarrow \T(S)
\]
is equivalent to the cosimplicial object in $\T(S)$ obtained from the cospan
\[
\begin{tikzcd}
{[S \overset{\id}\rightarrow S]} \arrow[r, "s"] & {[ X \overset{\pi}\rightarrow S]} & {[S \overset{\id}\rightarrow S]} \arrow[l, "s"']
\end{tikzcd}
\]
in $\Sm_{/S}$. In particular there is a canonical map
\begin{equation} \label{eqn:EtaLog}
\eta:\underline\1_S  \rightarrow s^*\mathscr{A}_{(X,s)}
\end{equation}
in $\Fun(\Delta,\T(S))$ induced by the map of cospans
\[
\begin{tikzcd}
{[S \overset{\id}\rightarrow S]} \arrow[d, "\id"] \arrow[r, "\id"] & {[S \overset{\id}\rightarrow S]} \arrow[d, "s"] & {[S \overset{\id}\rightarrow S]} \arrow[d, "\id"'] \arrow[l, "\id"'] \\
{[S \overset{\id}\rightarrow S]} \arrow[r, "s"]                    & {[ X \overset{\pi}\rightarrow S]}               & {[S \overset{\id}\rightarrow S]} \arrow[l, "s"']                    
\end{tikzcd}
\]
in $\Sm_{/S}$.

\begin{lem} \label{lem:CompEtas^*EpsilonIsID}
Let $(\pi:X \rightarrow S,s)$ be a smooth pointed $S$-scheme. Then the composition
\[
\underline\1_S \overset{\eta}\longrightarrow s^*\mathscr{A}_{(X,s)} \overset{s^* \varepsilon}\longrightarrow \underline\1_S
\]
is equivalent to the identity. 
\end{lem}

\begin{proof}
This follows from the fact that this composition is induced by the top to bottom composition of cospans
\[
\begin{tikzcd}
{[S \overset{\id}\rightarrow S]} \arrow[d, "\id"] \arrow[r, "\id"] & {[S \overset{\id}\rightarrow S]} \arrow[d, "s"]    & {[S \overset{\id}\rightarrow S]} \arrow[d, "\id"] \arrow[l, "\id"'] \\
{[S \overset{\id}\rightarrow S]} \arrow[r, "s"] \arrow[d, "\id"]   & {[ X \overset{\pi}\rightarrow S]} \arrow[d, "\pi"] & {[S \overset{\id}\rightarrow S]} \arrow[l, "s"'] \arrow[d, "\id"]   \\
{[S \overset{\id}\rightarrow S]} \arrow[r, "\id"]                  & {[S \overset{\id}\rightarrow S]}                   & {[S \overset{\id}\rightarrow S]} \arrow[l, "\id"']                 
\end{tikzcd}
\]
which is the identity.

\end{proof}

\begin{definition} Let $\pi: X \rightarrow S$ a smooth map and $\T(\_)$ a six functor formalism. We denote the smallest stable full subcategory of $\T(X)$ containing the objects of the form $\pi^*N$ for all $N$ in $\T(S)$ by $\Uni_\pi(X)$. An object $M$ in $\Uni_\pi(X)$ is called \textit{unipotent}.
\end{definition}

\begin{remark} For a smooth pointed $S-$scheme $(X,s)$ we have 
\[ \mathscr{A}_{(X,s)}([i]) \simeq \Sigma^\infty [X \times_S X^{\times i} \overset{p_1}\rightarrow X] \simeq \pi^* \Sigma^\infty [X^{\times i} \rightarrow S]. 
\]
Hence $\mathscr{A}_{(X,s)}$ factors as
\[
\mathscr{A}_{(X,s)}: \Delta \longrightarrow \Uni_\pi(X) \subset \T(X).
\]
\end{remark}

\begin{lem} \label{lem:GenOfAyoubsThm}
Let $(\pi:X \rightarrow S,s)$ be a smooth pointed $S$-scheme. Then for any $N$ in $\T(S)$ the composition
\[
N \overset{\unit}\longrightarrow \pi_* \pi^* N \overset{\varepsilon^\#}\longrightarrow \colim_{\Delta^{op}}\pi_*\Hom(\mathscr{A}_{(X,s)},\pi^* N ),
\]
is an equivalence. Here the last arrow is induced by $\varepsilon: \mathscr{A}_{(X,s)} \rightarrow \underline\1_X$ (\ref{eqn:EpsilonLog}).
\end{lem}

\begin{proof}
Let us consider the maps $(\theta, \id_\Delta) :(\mathcal{A}_{(X,s)}, \Delta) \rightarrow (X,\Delta)$, $(\pi, \id_\Delta): (X, \Delta) \rightarrow (S, \Delta)$ and $p_\Delta: (S,\Delta) \rightarrow (S, \star)$ in $\DiaSch_{/S}$. We claim that
\[
\colim_{\Delta^{op}}\pi_*\Hom(\mathscr{A}_{(X,s)},\pi^* N ) \simeq (p_\Delta)_\# (\pi, \id_\Delta)_*(\theta, \id_{\Delta})_*(\theta, \id_{\Delta})^*(\pi, \id_\Delta)^* (p_\Delta)^* N.
\]
Indeed we can follow the proof of Lemma \ref{lem:CoherenceLemma} verbatim. Moreover we see from the construction of this equivalence that the claim of the Lemma is equivalent to saying that the composition
\begin{align*}
N &\longrightarrow \pi_* \pi^* N \\
&\overset{\sim}\longrightarrow (p_\Delta)_\# (\pi, \id_\Delta)_*(\pi, \id_\Delta)^* (p_\Delta)^* N \\
&\overset{\unit}\longrightarrow (p_\Delta)_\# (\pi, \id_\Delta)_*(\theta, \id_{\Delta})_*(\theta, \id_{\Delta})^*(\pi, \id_\Delta)^* (p_\Delta)^* N
\end{align*}
is an equivalence. We can show this by following \cite[3.4.9]{AyoubThesisII} verbatim where the case $(X,s) = (\G_{m,S},1)$ is treated. 
\end{proof}

\noname \label{noname:CanMapForUnivPropOfLog} There is a canonical comparison map 
\begin{align*}
\colim_{\Delta^{op}} \pi_* \Hom(\mathscr{A}_{(X,s)},  \_) & \overset{\unit}\longrightarrow \colim_{\Delta^{op}} \pi_* s_* s^* \Hom(\mathscr{A}_{(X,s)}, \_) \\
& \overset{\sim}\longrightarrow \colim_{\Delta^{op}} s^* \Hom(\mathscr{A}_{(X,s)}, \_) \\
& \longrightarrow  \colim_{\Delta^{op}} \Hom(s^*\mathscr{A}_{(X,s)}, s^* \_) \\
& \overset{\eta^\#}\longrightarrow s^* \_,
\end{align*}
in $\Fun^{\ex}(\T(X), \T(S))$, where the last arrow is induced by $\eta:\1  \rightarrow s^*\mathscr{A}_{(X,s)}$.

\begin{thm} \label{thm:UnivPropOfLog}
Let $(\pi:X \rightarrow S,s)$ be a smooth pointed $S$-scheme. Then for any $M$ in $\Uni_\pi(X)$ the comparison map
\[
\colim_{\Delta^{op}} \pi_* \Hom(\mathscr{A}_{(X,s)}, M) \longrightarrow s^* M
\]
obtained from \ref{noname:CanMapForUnivPropOfLog} is an equivalence.
\end{thm}

\begin{proof}
It suffices to show the claim in the case where $M = \pi^*N$ for an $N$ in $\T(S)$. Consider the following diagram:
\[
\begin{tikzcd}
N \arrow[r, "\unit"] \arrow[rd, "\sim"'] & \pi_*\pi^*N \arrow[d, "\unit"] \arrow[r, "\varepsilon^\#"]        & {\colim_{\Delta^{op}}\pi_*\Hom(\mathscr{A}_{(X,s)}, \pi^*N)} \arrow[d, "\unit"]                     \\
                                         & \pi_*s_* s^*\pi^*N \arrow[r, "\varepsilon^\#"] \arrow[d, "\sim"'] & {\colim_{\Delta^{op}}\pi_*s_* s^*\Hom(\mathscr{A}_{(X,s)}, \pi^*N)} \arrow[d, "\sim"]               \\
                                         & s^* \pi^* N \arrow[r, "\varepsilon^\#"] \arrow[d, "\id"']                        & { \colim_{\Delta^{op}} s^*\Hom(\mathscr{A}_{(X,s)}, \pi^*N)} \arrow[d] \arrow[ld, "(1)" description, phantom] \\
                                         & s^* \pi^* N \arrow[r, "(s^*\varepsilon)^\#"] \arrow[rd, "\id"']     & {\colim_{\Delta^{op}} \Hom( s^*\mathscr{A}_{(X,s)},  s^*\pi^*N)} \arrow[d, "\eta^\#"]                   \\
                                         &                                                                &  s^* \pi^* N.                                                            
\end{tikzcd}
\]
Here the supscript $(\_)^\#$ denotes the map induced by the respective map. The composition of the top horizontal row is an equivalence by Lemma \ref{lem:GenOfAyoubsThm} and the right vertical composition is the comparison map of \ref{noname:CanMapForUnivPropOfLog}. Hence it remains to show that the diagram commutes. The commutativity of the bottom triangle follows from Lemma \ref{lem:CompEtas^*EpsilonIsID} and commutativity of the unmarked squares is obvious. In order to check that (1) commutes let us rewrite the diagram in terms of functors induced by maps in $\DiaSch_{/S}$ as in the proof of Lemma \ref{lem:GenOfAyoubsThm}. Let us denote the diagram
\[
\Delta \overset{\mathcal{A}_{(X,s)}}\longrightarrow \Sm_{/X} \overset{\_ \times_X S}\longrightarrow \Sm_{/S}
\]
by $s^*\mathcal{A}_{(X,s)}$. Here $\_ \times_X S$ denotes the functor given by pullback along $s: S \rightarrow X$. Write $(\theta_S, \id_\Delta): (s^*\mathcal{A}_{(X,s)}, \Delta) \rightarrow (S, \Delta)$ for the canonical map in $\DiaSch_{/S}$. Then commutativity of (1) is equivalent to commutativity of the square
\[
\begin{tikzcd}
{(p_\Delta)_\# (s, \id_\Delta)^*(\pi, \id_\Delta)^* (p_\Delta)^* N} \arrow[r, "\unit"] \arrow[d, "\id"'] & {(p_\Delta)_\# (s, \id_\Delta)^*(\theta, \id_{\Delta})_*(\theta, \id_{\Delta})^*(\pi, \id_\Delta)^* (p_\Delta)^* N} \arrow[d, "\Ex"] \\
{(p_\Delta)_\# (s, \id_\Delta)^*(\pi, \id_\Delta)^* (p_\Delta)^* N} \arrow[r, "\unit"]                   & {(p_\Delta)_\# (\theta_S, \id_{\Delta})_*(\theta_S, \id_{\Delta})^*(s, \id_\Delta)^*(\pi, \id_\Delta)^* (p_\Delta)^* N.}             
\end{tikzcd}
\]
This follows from the pasting property of exchange maps (see \ref{noname:ExchangeMaps}).
\end{proof}

\noname Let us denote by $\Uni_{\pi}^{\cons}(X)$ the full subcategory of $\Uni_\pi(X)$ consisting of constructible unipotent objects. In particular $\Uni_{\pi}^{\cons}(X)$ is a small $\infty$-category. Clearly $\mathscr{A}_{(X,s)}$ factors as
\[
\mathscr{A}_{(X,s)}: \Delta \longrightarrow \Uni_{\pi}^{\cons}(X) \subset \T(X).
\]
Consider the composition
\begin{equation} \label{eqn:DiagForSimplicialLog}
 \Delta \overset{\mathscr{A}_{(X,s)}}\longrightarrow \Uni_{\pi}^{\cons}(X) \overset{j}\longrightarrow \Fun(\Uni_{\pi}^{\cons}(X), \Spc)^{\op},
\end{equation}
where $j$ denotes the opposite Yoneda functor ( i.e. the functor sending $M$ to $\map_{\Uni_{\pi}^{\cons}(X)}(M, \_)$). We denote the colimit of (\ref{eqn:DiagForSimplicialLog}) in $\Fun(\Uni_{\pi}^{\cons}(X), \Spc)^{\op}$ by $\underline{\mathscr{A}}_{(X,s)}$.

\noname Let $\C$ be a small $\infty$-category. Recall that $\Pro (\C)$ is the full subcategory of $\Fun(\C, \Spc)^{\op}$ consisting of the finite limit preserving functors. Equivalently $\Pro (\C) = (\Ind (\C^{\op}))^{\op}$ is the subcategory of $\Fun(\C, \Spc)^{\op}$ consisting of those functors which can be written as a small co-filtered limit of representable functors (see \cite[5.3.5.4]{lurie2009higher}).

\begin{cor} \label{cor:LogIsProObject}
Assume that the tensor unit $\1$ in $\T(S)$ is compact. Then there is a natural equivalence
\[
\map_{\Fun(\Uni_{\pi}^{\cons}(X), \Spc)^{\op}}(\underline{\mathscr{A}}_{(X,s)}, j \_) \overset{\sim}\longrightarrow \map_{\T(S)}(\1, s^* \_)
\]
in $\Fun(\Uni_{\pi}^{\cons}(X), \Spc)^{\op}$. In particular $\underline{\mathscr{A}}_{(X,s)}$ belongs to $\Pro (\Uni_{\pi}^{\cons}(X))$.
\end{cor}

\begin{proof}
The comapctness of $\1$ in $\T(S)$ implies the equivalences
\begin{align*}
\map_{\Fun(\Uni_{\pi}^{\cons}(X), \Spc)^{\op}}(\underline{\mathscr{A}}_{(X,s)}, j \_) &\simeq \colim_{\Delta^{\op}} \map_{\Uni_{\pi}^{\cons}(X)}(\mathscr{A}_{(X,s)}, \_) \\
& \simeq \colim_{\Delta^{\op}} \map_{\T(X)}(\1 , \Hom(\mathscr{A}_{(X,s)}, \_)) \\
& \simeq \colim_{\Delta^{\op}} \map_{\T(S)}(\1 , \pi_*\Hom(\mathscr{A}_{(X,s)}, \_)) \\
& \simeq  \map_{\T(S)}(\1 , \colim_{\Delta^{\op}} \pi_*\Hom(\mathscr{A}_{(X,s)}, \_)).
\end{align*}
Applying $\map_{\T(S)}(\1, \_)$ to the composition in \ref{noname:CanMapForUnivPropOfLog} gives rise to a natural transformation
\[
\map_{\T(S)}(\1 , \colim_{\Delta^{\op}} \pi_*\Hom(\mathscr{A}_{(X,s)}, \_)) \longrightarrow \map_{\T(S)}(\1 , s^* \_)
\]
which is an equivalence when restricted to unipotent objects by Theorem \ref{thm:UnivPropOfLog}. The last sentence follows from the fact that $\map_{\T(S)}(\1 , s^* \_)$ commutes with finite limits.
\end{proof}

\begin{cor} \label{cor:CosimplcialLogAgreesWithClassicalLogByHK}
Assume that $\T(\_)$ is $\DA_{\et}( \_, \Lambda)$, where $\Lambda$ is a $\Q$-algebra. Let $S$ be a finite dimensional noetherian scheme, $X$ a smooth commutative group scheme over $S$ and  $s: S \rightarrow X$ the unit section. If either $S$ is of characteristic $0$ or $X$ is affine, then $\underline{\mathscr{A}}_{(X,s)}$ is equivalent in $\Pro (\Uni_{\pi}^{\cons}(X))$ to the logarithm motive $\Log$ defined in \cite[\S 4]{HuberKingsPolylog}.
\end{cor}

\begin{proof}
Since $S$ is finite dimensional noetherian it is of finite \'etale cohomological dimension for $\Lambda$-coefficients by \cite[1.1.4]{CisinskiDegliseEtale}. By standard arguments this implies that the compact objects of $\DA_{\et}( S, \Lambda)$ and $\DA_{\et}( S, \Lambda)$ are precisely the constructible objects and $\pi_*$ commutes with small colimits. It is clear from the construction in  \cite[\S 4]{HuberKingsPolylog} that $\Log$ is a pro-object in $\Uni_{\pi}^{\cons}(X)$. Moreover by applying $\map_{\DA_{\et}(S, \Lambda)}(\1, \_)$ to the equivalence in \cite[4.5.2]{HuberKingsPolylog} we get an equivalence
\[
\map_{\Fun(\Uni_{\pi}^{\cons}(X), \Spc)^{\op}}(\Log, j \_) \overset{\sim}\longrightarrow \map_{\DA_{\et}(S, \Lambda)}(\1, s^* \_)
\]
in $\Fun(\Uni_{\pi}^{\cons}(X), \Spc)^{\op}$. Hence Corollary \ref{cor:LogIsProObject} shows that there is an equivalence
 \[
 \Log \overset{\sim}\longrightarrow \underline{\mathscr{A}}_{(X,s)}
 \] 
 in $\Pro (\Uni_{\pi}^{\cons}(X))$.
\end{proof}

\begin{remark}
This Corollary tells us that we may consider $\mathscr{A}_{(X,s)}$ as a cosimplicial representation of the classical logarithm motive. Let us note that while the construction of $\Log$ needs rational coefficients there are no restrictions to the coefficients for $\underline{\mathscr{A}}_{(X,s)}$. Hence $\underline{\mathscr{A}}_{(X,s)}$ might be useful for defining polylogarithm classes as in \cite[\S5]{HuberKingsPolylog} integrally. 
\end{remark}

\begin{prop} \label{prop:CharacterisationOfFunctOfLog}

Assume that $\1$ is compact in $\T(S)$ and let $f: (X,x) \rightarrow (Y,y)$ be a morphism of smooth pointed $S$-schemes. Then there exists a unique (up to homotopy) map $\gamma_f: \underline{\mathscr{A}}_{(X,x)} \rightarrow f^* \underline{\mathscr{A}}_{(Y,y)}$ in $\Pro (\Uni_{\pi}^{\cons}(X))$ 
with the property that 
\[
\begin{tikzcd}
                                             & \1 \arrow[ld, "{\eta_{(X,x)}}"'] \arrow[rd, "{\eta_{(Y,y)}}"] &                   \\
{x^* \underline{\mathscr{A}}_{(X,x)}} \arrow[r, "x^* \gamma_f"] & {x^* f^* \underline{\mathscr{A}}_{(Y,y)}} \arrow[r, "\sim", no head]             & {y^*\underline{\mathscr{A}}_{(Y,y)}}
\end{tikzcd}
\]
commutes. 

\end{prop}

\begin{proof}
By Corollary \ref{cor:LogIsProObject} we have an equivalence
\[
\map_{\Pro (\Uni_{\pi}^{\cons}(X))}(\underline{\mathscr{A}}_{(X,x)}, f^* \underline{\mathscr{A}}_{(Y,y)}) \simeq \map_{\Pro (\T^{\cons}(S))}(\1, x^* f^* \underline{\mathscr{A}}_{(Y,y)}).
\]
Unwinding the construction of this equivalence via the comparison map \ref{noname:CanMapForUnivPropOfLog} we get the claim. 
\end{proof}

\begin{remark} \label{rem:ExplicitDescriptionOfGammaViaSpans} In fact we can give a very explicit construction of the map $\gamma_f$ in Proposition \ref{prop:CharacterisationOfFunctOfLog} above. Note that $f^* \underline{\mathscr{A}}_{(Y,y)}$ is the pro-object associated to the cosimplicial object in $\Sm_{/X}$ obtained from the cospan
\[
[X \overset{\pi_X}\rightarrow S] \overset{f}\longrightarrow [Y \overset{\pi_Y}\rightarrow S] \overset{y}\longleftarrow [S \overset{\id}\rightarrow S].
\]
This is a priori a cosimplicial object in $\Sch^{\qcqs}_{/S}$ but as in \ref{noname:ConstrCosimplLog} we can consider it as a cosimplicial object in $\Sm_{/X}$ via the projection maps to $X$. Then the map of cospans
\[
\begin{tikzcd}
{[X \overset{\pi_X}\rightarrow S]} \arrow[r, "\id"] \arrow[d, "\id"] & {[X \overset{\pi_X}\rightarrow S]} \arrow[d, "f"] & {[S \overset{\id}\rightarrow S]} \arrow[l, "x"'] \arrow[d, "\id"] \\
{[X \overset{\pi_X}\rightarrow S]} \arrow[r, "f"]                    & {[Y \overset{\pi_Y}\rightarrow S]}                & {[S \overset{\id}\rightarrow S]} \arrow[l, "y"']                 
\end{tikzcd}
\]
gives rise to a map $\tilde{\gamma}_f: \underline{\mathscr{A}}_{(X,x)} \rightarrow f^* \underline{\mathscr{A}}_{(Y,y)}$. It is easy to check that  $x^* \underline{\mathscr{A}}_{(X,x)}$ is the pro object associated to the cospan
\[
\begin{tikzcd}
{[S \overset{\id}\rightarrow S]} \arrow[r, "x"] & {[X \overset{\pi_X}\rightarrow S]} & {[S \overset{\id}\rightarrow S]} \arrow[l, "x"']
\end{tikzcd}
\]
and $x^* f^* \underline{\mathscr{A}}_{(Y,y)} \simeq y^* \underline{\mathscr{A}}_{(Y,y)}$ is associated to
\[
\begin{tikzcd}
{[S \overset{\id}\rightarrow S]} \arrow[r, "y"] & {[Y \overset{\pi_Y}\rightarrow S]} & {[S \overset{\id}\rightarrow S].} \arrow[l, "y"']
\end{tikzcd}
\]
Moreover the map $x^* \tilde{\gamma}_f: x^* \underline{\mathscr{A}}_{(X,x)} \rightarrow x^* f^* \underline{\mathscr{A}}_{(Y,y)} \simeq y^* \underline{\mathscr{A}}_{(Y,y)}$ is induced by the map of cospans
\[
\begin{tikzcd}
{[S \overset{\id}\rightarrow S]} \arrow[r, "x"] \arrow[d, "\id"] & {[X \overset{\pi_X}\rightarrow S]} \arrow[d, "f"] & {[S \overset{\id}\rightarrow S]} \arrow[l, "x"'] \arrow[d, "\id"] \\
{[S \overset{\id}\rightarrow S]} \arrow[r, "y"]                  & {[Y \overset{\pi_Y}\rightarrow S]}                & {[S \overset{\id}\rightarrow S].} \arrow[l, "y"']                 
\end{tikzcd}
\]
Now taking into account the definition of (\ref{eqn:EtaLog}) via a map of cospans we see that the diagram 
\[
\begin{tikzcd}
                                             & \1 \arrow[ld, "{\eta_{(X,x)}}"'] \arrow[rd, "{\eta_{(Y,y)}}"] &                   \\
{x^* \underline{\mathscr{A}}_{(X,x)}} \arrow[r, "x^* \tilde{\gamma}_f"] & {x^* f^* \underline{\mathscr{A}}_{(Y,y)}} \arrow[r, "\sim", no head]             & {y^*\underline{\mathscr{A}}_{(Y,y)}}
\end{tikzcd}
\]
commutes. Therefore the uniqueness in Proposition \ref{prop:CharacterisationOfFunctOfLog} shows that $\tilde{\gamma}_f$ and $\gamma_f$ are equivalent. 
\end{remark}

\begin{lem} \label{lem:GammaFIsFofHK}
Assume that $\T(\_)$ is $\DA_{\et}( \_, \Lambda)$ where $\Lambda$ is a $\Q$-algebra. Let $S$ be a finite dimensional noetherian scheme and $f: X \rightarrow Y$ a morphism of smooth commutative group schemes with connected fibers over $S$. Assume that either $S$ is of characteristic zero or that $X$ and $Y$ are affine. Let us denote the logarithm motives associated to $X$ and $Y$ by $\Log_X$ and $\Log_Y$ respectively. Then
\[
\begin{tikzcd}
\Log_X \arrow[d, "\sim"] \arrow[r, "f_{\#}"]             & f^*\Log_Y \arrow[d, "\sim"]           \\
{\underline{\mathscr{A}}_{(X,x)}} \arrow[r, "\gamma_f"'] & {f^* \underline{\mathscr{A}}_{(X,x)}}
\end{tikzcd}
\]
commutes. Here the vertical equivalences are the equivalences from Corollary \ref{cor:CosimplcialLogAgreesWithClassicalLogByHK} and the map $f_{\#}$ is the map constructed in \cite[4.4.1]{HuberKingsPolylog}. 
\end{lem}

\begin{proof}
We are in the situation of Corollary \ref{cor:CosimplcialLogAgreesWithClassicalLogByHK}. In \cite[4.3.1]{HuberKingsPolylog} a map $\eta^{\Log_X}: \1 \rightarrow s^*\Log_X$ is constructed. It is not hard to see that 
\[
\begin{tikzcd}
\1 \arrow[r, "\eta^{\Log_X}"] \arrow[rd, "{\eta_{(X,x)}}"'] & \Log_X \arrow[d, "\sim"]          \\
                                                            & {\underline{\mathscr{A}}_{(X,x)}}
\end{tikzcd}
\]
commutes: Indeed the construction of the equivalence in \cite[4.5.2]{HuberKingsPolylog} shows that $\eta^{\Log_X}$ corresponds to $\id: \Log_X \rightarrow \Log_X$ under the equivalence
\[
\map_{\Pro(\Uni_{\pi}^{\cons}(X))}(\Log_X, \Log_X) \overset{\sim}\longrightarrow \map_{\Pro (\DA_{\et}(S, \Lambda))}(\1, s^* \Log_X).
\]
Similarly the construction of the equivalence in Theorem \ref{thm:UnivPropOfLog} shows that $\eta_{(X,x)}$ corresponds to $\id: \underline{\mathscr{A}}_{(X,x)} \rightarrow \underline{\mathscr{A}}_{(X,x)}$ under the equivalence
\[
\map_{\Pro(\Uni_{\pi}^{\cons}(X))}(\underline{\mathscr{A}}_{(X,x)}, \underline{\mathscr{A}}_{(X,x)}) \overset{\sim}\longrightarrow \map_{\Pro (\DA_{\et}(S, \Lambda))}(\1, s^* \underline{\mathscr{A}}_{(X,x)}).
\]

It follows straightforward from the explicit constructions in \cite[\S 4]{HuberKingsPolylog} that
\[
\begin{tikzcd}
                                             & \1 \arrow[ld, "{\eta^{\Log_X}}"'] \arrow[rd, "{\eta^{\Log_Y}}"] &                   \\
{x^* \Log_X} \arrow[r, "x^* f_{\#}"] & {x^* f^* \Log_Y} \arrow[r, "\sim", no head]             & {y^*\Log_Y}
\end{tikzcd}
\]
commutes. Hence the uniqueness property of Proposition \ref{prop:CharacterisationOfFunctOfLog} implies the claim.
\end{proof}

\section{The nearby cycles functors with rational coefficients}

\noname \label{noname:ABulletIsLog} Let $\Lambda$ be a $\Q$-algebra and consider the pro-object $\Log$ of $\DA_{\et}(\G_m, \Lambda)$ constructed in \cite[\S 6.4]{HuberKingsPolylog}. Note that the construction of the pro-object $\Log$ is dual to the construction of the ind-object $\Log^\vee$ in \cite[\S 3.6]{AyoubThesisII}. It is shown in \cite[11.14]{AyoubRealizationEtale} that we may essentially replace $\mathscr{A}_{S}$ by $\Log$ in the description of $\Upsilon_f$ in Proposition \ref{prop:NearbyCyclesViaLogAndColimits} when $f$ is of fintie type. With Corollary \ref{cor:CosimplcialLogAgreesWithClassicalLogByHK} one can make this precise. The logarithm motive is crucial for the following Lemma:

\begin{lem} \label{lem:MapInTransitionMapIsEquiForQCoeff}
Consider the situation of Proposition \ref{prop:NearbyCyclesViaLogAndColimits} and assume that $\Lambda$ is a $\Q$-algebra. Then for all $n$ in $\N'^\times$ the map
\[
\colim_{\Delta^{op}} \Hom((f_n)_\eta^* \pi_n^* e_k^*\mathscr{A}_{S_n}, t_n^*M)
\longrightarrow   \colim_{\Delta^{op}}  \Hom((f_n)_\eta^* \pi_n^* \mathscr{A}_{S_n}, t_n^* M)
\]
induced by the canonical map $\varphi_k: \mathscr{A}_{S_n} \rightarrow e_k^* \mathscr{A}_{S_n}$ is an equivalence. 
\end{lem}

\begin{proof}
The map of cosimplicial objects $\varphi_k: \mathscr{A}_{S_n} \rightarrow e_k^* \mathscr{A}_{S_n}$ in $\DA_{\et}(\G_{m, S_n}, \Lambda)$ induces a map  $\underline{\varphi}_k: \underline{\mathscr{A}}_{S_n} \rightarrow e_k^* \underline{\mathscr{A}}_{S_n}$ in $\Pro \Uni_\pi^{\cons}(\G_{m,S_n})$. Hence it suffices to show that the map of pro-objects $\underline{\varphi}_k$ is an equivalence. 

From the explicit construction of $\varphi_k$ in \ref{noname:Abullet} and the construction in Remark \ref{rem:ExplicitDescriptionOfGammaViaSpans} it follows that $\underline{\varphi}_k$ is equivalent to the map $\gamma_{e_k}$ of Proposition \ref{prop:CharacterisationOfFunctOfLog}. Since $e_k: \G_{m,S_n} \rightarrow \G_{m,S_n}$ is an isogeny $(e_k)_{\#}: \Log \rightarrow e_k^* \Log$ is an equivalence by \cite[4.4.2]{HuberKingsPolylog}. Hence Lemma \ref{lem:GammaFIsFofHK} implies that $\underline{\varphi}_k$ is an equivalence as desired. 
\end{proof}

\begin{lem} \label{lem:UnitAdmitsRetraction}
Let $h: Y \rightarrow X$ be a finite surjective morphism where $X$ is a normal noetherian finite dimensional scheme and consider a $\Q$-algebra $\Lambda$. Then the unit map
\[
M \longrightarrow h_* h^* M
\]
admits a retraction for every $M$ in $\DA_{\et}(X, \Lambda)$. Moreover this is true after basechange along any map $X' \rightarrow X$.
\end{lem}

\begin{proof}
Using \cite[5.5.10]{CisinskiDegliseEtale} and \cite[4.2.13]{CisinskiDegliseBook} we can reduce to the case where $M\simeq f_* \1$ for some proper $f:W \rightarrow X$. Using \cite[5.5.12]{CisinskiDegliseEtale}  one observes that it suffices to prove the Lemma in the case where $\Lambda \simeq \Q$. Rational \'etale motives $\DA_{\et}(\_, \Q)$ are equivalent to Beilinson motives $\DM_{\BE}(\_)$ when restricted to finite dimensional noetherian schemes by \cite[5.2.2]{CisinskiDegliseEtale}. This implies that $\DA_{\et}(\_,\Q)$ is separated by \cite[14.3.3]{CisinskiDegliseBook} when restricted to finite dimensional noetherian schemes. Thus we may apply \cite[3.3.40]{CisinskiDegliseBook}. 
\end{proof}

\begin{lem} \label{lem:DriectSummandInFilteredColim}
Let $I$ be a countable filtered poset and $\mathcal{C}$ a stable $\infty$-category which admits countable  colimits. Consider a functor $F: I \rightarrow \C$ such that for any map $i \rightarrow j$ the induced morphism $F(i) \rightarrow F(j)$ is the inclusion of a direct summand. Then for every $i$ in $I$ the canonical map $F(i) \rightarrow \colim F$ is the inclusion of a direct summand. 
\end{lem}

\begin{proof}
First let us note that we may find a cofinal functor $\varphi: \N \rightarrow I$, where $\N$ denotes the linearly ordered set $\{ 0< 1 < 2< \dots  \}$ of natural numbers. This is well known, we give the proof for completeness. Since $I$ is countable we find a bijective map $\sigma: \N \rightarrow \Ob I$.  We may assume that $I$ has an initial object $\sigma(0)$. For any two objects $i, j$ in $I$ let us write $i \prec j$ whenever there is a map $i \rightarrow j$ in the poset $I$. Now inductively define a functor $\varphi: \N \rightarrow I$ by the rules: $\varphi(0) = \sigma(0)$, $\varphi(m) \prec \varphi(n)$ and $ \sigma(m) \prec  \varphi(n)$ for all $m<n$. The second rule makes sure that $\varphi$ is a functor and the third rule implies that $\varphi$ is cofinal. 

Therefore it is sufficient to prove the Lemma in the case $I = \N$ which boils down to the following: Given a $\N$- indexed family of objects $\{x_i\}_{i \in \N}$ in $\mathcal{C}$ the colimit of the diagram
\[
\dots \longrightarrow \bigoplus_{i \leq n} x_i \longrightarrow \bigoplus_{i \leq n+1} x_i \longrightarrow \dots
\]
in $\mathcal{C}$ is equivalent to $\bigoplus_{i \in \N} x_i$. This is clear. 
\end{proof}

\begin{thm}\label{thm:UpsilonRetractOfPsiTame}
Let $S$ be the spectrum of a strictly henselian discrete valuation ring and assume that $\Lambda$ is a $\Q$-algebra. Then for every morphism of finite type $f:X \rightarrow S$, every $M$ in $\DA_{\et}(X_\eta, \Lambda)$ and every $n$ in $\N'^\times$  the canonical map 
\[ \Upsilon_{f_n}(t_n^* M) \longrightarrow \Psi^{\tame}_f(M)\]
 is the inclusion of a direct summand.
\end{thm}

\begin{proof}
Considering the description of $\Psi_f^{\tame}$ in Proposition \ref{prop:NearbyCyclesViaLogAndColimits} and Lemma \ref{lem:DriectSummandInFilteredColim} above it suffices to show that for any $\varphi: n \rightarrow m$ in $\N'^\times$ the transition map
\[
\tau_ \varphi: \Upsilon_{f_m}(t_m^* M) \longrightarrow \Upsilon_{f_n}(t_n^* M)
\]
admits a retract. 

Since $f$ is of finite type the functors $j_*$ and $t_{k*}$ commute with colimits by \ref{prop:SstrLocImpliesCompactGen}. Hence Lemma \ref{lem:MapInTransitionMapIsEquiForQCoeff} shows that the maps in \ref{remark:ExplicitTransitionMapsNearby} (1) induce an equivalence
\[
\Upsilon_{f_n}(t_n^*M) \simeq  i^*j_* t_{k *} t_k^* \colim_{\Delta^{\op}}\Hom((f_m)_\eta^* \pi_m^* \mathscr{A}_{S_m}, t_m^*M), 
\]
and the transition map $\tau_ \varphi$ is simply induced by the unit map $\id \rightarrow  t_{k *} t_k^*$. Therefore $\tau_\varphi$ admits a retraction by Lemma \ref{lem:UnitAdmitsRetraction}.
\end{proof}

\begin{cor} \label{cor:UpsilonfnIsPsitame}
Let $S$ be the spectrum of an excellent strictly henselian discrete valuation ring and assume that $\Lambda$ is a $\Q$-algebra. Then for every $f:X \rightarrow S$ of finite type and every $M$ in $\DA^{\cons}_{\et}(X_\eta, \Lambda)$ there exists an $n$ in $\N'^\times$ such that the canonical map
\[
\Upsilon_{f_n}(t_n^* M) \longrightarrow \Psi_f^{\tame}(M) 
\]
is an equivalence.
\end{cor}

\begin{proof}
By Proposition \ref{prop:SstrLocImpliesCompactGen} the constructible objects in $\DA_{\et}(X_\sigma, \Lambda)$ agree with the compact objects and therefore $\Psi^{\tame}_f(M)$ is compact by Theorem \ref{thm:PropertiesOfEtaleMotivicNearby} (3). The canonical maps into the colimit
\[
\Upsilon_{f_n}(t_n^* M) \longrightarrow \Psi_{f}^{\tame}(M)
\]
are inclusions of direct summands by Theorem \ref{thm:UpsilonRetractOfPsiTame}. By compactness, the identity of  $\Psi^{\tame}_f(M)$ factors as
\[
\Psi^{\tame}_f(M) \longrightarrow \Upsilon_{f_n}(t_n^*M) \longrightarrow \Psi^{\tame}_f(M)
\]
for some $n$ in $\N'^\times$. In particular $\Psi^{\tame}_f(M)$ is a direct summand of $\Upsilon_{f_n}(t_n^*M)$. But since $\Upsilon_{f_n}(t_n^* M) $ is a direct summand of $\Psi_{f}^{\tame}(M)$ this implies that  the canonical map
\[
\Upsilon_{f_n}(t_n^* M) \longrightarrow \Psi^{\tame}_f(M)
\] 
is an equivalence. 
\end{proof}

\begin{cor} \label{cor:UpsilonRetractOfPsi}
Let $S$ be the spectrum of an excellent strictly henselian discrete valuation ring and assume that $\Lambda$ is a $\Q$-algebra. Then for every $f:X \rightarrow S$ of finite type and every $M$ in $\DA^{\cons}_{\et}(X_\eta, \Lambda)$ the canonical map 
\[ \Upsilon_{(t_L \circ f_L)_n}(t_n^* t_L^* M) \longrightarrow \Psi_f(M) \]
is the inclusion of a direct summand for all $n$ in $\N'^\times$ and $L$ in $\Xi_\tau$. In particular we can find an $n$ in $\N'^\times$ and $L$ in $\Xi_\tau$ such that this map is an equivalence.
\end{cor}

\begin{proof} Since the map in question factors as 
\[ \Upsilon_{(t_L \circ f_L)_n}(t_n^* t_L^* M) \longrightarrow \Psi_{(t_L \circ f_L)_n}(t_n^* t_L^* M) \overset{\sim}\longrightarrow\Psi_f(M) \]
we may assume that $K=L$ and $n=1$.

By Theorem \ref{thm:PropertiesOfEtaleMotivicNearby} (2) there exists a $u: K \rightarrow F$ in $\Xi_\tau$ such that the canonical map
\[
\Psi_{t_{F} \circ f_{F}}^{\tame}( t_{F}^* M) \longrightarrow \Psi_f(M)
\]
is an equivalence. Let us consider the composition
\[
\tilde{\tau}_u: 
\Upsilon_{f}(M) \simeq u_\sigma^* \Upsilon_{f}(M) \overset{\Ex^*}\longrightarrow \Upsilon_{f \circ u}(u_\eta^* M) \simeq \Upsilon_{t_{F} \circ f_{F}}(t_{F}^* M)
\]
(using the notations of Remark \ref{remark:ExplicitTransitionMapsNearby} (2)).
From the explicit description of the transition map in Remark \ref{remark:ExplicitTransitionMapsNearby} (2) and the fact that $\Upsilon \rightarrow \Psi^{\tame}$ is a map of specialization systems we see that 
\[
\begin{tikzcd}
\Upsilon_{f}(M) \arrow[d, "\tilde{\tau}_u"'] \arrow[r]            & \Psi^{\tame}_{f}(M) \arrow[d, "\tau_u"]  \\
\Upsilon_{ f_F \circ \tau_F}(t_F^*M) \arrow[r] & \Psi^{\tame}_{ f_F \circ \tau_F}(t_F^*M)
\end{tikzcd}
\]
commutes. Since the bottom horizontal map is the inclusion of a direct summand by Theorem \ref{thm:UpsilonRetractOfPsiTame} it suffices to show that $\tilde{\tau}_u$ is the inclusion of a direct summand. 

This follows from Lemma \ref{lem:UnitAdmitsRetraction} since $\tilde{\tau}_u$ is given by
\begin{align*}
\colim_{\Delta^{\op}} i^*j_*\Hom(f_\eta^* \pi^* \mathscr{A}_S, M) & \overset{\sim}\longrightarrow  i^*j_*\colim_{\Delta^{\op}}\Hom(f_\eta^* \pi^* \mathscr{A}_S, M) \\
& \overset{\unit}\longrightarrow  i^*j_* t_{F*} t_F^* \colim_{\Delta^{\op}}\Hom(f_\eta^* \pi^* \mathscr{A}_S, M) \\
& \overset{\sim}\longrightarrow  i^*j_* \colim_{\Delta^{\op}}\Hom(t_F^* f_\eta^* \pi^* \mathscr{A}_S, t_F^* M) \\
&\overset{\sim}\longrightarrow \colim_{\Delta^{\op}} i^*j_*\Hom((f \circ t_F)_\eta^* \pi^* \mathscr{A}_S, t_F^* M).
\end{align*}

The last sentence follows analogue to Corollary \ref{cor:UpsilonfnIsPsitame}.

\end{proof}

\chapter{A local monodromy theorem for $J$-adic realizations of \'etale motives}

Our main Goal in this chapter is to prove two generalizations of Grothendieck's local monodromy theorem: Corollary \ref{cor:MonodromyThm} and Corollary \ref{cor:LocMonForPerverseThings}. It turns out that this is a rather easy consequence of a theorem by Ayoub (Theorem \ref{thm:MonodromySquareAyoub}) and Corollary \ref{cor:UpsilonfnIsPsitame}.

We start off by introducing the $J$-adic realization of \'etale motives and the $J$-adic nearby cycles functor. 

\section{The $J$-adic realization of the motivic nearby cycles functor} \label{section:AdicRealNearby}

\noname \label{noname:SetupGalois} Throughout this chapter let us fix the following situation: $S$ is the spectrum of an excellent strictly henselian discrete valuation ring $R$ with fixed uniformizer $\pi$. We write $K := \Frac(S)$, $\kappa := S/(\pi)$ and denote by $p$ the characteristic of $\kappa$. We fix a separable closure $\bar{K}$ of $K$ and write $\tilde{K}:=  K(\pi^{1/n} | n \in \N'^\times) \subset \bar{K}$. Recall from \ref{const:TotalPsi} that there is a short exact sequence of groups 
\[
0 \longrightarrow P \longrightarrow \Gal(\bar{K}/K) \overset{\chi}\longrightarrow \widehat{\Z}'(1) \longrightarrow 0 
\]
and an isomorphism $\Gal(\tilde{K}/K) \simeq \widehat{\Z}'(1)$ under which $\chi$ corresponds to the restriction map $\Gal(\bar{K}/K) \rightarrow \Gal(\tilde{K}/K)$. 

Finally let $\Lambda$ be a ring and $J \subset \Lambda$ an ideal such that $\Lambda/J$ is of positive characteristic $m$ invertible in $\mathcal{O}(S)$.

\noname Let $Y$ be a scheme. Since we do not need any $\infty$-categorical tools for this chapter let us denote by $\DA_{\et}(Y, \Lambda)$ the triangulated category obtained as the homotopy category of the $\infty$-category defined in \ref{noname:SymMonStrOfDA}.

\noname \label{noname:JAdicRealization} Let $\N$ denote the poset $\{ 0 \rightarrow 1 \rightarrow 2 \rightarrow \dots \}$ and let $\Lambda / J^*$ denote the diagram
\[
\left\lbrace  \dots \longrightarrow \Lambda/ J^3 \longrightarrow \Lambda/ J^2 \longrightarrow \Lambda/ J \longrightarrow 0 \right\rbrace
\]
of rings indexed by the poset $\N^{op}$. We may consider $\Lambda / J^*$ as a ring object in the topos of presheaves (of sets) on $\N$. In particular we can talk about the abelian category of $\Lambda / J^*$-modules. 

Let $Y$ be a scheme such that the characteristic of $\Lambda/J$ is invertible in $\mathcal{O}(Y)$. We denote by $\T_{\et}(Y, \Lambda/J^*)$ the unbounded derived category of \'etale sheaves on $Y$ with values in the abelian category of $\Lambda/J^*$-modules. If $Y$ is locally noetherian we may follow \cite[\S 5]{AyoubRealizationEtale} to obtain a triangulated symmetric monoidal functor
\[
\R_J: \DA_{\et}(Y, \Lambda) \longrightarrow {\T}_{\et}(Y, \Lambda/ J^*).
\]
For $s$ in $\N$ there are canonical functors
\[
s^* : {\T}_{\et}(Y, \Lambda/ J^*) \longrightarrow {\T}_{\et}(Y, \Lambda/ J^s),
\]
which form a conservative family by \cite[5.4]{AyoubRealizationEtale}. Let
\[
\hat{\T}_{\et}(Y, \Lambda_J) \subset {\T}_{\et}(Y, \Lambda/ J^*)
\]
denote the full subcategory consisting of objects $K$ such that the canonical map
\[
(s+1)^* K \otimes_{\Lambda/J^{s+1}} \Lambda/J^{s} \longrightarrow s^* K
\]
is an equivalence for all $s$ in $\N$. (Note that we always talk about dervied functors without explicitly saying so.) Then $\R_J$ factors through $\hat{\T}_{\et}(Y, \Lambda_J)$ by \cite[5.8]{AyoubRealizationEtale}. By slight abuse of notation we write
\[
\R_J: \DA_{\et}(Y, \Lambda) \longrightarrow \hat{\T}_{\et}(Y, \Lambda_J)
\]
and call it the \textit{$J$-adic realization functor}.

The formation of ${\T}_{\et}(\_, \Lambda/ J^*)$ comes equipped with the six functors. The functors $f^*,f_*,f_!, f^!$ and $\otimes$ restrict to the subcategories $\hat{\T}_{\et}(\_, \Lambda_J)$ when restricted to morphisms of finite type between noetherian finite dimensional schemes (Using Proposition \ref{prop:ChangeOfCoeffCompWSixFunctors} we can extend the proof of \cite[6.7]{AyoubRealizationEtale}). As in \cite[6.9]{AyoubRealizationEtale} we can deduce that the $J$-adic realization
\[
\R_J: \DA_{\et}(\_, \Lambda) \longrightarrow \hat{\T}_{\et}(\_, \Lambda_J)
\]
commutes with the six functors $f^*,f_*,f_!, f^!$ and $\otimes$ when restricted to morphisms of finite type between noetherian finite dimensional schemes. Moreover it commutes with the formation of $\Hom$ when restricted to constructible objects in $\DA_{\et}^{\cons}(\_, \Lambda)$. 

\noname An object $K$ in $\hat{\T}_{\et}(Y, \Lambda_J)$ is called \textit{constructible} if for all $s$ in $\N$ and $n \in \Z$ the object $H^n(s^*K)$ is a constructible \'etale sheaf the classical sense (i.e in the sense of \cite[IX, 2.3]{SGA4.3}) and $H^n(s^*K) = 0$ for all but finitely many $n$. We write $\hat{\T}^{\cons}_{\et}(Y, \Lambda_J)$ for the full subcategory of $\hat{\T}_{\et}(Y, \Lambda_J)$ consisting of constructible objects.

We define $\hat{\T}^{\cons}_{\et}(Y, \Lambda_J \otimes \Q)$  to be the pseudo abelian envelope of $\hat{\T}^{\cons}_{\et}(Y, \Lambda_J) \otimes \Q.$
Here $\hat{\T}^{\const}_{\et}(Y, \Lambda_J) \otimes \Q$ denotes the category whose objects are the objects of $\hat{\T}^{\cons}_{\et}(Y, \Lambda_J)$ and 
\[
\map_{\hat{\T}^{\cons}_{\et}(Y, \Lambda_J) \otimes \Q}(A,B) = \map_{\hat{\T}^{\cons}_{\et}(Y, \Lambda_J)}(A,B) \otimes \Q\]
for $A,B$ in $\hat{\T}^{\cons}_{\et}(Y, \Lambda_J)$. As remarked in \cite[9.4]{AyoubRealizationEtale} $\hat{\T}^{\cons}_{\et}(Y, \Lambda_J \otimes \Q)$ carries a canonical structure of a triangulated category. Moreover the functor $\R_J$ above induces a triangulated functor
\[
\R_J: \DA^{\cons}_{\et}(Y, \Lambda \otimes_\Z \Q) \longrightarrow \hat{\T}^{\cons}_{\et}(Y, \Lambda_J \otimes \Q)
\]
by \cite[9.5]{AyoubRealizationEtale}.

The formation of $\hat{\T}^{\cons}_{\et}(\_, \Lambda_J \otimes \Q)$ comes equipped with the six functors when restricted to morphisms of finite type between quasi-excellent noetherian schemes of finite dimension and the realization functors
\[
\R_J: \DA^{\cons}_{\et}(\_, \Lambda \otimes_\Z \Q) \longrightarrow \hat{\T}^{\cons}_{\et}(\_, \Lambda_J \otimes \Q)
\]
commute with the six operations under these restrictions. (One extends the proof of  \cite[9.7]{AyoubRealizationEtale} using Theorem \ref{thm:SixFunctorsPresConstr}.) 

Recall the situation fixed in \ref{noname:SetupGalois}:  $S$ is the spectrum of an excellent strictly henselian discrete valuation ring $R$. In particular in this situation the $J$-adic realization commutes with the six functors when restricted to schemes of finite type over $S$.

\noname Consider a morphism of finite type  $f: X \rightarrow S$ and let us denote by $\tilde{S}$ (resp. $\bar{S}$) the integral closure of $S$ in $\tilde{K}$ (resp. $\bar{K}$). We write $\eta := \Spec K$, $\tilde{\eta} := \Spec \tilde{K}$, $\bar{\eta}:= \Spec \bar{K}$ and $\sigma := \Spec \kappa$. Then there are decompositions
\[
\tilde{\eta} \overset{\tilde{j}} \longrightarrow \tilde{S} \overset{\tilde{i}}\longleftarrow \sigma 
\]
and
\[
\bar{\eta} \overset{\bar{j}} \longrightarrow \bar{S} \overset{\bar{i}}\longleftarrow \sigma 
\]
of $\tilde{S}$ and $\bar{S}$ into a closed immersion and its open complement. 

Denote by $\tilde{f}: \tilde{X} \rightarrow \tilde{S}$ (resp. $\bar{f}: \bar{X} \rightarrow \bar{S}$) the pullback of $f$ along $\tilde{\theta}: \tilde{S}\rightarrow S$ (resp. $\bar{\theta}: \bar{S}\rightarrow S$). Then we obtain the diagram 
\[
\begin{tikzcd}
X_{\tilde{\eta}} \arrow[r, "\tilde{j}"] \arrow[d, "f_{\tilde{\eta}}"] & \tilde{X} \arrow[d, "\tilde{f}"] & X_\sigma \arrow[l, "\tilde{i}"'] \arrow[d, "f_\sigma"] \\
\tilde{\eta} \arrow[r, "\tilde{j}"]                                 & \tilde{S}                        & \sigma. \arrow[l, "\tilde{i}"']                        
\end{tikzcd}
\]
by taking pullbacks. We define the \textit{tame $J$-adic nearby cycles functor} by
\[
\Psi^{J,\tame}_f := \tilde{i}^*\tilde{j}_*\tilde{\theta}^*: \hat\T_{\et}^{\cons}(X_\eta, \Lambda_J \otimes \Q) \longrightarrow \hat\T_{\et}^{\cons}(X_\sigma, \Lambda_J \otimes \Q).
\]
Analogously we obtain a diagram
\[
\begin{tikzcd}
X_{\bar{\eta}} \arrow[r, "\bar{j}"] \arrow[d, "f_{\bar{\eta}}"] & \bar{X} \arrow[d, "\bar{f}"] & X_\sigma \arrow[l, "\bar{i}"'] \arrow[d, "f_\sigma"] \\
\bar{\eta} \arrow[r, "\bar{j}"]                                 & \bar{S}                        & \sigma. \arrow[l, "\bar{i}"']                        
\end{tikzcd}
\]
and define the \textit{total $J$-adic nearby cycles functor} by
\[
\Psi^{J}_f := \bar{i}^*\bar{j}_*\bar{\theta}^*: \hat\T_{\et}^{\cons}(X_\eta, \Lambda_J \otimes \Q) \longrightarrow \hat\T_{\et}^{\cons}(X_\sigma, \Lambda_J \otimes \Q).
\]

By \cite[10.12, 10.17]{AyoubRealizationEtale} the $J$-adic realization is compatible with the formation of (tame) nearby cycles in the sense that
\[
\begin{tikzcd}
{\DA_{\et}^{\cons}(X_\eta, \Lambda \otimes \Q))} \arrow[d, "\Psi_f^{\tame}"'] \arrow[r, "\R_J"] & {\hat\T_{\et}^{\cons}(X_\eta, \Lambda_J \otimes \Q)} \arrow[d, "{\Psi_f^{J, \tame}}"] \\
{\DA^{\cons}_{\et}(X_\sigma, \Lambda \otimes \Q))} \arrow[r, "\R_J"]                            & {\hat\T^{\cons}_{\et}(X_\sigma, \Lambda_J \otimes \Q)}                                      
\end{tikzcd}
\]
and
\[
\begin{tikzcd}
{\DA^{\cons}_{\et}(X_\eta, \Lambda \otimes \Q))} \arrow[d, "\Psi_f^{}"'] \arrow[r, "\R_J"] & {\hat\T^{\cons}_{\et}(X_\eta, \Lambda_J \otimes \Q)} \arrow[d, "\Psi_f^{J}"] \\
{\DA^{\cons}_{\et}(X_\sigma, \Lambda \otimes \Q))} \arrow[r, "\R_J"]                       & {\hat\T^{\cons}_{\et}(X_\sigma, \Lambda_J \otimes \Q)}                             
\end{tikzcd}
\]
commute.

\noname \label{noname:ActionsOfGalOnDifferentNearbyCycles} Let $A$ be an object of  $\hat\T_{\et}^{\cons}(X_\eta, \Lambda_J \otimes \Q)$ and $\xi$ an element of $\widehat{\Z}'(1) \simeq \Gal(\tilde{K}/K)$. By slight abuse of notation let us write $\xi: \tilde{\eta} \rightarrow \tilde{\eta}$ and $\xi: \tilde{S} \rightarrow \tilde{S}$ for the induced morphisms of affine schemes. Also we write $\xi: X_{\tilde{\eta}} \rightarrow X_{\tilde{\eta}}$, $\xi: \tilde{X} \rightarrow \tilde{X}$, $\tilde{\theta}: \tilde{X} \rightarrow X$, $\tilde{\theta}: X_{\tilde{\eta}} \rightarrow X_\eta$ for the morphisms of schemes induced via base change along $f: X \rightarrow S$. Consider the diagram
\[
\begin{tikzcd}
X_{\tilde{\eta}} \arrow[r, "\tilde{j}"] \arrow[d, "\xi"']          & \tilde{X} \arrow[d, "\xi"]          & X_\sigma \arrow[d, "\id"] \arrow[l, "\tilde{i}"'] \\
X_{\tilde{\eta}} \arrow[r, "\tilde{j}"] \arrow[d, "\tilde\theta"'] & \tilde{X} \arrow[d, "\tilde\theta"] & X_\sigma \arrow[l, "\tilde{i}"'] \arrow[d, "\id"] \\
X_\eta \arrow[r, "j"]                                             & X                                   & X_\sigma. \arrow[l, "i"']                          
\end{tikzcd}                    
\]
Then we may define a map $\xi: \Psi^{J, \tame}_f(A) \rightarrow \Psi^{J,\tame}_f(A)$ as the composition
\[
\tilde{i}^*\tilde{j}_* \tilde\theta^*A 
 \overset{\sim}\longrightarrow \tilde{i}^* \xi^* \tilde{j}_* \tilde\theta^*A \overset{\sim}\longrightarrow \tilde{i}^*  \tilde{j}_* \xi^* \tilde\theta^*A \overset{\sim}\longrightarrow \tilde{i}^*\tilde{j}_* \tilde\theta^*A.
\]
Completely analogously for $\lambda$ in $\Gal(\bar{K}/K)$ we write $\lambda: \bar{\eta} \rightarrow \bar{\eta}$, $\lambda: \bar{S} \rightarrow \bar{S}$ for the induced morphisms of affine schemes and $\lambda: X_{\bar{\eta}} \rightarrow X_{\bar{\eta}}$, $\lambda: \bar{X} \rightarrow \bar{X}$, $\tilde{\theta}: \tilde{X} \rightarrow X$, $\tilde{\theta}: X_{\tilde{\eta}} \rightarrow X_\eta$ for the morphisms obtained by base change along $f$. Consider the diagram
\[
\begin{tikzcd}
X_{\bar{\eta}} \arrow[r, "\bar{j}"] \arrow[d, "\lambda"']          & \bar{X} \arrow[d, "\lambda"]          & X_\sigma \arrow[d, "\id"] \arrow[l, "\bar{i}"'] \\
X_{\bar{\eta}} \arrow[r, "\bar{j}"] \arrow[d, "\bar\theta"'] & \bar{X} \arrow[d, "\bar\theta"] & X_\sigma \arrow[l, "\bar{i}"'] \arrow[d, "\id"] \\
X_\eta \arrow[r, "j"]                                             & X                                   & X_\sigma. \arrow[l, "i"']                          
\end{tikzcd}                    
\]
Then we may define a map $\lambda: \Psi^{J}_f(A) \rightarrow \Psi^{J}_f(A)$ as the composition
\[
\bar{i}^*\bar{j}_* \bar\theta^*A 
 \overset{\sim}\longrightarrow \bar{i}^* \lambda^* \bar{j}_* \bar\theta^*A \overset{\sim}\longrightarrow \bar{i}^*  \bar{j}_* \lambda^* \bar\theta^*A \overset{\sim}\longrightarrow \bar{i}^*\bar{j}_* \bar\theta^*A.
\]

It is straightforward to check that these rules give rise to actions of $\Gal(\tilde{K}/K)$ on $\Psi^{J, \tame}_f(A)$ and of $\Gal(\bar{K}/K)$ on $\Psi^{J}_f(A)$ in $\hat\T_{\et}^{\cons}(X_\sigma, \Lambda_J \otimes \Q)$.

\noname Let $m$ be the characteristic of $\Lambda/J$ and denote by $\Lambda_J(1)$ the $\N_{>0}$ indexed system of abelian groups $\left\lbrace\mu_{m^k}(K) \otimes_{\Z / m^k \Z} \Lambda/J^k \right\rbrace_{k \in \N>0}. $ Via the group homomorphism $\chi$ of \ref{noname:SetupGalois} this system is endowed with an action of $\Gal(\bar{K}/K)$ and hence we can consider $\Lambda_J(1)$ as an object in $\hat\T_{\et}(\eta , \Lambda_J)$. Since $\R_{\modulo J^k}\1(1) \simeq \mu_{m^k}(K)\otimes_{\Z / m^k \Z} \Lambda/J^k$ in $\T_{\et}(\eta, \Lambda/ J^k)$ (see the proof of Theorem \ref{thm:Rigidity}) it follows from the construction of $\R_J$ that $\R_J \1(1) \simeq \Lambda_J(1)$. 

Via the composition
\[\widehat{\Z}'(1) \simeq \lim_{n \in \N'^\times} \left( \mu_n(K) \right) \rightarrow \lim_{k\in \N} \left( \mu_{m^k}(K) \right) \rightarrow \lim_{k\in \N} \left(\mu_{m^k}(K) \otimes_{\Z / m^k \Z} \Lambda/J^k \right)
\] 
an element $\xi \in \widehat{\Z}'(1)$ gives rise to an element in $\Lambda_J (1)$ or equivalently to a map
\[
\xi: \Lambda_J \longrightarrow \Lambda_J(1)
\]
in $\hat{\T}_{\et}(\eta, \Lambda_J)$.

\noname Let $f:X \rightarrow S$ be a morphism of schemes and $M$ in $\DA_{\et}(X_\eta, \Lambda \otimes \Q)$. By \cite[11.16]{AyoubRealizationEtale} there is a map
\[
N: \Upsilon_f(M) \longrightarrow \Upsilon_f(M)(-1),
\]
 the so called \textit{monodromy operator}, which fits into the  \textit{monodromy distinguished triangle}
\[
\chi_f(M) \longrightarrow \Upsilon_f(M) \overset{N}\longrightarrow \Upsilon_f(M)(-1).
\]
Moreover $N$ is nilpotent whenever $M$ is constructible and $f$ is of finite type.

\noname Combining the two observations we can define for any $f:X \rightarrow S$ of finite type and $M$ in $\DA^{\cons}_{\et}(X_\eta, \Lambda \otimes \Q)$ a map
\[
N \cdot \xi: \R_J \Upsilon_f(M) \simeq \R_J \Upsilon_f(M) \otimes \Lambda_J \overset{\R_J N \otimes \xi}\longrightarrow \R_J \Upsilon_f(M)(-1) \otimes \Lambda_J (1) \simeq \R_J \Upsilon_f(M).
\]
Since $N$ is nilpotent there is a well defined map
\[
\exp(N \cdot \xi):= \sum_{i=0}^\infty \frac{(N \cdot \xi)^{\circ i}}{i!} :\R_J \Upsilon_f(M)  \longrightarrow \R_J \Upsilon_f(M).
\]

\noname For a $M$ in $\DA_{\et}(X_\eta, \Lambda)$ and $k$ in $\N'^\times$ let us write 
\[
i_k: \Upsilon_{f_k} (t_k^*M) \longrightarrow \Psi_f(M) \simeq \colim_{n \in \N'^\times} \Upsilon_{f_n} (t_n^*M)
\]
for the canonical map into the colimit.  

\begin{thm} \label{thm:MonodromySquareAyoub}
Let $f:X \rightarrow S$ be of finite type, $M$ in $\DA^{\cons}_{\et}(X_\eta, \Lambda \otimes \Q)$ and $\xi \in\widehat{\Z}'(1)$. Then the diagram
\[
\begin{tikzcd}
 \R_J \Upsilon_{f}( M)  \arrow[d, "\exp(N \cdot \xi)"'] \arrow[r, "\R_J i_1"] & \R_J \Psi^{\tame}_{f}(M) \arrow[r, "\sim"] & { \Psi^{J,\tame}_{f}(\R_J M)} \arrow[d, "\xi"] \\
 \R_J \Upsilon_{f}(M)  \arrow[r, "\R_J i_1 "]                          & \R_J \Psi^{\tame}_{f}(M) \arrow[r, "\sim"] & { \Psi^{J,\tame}_{f}(\R_J M)}                 
\end{tikzcd}
\]
commutes.
\end{thm}

\begin{proof}
This is \cite[11.17]{AyoubRealizationEtale}.
\end{proof}

\section{Local monodromy}

\noname For a finite field extension $L/K$ in $\bar{K}$ let $S_L$ denote the normalization of $S$ in $L$ and write $f_L: X_L \rightarrow S_L$ for the base change of $f$ along the canonical map $S_L \rightarrow S$. We write $\eta_L := \Spec L$ and $t_L: \eta_L \rightarrow \eta$ for the induced map of affine schemes as well as $t_L: (X_L)_\eta \rightarrow X_\eta$ for the base change of $t_L$ along $f_\eta$. Then there is a canonical equivalence of functors
\[
\varphi_L:  \Psi^{J}_{f_L}(t_L^*\_) \simeq \bar{i}^*\bar{j}_*\bar{\theta}_L^* t_L^ * \simeq \bar{i}^*\bar{j}_*\bar{\theta}^*\simeq \Psi^{J}_{f}(\_).
\]

\begin{lem} \label{lem:ActionOfGLonPsiLandPsiIsComp}
 For any finite separable extension $L/K$ in $\bar{K}$, $\lambda \in \Gal(\bar{K}/L) \subset \Gal(\bar{K}/K)$ and $A$ in $\hat\T_{\et}^{\cons}(X_\eta, \Lambda_J \otimes \Q)$ the square
\[
\begin{tikzcd}
\Psi_{f_L}^{J}(t_L^*A) \arrow[d, "\lambda"'] \arrow[r, "\varphi_L"] & \Psi^{J}_{f}(A) \arrow[d, "\lambda"] \\
\Psi^{J}_{f_L}(t_L^*A) \arrow[r, "\varphi_L"]                      & \Psi^{J}_{f}(A)                    
\end{tikzcd}
\]
commutes. 
\end{lem}

\begin{proof}
This is expressed by the commutativity of 
\[
\begin{tikzcd}
\bar{i}^*\bar{j_*}\bar{\theta}^* A \arrow[d, "\sim"] \arrow[r, "\sim"] & \bar{i}^* \lambda^* \bar{j_*}\bar{\theta}^* A \arrow[d, "\sim"] \arrow[r, "\sim"] & \bar{i}^* \bar{j_*} \lambda^*\bar{\theta}^* A \arrow[r, "\sim"] \arrow[d, "\sim"] & \bar{i}^*\bar{j_*}\bar{\theta}^* A \arrow[d, "\sim"] \\
\bar{i}^*\bar{j_*}\bar{\theta}_L^* t_L^*A \arrow[r, "\sim"]            & \bar{i}^* \lambda^*\bar{j_*}\bar{\theta}_L^* t_L^*A \arrow[r, "\sim"]             & \bar{i}^* \bar{j_*}\lambda^*\bar{\theta}_L^* t_L^*A \arrow[r, "\sim"]             & \bar{i}^*\bar{j_*}\bar{\theta}_L^* t_L^*A.           
\end{tikzcd}
\]

\end{proof}

\noname For $n$ in $\N'^\times$ there is an isomorphism of groups $\Gal(\tilde{K}/K(\pi^{1/n})) \simeq \widehat{\Z}'(1)$ under which the inclusion $\Gal(\tilde{K}/K(\pi^{1/n})) \subset \Gal(\tilde{K}/K)$ corresponds to the injective group homomorphism
\[
(.)^n: \widehat{\Z}'(1) \longrightarrow (\widehat{\Z}'(1))^n \subset \widehat{\Z}'(1). 
\]
As in \ref{noname:SetupPropNearbyCycleViaColimits} let us denote by $S_n$ the normalization of $S$ in $K(\pi^{1/n})$ and write $f_n: X_n \rightarrow S_n$ for the base change of $f$ along the canonical map $S_n \rightarrow S$. We write $\eta_n := \Spec K(\pi^{1/n})$ and  $t_n: \eta_n \rightarrow \eta$ for the induced map of affine schemes as well as $t_n: (X_n)_\eta \rightarrow X_\eta$ for the base change of $t_n$ along $f_\eta$. Then there is a canonical equivalence of functors
\[
\varphi_n:\Psi^{J, \tame}_{{f}_n}(t_n^*\_)  \simeq \tilde{i}^*\tilde{j}_*\tilde{\theta}_n^* t_n^*  \simeq \tilde{i}^*\tilde{j}_*\tilde{\theta}^* \simeq \Psi^{J,\tame}_{f}(\_).
\]

\begin{lem} \label{lem:TameActionIsComp}
For any $n$ in $\N'^\times$, $\xi \in \widehat{\Z}'(1) \simeq \Gal(\tilde{K}/ K(\pi^{1/n}))$ and $A$ in $\hat\T_{\et}^{\cons}(X_\eta, \Lambda_J \otimes \Q)$ the square
\[
\begin{tikzcd}
\Psi_{{f}_n}^{J, \tame}(t_n^*A) \arrow[d, "\xi"'] \arrow[r, "\varphi_n"] & \Psi^{J, \tame}_{f}(A) \arrow[d, "\xi^n"] \\
\Psi^{J, \tame}_{{f}_n}(t_n^*A) \arrow[r, "\varphi_n"]                      & \Psi^{J, \tame}_{f}(A)                    
\end{tikzcd}
\]
commutes. 

\end{lem}

\begin{proof}
Analogous to Lemma \ref{lem:ActionOfGLonPsiLandPsiIsComp}.
\end{proof}

\noname Let us write $\theta: \bar{\eta} \rightarrow \tilde{\eta}$ for the canonical map and denote maps induced by base change $X_{\bar{\eta}} \rightarrow X_{\tilde{\eta}}$ and $\bar{X} \rightarrow \tilde{X}$ also by $\theta.$ Consider the diagram 
\[
\begin{tikzcd}
X_{\bar{\eta}} \arrow[r, "\bar{j}"] \arrow[d, "\theta"'] & \bar{X} \arrow[d, "\theta"] & X_\sigma \arrow[l, "\bar{i}"'] \arrow[d, "\id"] \\
X_{\tilde{\eta}} \arrow[r, "\tilde{j}"']                  & \tilde{X}                   & X_\sigma \arrow[l, "\tilde{i}"]               
\end{tikzcd}
\]
consisting of pullback squares. Then the composition
\[
\tilde{i}^* \tilde{j}_* \tilde{\theta}^* \overset{\sim}\longrightarrow \bar{i}^* \theta^* \tilde{j}_* \tilde{\theta}^* \overset{\Ex}\longrightarrow \bar{i}^* \bar{j}_* \theta^* \tilde{\theta}^* \simeq \bar{i}^* \bar{j}_* \bar{\theta}^*
\]
defines a natural transformation $\alpha: \Psi^{J, \tame}_f(\_) \rightarrow \Psi_f^{J}(\_)$. Here $\Ex$ denotes the exchange map associated to the left  pullback square.

\begin{lem} \label{lem:lambdaActsOnTrameViaChiLambda}
For any $\lambda$ in $\Gal(\bar{K}/K)$ and $A$ in $\hat\T_{\et}^{\cons}(X_\eta, \Lambda_J \otimes \Q)$ the diagram
\[
\begin{tikzcd}
{\Psi_f^{J,\tame}(A)} \arrow[d, "\alpha"'] \arrow[r, "\chi(\lambda)"] & {\Psi_f^{J,\tame}(A)} \arrow[d, "\alpha"] \\
\Psi_f^{J}(A) \arrow[r, "\lambda"]                                  & \Psi_f^{J}(A)                           
\end{tikzcd}
\]
commutes.
\end{lem}

\begin{proof}
Note that the diagram
\[
\begin{tikzcd}
\bar{X} \arrow[r, "\lambda"] \arrow[d, "\theta"'] & \bar{X} \arrow[d, "\theta"] \\
\tilde{X} \arrow[r, "\chi(\lambda)"]             & \tilde{X}                  
\end{tikzcd}
\]
commutes. Hence it is easy to check that the inner squares of the  diagram
\[
\begin{tikzcd}
\tilde{i}^* \tilde{j}_* \tilde{\theta}^* \arrow[r, "\sim", no head] \arrow[d, "\sim", no head] & \tilde{i}^* \chi(\lambda)^* \tilde{j}_* \tilde{\theta}^* \arrow[r, "\Ex"] \arrow[d, "\sim", no head]        & \tilde{i}^*  \tilde{j}_* \chi(\lambda)^* \tilde{\theta}^* \arrow[r, "\sim", no head] \arrow[d, "\sim", no head] & \tilde{i}^* \tilde{j}_* \tilde{\theta}^* \arrow[d, "\sim", no head] \\
\bar{i}^* \theta^* \tilde{j}_* \tilde{\theta}^* \arrow[dd, "\Ex"] \arrow[r, "\sim", no head]   & \bar{i}^* \theta^* \chi(\lambda)^* \tilde{j}_* \tilde{\theta}^* \arrow[r, "\Ex"] \arrow[d, "\sim", no head] & \bar{i}^*  \theta^* \tilde{j}_* \chi(\lambda)^* \tilde{\theta}^* \arrow[r, "\sim", no head] \arrow[d, "\Ex"]    & \bar{i}^* \theta^* \tilde{j}_* \tilde{\theta}^* \arrow[dd, "\Ex"]   \\
                                                                                               & \bar{i}^* \lambda^* \theta^* \tilde{j}_* \tilde{\theta}^* \arrow[d, "\Ex"]                                  & \bar{i}^*  \tilde{j}_* \theta^* \chi(\lambda)^* \tilde{\theta}^* \arrow[d, "\sim", no head]                     &                                                                     \\
\bar{i}^* \bar{j}_* \bar{\theta}^* \arrow[r, "\sim", no head]                                  & \bar{i}^* \lambda^* \bar{j}_* \bar{\theta}^* \arrow[r, "\Ex"]                                               & \bar{i}^*  \bar{j}_* \lambda^* \bar{\theta}^* \arrow[r, "\sim", no head]                                        & \bar{i}^* \bar{j}_* \bar{\theta}^*                                 
\end{tikzcd}
\]
commute. Here $\Ex$ denote the respective canonical exchange maps. 
\end{proof}

\noname Let $\C$ be an additive category, $G$ a group and $X$ an object of $\C$ with an action of $G$. By this we mean that there is a group homomorphism
\[
\rho: G \longrightarrow \End_{\C}(X).
\]
For a $g \in G$ and $m \in \Z_{>0}$ we write $(\rho(g)-1)^m$ for the $m$-fold composition of the endomorphism $\rho(g)-\id: X \rightarrow X$. We say that \textit{the action of $(g- 1)^{m}$ on $X$ is zero} if $(\rho(g)- \id)^{m}: X \rightarrow X$ is the zero map. The action of a $g \in G$ on $X$ is sometimes called \textit{unipotent} if there is a $m \in \Z_{>0}$ such that the action of $(g- 1)^{m}$ on $X$ is zero.

\begin{thm} \label{thm:OpenSubgroupActsUnipotentlyOnPSi}
Let $S$ be the spectrum of an excellent strictly henselian discrete valuation ring, $f:X \rightarrow S$ a morphism of finite type and $M$ in $\DA_{\et}^{\cons}(X, \Lambda \otimes \Q)$. Then there exists an open subgroup $H \subset \Gal(\bar{K}/K)$ and a positive integer $m$ such that the action of $(\lambda- 1)^{m}$ on $\Psi^J_f(\R_J(M))$ is zero for all $\lambda \in H$.
\end{thm}

\begin{proof}
Let $L/K$ in $\Xi_\tau$ be a finite extension as in Theorem \ref{thm:PropertiesOfEtaleMotivicNearby}(2) such that the composition of the canonical maps
\[
\Psi^{\tame}_{f_L}(t_L^* M) \longrightarrow \Psi_{f_L}(t_L^* M) \overset{\sim}\longrightarrow \Psi_f(M)
\]
is an equivalence. Since $\Gal(\bar{K}/L) \subset \Gal(\bar{K}/K)$ is an open subgroup it suffices to find an open subgroup $H \subset \Gal(\bar{K}/L)$ satisfying the claim. Hence using Lemma \ref{lem:ActionOfGLonPsiLandPsiIsComp} we may assume that the canonical map
\begin{equation} \label{eqn:ComparisonMapPsiTameandPsi}
\Psi^{\tame}_{f}(M) \longrightarrow \Psi_{f}(M)
\end{equation}
is an equivalence.

By Corollary \ref{cor:UpsilonfnIsPsitame} there exists a $k$ in $\N'^\times$ such that 
\[
i_k: \Upsilon_{{f}_k}(t_k^* M) \longrightarrow \Psi^{\tame}_{f}(M)
\]
is an equivalence. Since $i_k$ factors as
\[
\begin{tikzcd}
\Upsilon_{{f}_k}(t_k^* M) \arrow[rd, "i_1"'] \arrow[rr, "i_k"] &                                                & \Psi^{\tame}_{f}(M) \\
                                                               & \Psi^{\tame}_{f_k}(t_k^*M) \arrow[ru, "\sim"'] &                    
\end{tikzcd}
\]
we see in particular that 
\[
i_1: \Upsilon_{{f}_k}(t_k^* M) \longrightarrow \Psi^{\tame}_{{f}_k}(t_k^*M)
\]
is an equivalence. Let $(\widehat{\Z}'(1))^k \subset \widehat{\Z}'(1)$ be the finite index subgroup consisting of the elements of the form $\xi^k$ for all $\xi$ in $\widehat{\Z}'(1)$ and set $H := \chi^{-1} ((\widehat{\Z}'(1))^k) \subset \Gal(\bar{K}/K)$.

By Lemmas \ref{lem:TameActionIsComp} and \ref{lem:lambdaActsOnTrameViaChiLambda} we are reduced to showing that there exists a positive integer $m$ such that the action of $(\chi(\Lambda) - 1)^m$  on $\Psi^{J, \tame}_{{f}_k}(\R_J t_k^*M)$ is zero for all $\lambda \in H$. Since
\[
N: \Upsilon_{{f}_k}(t_k^*M) \longrightarrow \Upsilon_{{f}_k}(t_k^* M) (-1)
\]
is nilpotent by \cite[11.16]{AyoubRealizationEtale} there exists a positive integer $m$ such that
\[
(\exp(N \cdot \xi) - \id)^m: \Upsilon_{{f}_k}(t_k^*M) \longrightarrow \Upsilon_{{f}_k}(t_k^* M)
\]
is the zero map for every $\xi$ in $\widehat{\Z'}(1) \simeq \Gal(\tilde{K}/ K(\pi^{\frac{1}{k}}))$.

By Theorem \ref{thm:MonodromySquareAyoub} the square 
\[
\begin{tikzcd}
 \R_J \Upsilon_{{f}_k}(t_k^* M)  \arrow[d, "\exp(N \cdot \chi(\lambda))"'] \arrow[r, "\R_J i_1"] & \R_J \Psi^{\tame}_{{f}_k}(t_k^*M) \arrow[r, "\sim"] & { \Psi^{J,\tame}_{{f}_k}(\R_J t_k^* M)} \arrow[d, "\chi(\lambda)"] \\
 \R_J \Upsilon_{{f}_k}(t_k^* M)  \arrow[r, "\R_J i_1"]                          & \R_J \Psi^{\tame}_{{f}_k}(t_k^* M) \arrow[r, "\sim"] & { \Psi^{J,\tame}_{{f}_k}(\R_J t_k^* M)}                   
\end{tikzcd}
\]
commutes and its horizontal maps are equivalences since $i_1$ is an equivalence by construction. This shows that  the action of $(\chi(\lambda) - 1)^m$ on $\Psi^{J, \tame}_{{f}_k}(\R_J t_k^*M)$ is zero  as desired. 
\end{proof}

\noname \label{noname:JAdicCohomology} Let $S$ be the spectrum of an excellent strictly henselian discrete valuation ring and $g: Y \rightarrow \eta$ a morphism of finite type. Consider the pullback square
\[
\begin{tikzcd}
\bar{Y} \arrow[r, "\bar{\theta}"] \arrow[d, "\bar{g}"'] & Y \arrow[d, "g"] \\
\bar{\eta} \arrow[r, "\bar{\theta}"']                   & \eta.            
\end{tikzcd}
\]
For an $A$ in $\hat\T_{\et}^{\cons}(Y, \Lambda_J \otimes \Q)$ let us write
\[
R \Gamma(\bar{Y}, A|_{\bar{Y}}) := \bar{g}_* \bar{\theta}^* A
\]
and
\[
R \Gamma_c(\bar{Y}, A|_{\bar{Y}}) := \bar{g}_! \bar{\theta}^* A.
\]
For a $\lambda$ in $\Gal(\bar{K}/K)$ consider the induced morphism of schemes $\lambda: \bar{Y} \rightarrow \bar{Y}$. Then $\lambda$ acts on $R \Gamma(\bar{Y}, A|_{\bar{Y}})$ via
\[
\lambda:  \bar{g}_* \bar{\theta}^* A \overset{\unit}\longrightarrow \bar{g}_* \lambda_* \lambda^* \bar{\theta}^* A \simeq \bar{g}_* \bar{\theta}^* A
\]
and similarly on $R \Gamma_c(\bar{Y}, A|_{\bar{Y}})$.

Consider the case where $\Lambda = \Z$ and $J= (\ell)$ for some prime $\ell \neq p$ and write $\hat\T(\Q_\ell) := \hat\T^{\cons}(\Spec \bar{K}, \Z_\ell \otimes \Q)$ and $\hat\T(\Z_\ell) := \hat\T(\Spec \bar{K}, \Z_\ell)$. Moreover let us denote the abelian group of morphisms in these categories by $\HOM_{\hat{\T}(\Q_\ell)}(\_,\_ )$ and $\HOM_{\hat{\T}(\Z_\ell)}(\_,\_ )$ respectively.

Then we have the formula
\begin{align*}
H^{i}(R \Gamma(\bar{Y}, A|_{\bar{Y}})) &\simeq \HOM_{\hat{\T}(\Q_\ell)}(\1, \bar{g}_* \bar{\theta}^* A [i]) \\
&\simeq \HOM_{\hat{\T}(\Z_\ell)}(\1, \bar{g}_* \bar{\theta}^* A [i]) \otimes_\Z \Q \\
&\simeq H^{i}\left(\Hom_{\hat{\T}(\Z_\ell)}(\1, \bar{g}_* \bar{\theta}^* A) \right) \otimes_\Z \Q \\
&\overset{}\simeq H^{i}\left( R\varprojlim_k \Hom_{\hat{\T}(\Z_\ell)}(\1, (\bar{g}_* \bar{\theta}^* A)/\ell^k ) \right) \otimes_\Z \Q \\
&\overset{(1)}\simeq \varprojlim_k H^{i}\left( \Hom_{{\T}(\Z/ \ell^k)}(\1, \bar{g}_* \bar{\theta}^* (A/\ell^k) ) \right) \otimes_\Z \Q \\
&\simeq \varprojlim_k H^{i}_{\et}(\bar{Y}, A/\ell^k|_{\bar{Y}})\otimes_\Z \Q.
\end{align*}
Here $H^{i}_{\et}(\bar{Y}, A/\ell^k|_{\bar{Y}})$ simply denotes the classical \'etale cohomology with values in the $\Z / \ell^k$-module $ \bar{\theta}^* A/ \ell^k$. The equivalence (1) follows from the fact that the system 
\[
\left\lbrace \Hom_{\hat{\T}(\Z_\ell)}(\1, (\bar{g}_* \bar{\theta}^* A)/\ell^k ) \right\rbrace_{k \in \N}
\]
is a Mittag-Leffler system and therefore the inverse limit is exact. This can be deduced from Artin-vanishing and the fact that $A$ is constructible. 

A similar formula holds for for $ H^{i} (R \Gamma_c(\bar{Y}, A|_{\bar{Y}}))$. 

\begin{cor} \label{cor:MonodromyThm}
Let $S$ be the spectrum of an excellent strictly henselian discrete valuation ring, $g: Y \rightarrow \eta$ a separated morphism of finite type and $M$ in $\DA^{\cons}_{\et}(Y, \Lambda \otimes \Q)$. Then:
\begin{enumerate}
\item There exists an open subgroup $H \subset \Gal(\bar{K}/K)$ and a positive integer $m$ such that the action of $(\lambda- 1)^{m}$ on $R\Gamma(\bar{Y}, \R_J M |_{\bar{Y}})$ is zero for all $\lambda \in H$.
\item There exists an open subgroup $H' \subset \Gal(\bar{K}/K)$ and a positive integer $m'$ such that the action of $(\lambda- 1)^{m'}$ on $R\Gamma_c(\bar{Y}, \R_J M |_{\bar{Y}})$ is zero for all $\lambda \in H'$.
\end{enumerate} 
\end{cor} 

\begin{proof}
By using Temkin's version of Nagata compactification \cite{TemkinRelative} we find a proper morphism $f: X \rightarrow S$ and an affine open immersion $Y \rightarrow X$ making the diagram
\[
\begin{tikzcd}
Y \arrow[d, "g"'] \arrow[r] & X \arrow[d, "f"] \\
\eta \arrow[r, "j"]         & S               
\end{tikzcd}
\]
commute. Consider the diagram
\[
\begin{tikzcd}
X_{\bar{\eta}} \arrow[d, "f_{\bar{\eta}}"'] \arrow[r, "\bar\theta"] & X_\eta \arrow[d, "f_\eta"] \arrow[r, "j"] & X \arrow[d, "f"] \\
\bar{\eta} \arrow[r, "\bar\theta"]                              & \eta \arrow[r, "j"]                       & S               
\end{tikzcd}
\]
obtained by pullback. The pullbacks above induce canonical maps $k: Y \rightarrow X_{\eta}$ and  $\bar{k}: \bar{Y} \rightarrow X_{\overline{\eta}}$ which are open affine and fit into a cartesian square
\[
\begin{tikzcd}
\bar{Y} \arrow[r, "\bar{\theta}"] \arrow[d, "\bar{k}"'] & Y \arrow[d, "k"] \\
X_{\bar\eta} \arrow[r, "\bar\theta"]                    & X_{\eta}.        
\end{tikzcd}
\]
Proper and smooth base change give rise to an equivalence
\begin{equation} \label{eqn:GalEquivEqui}
R\Gamma(\bar{Y}, \R_J M |_{\bar{Y}}) \simeq (f_{\bar{\eta}})_* \bar{k}_* \bar{\theta}^* \R_J M \simeq (f_{\sigma})_* \Psi^J_f(\R_J k_* M )
\end{equation}
which is $\Gal(\bar{K}/K)$-equivariant. From Theorem \ref{thm:OpenSubgroupActsUnipotentlyOnPSi} we get a subgroup $H$ and an integer $m$ such that $(\lambda-1)^m$ acts as the zero map on $\Psi^J_f(\R_J k_* M )$. Hence (\ref{eqn:GalEquivEqui}) implies that this $H$ and $m$ satisfy (1). By replacing $k_*$ with $k_!$ in the argument above we get (2). 
\end{proof}

\begin{remark}
\begin{enumerate}
\item In the case where $\Lambda= \Z$ and $J =(\ell)$ for some prime $\ell \neq p$  we can see from the proof that the $H$ and $m$ in Theorem \ref{thm:OpenSubgroupActsUnipotentlyOnPSi} and Corollary \ref{cor:MonodromyThm} are independent of $\ell$. By this we mean that the $H$ and $m$ which we construct in the proof of Theorem \ref{thm:OpenSubgroupActsUnipotentlyOnPSi} only depend on the motive $M$ and work for all $\ell \neq p$ simultaneously.

\item In the case where $\Lambda= \Z$, $J= (\ell)$ for some prime $\ell \neq p$ and $M = \1$ in $\DA^{\cons}_{\et}(Y, \Lambda \otimes \Q)$ Corollary \ref{cor:MonodromyThm} recovers Grothendieck's classical local monodromy theorem (see \cite[\S1]{IllusieAutour}) under the minor additional assumption that $S$ is excellent. Note that this proof is completely independent of the proofs existing in the literature. Interestingly we can deduce from this that the classical monodromy operator 
\[
N: H^{i}_{\et}(\bar{Y}, \Q_\ell) \longrightarrow H^{i}_{\et}(\bar{Y}, \Q_\ell)(-1)
\]
is induced by the (motivic!) monodromy operator $N:  \Upsilon_f \rightarrow \Upsilon_f(-1)$. In particular it is induced by the monodromy of the logarithm motive (see \cite[p. 86]{AyoubRealizationEtale}). A similiar observation in the $p$-adic setting was made in \cite[Appendix A]{BindaKatoVezzani}.
\end{enumerate}
\end{remark}

\noname $\hat\T^{\cons}_{\et}(\_ ,\Lambda_J \otimes \Q)$ satisfies the axioms of \cite[1.4.3]{BBD} when restricted to quasi-excellent noetherian schemes of finite dimension. Hence following \textit{loc. cit} we can define for such a scheme $W$ the perverse t-structure on $\hat\T^{\cons}(W ,\Lambda_J \otimes \Q)$. For an $A$ in $\hat\T^{\cons}_{\et}(W,\Lambda_J \otimes \Q)$ let us denote by ${}^p\mathcal{H}^{j}(A)$ the $j$-th cohomology with respect to the perverse $t$-strucutre. 

\begin{cor} \label{cor:LocMonForPerverseThings}
The statement of Corollary \ref{cor:MonodromyThm} holds true if we replace $\R_J M$ in (1) and (2) by ${}^p\mathcal{H}^j(\R_J M)$ for any integer $j$. 
\end{cor}

\begin{proof}
We use the notations of Corollary \ref{cor:MonodromyThm} and its proof. Since $k$ is open affine $k_*$ and $k_!$ are t-exact with respect to the perverse t-structure by \cite[4.1.3]{BBD}. Moreover $\Psi_f^J [-1]$ is t-exact by \cite[4.4.2]{BBD} combined with \cite[4.2]{IllusieAutour}. This implies that there is a $\Gal(\bar{K}/K)$-equivariant equivalence
\[
\Psi^J_f(k_*{}^p\mathcal{H}^j(\R_J M))[-1] \simeq {}^p\mathcal{H}^j (\Psi^J_f((\R_J k_* M))[-1]). 
\]
Hence Theorem \ref{thm:OpenSubgroupActsUnipotentlyOnPSi} gives rise to a subgroup $H \subset \Gal(\bar{K}/K)$ and an integer $m$ such that the action of $(\lambda -1)^m$ on $\Psi_f(k_*{}^p\mathcal{H}^j(\R_J M))$ is zero for all $\lambda \in H$. Again there is a $\Gal(\bar{K}/K)$-equivariant equivalence
\[
R\Gamma(\bar{Y}, {}^p\mathcal{H}^j(\R_J M)|_{\bar{Y}}) \simeq (f_{\sigma})_* \Psi^J_f(k_*{}^p\mathcal{H}^j(\R_J M))
\]
which allows us to conclude analogue as in the proof of Corollary \ref{cor:MonodromyThm} above.
\end{proof}

\begin{remark}
In particular Corollary \ref{cor:LocMonForPerverseThings} implies that a local monodromy theorem is true for "sheaves of geometric origin" (see \cite[\S 6.2]{BBD}). This was already expected to be true by Illusie in \cite[\S 1]{IllusieAutour}. 
\end{remark}

\chapter{Universal local acyclicity for motives and the nearby cycles functor}

We start off with developing the theory of universal local acyclicity for motivic $\infty$-categories. A good part of this is analogue to the case of \'etale sheaves as developed in \cite{lu_zheng_2022} and \cite{HansenScholzeRelative}. 

We establish a very useful characterization of universal local acyclicity in terms of K\"unneth-type formulas (Proposition \ref{thm:EquivCharOfULA}) and without much effort we get a 'Generic universal local acyclicity theorem' (Proposition \ref{thm:GenericULA}) which generalizes Deligne's classical result.  

Then we turn our attention to proving the main result of this chapter: Theorem \ref{thm:ULAinDAdetectedByNearby} shows that for \'etale motives universal local acyclicity over an excellent regular 1-dimensional base is detected by the nearby cycles functor. The proof claims a good part of this chapter and is unfortunately very technical.

Finally as an application we introduce Beilinson's weak singular support of an \'etale motive and show that it can be characterized via the nearby cycles functor.

\section{Cohomological correspondences}

\begin{construction} \label{noname:ConstrOfCS} Let $S$ be a qcqs scheme and $\T(\_)$ a motivic $\infty$-category over $S$ (see \ref{noname:MotivicInftyCat}). Let us denote by $\Sch^{\text{ft}}_{/S}$ the category of schemes of finite type over $S$. We consider the bicategory $\C_{S, \T}$ of cohomological correspondences with values in $\T( \_)$ as constructed for example in \cite{lu_zheng_2022}. This is the bicategory where:
\begin{enumerate}
\item Objects are pairs $(X,M)$ where $X \in \Sch^{\text{ft}}_{/S}$ and $M \in \T(X)$.
\item For any two objects $X,Y$ in $\Sch^{\text{ft}}_{/S}$ let us fix a choice of pullback $X \times_SY$. Different choices of pullbacks will give rise to equivalent bicategories. A morphism $(C,\alpha): (X, M) \rightarrow (Y,N)$ consists of a correspondence
\[
{(\overleftarrow{c},\overrightarrow{c})}: C \longrightarrow X \times_S Y
\]
and a map $\alpha: \overleftarrow{c}^*M \rightarrow \overrightarrow{c}^! N$ in $h\T(C)$. Given another map $(D, \beta): (Y, N) \rightarrow (Z,P)$ in order to define the composite $(D, \beta)\circ(C, \alpha)$ consider the diagram of schemes
\[
\begin{tikzcd}
  &                                                                     & C \times_Y D \arrow[ld, "\overleftarrow{e}"'] \arrow[rd, "\overrightarrow{e}"] &                                                                     &   \\
  & C \arrow[ld, "\overleftarrow{c}"'] \arrow[rd, "\overrightarrow{c}"] &                                                                                & D \arrow[ld, "\overleftarrow{d}"'] \arrow[rd, "\overrightarrow{d}"] &   \\
X &                                                                     & Y                                                                              &                                                                     & Z
\end{tikzcd}
\]
obtained by taking pullback. Then $(D, \beta)\circ(C, \alpha)$ is given by the correspondence
\[
(\overleftarrow{c}\overleftarrow{e}, \overrightarrow{e} \overrightarrow{d}): C \times_Y D \longrightarrow X \times_S Z
\]
and the map
\[
\gamma: \overleftarrow{e}^*\overleftarrow{c}^*M \overset{\overleftarrow{e}^*\alpha}\longrightarrow \overleftarrow{e}^* \overrightarrow{c}^!N \overset{\Ex}\longrightarrow \overrightarrow{e}^! \overleftarrow{d}^*N \overset{\overrightarrow{e}^! \beta}\longrightarrow \overrightarrow{e}^! \overrightarrow{d}^! P
\]
in $h\T(C \times_Y D)$, where $\Ex$ is the exchange map with respect to the adjunction $(\_)_! \dashv (\_)^!$ and the square
\[
\begin{tikzcd}
\T(C) \arrow[r, "\overleftarrow{e}^*"] \arrow[d, "\overrightarrow{c}_!"'] & \T(C \times_Y D) \arrow[d, "\overrightarrow{e}_!"] \\
\T(Y) \arrow[r, "\overleftarrow{d}^*"'] \arrow[ru, Rightarrow, "\Ex", shorten <=2.5ex, shorten >=2.5ex ]            & {\T(D),}                                          
\end{tikzcd}
\]
whose 2-cell is invertible by proper base change. Since this composition depends on a choices of pullbacks it is not well defined as a strict composition. Given an object $(X,M)$ the identity is (up to equivalence) the morphism 
\[
(\Delta_X, {\id}_M): \id^* M \simeq M \overset{\id}\longrightarrow  M \simeq \id^!M,
\]
where $\Delta_X = (\id, \id): X \rightarrow X \times_S X$ is the diagonal morphism.

\item Given two morphisms $(C,\alpha),(D, \beta): (X, M) \rightarrow (Y,N)$ a 2-cell $(\Theta,f): (C,\alpha) \rightarrow (D, \beta)$ is given by a proper morphism $f: C \rightarrow D$ fitting into the diagram
\[
\begin{tikzcd}
C \arrow[rd, "{(\overleftarrow{c}, \overrightarrow{c})}"'] \arrow[rr, "f"] &              & D \arrow[ld, "{(\overleftarrow{d},\overrightarrow{d})}"] \\
                                                                           & X \times_S Y &                                                         
\end{tikzcd}
\]
such that $\beta$ is equal to the composition
\[
f_* \alpha: \overleftarrow{d}^*M \rightarrow f_*f^*\overleftarrow{d}^* M \simeq f_* \overleftarrow{c}^* M \overset{f_* \alpha}\longrightarrow f_* \overrightarrow{c}^!N \simeq  f_* f^! \overrightarrow{d}^!N \rightarrow \overrightarrow{d}^!N,
\]
where the first and the last arrow are induced by the unit $\id \rightarrow f_* f^*$ and counit $f_* f^! \simeq f_! f^! \rightarrow \id$ respectively. This notion of 2-cells makes the composition in (2) well defined in the appropriate weak sense. 
\end{enumerate} 
\end{construction}

\begin{noname}

 We can equip the bicategory $\C_{S, \T}$ with a symmetric monoidal structure where the tensor product is given by
\[
(X,M) \otimes (Y,N) := (X \times_S Y, M \boxtimes N). 
\]
Then given morphisms $(C,\alpha): (X, M) \rightarrow (Z,P)$ and $(D, \beta): (Y, N) \rightarrow (W,Q)$ the induced morphism 
\[
(C, \alpha) \otimes (D, \beta): (X \times_S Y, M \boxtimes N) \rightarrow (Z \times_S W, P \boxtimes Q)
\]
is given by the correspondence 
\[
(\overleftarrow{e}, \overrightarrow{e}) = ( \overleftarrow{c} \times \overleftarrow{d}, \overrightarrow{c} \times \overrightarrow{d}): C\times_S D \rightarrow (X \times_S Y) \times_S (Z \times_S W)
\]
and the map
\[
\overleftarrow{e}^*(M \boxtimes N) \simeq \overleftarrow{c}^*M \boxtimes \overleftarrow{d}^*N \overset{\alpha \boxtimes \beta}\longrightarrow \overrightarrow{c}^!P \boxtimes \overrightarrow{d}^!Q \overset{}\longrightarrow \overrightarrow{e}^!( P \boxtimes Q),
\]
where the last arrow is the canonical K\"unneth morphism. 

\end{noname}

\begin{lem} \label{lem: C_STAdmitsInternalHom}
For any object $(X,M)$ in $\C_{S,\T}$ the functor
\[
\_ \otimes (X,M): \C_{S, \T} \rightarrow \C_{S,\T}
\]
has a right adjoint ${\Hom}_{C_{S,\T}}((X,M), \_)$. These internal mapping objects are given by
\[
{\Hom}_{C_{S, \T}}((X,M), (Y,N)) \simeq (X \times_S Y, {\Hom}_{\T(X \times_S Y)}(p_X^* M, p_Y^!N)).
\]
\end{lem}

\begin{proof}
See \cite[2.8.]{lu_zheng_2022}.
\end{proof}

\begin{construction} \label{noname:FuntcorOnC_SIndByAnyMap} Given a morphism  $g:T \rightarrow S$ between qcqs schemes let us denote by $\T(\_)|_T$ the restriction of $\T(\_)$ to $\Sch^{\qcqs}_{/T}$. We denote the pull back of a scheme $f: X \rightarrow S$ along $g$ by $f_T: X_T \rightarrow T.$ Further for a scheme $X$ over $S$ we write by slight abuse of notation $g^*: \T(X) \rightarrow \T (X_T)$ for the induced inverse image functor. 

We define a functor $g^\diamondsuit: \C_{S,\T} \rightarrow \C_{T, \T|_T}$ as follows: We send a morphism $(C, \alpha): (X,M) \rightarrow (Y,N)$ to $(C_T, g^*\alpha): (X_T, g^*M) \rightarrow (Y_T, g^* N)$ where $g^* \alpha$ denotes the composition
\[
\overleftarrow{c_T}^* g^* M \simeq g^* \overleftarrow{c}^*M \overset{g^* \alpha}\longrightarrow g^* \overrightarrow{c}^! N \overset{\Ex}\longrightarrow \overrightarrow{c_T}^!g^*  N.
\]
Given two morphisms $(C,\alpha),(D, \beta): (X, M) \rightarrow (Y,N)$ a 2-cell $(\Theta,f): (C,\alpha) \rightarrow (D, \beta)$ given by a proper morphism $f: C \rightarrow D$ is sent to the 2-Morphism $g^\diamondsuit \Theta: (C_T,g^*\alpha) \rightarrow (D_T, g^*\beta)$ induced by the proper morphism $f_T: C_T \rightarrow D_T$. Due to proper base-change this is easily verified to be a well defined functor of bicategories. Moreover the equivalences 
\[
((X\times_S Y)_T, g^*(M \boxtimes N)) \rightarrow ((X_T \times_T Y_T), g^*M \boxtimes g^*N)
\]
and 
\[
(T , g^* \1) \rightarrow (T, \1)
\]
equip $g^\diamondsuit: \C_{S,\T} \rightarrow \C_{T, \T}$ with the structure of a strict monoidal functor. 
\end{construction} 

\section{Universal local acyclicity}

\noname Let us fix a qcqs scheme $S$ and let $\T(\_)$ be a motivic $\infty$-category over $S$. Most of what we are treating in this section is essentially due to \cite{lu_zheng_2022} and \cite{HansenScholzeRelative}.

\begin{definition} \label{def:ULA}
Let $f: X \rightarrow S$ be of finite type and $M \in \T(X)$. $M$ is called \textit{ universally locally acyclic with respect to $f$}  if $(f:X \rightarrow S,M)$ is dualizable in $\C_{S,\T}$ (see Definition \ref{def:StronglyDualizable}).
\end{definition}

\noname For convenience we will also write '$f$-ULA' or 'ULA with respect to $f$' instead of 'universally locally acyclic with respect to $f$'.

\begin{remark}
It follows right from the definition that an object $M$ is a dualizable object of $\T(S)$ if and only if it is ULA with respect to the identity of $S$.
\end{remark}

\noname For any $f:X \rightarrow S$ of finite type and $M$ in $\T(X)$ let us write
\[
\D_S(M) := \Hom(M, f^! \1).
\]
Note that by Lemma \ref{lem: C_STAdmitsInternalHom} we have 
\[
{\Hom}_{\C_{S, \T}}((X,M), (S, \1_S)) \simeq (X, \D_S(M))
\]
in $\C_{S, \T}$.

\noname For any two morphisms of schemes $f: X \rightarrow S$ and $g: Y \rightarrow S$ let us denote by $p_1: X \times_S Y \rightarrow X$ the first projection and by $p_2: X \times_S Y \rightarrow Y$ the second projection.  

\begin{prop} \label{prop:equivcharofULA}
Let $f: X \rightarrow S$ be of finite type and $M$ in $\T(X)$. Then the following are equivalent:
\begin{enumerate}
\item $M$ is ULA with respect to $f$.
\item  For every $g: Y \rightarrow S$ in $\Sch^{\ft}_{/S}$ and $N$ in $\T(Y)$ the canonical map
\[
p_1^* \Dual_S(M) \otimes p_2^* N \rightarrow \Hom (p_1^* M, p_2^! N)
\]
is an equivalence.
\item  The canonical map
\[
p_1^* \Dual_S(M) \otimes p_2^* M \rightarrow \Hom (p_1^* M, p_2^! M)
\]
is an equivalence.
\end{enumerate}
\end{prop}

\begin{proof}
This follows from Lemma \ref{lem:StrongDualAndIntHom} using the descriptions of the symmetric product and internal Hom of $\C_{S, \T}$ of Lemma \ref{lem: C_STAdmitsInternalHom}.
\end{proof}

\begin{lem} \label{lem:ULAandProperSmooth}
Let $f: X \rightarrow S$ be of finite type and $M$ in $\T(X)$ ULA with respect to $f$. 
\begin{enumerate} 
\item If  $h:X \rightarrow Y$ in $\Sch^{\qcqs}_{/S}$ is proper, then $h_* M$ is $k$-ULA, where $k:Y \rightarrow S$ denotes the structure morphism of $Y$.
\item If  $h:Y \rightarrow X$ in $\Sch^{\qcqs}_{/S}$ is smooth, then $h^* M$ is $f \circ h$-ULA.
\item If  $g: S \rightarrow S'$ is smooth, then $M$ is $g \circ f$-ULA. 
\end{enumerate}
\end{lem}

\begin{proof}
Let us prove (1). The claims (2) and (3) will follow similarly. We have
\begin{align*}
p_1^* \Dual_S(h_* M) \otimes p_2^* h_*M & \simeq (1 \times h_*) (p_1^* \Dual_S(h_*M) \otimes p_2^*M) \\
&\simeq (h \times h)_* (p_1^* \Dual_S(M) \otimes p_2^*M) \\
&\simeq (h \times h)_*\Hom (p_1^* M, p_2^! M) \\
&\simeq (h \times 1)_*\Hom (p_1^* M, (1 \times h)_* p_2^! M)\\
&\simeq \Hom (p_1^* h_*M,  p_2^! h_* M).
\end{align*}
Here the first two equivalences follow from the projection formula and proper base change, the third equivalence follows from Proposition \ref{prop:equivcharofULA} and the last two equivalences follow simply from the adjunctions $(1 \times h)^* \dashv (1 \times h)_*$ and $(h \times 1)_!\dashv (h \times 1)^!$. Using the characterization (3) of Proposition \ref{prop:equivcharofULA} this implies that $h_* M$ is $k$-ULA . \end{proof}

\begin{prop} \label{prop:propertiesOfULA}
Let $f: X \rightarrow S$ be of finite type and conisder a $M$ in $\T(X)$ which is ULA with respect to $f$.
\begin{enumerate} 
\item Let $g: S' \rightarrow S$ be a map of qcqs schemes, $f': X \times_S S' \rightarrow S'$ the base change of $f$ along $g$ and $M'$ the pullback of $M$ along $X \times_S S' \rightarrow X$. Then $M'$ is $f'$-ULA.
\item The canonical map
\[
M \rightarrow \Dual_{S} \Dual_{S} (M)
\]
is an equivalence.
\item For any $Y$ in $\Sch^{\ft}_{/S}$ and $N \in \T(Y)$ the canonical map
\[
M \boxtimes N \rightarrow \Hom( p_1^* \Dual_{S}M , p_2^! N)
\]
is an equivalence. 
\item For any $Y$ in $\Sch^{\ft}_{/S}$ and $N \in \T(Y)$ the canonical map
\[
\D_S(M \boxtimes N) \rightarrow \D_S(M) \boxtimes \D_S(N)
\]
is an equivalence.
\item  For any morphism $h: Y' \rightarrow Y$ in $\Sch^{\ft}_{/S}$ and $N \in \T(Y)$ the canonical map
\[
M \boxtimes h^! N \rightarrow (1 \times h)^! ( M \boxtimes N)
\] 
is an equivalence.
\item  For any morphism $h: Y' \rightarrow Y$ in $\Sch^{\ft}_{/S}$ and $N \in \T(Y')$ the canonical map
\[
M \boxtimes h_* N \rightarrow (1 \times h)_* ( M \boxtimes N)
\] 
is an equivalence.
\end{enumerate}
\end{prop}

\begin{proof}
This is essentially \cite[2.11.]{lu_zheng_2022}. We give the arguments for the sake of completeness. Indeed (1) follows from Corollary \ref{cor:1.ImageOfDualizableObjectisDualizablw2} since the induced functor $g^\diamondsuit: \C_{S, \T} \rightarrow \C_{S', \T|_{S'}}$ of \ref{noname:FuntcorOnC_SIndByAnyMap} is symmetric monoidal. (2) is clear since $(X, \D_S(M))$ is the dual of $(X,M)$ in $\C_{S, \T}$ and hence the canonical map $(X, M) \rightarrow (X, \D_S\D_S(M))$ is an equivalence in $\C_{S, \T}$. (3) is the combination of (2) and  the characterization (2) of Proposition \ref{prop:equivcharofULA}. The equivalence (4) follows from the chain of equivalences
\begin{align*}
\D_S(M) \boxtimes \D_S(N) &\simeq \Hom(p_1^* M, p_2^! \D_S(N)) \\
&\simeq \Hom(p_1^* M, \D_S(p_2^* N)) \\
& \simeq \Hom(p_1^* M \otimes p_2^* N, p_2^! g^! \1)\\
& \simeq \D_S( M \boxtimes N),
\end{align*}
where the first equivalence is given by (3). (5) follows from (3) via
\begin{align*}
M \boxtimes h^! N &\simeq \Hom(p_1^* \D_S(M), p_2^! h^! N) \\
&\simeq \Hom((1 \times h)^*p_1^* \D_S(M), (1 \times h)^! p_2^!  N) \\
&\simeq (1 \times h)^!\Hom(p_1^* \D_S(M), p_2^!  N) \\
&\simeq (1 \times h)^! (M \boxtimes N).
\end{align*}
Similarly (6) follows from (3): 
\begin{align*}
M \boxtimes h_* N &\simeq \Hom(p_1^* \D_S(M), p_2^! h_* N) \\
&\simeq \Hom(p_1^* \D_S(M), (1 \times h)_* p_2^!  N) \\
&\simeq (1 \times h)_*\Hom(p_1^* \D_S(M), p_2^!  N) \\
&\simeq (1 \times h)_* (M \boxtimes N).
\end{align*}
\end{proof}

\begin{lem} \label{lem:*Künnethvs!Künneth}
Let $f: X \rightarrow S$ be of finite type and consider $M \in \T(X)$. Then the following are equivalent:
\begin{enumerate}
\item For any morphism $h: Y' \rightarrow Y$ in $\Sch^{\ft}_{/S}$ and $N \in \T(Y)$ the canonical map
\[
M \boxtimes h^! N \rightarrow (1 \times h)^! ( M \boxtimes N)
\] 
is an equivalence.
\item  For any morphism $h: Y' \rightarrow Y$ in $\Sch^{\ft}_{/S}$ and $N \in \T(Y')$ the canonical map
\[
M \boxtimes h_* N \rightarrow (1 \times h)_* ( M \boxtimes N)
\] 
is an equivalence.    
\end{enumerate}
\end{lem}

\begin{proof}
Let us fix maps $g:Y \rightarrow S$ and $h: Y' \rightarrow Y$ in $\Sch^{\ft}_{/S}$. Assume that $M$ satisfies (1). Property (2) can be checked Zariski-local on $Y$. Hence we can assume that $h$ is separated and thus we may factor $h$ as an open immersion followed by a proper map by using Nagata compactification. Since the projection formula holds for for proper maps we are reduced to to check (2) for open immersions $j:U \rightarrow Y$. Let us denote the complement by $i: Z \rightarrow Y$. Consider the map of fiber sequences obtained from the localization sequence
\begin{equation}\label{eqn:mapofLocSeqKünneth}
\begin{tikzcd}
M \boxtimes i_*i^! N \arrow[r] \arrow[d]              & M \boxtimes N \arrow[r] \arrow[d] & M \boxtimes j_*j^* N \arrow[d]              \\
(1 \times i)_*(1 \times i)^!(M \boxtimes N) \arrow[r] & M \boxtimes N \arrow[r]           & (1 \times j)_*(1 \times j)^*(M \boxtimes N).
\end{tikzcd}
\end{equation}
It is straightforward to check that all squares above commute. Property (1) applied to the map $i: Z \rightarrow Y$ and the projection formula for the closed immersion $i: Z\rightarrow Y$ imply that the left vertical arrow is an equivalence. Hence the right vertical map is an equivalence which shows property (2) for $j: U \rightarrow Y$ as desired. 

Conversely assume that $M$ satisfies (2). We may check (1) Zariski locally on $Y$ and hence assume that $h$ is quasi projective. Thus we can factor $f$ into a closed immersion followed by a smooth morphism. This reduces (1) to the case where $h$ is a closed immersion $i: Z \rightarrow Y$. Let $j: U \rightarrow Y$ be the open complement of $i$. Then we may again consider diagram (\ref{eqn:mapofLocSeqKünneth}) for these $i$ and $j$. Now the right vertical map is an equivalence since $A$ satisfies (2) and since condition (1) is always satisfied for any $M$ whenever $h$ is an open immersion.
\end{proof}

\noname \label{noname:DefOfConstr} For any $X$ in $\Sch^{\qcqs}_{/S}$ we denote by $\T^{\cons}(X)$ the smallest idempotent complete full stable subcategory of $\T(X)$ containing the objects of the form $h_! \1(n)$ for all $n \in \mathbb{Z}$ and $h: Y \rightarrow X$ smooth. We call an object $M$ in $\T(X)$ \textit{constructible} if it belongs to  $\T^{\cons}(X)$.

\begin{lem} \label{lem:*KünnethimpliesCanmapisIso}
Let $f: X \rightarrow S$ be of finite type and $M \in \T(X)$. Assume that $M$ satisfies the two equivalent conditions of Lemma \ref{lem:*Künnethvs!Künneth}. Then for all $g:Y \rightarrow S$ and $N \in \T^{\cons}(Y)$ the canonical map
\[
\Dual_{S}(N) \boxtimes M \rightarrow \Hom(p_1^* N, p_2^! M)
\]
is an equivalence.
\end{lem}

\begin{proof} Since all operations involved commute with finite colimits and Tate twists it suffices to prove the claim for $N=h_! \1$, where $h: T \rightarrow Y$ is a smooth morphism. To fix notations consider the diagram
\[
\begin{tikzcd}
T \times_S X \arrow[r, "p_1"] \arrow[d, "1 \times h"'] \arrow[dd, "1 \times k"', bend right=60] & T \arrow[d, "h"] \arrow[dd, "k", bend left=60] \arrow[d] \\
Y \times_S X \arrow[d, "p_2"'] \arrow[r, "p_1"]                                                   & Y \arrow[d, "g"]                                         \\
X \arrow[r, "f"']                                                                                 & S.                                                       
\end{tikzcd}
\]
Then there are equivalences
\begin{align*}
\Dual_{S}(h_! \1) \boxtimes M & \simeq (h_* k^! \1) \boxtimes M \\
&\overset{(1)}\simeq ( h \times 1)_* (k^! \1 \boxtimes M) \\
& \overset{(2)}\simeq ( h \times 1)_* (1 \times k)^! M  \\
&\simeq ( h \times 1)_* \Hom( \1,(1 \times k)^! M ) \\
&\simeq \Hom ( p_1^* h_! \1 , p_2^! M),
\end{align*}
where $(1)$ and $(2)$ follow from equivalent conditions of Lemma \ref{lem:*Künnethvs!Künneth}.
\end{proof}

\begin{prop} \label{thm:EquivCharOfULA}
Let $f:X \rightarrow S$ be of finite type and $M \in \T^{\cons}(X) \subset\T(X)$. The following are equivalent:
\begin{enumerate}
\item $M$ is ULA with respect to $f$.
\item For any morphism $h: Y' \rightarrow Y$ in $\Sch^{\ft}_{/S}$ and $N \in \T(Y)$ the canonical map
\[
M \boxtimes h^! N \rightarrow (1 \times h)^! ( M \boxtimes N)
\] 
is an equivalence.
\item  For any morphism $h: Y' \rightarrow Y$ in $\Sch^{\ft}_{/S}$ and $N \in \T(Y')$ the canonical map
\[
M \boxtimes h_* N \rightarrow (1 \times h)_* ( M \boxtimes N)
\] 
is an equivalence. 
\end{enumerate}
\end{prop}

\begin{proof}
We have seen the implication (1) $\Rightarrow$ (2),(3) in Proposition \ref{prop:propertiesOfULA}. Lemma \ref{lem:*Künnethvs!Künneth} implies that (2) and (3) are equivalent. Assume that the equivalent conditions (2) and (3) hold. Then applying Lemma \ref{lem:*KünnethimpliesCanmapisIso} to the case $g=f$ and $M=N$ implies that the canonical map
\[
p_1^* \Dual_X(M) \otimes p_2^* N \rightarrow \Hom (p_1^* M, p_2^! N)
\]
is an equivalence. Hence $M$ is ULA with respect to $f$ by Proposition \ref{prop:equivcharofULA}. 
\end{proof}

\begin{remark}
Let $f: X \rightarrow S$ be of finite type and consider $M \in \T(X)$. Then in \cite[4.1.2]{JinYangMotNearbyPosChar} $M$ is called \textit{strongly locally acyclic} with respect to $f$ if for any $h:T \rightarrow S$ and any $ N \in \T(T)$ the canonical map
\[
M \boxtimes h_*N \rightarrow (1 \times h)_* ( M \boxtimes N)
\] 
is an equivalence. Further $M$ is called \textit{universally strongly locally acyclic} with respect to $f$ if the analogue holds true after pulling $f$ and $M$ back along every $S' \rightarrow S$. It is an easy check that $M$ is universally strongly locally acyclic with respect to $f$ if and only if it satisfies the equivalent conditions of Lemma \ref{lem:*Künnethvs!Künneth}. In particular a constructible $M$ is universally strongly locally acyclic with respect to $f$ if and only if it is ULA with respect to $f$.
\end{remark}

\noname \label{noname:DefOfConstGen} For $f:X \rightarrow S$ in $\Sch^{\qcqs}_{/S}$ we say $\T(X)$ is \textit{compactly generated by constructible objects} if $\T(X)$ is compactly generated and every constructible object is compact. Note that in this case an object is compact if and only if it is constructible (see \cite[1.4.11]{CisinskiDegliseBook}). 

\begin{prop} \label{prop:MULAAndConstrGenImpliesConstr}
Let $f: X \rightarrow S$ be of finite type and assume that $\T(X)$ and $\T(X \times_S X)$ are compactly generated by constructible objects. If $M$ in $\T(X)$ is ULA with respect to $f$, then $M$ is constructible.  
\end{prop}

\begin{proof}
This is essentially \cite[3.4 (iii)]{HansenScholzeRelative}. Since $M$ is ULA with respect to $f$ we have for any $N$ in $\T(X)$
an equivalence
\[
p_1^*\D_{S}(M) \otimes p_2^* N \simeq \Hom(p_1^*M, p_2^!N). 
\]
Applying $\Delta_X^!$ and $\map_{\T(X)}( \1_X, \_)$ to this equivalence we get
\begin{align*}
\map_{\T(X)}(\1_X, \Delta^!(p_1^*\D_{S}(M) \otimes p_2^* N)) &\simeq \map_{\T(X)}(\1_X, \Delta^!(\Hom(p_1^*M, p_2^!N)) \\
& \simeq \map_{\T(X)}(\1_X, \Hom(M, N)) \\
& \simeq \map_{\T(X)}(M, N). 
\end{align*}
Note that $\1_X$ is by assumption compact and $\Delta^!$ commutes with small filtered colimits: For every compact generator $A$ in $\T(X)$ (i.e. constructible $A$ by above) and every small filtered diagram $F: I \rightarrow \T(X \times_S X)$ we have 
\begin{align*}
\map_{\T(X)}(A, \Delta^! (\colim_I F(i))) & \simeq \map_{\T(X \times_S X)}( \Delta_! A, \colim_I F(i)) \\
& \simeq \colim_I \map_{\T(X \times_S X)}( \Delta_! A, F(i)) \\
& \simeq \colim_I \map_{\T(X)}( A, \Delta^! F(i)) \\
& \simeq \map_{\T(X)}( A, \colim_I \Delta^! F(i)).
\end{align*}
Here we used that $\Delta_!$ preserves constructible objects (see \cite[2.60]{Adeel6Functor}). This shows  that $\map_{\T(X)}(\1_X, \Delta^!(p_1^*M \otimes p_2^* N))$ commutes with small filtered colimits in $N$ and hence $M$ is a compact object of $\T(X).$ 
\end{proof}

\noname \label{noname:FinitaryFunctor} We call a functor
\[
\mathcal{F}: (\Sch^{\qcqs}_{/S})^{op} \longrightarrow \widehat{\Cat}_\infty
\]
\textit{\'etale-continuous} if for every small cofiltered diagram $\{Y_i\}_I: I \rightarrow \Sch^{\qcqs}_{/S}$ with affine and \'etale transition maps the canonical functor
\[
\colim_{i \in I} \mathcal{F}(Y_i) \longrightarrow  \mathcal{F}(\slim_{i \in I} Y_i)
\]
is an equivalence.

\noname Given a morphism $f:X \rightarrow S$ of finite type and a map of qcqs schemes $g: T \rightarrow S$ let us denote by 
\[
\T^{\text{ULA}}(X \times_S T / T) \subset \T(X \times_S T)
\]
the full subcategory consisting of elements which are ULA with respect to the pullback of $f$ along $g$. This gives rise to a subfunctor
\[
\T^{\text{ULA}}(X \times_S \_ / \_) \subset \T(X \times_S \_).
\]

\begin{lem} \label{lem:DULAisFinitary}
Let $f:X \rightarrow S$ be of finite type and assume that $\T^{\cons}(\_)$ is \'etale-continuous. Then $\T^{\text{ULA}}(X \times_S \_ / \_) \cap \T^{\cons}(X \times_S \_) \subset \T^{\cons}(X \times_S \_)$ is an \'etale-continuous subfunctor.
\end{lem}

\begin{proof}
Let $T = \lim_I T_i$ be an inverse limit in $\Sch^{\qcqs}_{/S}$ with affine \'etale transition maps and write $f_T: X_T \rightarrow T$ and $f_{T_i}: X_{T_i} \rightarrow T_i$ for the pullbacks of $f$ respectively. Moreover for any $\varphi: j \rightarrow i$ in $I$ we write $\varphi: X_{T_j} \rightarrow X_{T_i}$ for the induced morphism of schemes.

Consider an $M$ in $\T^{\text{ULA}}(X_T / T) \cap \T^{\cons}(X_T) \subset \T^{\cons}(X_T)$. Since $\T^{\cons}(\_)$ is continuous there exists an $i \in I$ and an $M_i$ in $\T^{\cons}(X_{T_i})$ such that its pullback to $X_T$ is $M$. Consider the canonical map
\[
\theta_{M_i}: p_1^* \Dual_S(M_i) \otimes p_2^* M_i \rightarrow \Hom (p_1^* M_i, p_2^! M_i)
\]
in $\T(X_{T_i} \times_{T_i} X_{T_i})$. Since the transition maps are \'etale, the vertical maps of the commutative diagram
\[
\begin{tikzcd}
\varphi^*(p_1^* \Dual_S(M_i) \otimes p_2^* M_i) \arrow[r, "\theta_{M_i}"] \arrow[d] & {\varphi^*\Hom (p_1^* M_i, p_2^! M_i)} \arrow[d] \\
p_1^*\D_S(\varphi^*M_i) \otimes p_2^* \varphi^* M_i \arrow[r, "\theta_{\varphi^*M_i}"]         & {\Hom (p_1^* \varphi^*M_i, p_2^! \varphi^*M_i)}           
\end{tikzcd}
\]
are equivalences for any $\varphi: j \rightarrow i$ in $I$ . Since $M$ is $f-ULA$, $\theta_M$ is an equivalence by criterion (3) of Proposition \ref{prop:equivcharofULA}. Hence there exists a  $\varphi: j \rightarrow i$ in $I$ such that $\theta_{\varphi^* M_i}$ is an equivalence. In particular $\varphi^*M_i$ is ULA with respect to $f_{T_j}$ again by criterion (3) of Proposition \ref{prop:equivcharofULA}.
\end{proof}

\noname We say that $\T(\_)$ \textit{satisfies the K\"unneth-formula} if the following holds: Consider a pair of commutative triangles
\[
\begin{tikzcd}
X \arrow[r, "k"] \arrow[rd] & X' \arrow[d] &  & Y \arrow[r, "l"] \arrow[rd] & Y' \arrow[d] \\
                            & \Spec(K)     &  &                             & \Spec(K)    
\end{tikzcd}
\]
in $\Sch^{\ft}_{/ S}$ where $K$ is a field. Then for every $M$ in $\T^{\cons}(X)$ and $N$ in $\T^{\cons}(Y)$ the canonical map
\[
k_*M \boxtimes l_*N \longrightarrow (k \times l)_*(M \boxtimes N)
\]
is an equivalence.

\begin{prop}
[Generic universal local acyclicity] \label{thm:GenericULA}
Assume that $\T(\_)$ satisfies the K\"unneth-formula and $\T^{\cons}(\_)$ is \'etale-continuous. Let $f: X \rightarrow S$ be of finite type and $M$ in $\T^{\cons}(X)$. Then there exists a dense open subscheme $U \subset S$ such that $M|_{X_U}$ is ULA with respect to $f_U: X \times_S U \rightarrow U$. 
\end{prop}

\begin{proof}
It suffices to find for any irreducible component $S_i$ of $S$ an open dense subset $U_i \subset S_i$ satisfying the claim of the Theorem. Hence we can assume that $S$ is irreducible with generic point $\eta.$ Then $M$ restricted to $f_\eta: X \times_S \eta \rightarrow \eta$ is ULA with respect to $f_\eta$ by the K\"unneth formula combined with Proposition \ref{thm:EquivCharOfULA}. Let $k(\eta)$ denote the residue field of $\eta$. Then $k(\eta) \simeq \colim_I A_i$, where the $A_i$ run over all affine open neighbourhoods of $\eta$ in $S$. Lemma \ref{lem:DULAisFinitary} implies that there exists an open affine neighborhood $A_i$ of $\eta$ such that $M$ restricted to $f_{A_i}: X \times_S \Spec(A_i) \rightarrow \Spec(A_i)$ is ULA with respect to $f_{A_i}$. Since $\Spec(A_i)$ contains the generic point it is open dense in $S$ as desired. 
\end{proof}

\begin{remark} The conditions of Proposition $\ref{thm:GenericULA}$ are for example satisfied in the following two cases: 
\begin{enumerate}
\item $\T(\_)$ is the motivic $\infty$-category $\SH(\_)[\mathcal{P}^{-1}]$, where $\mathcal{P}$ is the set consisting of those prime numbers appearing as residue characteristics of $S$. Indeed $\SH(\_)$ is continuous: This follows from \cite[2.18]{Adeel6Functor} and the fact that constructible objects and compact objects agree. Moreover after inverting the residue characteristics of $S$ it satisfies the K\"unneth-formula \cite[2.1.14]{EnlinFangzouKünneth}.

\item $\T(\_)$ is the motivic $\infty$-category of $h$-motives $\DM_{h}(\_, \Lambda)$ for a "good enough" ring $\Lambda$ (see \cite[6.3.6]{CisinskiDegliseEtale}). It is continuous by \cite[6.3.9]{CisinskiDegliseEtale} and satisfies the K\"unneth-formula by \cite[3.1.12]{Cisinski2021CohomologicalMI}. In particular whenever $S$ is locally noetherian and $\Lambda$ is a torsion ring with characteristic invertible in $S$ we recover Deligne's classical result \cite[Th. finitude, 2.13]{SGA4.5} by rigidity \cite[5.5.4]{CisinskiDegliseEtale}. Moreover whenever $S$ is noetherian of finite dimension Theorem \ref{thm:CompDAvsDM} implies that Proposition $\ref{thm:GenericULA}$ also applies to $\DA_{\et}(\_, \Lambda)$.
\end{enumerate} 
\end{remark}

\section{Detecting universal local acyclicity with the nearby cycles functor}

\noname Throughout this section let us fix a ring $\Lambda$. For the remainder of this chapter we will restrict ourselves to the motivic $\infty$-category $\T(\_)= \DA_{\et}(\_, \Lambda)$. Moreover let $S$ be the spectrum of an excellent strictly henselian discrete valuation ring. We fix a uniformizer $\pi$ of $S$ and a section $\tau$ of the short exact sequence in \ref{noname:ConstrOfPsi}. We construct all nearby cycles functors with respect to these choices of $\pi$ and $\tau$.

\noname \label{noname:setupForComparionSquareofDualityAndRecollement} Let $f: X \rightarrow S$ be a morphism of finite type and consider the decomposition
\[
\begin{tikzcd}
X_\eta \arrow[d, "f_\eta"] \arrow[r, "j"] & X \arrow[d, "f"] & X_\sigma \arrow[d, "f_\sigma"] \arrow[l, "i"'] \\
\eta \arrow[r, "j"']                      & S                & \sigma. \arrow[l, "i"]                         
\end{tikzcd}
\]

For any $M$ in $\DA_{\et}(X, \Lambda)$ there is a fiber sequence
\begin{equation} \label{eqn:i*appliedToLocSeq}
i^! M \rightarrow i^* M \rightarrow i^*j_*j^*M
\end{equation}
in $\DA_{\et}(X_\sigma, \Lambda)$ obtained by applying $i^*$ to the localization sequence. Furthermore there is a chain of equivalences
\begin{equation} \label{eqn:CompositionVarphi}
\begin{split}
 i^! \D_S(M) &= i^! \Hom(M, f^!\1) \\
&\simeq \Hom(i^*M, i^!f^!\1) \\
&\simeq \Hom(i^*M, f_\sigma^!i^!\1) \\
&\simeq \Hom(i^*M, f_\sigma^!\1)(-1)[-2]\\
&\simeq \D_\sigma(i^* M)(-1)[-2].
\end{split}
\end{equation}
Here we used that for $i: \sigma \rightarrow S $ we have $i^! \1 \simeq \1(-1)[-2]$ by relative purity (see\cite[5.6.2]{CisinskiDegliseEtale}). Moreover note that the equivalence in the second line of (\ref{eqn:CompositionVarphi}) is induced by the projection formula via the canonical equivalences
\begin{align*}
\map_{{\DA_{\et}(X_\sigma, \Lambda)}}(N, \Hom(i^*M, i^!f^!\1)) &\simeq \map_{\DA_{\et}(X_\sigma, \Lambda)}(N \otimes i^*M, i^! f^! \1) \\
&\simeq \map_{\DA_{\et}(X, \Lambda)}(i_! (N \otimes i^*M),  f^! \1) \\
&\simeq \map_{\DA_{\et}(X, \Lambda)}( i_! N \otimes  M,  f^! \1) \\
&\simeq \map_{\DA_{\et}(X, \Lambda)}( i_! N, \Hom( M,  f^! \1)) \\
&\simeq \map_{\DA_{\et}(X_\sigma, \Lambda)}( N, i^!\Hom( M,  f^! \1)) \\
\end{align*}
of mapping spaces for any $N \in \DA_{\et}(X, \Lambda)$.

\noname \label{noname:ComparisonMapsDuality}
Let $f: X \rightarrow S$ be a morphism of finite type. Let $M$ be in $\DA_{\et}(X_\eta, \Lambda)$ and denote by
\[
\id^t: M \otimes \D_\eta(M) \longrightarrow f_\eta^!\1
\]
the transpose of $\id: \D_\eta(M) \rightarrow \D_\eta(M) $ with respect to the $\otimes \dashv \Hom$ adjunction. The lax-monoidal structure of $\chi_f$ induces a canonical morphism
\begin{align*}
\chi_f(\D_\eta(M)) \otimes \chi_f(M) &\longrightarrow \chi_f(\D_\eta(M) \otimes M) \\
&\overset{\chi_f \id^t}\longrightarrow \chi_f (f_\eta^! \1) \\
&\overset{\Ex^!}\longrightarrow f_\sigma^! \chi_{\id}( \1) \\
&\longrightarrow f_\sigma^! i^! \1 [1] \\ &\overset{\sim}\longrightarrow f_\sigma^! \1 (-1)[-1].
\end{align*}
Here the second to last arrow is induced by the sequence (\ref{eqn:i*appliedToLocSeq}) and the last equivalence is relative purity.

Transposing this morphism with respect to the $\otimes \dashv \Hom$ adjunction we get a comparison morphism
\[
\comp_\chi: \chi_f \D_\eta(M) \rightarrow \D_\sigma\chi_f(M) (-1)[-1].
\]

\begin{lem} \label{lem:LocalizationSequCommWDuality}
Let $f:X \rightarrow S$ be a morphism of finite type. For every $M$ in $\DA_{\et}(X, \Lambda)$ the square 
\[
\begin{tikzcd}
i^*j_*j^* \D_S(M) \arrow[r] \arrow[d, "\comp_{\chi}"']        & i^! \D_S(M)[1] \arrow[d, "\varphi"]           \\
{\D_\sigma ( i^*i_*j^*M) (-1)[-1]} \arrow[r] & {\D_\sigma(i^*M)(-1)[-1]}
\end{tikzcd}
\]
commutes. Here the horizontal arrows are induced by the sequence (\ref{eqn:i*appliedToLocSeq}), the left vertical map is the comparison map constructed in \ref{noname:ComparisonMapsDuality} and the right vertical map $\varphi$ is the composition of the chain of equivalences (\ref{eqn:CompositionVarphi}). 
\end{lem}

\begin{proof}
By construction of the two vertical maps the transposed diagram factors as
\begin{equation} \label{eqn:TransposedDiagramOfLEmLocSeqAndDuality}
\begin{tikzcd}
i^*j_*j^* \D_S(M) \otimes i^*M \arrow[r] \arrow[dd]            & i^*j_*j^* \D_S(M) \otimes i^*j_*j^*M \arrow[d] \\
                                                               & i^*j_*j^* (\D_S(M) \otimes M) \arrow[d]        \\
{i^! \D_S(M)[1] \otimes i^*M} \arrow[d, "\varphi \otimes \id"']                        & i^*j_*j^* (f^! \1) \arrow[d]                   \\
{\Hom(i^*M, i^!f^! \1)[1] \otimes i^*M} \arrow[d]             & f_\sigma^! i^*j_*j^*\1 \arrow[d]               \\
{\Hom(i^*M, f_\sigma^! i^! \1)[1] \otimes i^*M} \arrow[r] & {f_\sigma^! i^! \1 [1]}.                        
\end{tikzcd}              
\end{equation}
For a morphism $g: Y \rightarrow X$ and objects $A$ in $\DA_{\et}(X, \Lambda)$ and $B$ in $\DA_{\et}(Y, \Lambda)$ we denote the composition
\[
(g_* B) \otimes A \overset{\text{unit}}\longrightarrow g_*g^*((g_* B) \otimes A)  \simeq  g_*(g^*g_* B \otimes g^*A) \overset{\text{counit}}\longrightarrow g_*( B \otimes g^*A)
\]    
by $\text{proj}_g$. Define a map 
\begin{equation} \label{eqn:ConstOfTheta}
\theta: i^! \D_S(M)[1] \otimes i^*M \longrightarrow i^! (\D_S(M) \otimes M)[1]
\end{equation}
as the transpose of 
\[
i_*(i^! \D_S(M)[1] \otimes i^*M) \overset{\text{proj}_i}\longleftarrow (i_*i^! \D_S(M)) \otimes M[1] \overset{\text{counit}}\longrightarrow \D_S(M) \otimes M [1].
\]
Here the first map is an equivalence by the projection formula for proper maps and the second map is induced by the counit map $\varepsilon: i_*i^! \rightarrow \id$ . The diagram
\begin{equation} \label{eqn:ObvCommForLocSequAndDualLemma}
\begin{tikzcd}
i^*j_*j^* (\D_S(M) \otimes M) \arrow[r] \arrow[d] & {i^! (\D_S(M) \otimes M)[1]} \arrow[d] \\
i^*j_*j^* (f^! \1) \arrow[d] \arrow[r]                                & {i^!f^! \1[1]} \arrow[d]               \\
f_\sigma^! i^*j_*j^*\1 \arrow[r]                                      & f_\sigma^! i^! \1[1]                     
\end{tikzcd}
\end{equation}
commutes. Therefore by pasting the diagrams (\ref{eqn:TransposedDiagramOfLEmLocSeqAndDuality}) and (\ref{eqn:ObvCommForLocSequAndDualLemma}) together it suffices to prove that the diagram
\[
\begin{tikzcd}
i^*j_*j^* \D_S(M) \otimes i^*M \arrow[r] \arrow[dd] \arrow[rdd, phantom, "
(1)"]         & i^*j_*j^* \D_S(M) \otimes i^*j_*j^*M \arrow[d] \\
                                                                        & i^*j_*j^* (\D_S(M) \otimes M) \arrow[d]        \\
{i^! \D_S(M)[1] \otimes i^*M} \arrow[d, "\varphi \otimes \id"'] \arrow[r, "\theta"] \arrow[rdd, phantom, "(2)"] & {i^! (\D_S(M) \otimes M)[1]} \arrow[d]         \\
{\Hom(i^*M, i^!f^! \1)[1] \otimes i^*M} \arrow[d]                      & {i^!f^! \1[1]} \arrow[d]                       \\
{\Hom(i^*M, f_\sigma^! i^!\1 )[1] \otimes i^*M} \arrow[r]          & f_\sigma^! i^! \1        [1]                     
\end{tikzcd}
\]
commutes. 

The diagram 
\[
\begin{tikzcd}
(j_*j^* \D_S(M)) \otimes M \arrow[r, "\id \otimes \text{unit}"] \arrow[rd, "\text{proj}_j"'] & (j_*j^* \D_S(M)) \otimes (j_* j^*M) \arrow[d] \\
                                                                                             & j_*j^*(\D_S(M) \otimes M)                    
\end{tikzcd}\]
commutes by an easy argument using the triangle equalities. Hence we may show that 
\[
\begin{tikzcd}
i^*j_*j^* \D_S(M) \otimes i^*M \arrow[r, "i^* \text{proj}_j"] \arrow[d] \arrow[rd, phantom, "(1)'"] & i^*j_*j^*(\D_S(M) \otimes M) \arrow[d] \\
i^! \D_S(M) \otimes i^*M \arrow[r, "\theta"]                                   & {i^!(\D_S(M) \otimes M)[1]}           
\end{tikzcd}
\]
commutes in order to show that (1) commutes. Consider the morphism of fiber sequences
\begin{equation} \label{eqn:MorphismOfLocSequAndTensorUsingTheta}
\begin{tikzcd}
(i_*i^!\D_S(M)) \otimes M \arrow[d, "(i_*\theta) \circ \text{proj}_i"'] \arrow[r] & \D_S(M) \otimes M \arrow[d, "\id"] \arrow[r] & (j_*j^*\D_S(M)) \otimes M \arrow[d, "\text{proj}_j"] \\
i_*i^!(\D_S(M) \otimes M) \arrow[r]                                               & \D_S(M) \otimes M \arrow[r]                  & j_*j^*(\D_S(M) \otimes M).                           
\end{tikzcd}
\end{equation}
Indeed showing that the right square commutes amounts to an easy check using the triangle equalities and the left square commutes directly by construction of $\theta.$  Then shifting (\ref{eqn:MorphismOfLocSequAndTensorUsingTheta}) by one square and applying $i^*$ gives commutativity of $(1)'$.

Clearly commutativity of $(2)$ reduces to commutativity of  
\[
\begin{tikzcd}
{i^! \D_S(M)[1] \otimes i^*M} \arrow[r, "\theta"] \arrow[d, "\varphi \otimes \id"'] \arrow[rd, phantom, "(2)'"] & {i^! (\D_S(M) \otimes M)[1]} \arrow[d, "i^! \id^t"] \\
{\Hom(i^*M, i^!f^! \1)[1] \otimes i^*M} \arrow[r, "\id^t "']                                      & {i^!f^! \1[1]}                                                                   
\end{tikzcd}
\]
where $\id^t$ denotes the respective transpose of the identity. This amounts to showing that 
\begin{equation} \label{eqn:lastTriangleOfPfOfLEmmaOFRecollemtAndDuality}
\begin{tikzcd}
{i^! \D_S(M)[1] \otimes i^*M} \arrow[r, "\theta"] \arrow[rd, "\varphi^t"'] & {i^! (\D_S(M) \otimes M)[1]} \arrow[d, "i^! \id^t"] \\
                                                                           & {i^!f^! \1[1]}                                               
\end{tikzcd}
\end{equation}
is commutative, where $\varphi^t$ denotes the transpose of $\varphi$. One can check directly that 
\[
i_*(i^! \D_S(M) \otimes M ) \overset{\text{proj}_i}\longleftarrow  (i_*i^! \D_S(M)) \otimes M\overset{\text{counit}}\longrightarrow \D_S(M) \otimes M \overset{\id^t}\longrightarrow f^! \1
\]
corresponds to $\varphi^t$ via
\[
\map_{\DA_{\et}(X_\sigma, \Lambda)}(i^! \D_S(M)\otimes i^*M, i^!f^!\1) \simeq \map_{{\DA_{\et}(X, \Lambda)}}(i_*(i^! \D_S(M) \otimes i^*M), f^! \1).
\]
By definiton of $\theta$ this means precisely that (\ref{eqn:lastTriangleOfPfOfLEmmaOFRecollemtAndDuality}) commutes.

\end{proof}

\noname Consider a morphism of schemes $f: X \rightarrow S$. Then as observed in \ref{remark:AllNearbyCyclesCoincWAyoubAndAreSpecSystems} (2) there are canonical natural transformations $\mu: i^*j_* \rightarrow \Upsilon_f$ and $\gamma: \Upsilon_f  \rightarrow \Psi_f$ giving rise to morphisms of specialization systems. Let us write for any $M$ in $\DA_{\et}(X, \Lambda)$
\[
\beta_M: i^*M \overset{\unit}\longrightarrow i^*j_*j^* M \overset{\mu_{j^*M}}\longrightarrow \Upsilon_f(j^*M)
\]
and
\[
\alpha_M: i^*M \overset{\beta_M} \longrightarrow \Upsilon_f(j^*M) \overset{\gamma_{j^*M}}\longrightarrow \Psi_f(j^*M)
\]
for the compositions.

\begin{lem} \label{lem:AlphaEquImplBetaEqu}
Let  $f: X \rightarrow S$  be a morphism of finite type and $M$ in $\DA^{\cons}_{\et}(X, \Lambda)$. Assume that $\Lambda$ is a $\Q$-algebra. If 
\[
\alpha_M: i^* M \longrightarrow \Psi_f(j^* M)
\]
is an equivalence, then
\[
\beta_M: i^*M \longrightarrow \Upsilon_f(j^*M)
\]
and
\[
\gamma_{j^*M}: \Upsilon_f(j^*M) \longrightarrow \Psi_f(j^*M)
\]
are both equivalences.
\end{lem}

\begin{proof}
This follows from the fact that $\gamma_{j^*M}$ is the inclusion of a direct summand by Corollary \ref{cor:UpsilonRetractOfPsi}.
\end{proof}

\noname \label{noname:MonodromySequenceAndComparisonDualitySpecSystems} Assume that $\Lambda$ is a $\Q$-algebra. Then for any morphism of schemes $f:X \rightarrow S$ there is a fiber sequence
\begin{equation} \label{eqn:monodromy sequence}
\chi_f \longrightarrow \Upsilon_f \overset{N}\longrightarrow \Upsilon_f(-1)
\end{equation}
by \cite[11.16]{AyoubRealizationEtale}. We call $N$ the \textit{monodromy operator}. 

If $f:X \rightarrow S$ is a morphism of finite type and $M$ is in $\DA_{\et}(X_\eta, \Lambda)$ we obtain as in \ref{noname:CompPsi} a comparison map
\[
\comp_\Upsilon: \Upsilon_f(\D_\eta(M)) \longrightarrow \D_\sigma (\Upsilon_f (M))
\]
as the transpose of the composition
\begin{align*}
\Upsilon_f(\D_\eta(M)) \otimes \Upsilon_f(M) &\longrightarrow \Upsilon_f(\D_\eta(M) \otimes M) \\
&\overset{\Upsilon_f \id^t}\longrightarrow \Upsilon_f (f_\eta^!\1) \\
&\overset{\Ex^!}\longrightarrow f_\sigma^! \Upsilon_{\id} (\1) \\
&\overset{\sim}\longrightarrow f_\sigma^! \1.
\end{align*}

By \cite[11.16]{AyoubRealizationEtale} this comparison map fits together with the comparison map from \ref{noname:ComparisonMapsDuality} into a morphism of fiber sequences
\begin{equation} \label{eqn:monodromySequAndDuality}
\begin{tikzcd}
\Upsilon_f \D_\eta(M) (1) \arrow[r, "-N_{\D_\eta(M)}(1)"] \arrow[d, "\comp_\Upsilon (1)"] & \Upsilon_f \D_\eta(M) \arrow[d,"\comp_\Upsilon"] \arrow[r] & {\chi_f \D_\eta(M)(1)[1]} \arrow[d, "\comp_\chi(1) {[1]} "] \arrow[r] & {\Upsilon_f \D_\eta(M)(1)[1]} \arrow[d, "\comp_\Upsilon(1) {[1]} "] \\
\D_\sigma(\Upsilon_f(M)))(1) \arrow[r, "\D_\sigma(N_M)"] & \D_\sigma(\Upsilon_f(M))) \arrow[r]       & \D_\sigma(\chi_f(M))) \arrow[r]               & {\D_\sigma(\Upsilon_f(M)))(1)[1]}.      
\end{tikzcd}
\end{equation}

Whenever $M$ is constructible 
\[
\comp_\Upsilon: \Upsilon_f(\D_\eta(M)) \longrightarrow \D_\sigma (\Upsilon_f (M))
\]
is an equivalence by \cite[10.21]{AyoubRealizationEtale}. In particular in this case the map of fiber sequences (\ref{eqn:monodromySequAndDuality}) implies that 
\[\comp_\chi: \chi_f \D_\eta(M) \longrightarrow \D_\sigma \chi_f( M)(-1)[-1] \]
is an equivalence as well.

\begin{lem} \label{lem:Beta_MequivImpliesBeta_DMequiv}
Let  $f: X \rightarrow S$  be a morphism of finite type and $M$ in $\DA^{\cons}_{\et}(X, \Lambda)$. Assume that $\Lambda$ is a $\Q$-algebra. If 
\[
\alpha_M: i^*M \longrightarrow \Psi_f(j^*M)
\]
is an equivalence, then so is
\[
\alpha_{\D_S(M)}: i^*\D_S(M) \longrightarrow \Psi_f(j^*\D_S(M)).
\]
\end{lem}

\begin{proof}
By Lemma \ref{lem:AlphaEquImplBetaEqu} the maps $\beta_M$ and $\gamma_{j^*M}$ are equivalences. Consider the monodromy sequence of $j^*M$:
\begin{equation} \label{eqn:MonSeqMLemma}
i^*j_*j^* M \longrightarrow \Upsilon_f(j^*M) \overset{N_{j^*M}}\longrightarrow \Upsilon_f(j^*M)(-1).  
\end{equation}
Since the equivalence $\beta_M$ factors as
\[
i^*M \longrightarrow i^*j_*j^* M \longrightarrow \Upsilon_f(j^*M)
\]
it gives rise to a splitting of (\ref{eqn:MonSeqMLemma}). Lemma  \ref{lem:LocalizationSequCommWDuality} implies that there is a map of fiber sequences
\[
\begin{tikzcd}
i^*\D_S(M) \arrow[r] \arrow[d, "\vartheta", dashed] & i^* j_*j^* \D_S(M)  \arrow[r] \arrow[d, "\comp_\chi"]    & {i^!\D_S(M)[1]} \arrow[d, "\varphi"] \\
{\D_\sigma(i^!M)(-1)[-2]} \arrow[r]    & {\D_\sigma(i^*j_*j^* M)(-1)[-1]} \arrow[r] & {\D_\sigma(i^* M)(-1)[-1]}          
\end{tikzcd}
\]
for some $\vartheta$ (which is an equivalence since $\comp_\chi$ and $\varphi$ are equivalences). Consider the commutative diagram
\begin{equation} \label{eqn:TopCompisbeta_DM}
\begin{tikzcd}
 i^* \D_S(M) \arrow[d, "\vartheta"] \arrow[r] & i^*j_*j^* \D_S(M) \arrow[d, "\comp_\chi"'] \arrow[r] & \Upsilon_f(j^* \D_S(M)) \arrow[d, "\comp_\Upsilon"] \\
{\D_\sigma(i^! M)(-1)[-2]} \arrow[r]      & {\D_\sigma ( i^*i_*j^*M)(-1)[-1]} \arrow[r]      & \D_\sigma(\Upsilon_f(j^*M)),                
\end{tikzcd}
\end{equation}
where the right square is part of (\ref{eqn:monodromySequAndDuality}). Note that the top composition of (\ref{eqn:TopCompisbeta_DM}) is precisely $\beta_{\D_S(M)}: i^* \D_S(M) \longrightarrow \Upsilon_f(j^* \D_S(M))$ and the vertical maps are equivalences. Thus in order to show that $\beta_{\D_S(M)}$ is an equivalence it suffices to show that the  bottom composition of (\ref{eqn:TopCompisbeta_DM}) is an equivalence. 

For this consider the morphism of fiber sequences
\[
\begin{tikzcd}
{\D_\sigma(i^!M)(-1)[-2]} \arrow[r] \arrow[d, "\omega"', dashed] & {\D_\sigma(i^*j_*j^*M)(-1)[-1]} \arrow[d,"\id"] \arrow[r] & {\D_\sigma(i^*M)(-1)[-1]} \arrow[d, "{\D_\sigma(\beta_M^{-1})(-1)[-1]}"] \\
\D_\sigma(\Upsilon_f(j^*M)) \arrow[r]                            & {\D_\sigma(i^*j_*j^*M)(-1)[-1]} \arrow[r, ""]                                & {\D_\sigma(\Upsilon_f(j^*M))(-1)[-1]},                                  
\end{tikzcd}
\]
where $\omega$ is the induced map and hence an equivalence. Here the top sequence is obtained from (\ref{eqn:i*appliedToLocSeq}) and the bottom sequence is obtained by applying $\D_\sigma (\_)$ to the fiber sequence
\[
\Upsilon_f(j^*M)(1)[1] \rightarrow i^*j_*j^* M (1)[1] \simeq \Upsilon_f(j^*M)(1)[1] \oplus \Upsilon_f(j^*M) \rightarrow \Upsilon_f(j^*M)
\]
obtained by splitting (\ref{eqn:MonSeqMLemma}) via $\beta_M$. By construction the bottom composition of (\ref{eqn:TopCompisbeta_DM}) is precisely $\omega$ and therefore an equivalence. This shows that $\beta_{\D_S(M)}$ is an equivalence. 

Since $\Upsilon \rightarrow \Psi$ is a morphism of lax-monoidal specialization systems, it is straightforward to check that the diagram
\[
\begin{tikzcd}
\Upsilon_f(\D_\eta j^*M) \arrow[d, "\comp_\Upsilon"'] \arrow[r, "\gamma_{\D_\eta j^*M}"] & \Psi_f(\D_\eta j^*M) \arrow[d, "\comp_\Psi"]                  \\
\D_\sigma \Upsilon_f(j^*M)                                                               & \D_\sigma \Psi_f(j^*M) \arrow[l, "\D_\sigma(\gamma_{j^*M})"']
\end{tikzcd}
\]
commutes. The vertical comparison maps are equivalences and, as observed in the very beginning of the proof, $\gamma_{j^*M}$ is an equivalence. Hence $\gamma_{\D_\eta j^*M}$ is an equivalence which implies that
\[
\alpha_{\D_S(M)}: i^* \D_S(M) \overset{\beta_{\D_S(M)}}\longrightarrow \Upsilon_{f} (j^*\D_S(M)) \overset{\gamma_{j^* \D_S M}}\longrightarrow \Psi_{f} (j^*\D_S(M))
\]
is an equivalence.
\end{proof}

\begin{lem} \label{lem:a_MisoImpliesa_MxDMiso} Let $f: X \rightarrow S$ be a morphism of finite type and $M$ in $\DA_{\et}(X, \Lambda)$. If the canonical maps
\[\can_{M} : i ^*M  \longrightarrow \Psi_{f}(j^* M )\]
and
\[ \can_{\D_S(M)} : i ^*\D_S(M)  \rightarrow \Psi_{f}(j^* \D_S(M) )
\]
are equivalences, then 
\begin{align*}\can_{M \boxtimes \D_S(M)} : i^*(M \boxtimes \Dual_S(M)) \longrightarrow \Psi_{f\times f}( j^* (M \boxtimes \Dual_S(M))) 
\end{align*}
is an equivalence. 
\end{lem}

\begin{proof}
Consider the diagram
\[
\begin{tikzcd}
i^*M \boxtimes i^*\D_S(M) \arrow[d, "\sim"'] \arrow[r] & \chi_f(j^*M) \boxtimes \chi_f(j^*\D_S(M)) \arrow[d] \arrow[r] \arrow[ld, phantom, "(1)"] & \Psi_f(j^*A) \boxtimes \Psi_f(j^*\D_S(M)) \arrow[d, "\sim"] \arrow[ld, phantom, "(2)"] \\
(i \times i)^*(M \boxtimes \D_S(M)) \arrow[r]            & \chi_{f \times f}(j^*M \boxtimes j^*\D_S(M)) \arrow[r]                   & \Psi_{f \times f}(j^*M \boxtimes j^*\D_S(M)).                            
\end{tikzcd}
\] 
The right vertical map is an equivalence by Theorem \ref{thm:PropertiesOfEtaleMotivicNearby} (3). The commutativity of (1) is an easy check considering the pseudo monoidal structure of $j_*$. The commutativity of (2) follows from the fact that $\chi \rightarrow \Psi$ is a morphism of lax-monoidal specialization systems (see Remark \ref{remark:AllNearbyCyclesCoincWAyoubAndAreSpecSystems} (3)). By assumption the top horizontal map is an equivalence. Hence the bottom horizontal map is an equivalence as desired. 
\end{proof}

\noname \label{noname:AlphaEquivGivesSection}Let $f: X \rightarrow S$ be of finite type, $M$ in $\DA^{\cons}_{\et}(X, \Lambda)$ and assume that $\Lambda$ is a $\Q$-algebra. Recall what we observed in the beginning of the proof of Lemma \ref{lem:Beta_MequivImpliesBeta_DMequiv}: If 
\[
\alpha_M: i^* M \longrightarrow \Psi_f(j^*M)
\]
is an equivalence this implies by Lemma \ref{lem:AlphaEquImplBetaEqu} that 
\[\beta_M: i^* M \overset{\unit}\longrightarrow \Upsilon_f(j^*M) \overset{\mu_{j^*M}}\longrightarrow \Upsilon_f(j^*M) \]
is an equivalence. In particular $\beta_M$ splits the monodromy sequence
\[
\chi_f(j^*M) \overset{\mu_{j^*M}}\longrightarrow \Upsilon_f(j^*M) \overset{N}\longrightarrow \Upsilon_f(j^*M) (-1). 
\]
and hence there is an equivalence
\begin{equation} \label{eqn:SplitMonodromySeq}
\chi_f(j^*M) \simeq \Upsilon_f(j^*M) \oplus \Upsilon_f(j^*M)(-1)[-1]. 
\end{equation}
In particular there are fiber sequences
\[
\Upsilon_f(j^*M) \overset{i_1}\longrightarrow \chi_f(j^*M) \overset{p_2}\longrightarrow \Upsilon_f(j^*M) (-1)[-1]
\]
and
\[
\Upsilon_f(j^*M)(-1)[-1] \overset{i_2}\longrightarrow \chi_f(j^*M) \overset{p_1}\longrightarrow \Upsilon_f(j^*M),
\]
where $i_1,i_2$ denote the inclusions and $p_1,p_2$ denote the projections with respect to the direct sum decomposition (\ref{eqn:SplitMonodromySeq}). 

Note that in the situation above $\alpha_{\D_S(M)}$ is also an equivalence by Lemma \ref{lem:Beta_MequivImpliesBeta_DMequiv} and thus $\beta_{\D_S(M)}$ induces a splitting of the monodromy sequence of $\D_S(M)$ analogue as above.

\begin{lem} \label{lem:DChi=ChiDCompatibleWithSplitting}
Let  $f: X \rightarrow S$  be a morphism of schemes and $M$ in $\DA^{\cons}_{\et}(X, \Lambda)$. Assume that $\Lambda$ is a $\Q$-algebra and that $\alpha_M: i^*M \rightarrow \Psi_f(j^*M)$ is an equivalence. Then the diagram
\[
\begin{tikzcd}
\chi_f(\D_\eta(j^*M)) \arrow[rr, "\sim", no head] \arrow[dd, "\comp_\chi"'] &  & \Upsilon_f(\D_\eta(j^*M)) \oplus \Upsilon_f(\D_\eta(j^*M)) (-1)(-1)  \arrow[dd, "\MatrixForThm"] \\
                                                                            &  &                                                                                                  \\
{\D_\sigma(\chi_f(j^*M))(-1)[-1]} \arrow[rr, "\sim", no head]                 &  & {\D_\sigma(\Upsilon_f(j^*M))(-1)[-1] \oplus \D_\sigma(\Upsilon_f(j^*M))}                            
\end{tikzcd}
\]
commutes. Here the horizontal equivalences are the direct sum decompositions induced by $\beta_M$ and $\beta_{\D_S(M)}$ as above.
\end{lem}

\begin{proof}
Consider the diagram 
\begin{equation} \label{eqn:BigDiagramLemmaSplittingAndDuality}
\begin{tikzcd}
\Upsilon_f(j^*\D_S(M))) \arrow[r, "i_1"] \arrow[d, "\beta_{\D_S(M)}^{-1}"] & \chi_f(j^*\D_S(M)) \arrow[r, "p_2"] \arrow[d, "\id"]                            & {\Upsilon_f(j^*\D_S(M)))(-1)[-1]} \arrow[d, dashed]                    \\
i^*\D_S(M) \arrow[r] \arrow[d, dashed]                                     & \chi_f(j^*\D_S(M)) \arrow[d, "\comp_{\chi}"] \arrow[r]                          & {i^!\D_S(M)[1]} \arrow[d, "\varphi"]                                   \\
{\D_\sigma(i^!M)(-1)[-2]} \arrow[r]                                        & {\D_\sigma(\chi_f(j^*M))(-1)[-1]} \arrow[r, "\D_\sigma(\unit)"]                 & {\D_\sigma(i^*M)(-1)[-1]}                                              \\
\D_\sigma(\Upsilon_f(j^*M)) \arrow[u, dashed] \arrow[r, "\D_\sigma(p_2)"]  & {\D_\sigma(\chi_f(j^*M))(-1)[-1]} \arrow[u, "\id"'] \arrow[r, "\D_\sigma(i_1)"] & {\D_\sigma(\Upsilon_f(j^*M))(-1)[-1]} \arrow[u, "{\D_\sigma(\beta_M)(-1)[-1]}"']
\end{tikzcd}
\end{equation}

The horizontal rows are fiber sequences, the solid squares commute (the right middle square commutes by
Lemma \ref{lem:LocalizationSequCommWDuality}) and the dashed arrows are the ones induced by completing the maps of fiber sequences. Note that all vertical maps are equivalences. We claim that the right vertical top to bottom composition of (\ref{eqn:BigDiagramLemmaSplittingAndDuality}) is equivalent to $\comp_\Upsilon (-1)[-1]$. Indeed this follows by careful inspection from the commutativity of the outer square of
\[
\begin{tikzcd}
{\Upsilon_f(j^*\D_S(M))(-1)[-1]} \arrow[d, "i_2"] \arrow[rr, "{\comp_{\Upsilon}(-1){[-1]}}"] &  & {\D_\sigma(\Upsilon_f(j^*M))(-1)[-1]} \arrow[d, "{\D_\sigma(\mu_{j^*M})(-1){[-1]}}"] \\
\chi_f(j^*\D_S(M)) \arrow[d] \arrow[rr, "\comp_\chi"]                                 &  & {\D_\sigma(\chi_f(j^*M))(-1)[-1]} \arrow[d, "{\D_\sigma(\unit)(-1){[-1]}}"]         \\
i^!\D_S(M) \arrow[rr, "\varphi"]                                                      &  & {\D_\sigma(i^*M)(-1)[-1].}                                                          
\end{tikzcd}
\]
Note for this that the right vertical composition is $\D_\sigma(\beta_M)(-1)[-1]$. The top square commutes by (\ref{eqn:monodromySequAndDuality}) and commutativity of the bottom square is Lemma \ref{lem:LocalizationSequCommWDuality}. Using the map of fiber sequences (\ref{eqn:monodromySequAndDuality}) we see that the left vertical composition of (\ref{eqn:BigDiagramLemmaSplittingAndDuality}) must already be equivalent to $\comp_\Upsilon$ which finishes the proof.  
\end{proof}

\begin{lem} \label{lem:constrOverFieldIsULA}
Let $k$ be a field of finite \'etale cohomological dimension for $\Lambda$-coefficients and $g: Y \rightarrow \Spec k$ a morphism of finite type. Then for $M$ in $\DA_{\et}(Y, \Lambda)$ the following are equivalent:
\begin{enumerate}
\item $M$ is constructible.
\item $M$ is ULA with respect to $g$.
\end{enumerate}
In particular this holds true if $k$ is the residue field of a point of a strictly local noetherian scheme. 
\end{lem}

\begin{proof}
We claim that $Y$ and $Y \times_{k} Y$ are of finite $\Lambda$-cohomological dimension. Indeed as in the proof of \cite[1.1.5]{CisinskiDegliseEtale} we may assume that $\Lambda= \Z/n\Z$ for some positive integer $n$.  Since $k$ is of finite $\Lambda$-cohomological dimension it follows from \cite[0F0V]{stacks-project} that $Y$ and $Y \times_{k} Y$ are of finite $\Lambda$-cohomological dimension. Hence $\DA_{\et}(Y, \Lambda)$ and $\DA_{\et}(Y \times_k X, \Lambda)$ are compactly generated by constructible objects by \cite[5.2.4]{CisinskiDegliseEtale}. Now \ref{prop:MULAAndConstrGenImpliesConstr} implies that every $M$ which is ULA with respect to $g$ is already constructible. 

Conversely any constructible $M$ is $g$-ULA: This follows from the fact that the K\"unneth formula holds for \'etale motives over a field (see \cite[3.1.12]{Cisinski2021CohomologicalMI}) and Proposition \ref{thm:EquivCharOfULA}. 

The last sentence of the Lemma follows from the fact that such $k$ have finite \'etale cohomological dimension by \cite[1.1.5]{CisinskiDegliseEtale}.
\end{proof}

\begin{remark}
Let $g: Y \rightarrow \Spec k$ be of finite type where $k$ is a field of finite $\Lambda$-cohomological dimension. Let us write $\D(\_):= \D_{\Spec k}(\_)= \Hom(\_, g^!\1)$. Then by Lemma  \ref{lem:constrOverFieldIsULA} we have an equivalence  
\[
\DA_{\et}^{\cons}(Y, \Lambda) \simeq \DA_{\et}^{\ULA}(Y/ \Spec k)
\]
of subcategories of $\DA_{\et}(Y, \Lambda)$. Hence Proposition \ref{prop:propertiesOfULA} implies that for any $M$ in $\DA_{\et}^{\cons}(Y, \Lambda)$ the canonical map $M \rightarrow \D \D M$ is an equivalence. In particular the functor
\[
\D(\_) : \DA_{\et}^{\cons}(Y, \Lambda)^{\op} \longrightarrow \DA_{\et}^{\cons}(Y, \Lambda)
\]
is an involution (i.e. satisfies $\D \circ \D \simeq \id$). This is sometimes referred to as \textit{Verdier duality}.

Let $h: Z \rightarrow Y$ be a morphism of finite type and $M$ in $\DA_{\et}(Y, \Lambda)$. The canonical map
\begin{equation} \label{eqn:ComparisonIsoShriekandD}
h^!\D M  \longrightarrow \D h^* M
\end{equation}
obtained as the transpose of
\[
h^!\D M  \otimes  h^* M \overset{\theta}\longrightarrow h^!(\D M  \otimes M) \longrightarrow h^! f^! \1
\]
(where $\theta$ is defined as in (\ref{eqn:ConstOfTheta})) is always an equivalence (see for example \cite[2.34 (iv)]{Adeel6Functor}). If $M$ is moreover constructible we get by Verdier duality an equivalence
\begin{equation} \label{eqn:ComparisonIsoShriekandDViaVerdierDuality}
h^* \D M \overset{\sim}\longrightarrow \D \D h^* \D M \overset{\sim}\longrightarrow \D h^! \D  \D M \overset{\sim}\longrightarrow \D h^! M.  
\end{equation}
\end{remark}

\begin{lem} \label{lem:CompAndPullback}
Let $f: X \rightarrow S$ and $g: Y \rightarrow X$ be of finite type and let $h: Y \overset{g}\rightarrow X \overset{f}\rightarrow  S$ denote the composition. For any $M$ in $\DA_{\et}(X_\eta, \Lambda)$  the diagram
\begin{equation} \label{eqn:FirstDiagOfLemCompAndPullback}
\begin{tikzcd}
\chi_h(g_\eta^!\D_\eta(M)) \arrow[r, "\Ex^!"] \arrow[d, "\sim"']     & g_\sigma^! \chi_f(\D_\eta(M)) \arrow[d, "g_\sigma^!\comp_\chi"] \\
\chi_h(\D_\eta(g_\eta^* M)) \arrow[d, "\comp_\chi"']      & {g_\sigma^! \D_\sigma(\chi_f(M))(-1)[-1]} \arrow[d, "\sim"]             \\
{ \D_\sigma(\chi_h(g_\eta^* M))(-1)[-1]} \arrow[r, "\D_\sigma \Ex^*"] & { \D_\sigma(g_\sigma^*\chi_f(M))(-1)[-1]}                      
\end{tikzcd}
\end{equation}
commutes. If  $M$ is moreover assumed to be constructible, then the diagram
\begin{equation} \label{eqn:SecondDiagOfLemCompAndPullback}
\begin{tikzcd}
g_\sigma^* \chi_f(\D_\eta(M)) \arrow[d, "g_\sigma^*\comp_\chi"'] \arrow[r, "\Ex^*"] & \chi_h(g_\eta^*\D_\eta(M)) \arrow[d, "\sim"]                 \\
{g_\sigma^* \D_\sigma(\chi_f(M))(-1)[-1]} \arrow[d, "\sim"']                               & \chi_h(\D_\eta(g_\eta^! M)) \arrow[d, "\comp_\chi"] \\
{ \D_\sigma(g_\sigma^!\chi_f(M))(-1)[-1]} \arrow[r, "\D_\sigma \Ex^!"]                               & { \D_\sigma(\chi_f(g_\eta^! M))(-1)[-1]}            
\end{tikzcd}
\end{equation}
also commutes. 
\end{lem}

\begin{proof} It is straightforward to check that the composition
\[
\chi_h(g_\eta^!\D_\eta(M)) \overset{\Ex^!}\longrightarrow g_\sigma^! \chi_f(\D_\eta(M)) \overset{g_\sigma^! \comp_\chi}\longrightarrow g_\sigma^! \D_\sigma(\chi_f(M))(-1)[-1] \overset{\sim}\longrightarrow \D_\sigma(g_\sigma^*\chi_f(M))(-1)[-1]
\]
is the transpose of the composition
\begin{align*}
\chi_h(g_\eta^!\D_\eta(M)) \otimes g_\sigma^* \chi_f(M) &\overset{\Ex^! \otimes \id}\longrightarrow g_\sigma^!\chi_f(\D_\eta(M)) \otimes g_\sigma^* \chi_f(M) \\
&\overset{\theta}\longrightarrow g_\sigma^!(\chi_f(\D_\eta(M) \otimes \chi_f(M)) \\
&\longrightarrow h_\sigma^! \1 (-1)[-1].
\end{align*}
Here the last map is the composition
\begin{align*}
g_\sigma^!(\chi_f(\D_\eta(M) \otimes \chi_f(M)) &\longrightarrow g_\sigma^!(\chi_f(\D_\eta(M) \otimes M)) \\
&\overset{g_\sigma^!(\chi_f \id^t)}
\longrightarrow g_\sigma^!(\chi_f(f_\eta^! \1))\\
&\overset{g_\sigma^!\Ex^!}\longrightarrow g_\sigma^!f_\sigma^! \chi_{\id}(\1) \\
&\longrightarrow h_\sigma^! \1(-1)[-1].
\end{align*}
Hence in order to show that (\ref{eqn:FirstDiagOfLemCompAndPullback}) commutes it suffices to show that its transposed diagram
\begin{equation} \label{eqn:FirstDiagOfLemCompAndPullbackTransposed}
\begin{tikzcd}
\chi_h(g_\eta^!\D_\eta(M)) \otimes g_\sigma^* \chi_f(M) \arrow[r, "\Ex^! \otimes \id"] \arrow[d, "\sim"] & g_\sigma^!\chi_f(\D_\eta(M)) \otimes g_\sigma^* \chi_f(M) \arrow[d, "\theta"] \\
\chi_h(\D_\eta(g_\eta^*M)) \otimes g_\sigma^* \chi_f(M) \arrow[d, "\id \otimes \Ex^*"]                   & g_\sigma^!(\chi_f(\D_\eta(M) \otimes \chi_f(M)) \arrow[d]           \\
\chi_h(\D_\eta(g_\eta^*M)) \otimes \chi_h(g_\eta^* M) \arrow[r]                                        & h_\sigma^! \1(-1)[-1]                                                      
\end{tikzcd}
\end{equation}
commutes. Note that
\[
\begin{tikzcd}
\chi_h(g_\eta^!\D_\eta(M)) \otimes g_\sigma^* \chi_f(M) \arrow[d, "\sim"] \arrow[r, "\id \otimes \Ex^*"]             & \chi_h(g_\eta^!\D_\eta(M)) \otimes  \chi_f(g_\eta^*M) \arrow[r]     & \chi_h(g_\eta^!\D_\eta(M) \otimes g_\eta^* M) \arrow[d, "\chi_h(\theta)"] \\
\chi_h(\D_\eta(g_\eta^*M)) \otimes g_\sigma^* \chi_f(M) \arrow[d, "\id \otimes \Ex^*"]            & {} & \chi_h(g_\eta^!(\D_\eta(M)\otimes M)) \arrow[d]           \\
\chi_h(\D_\eta(g_\eta^*M)) \otimes \chi_h(g_\eta^* M) \arrow[r] \arrow[rruu, "(1)" description, phantom] & \chi_h(\D_\eta(g_\eta^*M) \otimes g_\eta^* M) \arrow[ruu, "\sim"]  \arrow[ru, "(2)" description, phantom]                                    \arrow[r] & \chi_h(h^! \1)                                            
\end{tikzcd}
\]
commutes. Indeed (1) follows straight from functoriality of the lax-monoidal structure of $\chi_h$ and (2) commutes by construction of $(\ref{eqn:ComparisonIsoShriekandD})$. Hence in order to show that (\ref{eqn:FirstDiagOfLemCompAndPullbackTransposed}) commutes it suffices to show that 
\[
\begin{tikzcd}
\chi_h(g_\eta^!\D_\eta(M)) \otimes g_\sigma^* \chi_f(M) \arrow[d, "\sim"] \arrow[rr, "\Ex^! \otimes \id"]    &                                                                    & g_\sigma^!\chi_f(\D_\eta(M)) \otimes g_\sigma^* \chi_f(M) \arrow[d, "\theta"] \\
\chi_h(g_\eta^!\D_\eta(M)) \otimes \chi_h(g_\eta^* M) \arrow[dd]                                           &                                                                    & g_\sigma^!(\chi_f(\D_\eta(M) \otimes \chi_f(M)) \arrow[d]                     \\
                                                                                                           &                                                                    & g_\sigma^!\chi_h(\D_\eta(M) \otimes M) \arrow[d]                              \\
\chi_h(g_\eta^!\D_\eta(M) \otimes g_\eta^* M) \arrow[r, "\chi_h(\theta)"] \arrow[rruuu, "(3)" description, phantom] & \chi_h(g_\eta^!(\D_\eta(M) \otimes M)) \arrow[ru, "\Ex^!"] \arrow[r] & f_\sigma^! \1(-1)[-1]                                                                
\end{tikzcd}
\]
commutes. The commutativity of (3) is \cite[3.1.15]{AyoubThesisII} and the triangle is obvious.

Since $S$ is excellent and all schemes considered are of finite type over $S$ all functors preserve constructibility by Theorem \ref{thm:SixFunctorsPresConstr}. Assume that $M$ is constructible. We want to deduce the commutativity of (\ref{eqn:SecondDiagOfLemCompAndPullback}) from the commutativity of  (\ref{eqn:FirstDiagOfLemCompAndPullback}) using the involutions $\D_\sigma(\_)$ and $\D_\eta(\_)$. First we want to show that for every $N$ in $\DA_{\et}(X_\eta, \Lambda)$ and $k: Z \rightarrow S$ of finite type the square
\begin{equation} \label{eqn:PrelimClaimInCompBasechangeLemma}
\begin{tikzcd}
\chi_k(N) \arrow[r, "\sim"] \arrow[d, "\sim"']    & \D_\sigma \D_\sigma\chi_k(N) \arrow[d, "{\D_\sigma(\comp_\chi)(-1)[-1]}"] \\
\chi_k(\D_\eta \D_\eta N) \arrow[r, "\comp_\chi"] & {\D_\sigma \chi_k(\D_\eta N)(-1)[-1]}                                    
\end{tikzcd}
\end{equation}
commutes. The outer diagram of
\[
\begin{tikzcd}
\chi_k(N) \otimes \chi_k(\D_\eta N ) \arrow[dd, "\sim"] \arrow[rd] \arrow[rr,"\id \otimes \comp_\chi"] &                                                                                         & \chi_k(N) \otimes \D_\eta\chi_k(N)(-1)[-1] \arrow[dd,"(\id_{\D_\eta\chi_k(N)(-1){[-1]}})^t"] \\
{}                                                                            & \chi_k( N \otimes \D_\eta N) \arrow[d, "\sim"] \arrow[""'{name=foo}]{rd} \arrow[ru, "(4)" description, phantom] &                                               \\
\chi_k(\D_\eta \D_\eta N) \otimes \chi_k(\D_\eta N) \arrow[r]                 & \chi_k(\D_\eta \D_\eta N \otimes \D_\eta N) \arrow[r] \arrow[lu, "(5)" description, phantom] \arrow[rightarrow, to=foo, phantom, "(6)" description]     & {k^!\1(-1){[-1]}}                              
\end{tikzcd}
\]
is the transpose of (\ref{eqn:PrelimClaimInCompBasechangeLemma}). (4) commutes since both compositions are by construction transposed to
\[
\comp_\chi: \chi_k(\D_\eta N) \longrightarrow \D_\sigma \chi_k(N)(-1)[-1].
\]
Clearly (5) commutes. Finally in order to show that (6) commutes it suffices to show that 
\[
\begin{tikzcd}
N \otimes \D_\eta N \arrow[d, "\sim"'] \arrow[rrd, "(\id_{\D_\eta N})^t"]      &  &    \\
\D_\eta \D_\eta N \otimes \D_\eta N \arrow[rr, "(\id_{\D_\eta \D_\eta N})^t"'] &  & k^!\1
\end{tikzcd}
\]
commutes. This is true since both compositions are transposed to the canonical map $N \overset{\sim}\rightarrow \D_\eta \D_\eta N.$

We obtain a commutative diagram
\begin{equation} \label{eqn:DsigmaOfFirstDiagOfLemCompAndPullback}
\begin{tikzcd}
\D_\sigma \D_\sigma g_\sigma^* \chi_f(\D_\eta M) \arrow[r, "\D_\sigma \D_\sigma \Ex^*"] \arrow[d, "\sim"'] & \D_\sigma \D_\sigma \chi_f(g_\eta^*\D_\eta M) \arrow[d, "\D_\sigma \comp_\chi"] \\
\D_\sigma g_\sigma^! \D_\sigma \chi_f(\D_\eta M) \arrow[d, "\D_\sigma g_\sigma^! \comp_\chi"']           & {\D_\sigma \chi_f(\D_\eta g_\eta^*\D_\eta M)(-1)[-1]} \arrow[d, "\sim"]          \\
{\D_\sigma g_\sigma^! \chi_f(\D_\eta\D_\eta M)(-1)[-1]} \arrow[r, "\D_\sigma \Ex^!"]                       & {\D_\sigma \chi_f( g_\eta^! \D_\eta\D_\eta M)(-1)[-1]}                          
\end{tikzcd}
\end{equation}
by replacing $M$ with $\D_\eta M$ in (\ref{eqn:FirstDiagOfLemCompAndPullback}) an applying $\D_\sigma(\_)(-1)[-1]. $ Hence it suffices to show that 
\[
\begin{tikzcd}
g_\sigma^* \chi_f(\D_\eta M) \arrow[d, "g_\sigma^* \comp_\chi"'] \arrow[r, "\sim"] & \D_\sigma \D_\sigma g_\sigma^* \chi_f(\D_\eta M) \arrow[d, "\sim"] \arrow[ldd, "(7)" description, phantom] \\
{g_\sigma^* \D_\sigma \chi_f(M)(-1)[-1]} \arrow[d, "\sim"']                        & \D_\sigma g_\sigma^! \D_\sigma \chi_f(\D_\eta M) \arrow[d, "\D_\sigma g_\sigma^! \comp_\chi"]     \\
{\D_\sigma g_\sigma^! \chi_f( M)(-1)[-1]} \arrow[r, "\sim"]                        & {\D_\sigma g_\sigma^! \chi_f(\D_\eta\D_\eta M)(-1)[-1]}                                          
\end{tikzcd}
\]
and
\[
\begin{tikzcd}
\D_\sigma \D_\sigma \chi_h(g_\eta^*\D_\eta M) \arrow[d, "\D_\sigma(\comp_\chi)"'] &  \chi_h(g_\eta^*\D_\eta M) \arrow[l, "\sim"'] \arrow[d, "\sim"] \arrow[ldd, "(8)" description, phantom] \\
{\D_\sigma \chi_h(\D_\eta g_\eta^*\D_\eta M)(-1)[-1]} \arrow[d, "\sim"']          &  \chi_h(\D_\eta g_\eta^! M) \arrow[d, "\comp_\chi"]                                            \\
{\D_\sigma \chi_h( g_\eta^! \D_\eta\D_\eta M)(-1)[-1]}                            & {\D_\sigma \chi_f( g_\eta^!  M)(-1)[-1]} \arrow[l, "\sim
"]                                            
\end{tikzcd}
\]
commute. Then (\ref{eqn:SecondDiagOfLemCompAndPullback}) is the outer square of
\[
\begin{tikzcd}
\bullet \arrow[d] \arrow[r] & \bullet \arrow[d] \arrow[rr] \arrow[ldd, "(7)" description, phantom] &  & \bullet \arrow[d] \arrow[lldd, "(\ref{eqn:DsigmaOfFirstDiagOfLemCompAndPullback})" description, phantom] & \bullet \arrow[d] \arrow[ldd, "(8)" description, phantom] \arrow[l] \\
\bullet \arrow[d]           & \bullet \arrow[d]                                           &  & \bullet \arrow[d]                                      & \bullet \arrow[d]                                          \\
\bullet \arrow[r]           & \bullet \arrow[rr]                                          &  & \bullet                                                & \bullet \arrow[l]                                         
\end{tikzcd}
\]
and therefore commutes.

The commutativity of (7) follows from the commutativity of
\[
\begin{tikzcd}
                                                                                      &                                                                                                                                             & \D_\sigma\D_\sigma g_\sigma^* \chi_f(\D_\eta M) \arrow[d, "\sim"]                            \\
g_\sigma^* \chi_f(\D_\eta M) \arrow[r, "\sim"] \arrow[d, "\sim"'] \arrow[rru, "\sim"] & g_\sigma^* \D_\sigma\D_\sigma\chi_f(\D_\eta M) \arrow[d, "g_\sigma^* \D_\sigma \comp_\chi"] \arrow[r, "\sim"] \arrow[ld, "(9)" description, phantom] &  \D_\sigma g_\sigma^!\D_\sigma\chi_f(\D_\eta M) \arrow[d, "\D_\sigma g_\sigma^! \comp_\chi"] \\
g_\sigma^* \chi_f(\D_\eta \D_\eta \D_\eta M) \arrow[r, "g_\sigma^* \comp_\chi"]       & g_\sigma^* \D_\sigma\chi_f(\D_\eta\D_\eta M)(-1)[-1] \arrow[r, "\sim"]                                                                              & \D_\sigma g_\sigma^! \chi_f(\D_\eta\D_\eta M)(-1)[-1].                                               
\end{tikzcd}
\]
Indeed commutativity of (9) is (\ref{eqn:PrelimClaimInCompBasechangeLemma}), the rest is obvious and one checks easily that the composition
\[
\begin{tikzcd}
g_\sigma^* \chi_f(\D_\eta M) \arrow[d, "\sim"']                                 &                                                                &                                               \\
g_\sigma^* \chi_f(\D_\eta \D_\eta \D_\eta M) \arrow[r, "g_\sigma^* \comp_\chi"] & g_\sigma^* \D_\sigma\chi_f(\D_\eta\D_\eta M)(-1)[-1] \arrow[r, "\sim"] & \D_\sigma g_\sigma^! \chi_f(\D_\eta\D_\eta M)(-1)[-1]
\end{tikzcd}
\]
is equivalent to
\[
\begin{tikzcd}
g_\sigma^* \chi_f(\D_\eta M) \arrow[d, "g_\sigma^* \comp_\chi"'] &                                                         \\
{g_\sigma^* \D_\sigma \chi_f(M)(-1)[-1]} \arrow[d, "\sim"']      &                                                         \\
{\D_\sigma g_\sigma^! \chi_f( M)(-1)[-1]} \arrow[r, "\sim"]      & {\D_\sigma g_\sigma^! \chi_f(\D_\eta\D_\eta M)(-1)[-1].}
\end{tikzcd}
\]

Finally commutativity of (8) follows from the commutativity of 
\[
\begin{tikzcd}
\D_\sigma \D_\sigma \chi_h(g_\eta^*\D_\eta M) \arrow[d, "\D_\sigma(\comp_\chi)"'] &  \chi_h(g_\eta^*\D_\eta M) \arrow[l, "\sim"'] \arrow[d, "\sim"] \arrow[ld, "(10)" description, phantom] \\
{\D_\sigma \chi_h(\D_\eta g_\eta^*\D_\eta M)(-1)[-1]} \arrow[d, "\sim"']          &  \chi_h( \D_\eta \D_\eta g_\eta^*\D_\eta M) \arrow[d, "\sim"] \arrow[l, "\comp_\chi"']          \\
{\D_\sigma \chi_h( g_\eta^! \D_\eta\D_\eta M)(-1)[-1]}                            &  \chi_h( \D_\eta \D_\eta \D_\eta g_\eta^* M). \arrow[l, "\comp_\chi"']                          
\end{tikzcd}
\]
Here (10) commutes again by (\ref{eqn:PrelimClaimInCompBasechangeLemma}) and one checks that the compositions 
\[
\begin{tikzcd}
                                                       &  \chi_h(g_\eta^*\D_\eta M) \arrow[d, "\sim"]                         \\
                                                       &  \chi_h( \D_\eta \D_\eta g_\eta^*\D_\eta M) \arrow[d, "\sim"]        \\
{\D_\sigma \chi_h( g_\eta^! \D_\eta\D_\eta M)(-1)[-1]} &  \chi_h( \D_\eta \D_\eta \D_\eta g_\eta^* M) \arrow[l, "\comp_\chi"]
\end{tikzcd}
\]
and
\[
\begin{tikzcd}
                                                       &  \chi_h(g_\eta^*\D_\eta M) \arrow[d, "\sim"]        \\
                                                       &  \chi_h(\D_\eta g_\eta^! M) \arrow[d, "\comp_\chi"] \\
{\D_\sigma \chi_h( g_\eta^! \D_\eta\D_\eta M)(-1)[-1]} & {\D_\sigma \chi_f( g_\eta^!  M)(-1)[-1]} \arrow[l, "\sim"] 
\end{tikzcd}
\]
are equivalent.

\end{proof}

\begin{remark} \label{rem:BigRemBeforeBigProp}
\begin{enumerate}
\item Let us abbreviate $\DA_{\et}(\_ ,\Lambda)$ by $\DA(\_)$ and denote the restriction of $\DA(\_)$ to $\Sch^{\qcqs}_{/ \sigma}$ and $\Sch^{\qcqs}_{/ \eta}$ by $\DA|_\sigma(\_)$ and $\DA|_\eta(\_)$ respectively. We write $\C_{S, \DA_\eta}$ and $\C_{S, \DA_\sigma}$ for the categories obtained by applying Construction \ref{noname:ConstrOfCS} to 
\[
\T(\_)= \DA_\eta(\_): \Sch^{\qcqs}_{/S} \overset{\_ \times_S \eta }\longrightarrow \Sch^{\qcqs}_{/ \eta} \overset{\DA|_\eta}\longrightarrow \Pr^{L,\otimes, \st}
\]
and 
\[
\T(\_)= \DA_\sigma(\_): \Sch^{\qcqs}_{/S} \overset{\_ \times_S \sigma }\longrightarrow \Sch^{\qcqs}_{/ \sigma} \overset{\DA|_\sigma}\longrightarrow \Pr^{L, \otimes, \st}
\]
respectively. This means for example that an object in $\C_{S, \DA_\eta}$ is a pair $(f:X \rightarrow S, M)$, where $f$ is of finite type over $S$ and $M$ is a object of $\DA_{\et}(X_\eta, \Lambda)$.

For any specialization system $\sp$ over $(S, \eta, \sigma)$ the functoriality results in \cite[3.1]{AyoubThesisII} imply that there is a well defined functor $\sp: \C_{S, \DA_\eta} \rightarrow \C_{S, \DA_\sigma}$ of bicategories given as follows:  A morphism $(C, \alpha): (f:X \rightarrow S,M) \rightarrow (g: Y \rightarrow S,N)$ is sent to $(C, \sp (\alpha)): (f: X \rightarrow S, \sp_f(M)) \rightarrow (g: Y\rightarrow S, \sp_g(N))$ where $\sp (\alpha)$ denotes the composition
\[
\overleftarrow{c_\sigma}^* \sp_f(M) \overset{\Ex^*}\longrightarrow  \sp_{f \circ \overleftarrow{c}} (\overleftarrow{c_\eta}^*M) \overset{\sp_{f \circ \overleftarrow{c}} \alpha}\longrightarrow \sp_{g \circ \overrightarrow{c}} (\overrightarrow{c_\eta }^! N) \overset{\Ex^!}\longrightarrow \overrightarrow{c_\sigma}^!\sp_g (N).
\]

\item Assume that the residue characteristic of $S$ is invertible in $\Lambda$. The equivalences
\[
\1 \overset{\sim}\longrightarrow \Psi_{\id}(\1)
\]
and 
\[
\Psi_f(M) \boxtimes \Psi_g(N) \overset{\sim}\longrightarrow \Psi_{f\times g}(M\boxtimes N)
\]
for any two objects $(f:X \rightarrow S,M), (g: Y \rightarrow S,N)$ in $\C_{S,\DA_{\eta}}$ (see Theorem \ref{thm:PropertiesOfEtaleMotivicNearby}) imply that the specialization system $\Psi$ gives rise to a symmetric monoidal functor $\Psi: \C_{S, \DA_\eta} \rightarrow \C_{S,\DA_\sigma}$. In particular this implies that the canonical map
\[
\Psi_f(\D_\eta(M)) \rightarrow \D_\sigma(\Psi_f(M))
\]
is an equivalence for any dualizable object $(X,M)$ of $\C_{S,\DA_\eta}$. Note that this observation combined with Lemma \ref{lem:constrOverFieldIsULA} gives an alternative proof of Theorem \ref{thm:PropertiesOfEtaleMotivicNearby}(5). 

\item Let $k: W \rightarrow S$ be either $j: \eta \rightarrow S$ or $i:\sigma \rightarrow S$. Consider the functor
\[
\DA_W(\_ ): \Sch^{\qcqs}_{/S} \overset{\_ \times_S W}\longrightarrow \Sch^{\qcqs}_{/W} \overset{\DA|_{W}}\longrightarrow \Pr^{L, \st, \otimes}
\]
and let $\C_{S, \DA_W}$ be the category obtained by applying the Construction \ref{noname:ConstrOfCS} to $\T(\_)=\DA_W(\_)$. Then we can define a functor of bicategories
\[
k^*: \C_{S, \DA} \rightarrow \C_{S, \DA_W}
\]
by sending a morphism $(C, \alpha): (X,M) \rightarrow (Y,N)$ to $(C, k^*\alpha): (X, k^*M) \rightarrow (Y, k^* N)$ where $k^* \alpha$ denotes (by slight abuse of notation) the composition
\[
\overleftarrow{c_W}^* k^* M \simeq k^* \overleftarrow{c}^*M \overset{k^* \alpha}\longrightarrow k^* \overrightarrow{c}^! N \overset{\Ex}\longrightarrow \overrightarrow{c_W}^!k^*  N.
\]

Consider two morphisms $(C,\alpha),(D, \beta): (X, M) \rightarrow (Y,N)$ and a 2-cell $(\Theta,h): (C,\alpha) \rightarrow (D, \beta)$ given by a proper morphism $h: C \rightarrow D$. Then the proper morphism $h_W: C_W \rightarrow D_W$ obtained by base change along $k$ gives rise to a 2-cell $(k^*\Theta,h): (C,k^*\alpha) \rightarrow (D, k^*\beta)$. It is straightforward to check that this gives a well defined functor. The equivalences 
\[
(X\times_S Y, k^*(M \boxtimes N)) \rightarrow (X \times_S Y, k^*M \boxtimes k^*N)
\]
and 
\[
(S , k^* \1) \rightarrow (S, \1)
\]
are natural in $(X,M)$ and $(Y,N)$ and equip $k^*: \C_{S, \DA} \rightarrow \C_{S, \DA_W}$ with the structure of a symmetric monoidal functor.

The functor $k^\diamondsuit: \C_{S, \DA} \rightarrow \C_{W, \DA|_W}$ as defined in \ref{noname:FuntcorOnC_SIndByAnyMap} is not to be confused with $k^*$ above. Note though that for $f:X \rightarrow S$ of finite type and $M$ in $\DA_{\et}(X_W, \Lambda)$ the object $(X_W, M)$ is strongly dualizable in $\C_{W, \DA|_W}$ if and only if $(X, M)$ is strongly dualizable in $\C_{S, \DA_W}$.

\item It is straightforward to check that the natural transformations
\[
i^* \longrightarrow  \chi_f j^* \longrightarrow \Upsilon_f j^* \longrightarrow \Psi_f j^*
\]
for all morphisms of schemes $f: X \rightarrow S$ give rise to natural transformations
\[
i^* \longrightarrow  \chi j^* \longrightarrow \Upsilon j^* \longrightarrow \Psi j^*
\]
in $\Fun(\C_{S, \DA}, \C_{S, \DA_{\sigma}})$. Moreover the composition $\alpha: i^* \rightarrow \Psi j^*$ is a symmetric monoidal natural transformation. 
\item For all $X$ in $\Sch^{\qcqs}_{/S}$ the functors
\[
i_* : \DA_{\et}(X_\sigma, \Lambda) \rightarrow \DA_{\et}(X, \Lambda)
\]
and
\[
j_* : \DA_{\et}(X_\eta, \Lambda) \rightarrow \DA_{\et}(X, \Lambda)
\]
are fully faithful, admit left exact left adjoints $i^*$ and $j^*$ which are jointly conservative and $j^*i_* \simeq 0$ (see \cite[ $\S$ 2]{Adeel6Functor}). Hence $\DA_{\et}(X, \Lambda)$ is a recollement of $\DA_{\et}(X_\sigma, \Lambda)$ and $\DA_{\et}(X_\eta, \Lambda)$ in the sense of \cite[A.8.1]{lurie2016higher}. By \cite[A.8.11]{lurie2016higher} there exists a left exact correspondence $p: \mathcal{M} \rightarrow \Delta^1$ such that
$\DA(X)$ is equivalent to $\Fun_{\Delta^1}(\Delta^1, \mathcal{M})$. Unravelling the construction of $p: \mathcal{M} \rightarrow \Delta^1$ we can observe that an object of $\Fun_{\Delta^1}(\Delta^1, \mathcal{M})$ is a triple
\[
(F_\sigma, F_\eta, \varphi: F_\sigma \rightarrow i^*j_* F_\eta),
\]
where $F_\sigma$ is in $\DA_{\et}(X_\sigma, \Lambda)$ and $F_\eta$ is in $\DA_{\et}(X_\eta, \Lambda)$. The datum of a morphism $\Fun_{\Delta^1}(\Delta^1, \mathcal{M})$ is precisely the datum of a morphism $a_\sigma: \F_\sigma \rightarrow G_\sigma$ in $\DA_{\et}(X_\sigma, \Lambda)$ and a morphism $a_ \eta: F_\eta \rightarrow G_\eta$ in $\DA_{\et}(X_\eta, \Lambda)$ such that the diagram
\[
\begin{tikzcd}
F_\sigma \arrow[r] \arrow[d, "a_\sigma"] & i^*j_*F_\eta \arrow[d, "i^*j_* a_\eta"] \\
G_\sigma \arrow[r]           & i^*j_* G_\eta,         
\end{tikzcd}
\]
in $\DA_{\et}(X_\sigma, \Lambda)$ commutes. In particular the pair  
\begin{align*}
&i^*: \DA_{\et}(X, \Lambda) \rightarrow \DA_{\et}(X_\sigma, \Lambda), \\
&j^*: \DA_{\et}(X, \Lambda) \rightarrow \DA_{\et}(X_\eta, \Lambda)
\end{align*}
jointly detects identities. This implies that the pair of symmetric monoidal functors
\begin{align*}
&i^*: \C_{S, \DA} \rightarrow \C_{S, \DA_\sigma}, \\
&j^*: \C_{S, \DA} \rightarrow \C_{S, \DA_\eta} 
\end{align*}
jointly detect identities. In other words, map $(C, \alpha): (X,M) \rightarrow (X,M)$ is an identity in $\C_{S, \DA}$ if and only if $i^*(C, \alpha)$ as well as $j^*(C, \alpha)$ are identities of $(X, i^*M)$ in $\C_{S, \DA_\sigma}$ and $(X, j^*M)$ in $\C_{S, \DA_\eta}$ respectively. 
\end{enumerate}
\end{remark}

\begin{prop}\label{prop:InQsetupCanIsoImplULA}
Let $f: X \rightarrow S$ be a morphism of finite type and $M$ in $\DA_{\et}^{\cons}(X, \Lambda)$. Assume that $\Lambda$ is a $\Q$-algebra. If the canonical map
\[
\can_M: i^*M \longrightarrow \Psi_f(j^*M)
\]
is an equivalence, then $M$ is $f-$ULA. 
\end{prop} 

\begin{proof} We want to prove this by constructing an explicit duality datum for the object $(X,M):=(f:X \rightarrow S,M)$ in $\C_{S,\DA}$. We refer to Appendix \ref{app:A} for details about dualizability in a bicategory.

From the characterisation of universal local acyclicity in Proposition \ref{prop:equivcharofULA} we can see that the question whether $M$ is $f-$ULA is Zariski local on $X$. Since the map $\can_M$ is also compatible with inverse image along open immersions we may assume that $f:X \rightarrow S$ is separated. In particular the diagonal morphism $\Delta: X \rightarrow X \times_S X$ is a closed immersion. 

Let us write
\begin{equation} \label{eqn:EpsilonM}
(\Delta, \varepsilon_{M}): (X \times_S X, \D_S(M) \boxtimes M) \longrightarrow (S, \1) 
\end{equation}
for the maps transpose to $(\Delta,\id_{\D_S(M)}): (\D_S(M),M) \rightarrow  (\D_S(X),M)$ in $\C_{S, \DA}$ and set 
\[
(\Delta, \varepsilon_{j^*M}): (X \times_S X, \D_\eta(j^*M)) \boxtimes  j^*M ) \simeq (X \times_S X, j^*(\D_S(M)) \boxtimes M) ) \overset{j^* \varepsilon_M}\longrightarrow (S, \1). 
\]

Note that $j^*M$ is $f_\eta$-ULA by Lemma \ref{lem:constrOverFieldIsULA}. Thus by Lemma \ref{lem:NewRigidDualDataForDualizableObject} we see that $(\Delta, \varepsilon_{j^*M})$ and 
\begin{align*}
(\Delta, \eta_{j^*M}): (S, \1)  &\overset{\varepsilon_{j^*M}^t}\longrightarrow (X \times_S X,\D_\eta( \D_\eta(j^*M) \boxtimes j^*M))) \\
& \overset{\sim}\longrightarrow (X \times_S X,j^*M \boxtimes \D_\eta(j^*M))
\end{align*}
are a duality datum for $(X,j^*M)$ in $\C_{S, \DA_\eta}$. Let us write
$\Delta_\eta: X_\eta \rightarrow X_\eta \times_\eta X_\eta$
and $
\Delta_\sigma: X_\sigma \rightarrow X_\sigma \times_\sigma X_\sigma$ for the diagonal morphisms. Moreover we write 
\[
\eta_{j^*M}^t: \Delta_{\eta*}f_\eta^* \1 \rightarrow j^*M \boxtimes \D_\eta(j^*M)
\]
for the transpose of
\[
\eta_{j^*M}: f_\eta^* \1 \rightarrow \Delta_\eta^!(j^*M \boxtimes \D_\eta(j^*M)).
\]
Consider the diagram
\begin{equation} \label{equ:Big diagramForUnit}
\begin{tikzcd}
\Delta_{\sigma *} i^*f^* \1 \arrow[r] \arrow[d, "\simeq"'] & \Delta_{\sigma *} \chi_{f}( j^* f^*\1) \arrow[r] \arrow[d, "\simeq"] & \Delta_{\sigma *} \Psi_{f}(j^*f^* \1) \arrow[d, "\simeq"] \\
{i^* \Delta_*f^* \1} \arrow[dd, dashed, "\eta^t_i"'] \arrow[r] \arrow[rdd, phantom, "(1)"] & {\chi_{f \times f}(\Delta_{\eta *} j^* f^* \1)} \arrow[d, "\chi_{f \times f} (\eta_{j^*M}^t)"] \arrow[r]          & {\Psi_{f \times f }(\Delta_{\eta *} j^* f^* \1)} \arrow[d, "\Psi_{f \times f} (\eta^t_{j^*M})"]         \\
                                                                     & \chi_{f \times f}( j^*M \boxtimes \D(j^*M)) \arrow[d, "\simeq"] \arrow[r]                 & \Psi_{f \times f}( j^*M \boxtimes \D(j^*M)) \arrow[d, "\simeq"]                 \\
i^*(M \boxtimes \D(M)) \arrow[r]                                     & \chi_{f \times f}j^* ( M \boxtimes \D(M)) \arrow[r]                                       & \Psi_{f \times f}j^* ( M \boxtimes \D(M))                                      
\end{tikzcd}        
\end{equation}
where the top horizontal composition is $\Delta_*\can_{ i^*f^* \1}$ and the bottom horizontal composition is $\can_{M \boxtimes \D(M)}$. It is straightforward to check that the solid diagram commutes. By Lemma \ref{lem:a_MisoImpliesa_MxDMiso} the canonical map $\can_{M \boxtimes \D(M)}$ is an equivalence and hence there is a unique dotted map which we call (slightly suggestive) $\eta_i^t$ making the outer diagram commute.

We claim that the square (1) commutes. Let us explain how we can deduce the Proposition from this. In order to construct a duality datum for $(X,M)$ in $\C_{S,\DA}$ we want to construct a map
\[
\eta_M : f^*\1 \rightarrow \Delta^!( M \boxtimes \D(M))  
\]
or equivalently a map
\begin{equation} \label{equ:EtaOfMTranspose}
\eta_M^t : \Delta_* f^*\1 \rightarrow  M \boxtimes \D(M).  
\end{equation}
This gives rise to a map $(\Delta, \eta_M): (S, \1) \rightarrow (X \times X, M \boxtimes  \D(M))$ in $\C_{S,\DA}$. Consider $\DA_{\et}(X \times_S X, \Lambda)$ as the recollement of $\DA_{\et}(X_ \eta \times_\eta X_\eta, \Lambda)$ and $\DA_{\et}(X_\sigma \times_\sigma X_\sigma, \Lambda)$ via the left exact functor $\chi_{f\times f} = i^*j_*$. Then constructing a map (\ref{equ:EtaOfMTranspose}) is equivalent to constructing maps $\eta_i^t$ and $\eta_j^t$ which make the diagram
\begin{equation} \label{equ:EtaOfMAsREcollement}
\begin{tikzcd}
i^*\Delta_* f^*\1 \arrow[d, "\unit"] \arrow[r, "\eta_i^t"]     & i^*(M \boxtimes \D_S(M)) \arrow[d,"\unit"] \\
i^*j_*j^* \Delta_* f^* \1 \arrow[r, "i^*j_*(\eta_j^t)"'] & i^*j_*j^*(M \boxtimes \D_S(M))    
\end{tikzcd}
\end{equation}
commute. Hence commutativity of (1) gives rise to a map (\ref{equ:EtaOfMTranspose}). We want to show that the pair $(\Delta, \eta_M)$ and $(\Delta, \varepsilon_M)$ is a duality datum of $(X,M)$ in $\C_{S,\DA}$.

As observed in Remark \ref{rem:BigRemBeforeBigProp}(5) the functors $i^*: \C_{S, \DA} \rightarrow \C_{S, \DA_\sigma}$ and $j^*: \C_{S, \DA} \rightarrow \C_{S, \DA_\eta}$ jointly detect equivalences. Hence by Corollary \ref{cor:2.ConservativeFamilyOfFunctorsDetectsDualizability} it suffices to check that the two pairs $((\Delta, \eta_M)_{i^*}$, $(\Delta,\varepsilon_M)_{i^*})$ and $((\Delta,\eta_M)_{j^*}$, $(\Delta,\varepsilon_M)_{j^*})$ (using the notations of \ref{noname:EtaAndEpsilonIndBySymmMonFunctor}) give rise to duality data in $\C_{S, \DA_\sigma}$ and $\C_{S, \DA_\eta}$ respectively. Since $j:X_\eta \rightarrow X$ is an open immersion the canonical map
\[
j^*\D_S(M) \rightarrow \D_\eta(j^*M)
\] 
obtained as the transpose of
\[
j^* \D_S(M) \otimes j^*M \overset{\simeq}\longleftarrow j^*(\D_S(M) \otimes M) \overset{j^*(\id_{\D(M)}^t)}\longrightarrow j^*f^! \1 \longrightarrow f_\eta^! j^*\1. 
\]
is an equivalence. We know that the pair $((\Delta,\eta_{j^*M})$, $(\Delta,\varepsilon_{j^*M}))$ is a duality datum for $(X, j^*M)$ in $\C_{S, \DA_\eta}$. We claim that $(\Delta,\eta_M)_{j^*}$ is equivalent to $(\Delta,\eta_{j^*M})$ and $(\Delta,\varepsilon_M)_{j^*}$ is equivalent to $(\Delta,\varepsilon_{j^*M})$, which implies that $((\Delta,\eta_M)_{j^*}$, $(\Delta,\varepsilon_M)_{j^*})$ is a duality datum of $(X, j^*M)$ in $\C_{S, \DA_\eta}$. The first claim is equivalent to $\eta_{j^*M}^t$ being equivalent to 
\[
 \Delta_* f_\eta^* j^* \1 \simeq j^*\Delta_*f^* \1 \overset{j^* \eta_M^t}\longrightarrow j^*(M \boxtimes \D(M)) \overset{\simeq}\longrightarrow j^*M \boxtimes \D(j^*M)
\]
which is true by the construction of $\eta_M$ via the commutative square (1) in (\ref{equ:Big diagramForUnit}) and the second claim is true by the construction of $\varepsilon_{j^*M}$ above. 

It remains to show that $((\Delta, \eta_M)_{i^*}$, $(\Delta, \varepsilon_M)_{i^*})$ is a duality datum of $(X, i^*M)$. Since the natural transformation $\alpha: i^* \rightarrow \Psi j^*$ is symmetric monoidal there are commutative diagrams
\begin{equation} \label{eqn:mapEta_i^*toEta_psij^*}
\begin{tikzcd}
{(S, \1)} \arrow[d, "{(\Delta, (\eta_M)_{i^*})}"'] \arrow[rr, "\id"]                                      &  & {(S, \1)} \arrow[d, "{(\Delta, (\eta_M)_{\Psi j^*})}"]      \\
{(X \times_S X, i^*M \boxtimes i^*\D_S(M))} \arrow[rr, "{(\Delta, \alpha_M \boxtimes \alpha_{\D_S(M)})}"] &  & {(X \times_S X, \Psi_f(j^*M) \boxtimes \Psi_f(j^*\D_S(M)))}
\end{tikzcd}
\end{equation}
and
\begin{equation} \label{eqn:epsiloni*toEpsilonPsij*}
\begin{tikzcd}
{(X \times_S X, i^*M \boxtimes i^*\D_S(M))} \arrow[rr, "{(\Delta, \alpha_M \boxtimes \alpha_{\D_S(M)})}"] \arrow[d, "{(\Delta, (\varepsilon_M)_{i^*})}"'] &  & {(X \times_S X, \Psi_f(j^*M) \boxtimes \Psi_f(j^*\D_S(M)))} \arrow[d, "{(\Delta, (\varepsilon_M)_{\Psi j^*})}"] \\
{(S, \1)} \arrow[rr, "\id"]                                                                                                                                    &  & {(S, \1)}                                                                                                      
\end{tikzcd}
\end{equation}
in $\C_{S, \DA_\sigma}$. Note that $(\eta_M)_{\Psi j^*}$ and $(\varepsilon_M)_{\Psi j^*}$ is a duality datum for $(X, \Psi_f j^* M)$ in $\C_{S, \DA_\sigma}$: We saw above that $(\eta_M)_{j^*}$ and $(\varepsilon_M)_{j^*}$ is a duality datum for $(X,j^*M)$ in $\C_{S, \DA_\eta}$ and hence by Corollary \ref{cor:1.ImageOfDualizableObjectisDualizablw2} the pair $((\eta_M)_{j^*})_{\Psi}\simeq(\eta_M)_{\Psi j^*}$ and $((\varepsilon_M)_{j^*})_\Psi \simeq (\varepsilon_M)_{\Psi j^*}$ is a duality datum for $\Psi_f j^*M$. By assumption and Lemma \ref{lem:Beta_MequivImpliesBeta_DMequiv} the map $\alpha_M \boxtimes \alpha_{\D_S(M)}$ is an equivalence. Hence we may apply Lemma \ref{lem:IsoOfDualObjectsAndData} to the diagrams (\ref{eqn:mapEta_i^*toEta_psij^*}) and  (\ref{eqn:epsiloni*toEpsilonPsij*}). This shows that the pair $((\eta_M)_{i^*}$, $(\varepsilon_M)_{i^*})$ is a duality datum of $(X,i^*M)$.

Finally let us proof that the square $(1)$ in (\ref{equ:Big diagramForUnit}) commutes. Note that using the pasting property for exchange maps (see \ref{noname:ExchangeMaps}) one observes that
\begin{equation} \label{eqn:SmallBasechangeDiagramForProp}
\begin{tikzcd}
\Delta_{\sigma *} f_\sigma^* i^* \1 \arrow[dd, "\sim"] \arrow[r, "\unit"] & \Delta_{\sigma *} f_\sigma^*\chi_{\id} (j^* \1) \arrow[d, "\Ex"] \\
                                                                          & \Delta_{\sigma *} \chi_f (f_\eta^*j^* \1) \arrow[d, "\sim"]    \\
\Delta_{\sigma *} i^* f^* \1 \arrow[r, "\unit"]                           & \Delta_{\sigma *} \chi_f (j^* f^* \1)                         
\end{tikzcd}
\end{equation}
commutes. The maps 
\[
\alpha_\1: i^*\1 \rightarrow \chi_{\id} (j^* \1) \rightarrow \Psi_{\id}(j^* \1)
\]
and
\[
\can_{M \boxtimes \D_S (M)}: i^*(M \boxtimes \Dual_S(M)) \rightarrow \chi_{f\times f}( j^* (M \boxtimes \Dual_S(M))) \rightarrow \Psi_{f\times f}( j^* (M \boxtimes \Dual_S(M))) 
\]
are equivalences by Theorem \ref{thm:PropertiesOfEtaleMotivicNearby} (1) and Lemma \ref{lem:a_MisoImpliesa_MxDMiso} respectively. This implies by Lemma \ref{lem:AlphaEquImplBetaEqu} that $\Upsilon_{\id}(j^*\1) \simeq \Psi_{\id}(j^*\1)$ and $\Upsilon_{f \times f}(j^*(M \boxtimes \D_S(M))) \simeq \Psi_{f \times f}(j^*(M \boxtimes \D_S(M)))$. Moreover as explained in \ref{noname:AlphaEquivGivesSection} the maps $\beta_{\1}$ and $\beta_{M \boxtimes \D_SM}$ split the respective monodromy sequences which give rise to direct sum decompositions
\[
\chi_{\id}(j^* \1) \simeq \Upsilon_{\id}(j^*\1) \oplus \Upsilon_{\id}(j^*\1)(-1)[-1] ( \simeq \1 \oplus \1 (-1)[-1])
\]
and
\[
\chi_{f \times f}(j^*(M \boxtimes \D_S(M))) \simeq \Upsilon_{f \times f}(j^*(M \boxtimes \D_S(M)))  \oplus \Upsilon_{f \times f}(j^*(M \boxtimes \D_S(M))) (-1)[-1].
\]
Using the commutativity of (\ref{eqn:SmallBasechangeDiagramForProp}) we can observe that (1) commutes precisely if the composition 
\begin{equation} \label{eqn:CompositionNeedsToBeDiagMAtrix}
\Delta_{\sigma *} f_\sigma^*\chi_{\id} (j^* \1) \overset{\Ex^*}\longrightarrow \Delta_{\sigma *} \chi_f (f_\eta^*j^* \1) 
 \overset{\sim}\longrightarrow \chi_{f \times f}(\Delta_{\eta* } j^* f^* \1) 
 \overset{\chi_{f \times f}(\eta_{j^*M}^t)}\longrightarrow \chi_{f \times f}(j^*M \boxtimes \D_\eta(j^*M)) 
\end{equation}
is given by a diagonal matrix with respect to the direct sum decompositions above. Equivalently we may show that the composition 
\begin{equation} \label{eqn:CompositionNeedsToBeDiagMAtrixTransposed}
f_\sigma^* \chi_{\id}(\1) \overset{\Ex^*}\longrightarrow \chi_f(f_\eta^* \1) \overset{\chi_f(\eta_{j^*M})}\longrightarrow \chi_f(\Delta_\eta^! (j^*M \boxtimes \D_\eta(j^*M))) 
\overset{\Ex^!}\longrightarrow \Delta_\sigma^! \chi_{f\times f}(j^*M \boxtimes \D_\eta(j^*M))
\end{equation}
is given by a diagonal matrix with respect to the direct sum decompositions above. 

Recall the map $\varepsilon_M = \id^t: \D_S(M) \otimes M \rightarrow f^! \1 $ (\ref{eqn:EpsilonM}) from the very beginning of the proof. The diagram
\begin{equation} \label{eqn:CounitDiagram}
\begin{tikzcd}
\Delta_\sigma^*i^*(\D_S(M)\boxtimes M)) \arrow[d] \arrow[r, "\sim"]           & i^* \Delta^*(\D_S(M)\boxtimes M)) \arrow[d] \arrow[r, "i^* \varepsilon_M"]                & i^* f^! \1 \arrow[r, "\Ex"] \arrow[d]            & f_\sigma^! i^* \1 \arrow[d]             \\
\Delta_\sigma^*\chi_f (j^*(\D_S(M)\boxtimes M))) \arrow[d] \arrow[r, "\Ex^*"] & \chi_f (j^* \Delta^*(\D_S(M)\boxtimes M))) \arrow[d] \arrow[r, "\chi_fj^* \varepsilon_M"] & \chi_f (j^* f^! \1) \arrow[r, "\Ex^!"] \arrow[d] & f_\sigma^! \chi_{\id} (j^*\1) \arrow[d] \\
\Delta_\sigma^*\Psi_f (j^*(\D_S(M)\boxtimes M))) \arrow[r, "\Ex^*"]           & \Psi_f (j^*\Delta^*(\D_S(M)\boxtimes M))) \arrow[r, "\Psi_fj^* \varepsilon_M"]            & \Psi_f (j^* f^! \1) \arrow[r, "\Ex^!"]           & f_\sigma^! \Psi_{\id} (j^*\1)          
\end{tikzcd}
\end{equation}
is commutative. Let us denote the left to right composition of the middle row by $\theta$. Note that the left vertical composition is $\Delta_\sigma^* \can_{\D_S(M) \boxtimes M}$ and the right vertical composition is $f_\sigma^! \can_\1$. In particular they are both equivalences and induce direct sum decompositions
\[
\Delta_\sigma^* \chi_f(\D_\eta(j^*M) \boxtimes j^*M) \simeq \Delta_\sigma^* \Upsilon_f(\D_\eta(j^*M) \boxtimes j^*M) \oplus \Delta_\sigma^* \Upsilon_f(\D_\eta(j^*M) \boxtimes j^*M)(-1)[-1]
\]
and
\[
f_\sigma^! \chi_{\id} (j^*\1) \simeq f_\sigma^! \Upsilon_{\id} (j^*\1) \oplus f_\sigma^! \Upsilon_{\id} (j^*\1) (-1)[-1].
\]
Moreover commutativity of (\ref{eqn:CounitDiagram}) implies that the composition
\begin{equation} \label{eqn:CounitCompIsDiagonalMatrix}
\Delta_\sigma^* \chi_f(\D_\eta(j^*M) \boxtimes j^*M) \overset{\sim}\longrightarrow \Delta_\sigma^*\chi_f (j^*(\D_S(M)\boxtimes M)))
\overset{\theta}\longrightarrow f_\sigma^! \chi_{\id} (j^*\1)
\end{equation}
is a diagonal matrix with respect to these decompositions. Note that the composition
\[
\Delta_\sigma^* \chi_f(\D_\eta(j^*M) \boxtimes j^*M) \longrightarrow  \chi_f(\Delta_\eta^*(\D_\eta(j^*M) \boxtimes j^*M)) \overset{\chi_f(\varepsilon_{j^*M})}\longrightarrow \chi_f ( f_\eta^! \1) \overset{\Ex^!}\longrightarrow f_\sigma ^!\chi_{\id}( \1)
\]
is equivalent to the composition (\ref{eqn:CounitCompIsDiagonalMatrix}).

 We claim that the Proposition follows (i.e. (\ref{eqn:CompositionNeedsToBeDiagMAtrixTransposed}) is a diagonal matrix) if we can show that the diagram
\begin{equation} \label{eqn:KeyPropDiagramComparingDuality}
\begin{tikzcd}
f_\sigma^* \chi_{\id}(\1) \arrow[r] \arrow[d, "\sim"]                & \Delta_\sigma^!\chi_f(j^*M \boxtimes \D_\eta(j^*M)) \arrow[d, "\sim"]                               \\
f_\sigma^* \chi_{\id}(\D_\eta(\1)) \arrow[d, "f_\sigma^*\comp_\chi"] & \Delta_\sigma^!\chi_f(\D_\eta(\D_\eta(j^*M) \boxtimes j^*M)) \arrow[d, "\Delta_\sigma^!\comp_\chi"] \\
{f_\sigma^* \D_\sigma \chi_{\id}(\1) (-1)[-1]} \arrow[d, "\sim"]     & {\Delta_\sigma^! \D_\sigma ( \chi_{f}(\D_\eta(j^*M) \boxtimes j^*M )) (-1)[-1]} \arrow[d, "\sim"]   \\
{\D_\sigma (f_\sigma^! \chi_{\id}(\1)) (-1)[-1]} \arrow[r]           & {\D_\sigma ( \Delta_\sigma^*\chi_{f}(\D_\eta(j^*M) \boxtimes j^*M )) (-1)[-1]}                     
\end{tikzcd}
\end{equation}
commutes. Here the top horizontal map is the composition (\ref{eqn:CompositionNeedsToBeDiagMAtrixTransposed}) and the bottom horizontal map is $\D_\sigma(\_)(-1)[-1]$ applied to the composition (\ref{eqn:CounitCompIsDiagonalMatrix}). Indeed since all relevant direct sum decompositions are induced by the equivalences $\alpha_\1$, $\alpha_{M \boxtimes \D_S(M)}$ and $\alpha_{\D_S(M) \boxtimes M}$ it is easy to check that the left and right vertical maps are of the form
\[
\begin{psmallmatrix} 0 & *  \\ * & 0  \end{psmallmatrix}
\]
by Lemma \ref{lem:DChi=ChiDCompatibleWithSplitting} while the bottom horizontal map  (\ref{eqn:CounitCompIsDiagonalMatrix}) is of the form
\[
\begin{psmallmatrix} * & 0  \\ 0 & *  \end{psmallmatrix}
\]
with respect to these decompositions. This and the fact that the vertical compositions are equivalences implies that the top horizontal map of (\ref{eqn:KeyPropDiagramComparingDuality}) is a diagonal matrix as desired. 

In order to show that (\ref{eqn:KeyPropDiagramComparingDuality}) commutes we divide it in several small diagrams. First note that 
\[
\begin{tikzcd}
f_\sigma^* \chi_{\id}(\1) \arrow[d, "\sim"'] \arrow[r, "\Ex^*"] \arrow[d] \arrow[rd, "(4)" description, phantom] & \chi_f(f_\eta^* \1) \arrow[d, "\sim"] \\
f_\sigma^* \chi_{\id}(\D_\eta(\1)) \arrow[r, "\Ex^*"']                                                  & \chi_f(f_\eta^* \D_\eta(\1))         
\end{tikzcd}
\]
and
\[
\begin{tikzcd}
\chi_f(\Delta_\eta^!(j^*M \boxtimes \D_\eta(j^*M))) \arrow[r, "\Ex^!"] \arrow[d, "\sim"'] \arrow[r] \arrow[rd, "(5)" description, phantom] & \Delta_\sigma^!\chi_f(j^*M \boxtimes \D_\eta(j^*M)) \arrow[d, "\sim"] \\
\chi_f(\Delta_\eta^!(\D_\eta(\D_\eta(j^*M) \boxtimes j^*M))) \arrow[r, "\Ex^!"']                                                  & \Delta_\sigma^!\chi_f(\D_\eta(\D_\eta(j^*M) \boxtimes j^*M))         
\end{tikzcd}
\]
clearly commute. Essentially by construction of $\eta_{j^*M}$ (right at the beginning of the proof) the diagram
\[
\begin{tikzcd}
\chi_f(f_\eta^* \1) \arrow[rr, "\chi_f(\eta_{j^*M})"] \arrow[d, "\sim"'] \arrow[rrdd, "(6)", phantom] &  & \chi_f(\Delta_\eta^!(j^*M \boxtimes \D_\eta(j^*M))) \arrow[d, "\sim"]          \\
\chi_f(f_\eta^* \D_\eta(\1)) \arrow[d, "\sim"']                                                          &  & \chi_f(\Delta_\eta^!(\D_\eta(\D_\eta(j^*M) \boxtimes j^*M))) \arrow[d, "\sim"] \\
\chi_f( \D_\eta(f_\eta^!\1)) \arrow[rr, "\chi_f(\D_\eta(\varepsilon_{j^*M}))"]                           &  & \chi_f((\D_\eta(\Delta_\eta^*\D_\eta(j^*M) \boxtimes j^*M)))                  
\end{tikzcd}
\]
commutes. Next note that
\[
\begin{tikzcd}
f_\sigma^* \chi_{\id}(\D_\eta(\1)) \arrow[d, "f_\sigma^*\comp_\chi"] \arrow[r] \arrow[rdd, "(7)" description, phantom] & \chi_f(f_\eta^* \D_\eta(\1)) \arrow[d, "\sim"]          \\
{f_\sigma^* \D_\sigma \chi_{\id}(\1) (-1)[-1]} \arrow[d, "\sim"]                                              & \chi_{f}(\D_\eta (f_\eta^!\1))  \arrow[d, "\comp_\chi"] \\
{\D_\sigma (f_\sigma^! \chi_{\id}(\1)) (-1)[-1]} \arrow[r]                                                    & {\D_\sigma ( \chi_{f}(f_\eta^!\1)) (-1)[-1]}           
\end{tikzcd}
\]
and
\[
\begin{tikzcd}
\chi_f(\Delta_\eta^!(\D_\eta(\D_\eta(j^*M) \boxtimes j^*M))) \arrow[r] \arrow[d, "\sim"] \arrow[rdd, "(8)" description, phantom] & \Delta_\sigma^!\chi_f(\D_\eta(\D_\eta(j^*M) \boxtimes j^*M)) \arrow[d, "\Delta_\sigma^!\comp_\chi"] \\
\chi_f(\D_\eta (\Delta_\eta^*(\D_\eta(j^*M) \boxtimes j^*M))) \arrow[d, "\comp_\chi"]                                   & {\Delta_\sigma^! \D_\sigma ( \chi_{f}(\D_\eta(j^*M) \boxtimes j^*M )) (-1)[-1]} \arrow[d, "\sim"]   \\
{\D_\sigma ( \chi_{f}(\Delta_\eta^*(\D_\eta(j^*M) \boxtimes j^*M ))) (-1)[-1]} \arrow[r]                                & {\D_\sigma ( \Delta_\sigma^*\chi_{f}(\D_\eta(j^*M) \boxtimes j^*M )) (-1)[-1]}                     
\end{tikzcd}
\]
commute by Lemma \ref{lem:CompAndPullback}. Finally we claim that
\[
\begin{tikzcd}
\chi_{f}(\D_\eta (f_\eta^!\1))  \arrow[d, "\comp_\chi"'] \arrow[rr, "\chi_f(\D_\eta(\varepsilon_{j^*M}))"] \arrow[rrd, "(9)" description, phantom] &  & \chi_f(\D_\eta (\Delta_\eta^*(\D_\eta(j^*M) \boxtimes j^*M))) \arrow[d, "\comp_\chi"] \\
{\D_\sigma ( \chi_{f}(f_\eta^!\1)) (-1)[-1]} \arrow[rr, "\D_\sigma \chi_f(\varepsilon_{j^*M})(-1){[-1]}"]                                          &  & {\D_\sigma ( \chi_{f}(\Delta_\eta^*(\D_\eta(j^*M) \boxtimes j^*M ))) (-1)[-1]}       
\end{tikzcd}
\]
commutes. Consider its transposed diagram
\[
\begin{tikzcd}
\chi_{f}(\D_\eta (f_\eta^!\1)) \otimes  \chi_{f}(\Delta_\eta^*(\D_\eta(j^*M) \boxtimes j^*M ))  \arrow[d, "\chi_f(\D_\eta \varepsilon_{j^*M}) \otimes {\id}"'] \arrow[r, "\id \otimes \chi_f(\varepsilon_{j^*M})"] \arrow[rd, "(9)^t" description, phantom] & \chi_{f}(\D_\eta (f_\eta^!\1)) \otimes  \chi_{f}(f_\eta^! \1)  \arrow[d] \\
\chi_f(\D_\eta (\Delta_\eta^*(\D_\eta(j^*M) \boxtimes j^*M))) \otimes  \chi_{f}(\Delta_\eta^*(\D_\eta(j^*M) \boxtimes j^*M ))  \arrow[r]                                                                                                          & {f^!\1 (-1)[-1]}                                                        
\end{tikzcd}
\]
Here the right vertical and bottom horizontal map are defined as in \ref{noname:MonodromySequenceAndComparisonDualitySpecSystems}. It follows straight from this construction that in order to show that $(9)^t$ commutes it suffices to show that
\[
\begin{tikzcd}
\chi_{f}(\D_\eta (f_\eta^!\1) \otimes \Delta_\eta^*(\D_\eta(j^*M) \boxtimes j^*M ))  \arrow[d, "\chi_f(\D_\eta \varepsilon_{j^*M} \otimes {\id})"'] \arrow[r, "\chi_f(\id \otimes \varepsilon_{j^*M})"] & \chi_{f}(\D_\eta (f_\eta^!\1) \otimes  f_\eta^! \1)  \arrow[d, "\chi_f \id^t"] \\
\chi_f(\D_\eta (\Delta_\eta^*(\D_\eta(j^*M) \boxtimes j^*M) \otimes \Delta_\eta^*(\D_\eta(j^*M) \boxtimes j^*M ))  \arrow[r, "\chi_f \id^t"]                                                              & \chi_f(f^!\1)                                                              
\end{tikzcd}
\]
commutes. This follows from the fact that the transposed diagram of
\[
\begin{tikzcd}
\D_\eta (f_\eta^!\1) \otimes \Delta_\eta^*(\D_\eta(j^*M) \boxtimes j^*M )  \arrow[d, "\D_\eta \varepsilon_{j^*M} \otimes {\id}"'] \arrow[r, "\id \otimes \varepsilon_{j^*M}"] & \D_\eta (f_\eta^!\1) \otimes  f_\eta^! \1  \arrow[d, "\id^t"] \\
\D_\eta (\Delta_\eta^*(\D_\eta(j^*M) \boxtimes j^*M)) \otimes \Delta_\eta^*(\D_\eta(j^*M) \boxtimes j^*M )  \arrow[r, "\id^t"]                                                 & f^!\1                                                      
\end{tikzcd}
\]
is simply
\[
\begin{tikzcd}
\D_\eta  f^! \1 \arrow[d, "\id"'] \arrow[r, "\D_\eta \varepsilon_{j^*M}"] & \D_\eta (\Delta_\eta^*(\D_\eta(j^*M) \boxtimes j^*M))  \arrow[d, "\id"] \\
\D_\eta  f^! \1 \arrow[r, "\D_\eta \varepsilon_{j^*M}"]                  & \D_\eta (\Delta_\eta^*(\D_\eta(j^*M) \boxtimes j^*M)).                  
\end{tikzcd}
\] 

Since (\ref{eqn:KeyPropDiagramComparingDuality}) is the outer diagram of
\[
\begin{tikzcd}
\bullet \arrow[d] \arrow[r] \arrow[rd, "(4)" description, phantom]   & \bullet \arrow[d] \arrow[r] \arrow[rdd, "(6)" description, phantom] & \bullet \arrow[d] \arrow[r] \arrow[rd, "(5)" description, phantom]  & \bullet \arrow[d]  \\
\bullet \arrow[dd] \arrow[r] \arrow[rdd, "(7)" description, phantom] & \bullet \arrow[d]                                          & \bullet \arrow[r] \arrow[rdd, "(8)" description, phantom] \arrow[d] & \bullet \arrow[dd] \\
                                                            & \bullet \arrow[d] \arrow[r] \arrow[rd, "(9)" description, phantom]  & \bullet \arrow[d]                                          &                    \\
\bullet \arrow[r]                                           & \bullet \arrow[r]                                          & \bullet \arrow[r]                                          & \bullet           
\end{tikzcd}
\] 
this finishes the proof.
\end{proof}

\begin{prop} \label{prop:ULAImpliesCanIsEquiv}
Let $f:X \rightarrow S$ be a morphism of finite type and $M$ in $\DA_{\et}^{\cons}(X, \Lambda)$. If $M$ is $f$-ULA, then each of the canonical maps
\[
i^*M \longrightarrow \Upsilon_f(j^*M) \longrightarrow \Psi^{\tame}_f(j^*M) \longrightarrow \Psi_f(j^*M)
\]
is an equivalence. 
\end{prop}

\begin{proof}
This is essentially \cite[4.1.7.]{JinYangMotNearbyPosChar} using the equivalent characterizations of universal local acyclicity in Proposition \ref{thm:EquivCharOfULA}. There is a chain of equivalences 
\begin{align*}
\Upsilon_f(j^*M) \simeq \; &\colim_{\Delta^{\op}} i^*j_* \Hom(f_\eta^* \pi^* \mathscr{A}_S, j^*M) \\
\overset{(1)}\simeq &\colim_{\Delta^{\op}} i^* j_* ( j^*M \otimes  f_\eta^* \Hom( \pi^*  \mathscr{A}_S, \1)) \\
\overset{(2)}\simeq &\colim_{\Delta^{\op}} i^* ( M \otimes f^* j_* \Hom( \pi^* \mathscr{A}_S, \1)) \\
\simeq \; & i^* M \otimes  f^* \colim_{\Delta^{\op}} i^*j_* \Hom( \pi^* \mathscr{A}_S, \1)) \\
\simeq \; &i^*M \otimes f^* \Upsilon_{\id}(\1) \\
 \overset{(3)}\simeq &i^* M.
\end{align*} 
Here (1) follows from the fact that the cosimplicial object $\mathscr{A}_S$ is termwise dualizable in $\DA_{\et}(\G_{m,S}, \Lambda)$ (see \ref{noname:Abullet}). (2) is a direct consequence of $M$ being $f$-ULA and the equivalent condition (3) in Proposition \ref{thm:EquivCharOfULA}. Finally (3) follows from the unitality of $\Upsilon_{\id}$ (see \cite[10.2]{AyoubRealizationEtale}). It is not hard to check that the bottom to top composition of the equivalences above is equivalent to the desired map.

Recall that $M$ being $f-$ULA implies that $t_n^*M$ is $f_n$-ULA. Hence we have by the observations above that
\begin{align*}
\Psi^{\text{tame}}_f(j^*M) \simeq \colim_{n \in (\N'^\times)^{\op}} \Upsilon_{f_n} (t_n^* j^*M)) \simeq  \colim_{n \in (\N'^\times)^{\op}} i^* t_n^* M \simeq i^*M.
\end{align*}
Similarly we have that $t_L^* M$ is $f_L$-ULA and therefore $t_L^* M$ is $t_L \circ f_L$-ULA by Lemma \ref{lem:ULAandProperSmooth}. Thus we get
\[
\Psi_f(j^*M) \simeq \colim_{L \in \Xi_\tau} \Psi^{\text{tame}}_{t_L \circ f_L}(t_L^* j^*M) \simeq  \colim_{L \in \Xi_\tau}  i^* t_L^* M \simeq i^*M. 
\]

\end{proof}

\noname From what we established so far we can deduce the following interesting criterion on extension of universal local acyclicity: 

\begin{cor} \label{cor:ULAisGoodReduction}
Let $f: X \rightarrow S$ be of finite type where $S$ is the spectrum of an excellent strictly henselian discrete valuation ring and $\Lambda$ a $\mathbb{Q}$-algebra. For an $M$ in $\DA_{\et}^{\cons}(X_\eta, \Lambda)$ the following are equivalent:
\begin{enumerate}
\item There exists an $\widetilde{M}$ in $\DA_{\et} (X, \Lambda)$ which is ULA with respect to $f$ and such that $j^*\widetilde{M} \simeq M$.
\item The canonical map $\Upsilon_f (M) \rightarrow \Psi_f(M)$ is an equivalence and the monodromy operator $N$ in the exact sequence
\begin{equation} \label{eqn:MonSeqForExtensionThm}
\chi_{f} (M) \longrightarrow \Upsilon_{f}(M) \overset{N}\longrightarrow \Upsilon_{f}(M)(-1) 
\end{equation}
is equivalent to the zero map.
\end{enumerate}
\end{cor}

\begin{proof}
Let us assume that (1) is satisfied. Then it follows from Proposition \ref{prop:ULAImpliesCanIsEquiv} that the composition
\[
\can_{\widetilde{M}}: i^*\widetilde{M} \rightarrow i^* j_* j^* \widetilde{M} \simeq i^* j_* M \rightarrow \Upsilon_{f}(M) \rightarrow \Psi_{f}(M)
\]
is an equivalence. Hence the canonical map $\Upsilon_{f}(M) \rightarrow \Psi_{f}(M)$ is an equivalence Lemma \ref{lem:AlphaEquImplBetaEqu} and as in the proof of Lemma \ref{lem:Beta_MequivImpliesBeta_DMequiv} we see that the sequence (\ref{eqn:MonSeqForExtensionThm}) splits which implies that $N \simeq 0$. 

Conversely assume that (2) holds. Then  $\Upsilon_{f}(M) \simeq \Psi_f (M)$ is a direct factor of $i^*j_*M$ and hence the inlusion of this direct factor gives a map
\begin{equation} \label{eqn:MapForRecollement}
\Psi_{f}(M) \rightarrow i^*j_* M.
\end{equation}
Considering $\DA_{\et}(X, \Lambda)$ as the recollement of $\DA_{\et}(X_\sigma, \Lambda)$ and $\DA_{\et}(X_\eta, \Lambda)$ (see Remark \ref{rem:BigRemBeforeBigProp} (5)) the map (\ref{eqn:MapForRecollement}) gives rise to an element $\widetilde{M}$ in $\DA_{\et}^{\cons}(X, \Lambda)$ such that the composition
\[
i^* \widetilde{M} \rightarrow i^*j_*j^* \widetilde{M} \rightarrow \Psi_{f}(j^* \widetilde{M}) \simeq  \Psi_{f}(M)
\]
is an equivalence. Hence $\widetilde{M}$ is $f-$ULA by Proposition \ref{prop:InQsetupCanIsoImplULA}.
\end{proof}

\begin{remark}
The monodromy operator $N$ can be interpreted as the "logarithm of the monodromy action" (see \cite[11.17]{AyoubRealizationEtale} for a precise statement after \'etale realization). With this interpretation in mind Corollary \ref{cor:ULAisGoodReduction} says that $M$ in $\DA_{\et}^{\cons}(X_\eta, \Lambda)$ can be extended to a $\widetilde{M}$ which is ULA over $f$ if and only if the monodromy action on $\Upsilon_f(M) \simeq \Psi_f(M)$ is trivial. In other words "good reduction" of the motive is determined by the monodromy of $\Psi_f(M)$. 
\end{remark} 

\begin{lem} \label{lem:SuffToCheckULAAfterTensoringWQAndZ/p}
Let $g: Y \rightarrow T$ be of finite type where $T$ is a finite dimensional noetherian scheme and $M$ in $\DA^{\cons}_{\et}(Y, \Lambda)$. Assume that $\Lambda$ is flat over $\Z$. Then $M$ is $g$-ULA if and only if $\rho_{\Q}^* M$ and  $\rho_{\Z/\ell \Z}^* M$ are $g$-ULA, where $\ell$ runs through the set of all prime numbers. (Here $\rho_{\Q}^* $ and $\rho_{\Z/\ell \Z}^* $ are defined as in \ref{noname:rhoQrholambda}.)
\end{lem}

\begin{proof}
Since $\rho_{\Q}^*$ and $\rho_{\Z/\ell \Z}^*$ commute with the six functors by Proposition \ref{prop:ChangeOfCoeffCompWSixFunctors}, the claim follows from Proposition \ref{prop:QandZ/plinearizationIsConsFam} and the characterisation (2) of universal local acyclicity in Proposition \ref{prop:equivcharofULA}.
\end{proof}

\begin{thm} \label{thm:ULAinDAdetectedByNearby}
Let $f: X \rightarrow S$ be of finite type where $S$ is the spectrum of an excellent strictly henselian discrete valuation ring. Let $\Lambda$ be a noetherian ring flat over $\Z$ and $M$ in $\DA_{\et}^{\cons}(X, \Lambda)$. Then $M$ is ULA with respect to $f$ if and only if the canonical map
\[
\can_M: i^*M \longrightarrow \Psi_f(j^*M)
\]
is an equivalence. 
\end{thm}

\begin{proof} The "only if" part follows from Proposition \ref{prop:ULAImpliesCanIsEquiv}. For the "if" part it suffices by Lemma \ref{lem:SuffToCheckULAAfterTensoringWQAndZ/p} to check the cases where $\Lambda$ is a $\mathbb{Q}$-algebra and where $\Lambda$ is a $\mathbb{Z}/\ell \mathbb{Z}$-algebra for all primes $\ell$. If $\Lambda$ is a $\mathbb{Q}$-algebra this is Proposition \ref{prop:InQsetupCanIsoImplULA} since $j^*M$ is $f_\eta$-ULA by Lemma \ref{lem:constrOverFieldIsULA}. Let $p$ denote the residue characteristic of $S$. If $\Lambda$ is a $\mathbb{Z}/ \ell \mathbb{Z}$-algebra, where $\ell=p$, then $M$ is trivially $f$-ULA: Indeed let $k: X[1/p] \rightarrow X$ denote the open immersion. Then by Proposition \ref{prop:pOnBaseAutomaticallyInverted} the functor $k^*: \DA_{\et}(X, \Lambda) \rightarrow \DA_{\et}(X[1/p], \Lambda)$ is an equivalence and thus $M$ is ULA with respect to $f$ if and only if $k^*M$ is ULA with respect to $f[1/p]: X[1/p] \rightarrow S[1/p]$. Since $S[1/p]$ is either $\eta$ or the empty scheme, any $M$ in $\DA^{\cons}_{\et}(X[1/p], \Lambda)$ is ULA with respect to $f[1/p]$ either by Lemma \ref{lem:constrOverFieldIsULA} or trivially. If $\Lambda$ is a $\mathbb{Z}/ \ell \mathbb{Z}$-algebra, where $\ell \neq p$, then the mod-$\ell$ realization functors
\[
\mathfrak{R}_{\modulo \ell}: \DA_{\et}(X_\eta, \Lambda) \overset{\sim}\longrightarrow \mathcal{D}_{\et}(X_\eta, \Lambda)
\]
and 
\[
\mathfrak{R}_{\modulo \ell}: \DA_{\et}(X_\sigma, \Lambda) \overset{\sim}\longrightarrow \mathcal{D}_{\et}(X_\sigma, \Lambda)
\]
are equivalences of categories by rigidity (see Theorem \ref{thm:Rigidity}). By \cite[10.16]{AyoubRealizationEtale} these equivalences are compatible with the formation of nearby cycles functors in the sense that 
\[
\begin{tikzcd}
{\DA_{\et}(X_\eta, \Lambda)} \arrow[r, "\mathfrak{R}_{\modulo \ell}"] \arrow[d, "\Psi_f"'] & {\mathcal{D}_{\et}(X_\eta, \Lambda)} \arrow[d, "\Psi_f"] \\
{\DA_{\et}(X_\sigma, \Lambda)} \arrow[r, "\mathfrak{R}_{\modulo \ell}"]                    & {\mathcal{D}_{\et}(X_\sigma, \Lambda)}                  
\end{tikzcd}
\]
commutes, where the $\Psi_f$ on the right side denotes the classical nearby cycles functor for \'etale sheaves. Since $\mathfrak{R}_{\modulo \ell}$ is a symmetric monoidal equivalence, $(X,M)$ is dualizable in $\C_{S, \DA_{\et}}$ if and only if $(X,\mathfrak{R}_{\modulo \ell}(M))$ is dualizable in $\C_{S, \mathcal{D}_{\et}}$. By \cite[5.5.4]{CisinskiDegliseEtale} $\mathfrak{R}_{\modulo \ell}$ restricts to a fully faithfull embedding
\[
\DA_{\et}^{\cons}(X, \Lambda) \longrightarrow \mathcal{D}^{\text{cft}}_{\et}(X, \Lambda).
\]   
Finally it is proven in \cite[2.16]{lu_zheng_2022} that $\mathfrak{R}_{\modulo \ell} M \in \mathcal{D}^{\text{cft}}_{\et}(X, \Lambda)$ is ULA with respect to $f$ if and only if $\mathfrak{R}_{\modulo \ell}(\can_M)$ is an equivalence. 
\end{proof}

\begin{remark}
Theorem \ref{thm:ULAinDAdetectedByNearby} tells us in particular that we can detect universal local acyclicity over an excellent regular 1-dimensional scheme $T$ using the motivic nearby cycles functor. 
Indeed let $\Lambda$ be a noetherian ring flat over $\Z$, $g: Y \rightarrow T$ a morphism of finite type and $M$ in $\DA_{\et}(Y, \Lambda)$. Then it follows from \cite[4.3.9]{CisinskiDegliseBook} and the characterization of universal local acyclicity in Proposition \ref{prop:equivcharofULA} that $f$-universal local acyclicity of $M$ can be checked after pulling back to strict localizations of $T$. Then we are precisely in the right situation to apply Theorem \ref{thm:ULAinDAdetectedByNearby}.
\end{remark}

\section{Application: The weak singular support of an \'etale motive}
\noname Throughout this section let us fix a field $K$. By a smooth $K$-scheme we always mean a scheme equipped with a smooth morphism of finite type to $\Spec (K)$. Let us recall the definition of weak singular support after Beilinson (see \cite{BeilinsonHolonomic}). 

\noname Let $X$ be a smooth $K$-scheme and denote by $\mathbb{T}^*X$ its cotangent bundle. A morphism $f: X \rightarrow Y$ between smooth $K$-schemes induces a map of vector bundles 
\[d f:\mathbb{T}^*Y \times_Y X \longrightarrow \mathbb{T}^*X\]
 over $X$. 
A subset $C \subset \mathbb{T}^*X$ is called \textit{conical} if it is closed under the canonical $\G_{m,K}$-action on $\mathbb{T}^*X$.

\begin{definition} Let $X$ be a smooth $K$-scheme and $C \subset \mathbb{T}^*(X/K)$ a closed conical subset.
\begin{enumerate}
\item A morphism $h: U \rightarrow X$ between smooth $K$-schemes is called $C$-transversal at a geometric point $u \rightarrow U$ if 
\[
\ker(d h_u) \cap C_{h(u)} \setminus \{0\} = \emptyset.
\]
We say $h$ is $C$-transversal if it is $C$-transversal at all geometric points of $U$.
\item A morphism $f: X \rightarrow Y$ of smooth $K$-schemes is called $C$-transversal at a geometric point $x \rightarrow X$ if 
\[
(d f_x )^{-1} (C_x) \setminus \{0\} = \emptyset.
\]
We say $f$ is $C$-transversal if it is $C$-transversal at all geometric points of $X$.
\end{enumerate}
\end{definition}

\noname A \textit{test pair} $(h,f)$ is a correspondence $X \overset{h}\leftarrow U \overset{f}\rightarrow Y$ between smooth $K$-schemes. A \textit{weak test pair} is a test pair where $f$ is of the form $f: X \rightarrow \A^1_K$ and $h$ is either
\begin{enumerate}
\item an open immersion if $K$ is infinite or
\item the composition $U = V \times_{K} K' \rightarrow V \overset{h'}\rightarrow X$, where $K'$ is a finite extension of $K$ and $h'$ an open immersion, if $K$ is finite.
\end{enumerate}

Let $C \subset \mathbb{T}^*(X/K)$ be a closed conical subset and consider a weak test pair $(h,f)$. Clearly since such an $h$ is \'etale it is $C$-transversal. We call the test pair $(h,f)$  $C$-transversal if $f$ is also $C$-transversal.

\begin{definition}
Let $X$ be a smooth $K$-scheme, $C$ a closed conical subset of $\mathbb{T}^*(X/K)$ and $M$ in $\DA_{\et}^{\cons}(X, \Lambda)$. We say $M$ is \textit{weakly micro-supported on $C$} if $h^* M$ is $f$-ULA for all $C$-transversal weak test pairs $(h,f)$. We call the smallest conical subset of $\mathbb{T}^*(X/K)$ on which $M$ is weakly micro-supported the \textit{weak singular support} of $M$ and denote it by $SS^w(M)$.
\end{definition}

\begin{remark}
For a $M$ in $\DA_{\et}(X, \Lambda)$ let us denote by $\C'(M)$ the set of all closed conical subset of $\mathbb{T}^*(X/K)$ on which $M$ is weakly micro-supported. The weak singular support is well defined: Indeed as Beilinson noted in \cite{BeilinsonHolonomic} $\C'(M)$ is closed under intersections.
\end{remark}

\noname Let us fix some notation: Let $f: U \rightarrow \A^1_K$ be a morphism of schemes and $s \rightarrow \A^1_K$ a geometric point. We write $(\A^1_K)_{(s)}$ for the strict henselisation of $\A^1_K$ in $s$ and consider the pullback square
\[
\begin{tikzcd}
U \arrow[d, "f"'] & U_{(s)} \arrow[d, "f_{(s)}"] \arrow[l] \\
\A_K^1                 & (\A^1_K)_{(s)} \arrow[l]                     
\end{tikzcd}
\]
obtained by pulling back $f$ along the canonical map $(\A^1_K)_{(s)} \rightarrow \A_K^1 $. For a $M$ in $\DA_{\et}(U, \Lambda)$ we write $M_{(s)}$ for its restriction along $U_{(s)} \rightarrow U$. $(\A^1_K)_{(s)}$ is either the spectrum of a field (in the case where $s$ maps to the generic point) or the spectrum of a strictly henselian discrete valution ring. Let us fix a geometric point $0$ which maps to the closed point of $\A_K^1 = \Spec K[X]$ given by the ideal $(X)$. For a geometric point $s$ not mapping to the generic point consider the decomposition
\[
\begin{tikzcd}
U_s \arrow[d, "f_s"'] \arrow[r, "i"] & U_{(s)} \arrow[d, "f_{(s)}"] & U_{\eta_s} \arrow[d, "f_{\eta_s}"] \arrow[l, "j"'] \\
{s} \arrow[r, "i"']              & \A^1_{(s)}                   & \eta_s, \arrow[l, "j"]                             
\end{tikzcd}
\]
where $\eta_s$ denotes the generic point of $\A^1_{(s)}$. Then we can form the nearby cycles functor $\Psi_{f_{(s)}}(\_)$ and get the canonical natural transformation
\[
\alpha: i^* \_ \longrightarrow \Psi_{f_{(s)}}(j^* \_).
\]
We let $\Phi_{f_{(s)}}(\_)$ denote the cofiber of $\alpha$. $\Phi_{f_{(s)}}(\_)$ is sometimes called the \textit{vanishing cycles functor}.

\begin{prop} \label{prop:SSviaPsi}
Assume that $K$ is perfect, let $X$ be a smooth $K$-scheme, $M$ in $\DA^{\cons}_{\et}(X, \Lambda)$ and $C$ a closed conical subset of $\mathbb{T}^*(X/K)$. Moreover assume that $\Lambda$ is a noetherian ring flat over $\Z$. The following are equivalent:
\begin{enumerate}
\item $SS^w(M) \subset C$.
\item For all $C$-transversal weak test pairs $(h,f)$ the motive $(h^*M)_{(0)}$ is ULA with respect to $f_{(0)}: U_{(0)} \rightarrow (\A^1_K)_{(0)}$.
\item For all $C$-transversal weak test pairs $(h,f)$ we have $\Phi_{f_{(0)}}((h^*M)_{(0)}) \simeq 0$.
\end{enumerate}
\end{prop}

\begin{proof}
(2) and (3) are equivalent by Theorem \ref{thm:ULAinDAdetectedByNearby}. Clearly (1) implies (2) since universal local acyclicity is preserved under the basechange $(\A^1_K)_{(s)} \rightarrow \A^1_K$ by Proposition \ref{prop:propertiesOfULA}(1). 

Let us assume that (2) holds and let $(h,f)$ be a $C$-transversal weak test pair. By \cite[4.3.9]{CisinskiDegliseBook} and the characterization of universal local acyclicity in Proposition \ref{prop:equivcharofULA} $h^*M$ is ULA with respect to $f$ if and only if $(h^*M)_{(s)}$ is ULA with respect to $f_{(s)}$ for all geometric points $s$ of $\A^1_K$. 

By Lemma \ref{lem:constrOverFieldIsULA} $(h^*M)_{(s)}$ is ULA with respect to $f_{(s)}$ if $s$ maps to the generic point of $\A^1_K$. Hence we may assume that $s$ maps to a closed point. Since we may check universal local acyclicity \'etale locally on $\A^1_K$ and $K$ is perfect we may assume that the image of $s$ is given by the maximal ideal in $K[X]$ generated by $x-a$ for some $a \in K$. Let $-s: \A^1_K \rightarrow \A^1_K$ denote the isomorphism induced by the $K$-algebra map $K[X] \rightarrow K[X]$ mapping $x$ to $x-a$ and let us write $f-s := -s \circ f.$ The diagram
\[
\begin{tikzcd}
U \arrow[d, "f"']       & U_{(s)} \arrow[d, "f_{(s)}"] \arrow[l]     \\
\A^1_K \arrow[d, "-s"'] & (\A^1_K)_{(s)} \arrow[d, "\sim"] \arrow[l] \\
\A^1_K                  & (\A^1_K)_{(0)} \arrow[l]                  
\end{tikzcd}
\]
consists of pullback squares. Hence we see that $(h^*M)_{(s)}$ is $f_{(s)}$-ULA if and only if $(h^*M)_{(s)}$ is $(f-s)_{(0)}$-ULA. Thus we are left to show that $f-s$ is $C$-transversal. But this is clear since $d f = d (f-s)$. 
\end{proof}

\begin{remark}
Proposition \ref{prop:SSviaPsi} can be seen as the motivic analogue of \cite[8.6.4]{KashiwaraShapira}, where the complex analytic case was considered. 
\end{remark}

\appendix

\chapter{Dualizable objects in a bicategory} \label{app:A}
\numberwithin{thm}{chapter}
\noname Throughout this section let $(\C, \otimes, \1)$ be a symmetric monidal bicategory (see \cite{BenabouBicat},\cite[Chapter 2]{Schommer-PriesPhD}). Whenever we talk about a commutative diagram in a bicategory we implicitly mean the existence of an invertible 2-cell. Given two objects $X,Y$ in $\C$ we denote the (1-)category of maps between $X$ and $Y$ by $\map_{\C}(X,Y)$.

\begin{definition} \label{def:StronglyDualizable}
An object $X$ of $\C$ is called \textit{dualizable} if there exists an object $\widehat{X}$ in $\C$ together with maps
\[
\eta_X: \1 \rightarrow X \otimes \widehat{X}
\]
and
\[
\varepsilon_X: \widehat{X}\otimes X \rightarrow  \1
\]
in $\C$ such that the compositions
\[
X \overset{\eta_X \otimes X}\longrightarrow X \otimes \widehat{X} \otimes X \overset{X \otimes \varepsilon_X}\longrightarrow X
\]
and 
\[
\widehat{X} \overset{\widehat{X} \otimes \eta_X}\longrightarrow \widehat{X} \otimes X  \otimes \widehat{X} \overset{\varepsilon_X \otimes \widehat{X}}\longrightarrow \widehat{X}
\]
are equivalent to the identity. We call $\widehat{X}$ a \textit{dual of $X$} and the pair $\eta_X$ and $\varepsilon_X$ a \textit{duality datum} of $X.$
\end{definition}

\begin{lem} \label{lem:StongDualGivesAdj}
Let $X$ a dualizable object in $\C$ with dual $\widehat{X}$. Then:
\begin{enumerate}
\item $\widehat{X}$ is dualizable with dual $X$.
\item  The functor $X \otimes \_$ admits a right adjoint, namely  $\widehat{X} \otimes \_$. In particular $\widehat{X}$ is uniquely determined up to equivalence.
\end{enumerate}
\end{lem}

\begin{proof} A duality datum of $X$ with dual $\widehat{X}$ is also a duality datum for $\widehat{X}$ with dual $X$, which implies (1). For (2) note that a duality datum  gives rise to unit and counit maps for the desired adjunction which satisfy the triangle identities by definition.
\end{proof}

\noname We say that a symmetric monoidal bicategory is \textit{closed} if for any $X$ in $\C$ the functor
\[
X \otimes \_ : \C \longrightarrow \C
\]
admits a right adjoint. We denote this right adjoint by $\Hom(X, \_).$

\begin{lem} \label{lem:StrongDualAndIntHom}
Let $\C$ be a closed symmetric monidal bicategory and $X$ a dualizable object in $\C$. Then $\widehat{X} \simeq \Hom(X, \1)$ and moreover the following are equivalent:
\begin{enumerate}
\item $X$ is dualizable.
\item For any $Y$ in $\C$ the canonical morphism
\[
\widehat{X} \otimes Y \longrightarrow \Hom(X,Y)
\]
transpose to 
\[
\varepsilon_X \otimes \id_Y: X \otimes \widehat{X} \otimes Y \longrightarrow Y
\]
is an equivalence. 
\item The canonical morphism
\[
\widehat{X} \otimes X \longrightarrow \Hom(X,X)
\]
transpose to 
\[
\varepsilon_X \times \id_X: X \otimes \widehat{X} \otimes X \longrightarrow X
\]
is an equivalence. 
\end{enumerate}
\end{lem}
\begin{proof}
Lemma \ref{lem:StongDualGivesAdj} implies that $\widehat{X} \simeq \widehat{X} \otimes  \1 \simeq \Hom(X, \1)$. The implication (1) $\Rightarrow$ (2) follows from Yoneda using that for any $Z$ in $\C$
the diagram
\[
\begin{tikzcd}
{\map_{\C}(Z, \widehat{X} \otimes Y)} \arrow[rd, "\sim"', no head] \arrow[rr] &                             & {\map_{\C}(Z, \Hom(X,Y))} \arrow[ld, "\sim", no head] \\
                                                                                      & {\map_{\C}(X \otimes Z,Y))} &                                                      
\end{tikzcd}
\]
commutes by Lemma \ref{lem:StongDualGivesAdj} functorially in $Z$. (2) $\Rightarrow$ (3) is clear and (3) $\Rightarrow$ (1) is shown in \cite[1.4]{lu_zheng_2022}.  
\end{proof}

\begin{noname} \label{noname:EtaAndEpsilonIndBySymmMonFunctor} Let $F:\C \rightarrow \mathcal{D}$ be a strict monoidal functor between symmetric monoidal bicategories and consider two objects $X,Y$ of $\C$ together with morphisms 
$\eta: \1 \rightarrow X \otimes Y$ and $\varepsilon: Y \otimes X \rightarrow \1$. Then since $F$ is assumed to be strictly monoidal there are unique (up to equivalence) arrows $\eta_F$ and $\varepsilon_F$ making the diagrams 
\[
\begin{tikzcd}
F(\1) \arrow[r, "F(\eta)"] \arrow[d, "\simeq"'] & F(X \otimes Y) \arrow[d, "\simeq"] &  & F(Y \otimes X) \arrow[r, "F(\varepsilon)"] \arrow[d, "\simeq"'] & F(\1) \arrow[d, "\simeq"] \\
\1 \arrow[r, "\eta_F"']              & F(X) \otimes F(Y)        &  & F(Y) \otimes F(X) \arrow[r, "\varepsilon_F"]         & \1             
\end{tikzcd}
\]
commute. Here the vertical arrows are the canonical equivalences giving $F$ the structure of a strict monoidal functor. 
\end{noname}

\begin{prop} \label{prop:UnitCounitSymmMonFunctor}
In the situation of \ref{noname:EtaAndEpsilonIndBySymmMonFunctor} the morphisms $\eta_F$ and $\varepsilon_F$ make the diagram
\[
\begin{tikzcd}
F(X) \arrow[d, "\id"'] \arrow[r, "\simeq"] & F(\1 \otimes X) \arrow[r, "F(\eta \otimes X)"] \arrow[d, "\simeq"'] & F(X \otimes Y \otimes X) \arrow[r, "F(X \otimes \varepsilon)"]         & F(X \otimes \1) \arrow[d, "\simeq"] \arrow[r, "\simeq"] & F(X) \arrow[d, "\id"] \\
F(X) \arrow[r, "\simeq"]                   & \1 \otimes F(X) \arrow[r, "\eta_F \otimes F(X)"]                    & F(X) \otimes F(Y) \otimes F(X) \arrow[r, "F(X) \otimes \varepsilon_F"] & F(X) \otimes \1 \arrow[r, "\simeq"]                     & F(X)                 
\end{tikzcd}
\]
commute. Here the vertical maps are the canonical equivalences giving $F$ the structure of a strict symmetric monoidal functor. 
\end{prop}

\begin{proof}
This is an elementary observation using the axioms of a strict symmetric monoidal functor between bicategories. 
\end{proof}

\begin{cor} \label{cor:1.ImageOfDualizableObjectisDualizablw2}
Let $F:\C \rightarrow D$ be a strict monoidal functor between symmetric monoidal bicategories. If an object $X$ in $\C$ is dualizable with duality datum ($\eta, \varepsilon$), then $F(X)$ is dualizable with duality datum $(\eta_F, \varepsilon_F)$.
\end{cor}

\noname We say that a family $F_i: \C \rightarrow \mathcal{D}_i$, $i \in I$ of functors between bicategories is \textit{jointly detecting identities} if the following holds: A morphism $g: X \rightarrow X$ in $\C$ is equivalent to the identity if and only if $F_i(g)$ is equivalent to the identity for all $i \in I$. 

\begin{cor} \label{cor:2.ConservativeFamilyOfFunctorsDetectsDualizability}
Let $F_i: \C \rightarrow \mathcal{D}_i$, $i \in I$ be a family of strict monoidal functors between symmetric monoidal bicategories which is jointly detecting identities. Then an object $X$ in $\C$ is dualizable with duality data $\eta$ and $\varepsilon$ if and only if for all $i \in I $ $F_i(X)$ is dualizable with duality datum $\eta_{F_i}$ and $\varepsilon_{F_i}$. 
\end{cor}

\begin{lem} \label{lem:IsoOfDualObjectsAndData}
Let $\C$ be a symmetric monoidal bicategory and suppose there are equivalences $
\can:X \overset{\sim}\rightarrow Y$ and  $\can': X' \overset{\sim}\rightarrow Y'$ and commutative diagrams
\[
\begin{tikzcd}
\1 \arrow[r, "\eta_X"] \arrow[d, "\id"'] & X \otimes X' \arrow[d, "\can \otimes \can'"] &  & X' \otimes X \arrow[r, "\varepsilon_X"] \arrow[d, "\can' \otimes \can"'] & \1 \arrow[d, "\id"] \\
\1  \arrow[r, "\eta_Y"]                  & Y \otimes Y'                                 &  & Y' \otimes Y \arrow[r, "\varepsilon_Y"].                                  & \1                 
\end{tikzcd}
\]
Then the pair $(\eta_X, \varepsilon_X)$ is a duality datum for $X$ if and only if $(\eta_Y, \varepsilon_Y)$ is a duality datum for $Y$.
\end{lem}

\begin{proof}
Consider the commutative diagram 
\[
\begin{tikzcd}
X \arrow[d, "\can"'] \arrow[r, "\eta_X \otimes X"] & X \otimes X' \otimes X \arrow[r, "X \otimes \varepsilon_X"] \arrow[d, "\can \otimes \can' \otimes \can"] & X \arrow[d, "\can"] \\
Y \arrow[r, "\eta_Y \otimes Y"']                   & Y \otimes Y' \otimes Y \arrow[r, "Y \otimes \varepsilon_Y"']                                                  & Y.                  
\end{tikzcd}
\]
Clearly the top horizontal composition is equivalent to the identity if and only if the bottom horizontal composition is. Similarly the composition $(\varepsilon_X \otimes X')\circ(X' \otimes \eta_X)$ is equivalent to the identity if and only if $(\varepsilon_Y \otimes Y')\circ(Y' \otimes \eta_Y)$ is equivalent to the identity.
\end{proof}

\begin{lem} \label{lem:NewRigidDualDataForDualizableObject}
Let $\C$ be a closed symmetric monoidal bicategory and $X$ a dualizable object in $\C$. Then 
\[
\varepsilon_X: \Hom(X, \1) \otimes X \overset{\id^t}\longrightarrow \1
\]
and 
\[
\eta_X: \1  \overset{(\varepsilon_X)^t}\longrightarrow \Hom ( \Hom(X, \1) \otimes X, \1 ) \simeq X \otimes \Hom(X, \1)
\]
is a duality datum for $X$. Here in both cases the subscript $(\_)^t$ denotes the transpose with respect to the $\_ \otimes X \dashv \Hom(X, \_)$ adjunction.
\end{lem}

\begin{proof}
Let $\tilde{\eta}_X: \1 \rightarrow X \otimes \Hom(X, \1)$ and $\tilde{\varepsilon}_X: \Hom(X, \1) \otimes X \rightarrow \1$ be a duality datum. Then we claim that $\tilde{\eta}_X \simeq {\eta}_X$ and $\tilde{\varepsilon}_X \simeq {\varepsilon}_X$ from which the claim follows by Lemma \ref{lem:IsoOfDualObjectsAndData}. Indeed clearly both $\tilde{\varepsilon}_X$ and ${\varepsilon}_X$ are transpose to $\id: \Hom(X,\1) \rightarrow \Hom(X, \1)$ and thus equivalent. By construction $\eta_X$ corresponds to $\varepsilon_X$ via
\[
\map_\C(\1, X \otimes \Hom(X, \1)) \simeq \map_\C(\1, \Hom(\Hom(X, \1) \otimes X, \1 )) \simeq \map_\C(\Hom(X, \1)\otimes X, \1).
\]
Moreover $\varepsilon_X$ corresponds to the identity via
\[
 \map_\C(\Hom(X, \1)\otimes X, \1) \simeq \map_\C( X, \Hom(\Hom(X, \1),\1)) \simeq \map_\C(X,X).
\]
this implies that $\eta_X$ corresponds to the identity via 
\[
\map_\C(\1, X \otimes \Hom(X, \1)) \simeq \map_\C(X,X)
\]
and hence $\tilde{\eta}_X \simeq {\eta}_X$. 
\end{proof}

\let\appendixnameAlt\appendixname
\renewcommand{\appendixname}{}
\bibliography{references}
\bibliographystyle{alpha}
\renewcommand{\appendixname}{\appendixnameAlt}


\end{document}